\documentclass{article}
\usepackage[utf8]{inputenc}
\usepackage{cite}
\usepackage{amsfonts,amssymb}
\usepackage{amsmath,amscd,amsthm}
\usepackage{graphicx}

\usepackage{enumerate}
\usepackage{caption}
\usepackage{subcaption}
\usepackage{float}
\usepackage{tikz}
\usetikzlibrary{matrix,arrows}
\usetikzlibrary{decorations.pathreplacing}

\usepackage{chngcntr}

\newcommand{\R}{\mathbb R}
\newcommand{\Z}{\mathbb Z}
\newcommand{\N}{\mathbb N}

\newcommand{\diam}{\mathrm{diam}}

\DeclareMathOperator{\Lip}{Lip}
\DeclareMathOperator*{\argmax}{arg\,max}

\DeclareMathOperator{\Int}{Int}

\DeclareMathOperator{\conv}{conv}
\DeclareMathOperator{\codim}{codim}

\DeclareMathOperator{\ad}{ad}

\newcommand{\eqdef}{\overset{\mathrm{def}}{=\joinrel=}}

\newcommand{\ddt}{\frac{\displaystyle d}{\displaystyle dt}}
\newcommand{\dds}{\frac{\displaystyle d}{\displaystyle ds}}

\newtheorem{thm}{Theorem}
\newtheorem{lemma}{Lemma}
\newtheorem{prop}{Proposition}
\newtheorem{corollary}{Corollary}

\theoremstyle{definition}
\newtheorem{defn}{Definition}

\theoremstyle{remark}
\newtheorem{remark}{Remark}

\def\cal{\mathcal}
\newcommand{\texorpdfstring}[2]{#1}

\DeclareMathOperator{\bellman}{\cal B}
\DeclareMathOperator{\blowing}{\mathbf{B}}
\DeclareMathOperator{\blowingBig}{\mathbf{B}}

\begin{document}

\begin{centering}
	\LARGE Typicalness of chaotic fractal behaviour of integral vortexes in
	Hamiltonian systems with discontinuous right hand side\footnote{The work 
	was supported by the RFBR (grant N 14-01-00784)}.
\end{centering}

\bigskip

\begin{flushright}
	{\bf M.I. Zelikin}
	
	Lomonosov Moscow State University, Department of Mechanics and Mathematics
	
	\textit{mzelikin@mtu-net.ru}\\[0.5cm]
\end{flushright}

\begin{flushright}
	{\bf L.V. Lokutsievskiy}

	Lomonosov Moscow State University, Department of Mechanics and Mathematics

	\textit{lion.lokut@gmail.com}\\[0.5cm]
\end{flushright}

\begin{flushright}
	{\bf R. Hildebrand}

	Weierstrass Institute for Applied Analysis and Stochastics

	\textit{hildebra@wias-berlin.de}\\[0.5cm]
\end{flushright}

We consider a linear-quadratic deterministic optimal control problem where the control takes values in a two-dimensional simplex. The phase portrait of the optimal synthesis contains second-order singular extremals and exhibits modes of infinite accumulations of switchings in finite time, so-called chattering. We prove the presence of an entirely new phenomenon, namely the chaotic behaviour of bounded pieces of optimal trajectories. We find the hyperbolic domains in the neighbourhood of a homoclinic point and estimate the corresponding contraction-extension coefficients. This gives us the possibility to calculate the entropy and the Hausdorff dimension of the non-wandering set which appears to have a Cantor-like structure as in Smale's Horseshoe. The dynamics of the system is described by a topological Markov chain. In the second part it is shown that this behaviour is generic for piece-wise smooth Hamiltonian systems in the vicinity of a junction of three discontinuity hyper-surface strata. 

\tableofcontents

\section{Introduction}

The main tool for solving deterministic optimal control problems is the Pontryagin Maximum Principle (PMP) \cite{PBGM62}. It allows one to reduce the control problem to a two-point boundary value problem of Hamiltonian dynamics. Suppose the control variable $u$ takes values in some set $\Omega$. Then the Hamiltonian function $H$ defining the dynamics is given as the maximum $H(t,x,p) = \max_{u \in \Omega}{\cal H}(t,x,p,u)$
over the control $u$ of the {\it Pontryagin function} ${\cal H}$, and the optimal control $\hat u(t,x,p)$, if it exists, is found among the maximizers.

In general, the maximizer is unique in an open dense subset of the space of variables $t,x,p$ and depends smoothly on these variables. On this set, the Hamiltonian $H$ will be smooth, whereas on its boundary the derivatives of $H$ will experience discontinuities. Usually the Hamiltonian system as a whole will be piece-wise smooth on the extended phase space of the variables $t,x,p$. The extended phase space will be divided into disjoint domains $A_1,\dots,A_k$ on which the Hamiltonian is given by smooth functions $H_1,\dots,H_k$, respectively. The dynamics is described by an ordinary differential equation with discontinuous right-hand side. We consider situations when the set $\Omega$ is a convex polyhedron, and the domains $A_i$ are those regions where the optimal control resides in a particular vertex $v_i$ of the polyhedron. The set of points where the derivatives of $H$ are discontinuous is a stratified manifold, and on each stratum the optimal control is confined to a particular face of the polyhedron $\Omega$.

A trajectory of the Hamiltonian system evolving inside a smoothness domain is called {\it regular}. If a trajectory passes from one smoothness domain $A_i$ into another one $A_j$, then the corresponding optimal control will experience a jump from the vertex $v_i$ of the polyhedron $\Omega$ to the vertex $v_j$. This process is called {\it switching}, and the discontinuity hyper-surface is called {\it switching surface}. Typically optimal trajectories intersect the switching surface transversally, in which case they are called {\it bang-bang trajectories}. It may happen, however, that a trajectory  moves along the switching surface, in which case one speaks of a {\it singular trajectory}. Typically uniqueness of the solution does not hold in the vicinity of a singular trajectory, and many regular trajectories can join or leave from the same singular point. This is possible because the right-hand side of the underlying ordinary differential equation  experiences a discontinuity.

For a singular trajectory lying on a switching hyper-surface\footnote{Equivalently one can consider a problem with a one-dimensional control variable in the neighbourhood of the singular trajectory \cite{KelleyKoppMoyer}.} one can define an {\it order}, in dependence on up to which order the Poisson brackets of the adjoining smooth pieces $H_i$ of the Hamiltonian vanish. The order may be {\it local} or {\it intrinsic}, in dependence on whether the brackets vanish only on the trajectory itself or in a neighbourhood of it \cite{Lewis}. A good generalization of these notions (the so-called {\it natural order}) is given in \cite{Lokut_HSF}. If the natural order of the singular trajectory is even, then a regular trajectory cannot join it in a piece-wise smooth manner. In this case regular trajectories spiral around the singular trajectory and intersect the switching surface in an infinite number of points in finite time, such that the joining point is an accumulation point of switchings. This phenomenon is called {\it chattering}, and is well-studied for the situation where exactly two smoothness domains meet at the singular trajectory in question \cite{Marchal73}.

In this contribution we consider the situation where three
smoothness domains $A_1,A_2,A_3$ meet at a manifold ${\cal S}_{123}$ of codimension 2. This situation is equivalent to an optimal control problem with 2-dimensional control. The general case of singular extremals for $n$-dimensional control was explored in \cite{HLZOrderTorus}. The role of the order is played by a flag of orders. In the present contribution we consider a singular trajectory of intrinsic second order. This work can be seen as a continuation of the paper \cite{ZMHBasic}, where this case was first considered and the presence of the chattering phenomenon was proven. Here we observe an additional phenomenon appearing in this concrete problem, namely, the chaotic behaviour of bounded parts of optimal trajectories. This phenomenon was not yet seen in optimal control problems and is hence entirely new. The key to the proof is a new mathematical object: the system of ordinary differential equations for all Poisson brackets up to the fourth order between the restrictions $H_i$ of the Hamiltonian to the domains of continuity $A_i$ neighbouring the singular trajectory. We call this system the {\it descending system} of Poisson brackets.  

Our findings are not limited to optimal control problems, but rather hold for a whole class of piece-wise smooth Hamiltonian systems with three smoothness domains joining at surface ${\cal S}_{123}$ containing singular trajectories of second order. We will consider a particular {\it model} linear-quadratic optimal control problem in detail whose solution is prototypical for the general case. We call this problem the model problem for general piece-wise continuous Hamiltonian systems because the principal part of the descending system of ODEs for general Hamiltonian systems coincides with the descending system of ODEs for the model problem.

Finally, we state some conjectures and possible directions for further work.

The work consists of two parts. The first part is devoted to the complete design of the optimal synthesis for the model problem. The model problem is affine in the planar control, the latter takes values in an equilateral triangle centered at the origin. It is proved that the optimal synthesis contains an entirely new singularity. All optimal trajectories end at the fixed singular point in finite time, but nevertheless the control on some extremals has a chaotic nature: there exists a topological Markov chain
$\Sigma_\Gamma$ such that the sequence of control switchings on an optimal trajectory corresponds one to one to a trajectory through a point of $\Sigma_\Gamma$ under the right Bernoulli shift. The set  of non-wandering points in the model problem has a fractal structure (as in Smale's horseshoe) and non-integer Hausdorff and box dimensions.

In the second part a generalization of this phenomenon to Hamiltonian systems with discontinuous right-hand side is proved. The main role in demonstrating the generalization is played by a resolution of singularities for the Poincare mapping at the singular point, the so-called {\it strange point}. For this purpose we introduce an auxiliary descending system of ordinary differential equations on Poisson brackets of restrictions of the Hamiltonian to the domains of smoothness neighbouring the strange point. It turns out that the principal part of the descending system  coincides with the equations of the model problem that was investigated in the first part.

The discovered phenomenon appears in the neighbourhood of a generic singularity. Namely, in the second part we prove a theorem on the structural stability of the synthesis and determine the codimension of the manifold of singularities in question.

\section{Simplest properties of the model control problem}

In this section the model optimal control problem will be formulated and its basic properties described. The behaviour of the optimal trajectories of the problem is of chaotic nature. It is worth to note that the model problem possesses a rich group of symmetries, and hence its optimal synthesis can be studied to a sufficiently full extent. But at the same time it models the optimal synthesis for a broad class of Hamiltonian systems with discontinuous right-hand side. In the subsequent sections we show that the behaviour of optimal trajectories for systems in this class is similar to the optimal synthesis of the model problem.

\subsection{Formulation of the model problem}

Consider the optimal control problem

\begin{equation}
\label{problem:model}
	\begin{array}{l}
		J(x) = \frac{1}{2}\int\limits_0^{+\infty} \bigl<x(t),x(t)\bigr>\,dt\to \inf\\
		\ddot x = u;\qquad u\in\Omega\subset U;\\
	\end{array}
\end{equation}

\noindent with initial data

\begin{equation}
	x(0) = x_0,\ \ \dot x(0) = y_0.
\end{equation}

\noindent Here $x$ and $u$ belong to two-dimensional Euclidean space $U\simeq \R^2$ equipped with the scalar product  $\bigl<\cdot,\cdot\bigr>$, and $\Omega$ is a closed triangle with $0\in \Int \Omega$ (such triangles will be called admissible).

\begin{remark}
The deepest investigation of problem (\ref{problem:model}) will be conducted for the case when
$\Omega$ is an equilateral triangle centered at the origin. Some results related to the chaotic behaviour of the optimal trajectories will be obtained also for the more general case when $\Omega$ is not necessarily equilateral.
\end{remark}

Denote $y=\dot x$. Let $\phi$, $\psi$ be the adjoint variables to $x$ and $y$ as prescribed by the PMP. For brevity we shall sometimes write
$q=(x,y)$ and $p=(\phi,\psi)$. We employ the notation
 $M=U\oplus U=\{(x,y)\}$ for phase space, and the notation ${\cal M} = T^*M = \{(p,q)\}$ for extended phase space. The Pontryagin function has the form

\[{\mathcal H}= - \frac{\lambda_0}{2} \bigl<x,x\bigr> + \bigl<\phi,y\bigr> + \bigl<\psi,u\bigr>,\]

\noindent i.e., $\dot \phi = \lambda_0 x$ and
$\dot \psi=-\phi$. If $\lambda_0=0$, then the function $\psi$ is linear, and we immediately obtain a contradiction with the finiteness of integral (\ref{problem:model}). Setting $\lambda_0=1$
we obtain the following Hamiltonian system with  discontinuous right-hand side:

\begin{equation}
\label{eq:model_pmp_system}
	\left\{\begin{array}{l}
		\dot \psi = -\phi;\ \ \dot \phi = x; \ \dot x = y;\ \ \dot y = u.\\
		\bigl<\psi,u\bigr> \to \max\limits_{u\in\Omega}
	\end{array}\right.
\end{equation}

In problem (\ref{eq:model_pmp_system}) there exists a unique trajectory which is singular with respect to $\Omega$, namely  $x\equiv y\equiv \phi\equiv\psi\equiv u\equiv 0$. This trajectory is of second order (\cite{HLZOrderTorus}). In the second part of the paper we show that the behaviour of the optimal trajectories of problem (\ref{problem:model}) in the vicinity of this singular trajectory (i.e., of the origin) is typical for Hamiltonian systems with discontinuous right-hand side. This motivates us to investigate the optimal synthesis of problem (\ref{problem:model}) to a maximal extent.

\subsection{Symmetries of the problem}
\label{sec:symmetries}

The Hamiltonian system (\ref{eq:model_pmp_system}) has two important groups of symmetries. The first one is the scale group $\R\setminus 0$. Being one-dimensional, it allows to reduce the dimension of the phase space by one. The second group is the discrete permutation group $S_3$. It appears only in the case when $\Omega$ is an equilateral triangle and allows to prove the chaotic behaviour of the optimal trajectories.

\begin{enumerate}
	\item Let $\lambda\in\R\setminus 0$. Consider the following transformation $g(\lambda)\in GL(8,\R)$ of the extended phase space $\cal M$:
	
	\[
		x\mapsto \lambda^2 x,\ \ y\mapsto \lambda y,\ \ \phi\mapsto \lambda^3\phi,\ \ \psi\mapsto\lambda^4\psi.
	\]
	
	\noindent The transformation $g(\lambda)$ changes the Hamiltonian $\cal H$ by a factor of $\lambda^4$ and the symplectic form $\omega=dp\wedge dq=d\phi\wedge dx + d\psi\wedge dy$ by a factor of $\lambda^5$. Hence $g(\lambda)$
takes the set of trajectories of system (\ref{eq:model_pmp_system}) onto itself, but with a change in velocities by a factor of $\lambda$. If
$\lambda>0$, then optimal trajectories are mapped to optimal ones, because the functional $J$ changes by a factor of $\lambda^5$.

	\item If the triangle $\Omega$ is equilateral, then there exists a discrete subgroup $S_3\subset O(2,\R)$ of the orthogonal group of transformations of the plane $U$ which preserves the triangle $\Omega$. The simultaneous action of $S_3$ on the vectors $x$, $y$, $\phi$ and $\psi$ defines a transformation of the extended phase space that preserves the set of optimal trajectories, the Hamiltonian $H$ and the symplectic structure $\omega$.
	
\end{enumerate}


\subsection{A theorem on the passage to the adjoint variables}

One of the main tools for designing the optimal synthesis of problem (\ref{problem:model}) is the following theorem \ref{thm:model_problem_bellman} on the passage to the adjoint variables. This theorem allows to describe the structure of the manifold $M_+\subset{\cal M}=T^*M$ comprising those trajectories of system (\ref{eq:model_pmp_system}) which are optimal for problem (\ref{problem:model}).

Denote $q=(x,y)\in M=\R^4=U\oplus U$ and $p=(\phi,\psi)\in T^*_0M=\R^4$.

\begin{thm}
\label{thm:model_problem_bellman}

Let $\Omega$ be a convex compact set (not necessarily a triangle), and suppose $0\in\Int\Omega$. Then for any initial point $q_0=(x_0,y_0)$ the following holds:

 \begin{enumerate}
 \item There exists a unique optimal trajectory $\widehat q(t,q_0)$ of problem (\ref{problem:model}) with initial point $q_0$. The adjoint function $\widehat p(t,q_0)$ is also unique.

 \item There exists a time instant $T(q_0)\ge 0$ such that $\widehat q(t,q_0)=0$ and $\widehat p(t,q_0)=0$ for $t\ge T(q_0)$, and $q(t,q_0)\ne0$, $p(t,q_0)\ne0$ for $t<T(q_0)$. The function $T(\cdot)$ is continuous and
		
		\[
			C_1\max\{\sqrt{|x_0|},|y_0|\}\le T(q_0)\le C_2\max\{\sqrt{|x_0|},|y_0|\}
		\]
		
\noindent for some constants $C_1>0$ and $C_2>0$.
		
		\item The mapping $E:q_0 \mapsto \widehat p(0,q_0)=(\phi_0,\psi_0)$ is locally Lipschitz and bijective. Moreover,
		
		\[
			C_3\max\{\sqrt{|x_0|},|y_0|\}\le \max\{\sqrt[3]{|\phi_0|},\sqrt[4]{|\psi_0|}\}\le C_4\max\{\sqrt{|x_0|},|y_0|\}
		\]

		\noindent for some constants $C_3>0$ and $C_4>0$.
	\end{enumerate}
\end{thm}

The proof of the theorem consists of several lemmas.
\begin{lemma}
For any $q_0$ there exists a unique optimal solution
$\widehat q(t,q_0)$ of problem (\ref{problem:model}) with initial condition $q_0$.
\end{lemma}

\begin{proof}

Since for any measurable control $u(t)\in\Omega$ there exists a unique trajectory $q(t)$ with initial condition $q(0) = q_0$, we can consider $J$ as a functional on the space of controls $L_\infty(0,+\infty)$. The set of all admissible controls $\mathbb U$ is bounded, closed, and convex in
$L_\infty(0,+\infty)$ due to boundedness, closedness, and convexity of $\Omega$. Hence, due to the Banach-Alaoglu   theorem the set $\mathbb U$ is weakly $^*$ precompact
($L_\infty(0,+\infty) = L_1^*(0,+\infty)$). The weak $^*$ closedness of $\mathbb U$ follows trivially from the convexity and from the compactness of $\Omega$. Let us show that the functional $J$ is lower semi-continuous relative to the weak $^*$
topology.

Consider the operator $K:L_\infty(0,+\infty)\to AC[0,+\infty)$ given by
	
	\begin{equation}
	\label{eq:xt_by_x0_y0_u}
		x(t) = \bigl(Ku\bigr)(t) = x_0 + y_0t + \int_0^t (t-\tau) u(\tau)\,d\tau.
	\end{equation}
	
The operator $K$ takes the control $u(t)$ to the corresponding solution $x(t)$. Hence we have to prove that for any sequence of controls $u_n\xrightarrow{w^*} \widehat u$ 	
one has

	$$
		\liminf\limits_{n\to+\infty}\int_0^{+\infty} \bigl<\bigl(Ku_n\bigr)(t),\bigl(Ku_n\bigr)(t)\bigr>\,dt \ge
		\int_0^{+\infty} \bigl<\bigl(K\widehat u\bigr)(t),\bigl(K\widehat u\bigr)(t)\bigr>\,dt.
	$$

It is possible to rewrite (\ref{eq:xt_by_x0_y0_u}) in the form	 $\bigl(Ku\bigr)(t) = x_0 + y_0t + \int_0^{+\infty} (t-\tau)u(\tau)\theta(t-\tau)\,d\tau$,
	
where $\theta(\cdot)$ is the Heavyside function. Since
$\theta(t-\tau)\in L_1(0,+\infty)$, we have that $(Ku_n\bigr)(t)$ converges point-wise to $(K\widehat u\bigr)(t)$. The desired inequality follows from the Fatou theorem. The uniqueness of the optimal solution follows  immediately from the strict convexity of the functional $J$.
\end{proof}

Let us denote the Bellman function by $\bellman(q_0)$,

\begin{eqnarray*}
	\bellman(q_0) = \bellman(x_0,y_0)  = \inf \Bigl\{&&
	\frac{1}{2}\int_0^{+\infty} \bigl<x(t),x(t)\bigr>\,dt,\\
	&&\mbox{where } \ddot x \in\Omega\mbox{ and } x(0)=x_0, \dot x(0)=y_0 \Bigr\}.\\
\end{eqnarray*}

Obviously $0\le \bellman(q_0)<+\infty$ for any $q_0$.

\begin{lemma} \label{Bellman_convex}
	The Bellman function $\bellman$ is strictly convex.
\end{lemma}

\begin{proof}
	Let $x^1(t)$ and $x^2(t)$ be two admissible trajectories. For any $\lambda\in\R$ we have after some calculus
	
	$$
		J\bigl(\lambda x^1 + (1-\lambda)x^2\bigr) =
		\lambda J(x^1) + (1-\lambda) J(x^2) - \lambda(1-\lambda) J(x^2-x^1).
	$$

If $x^1(t)$ and $x^2(t)$ are optimal trajectories with initial conditions $q^1_0=(x^1_0,y^1_0)$ and $q^2_0=(x^2_0,y^2_0)$ respectively, then in view of the nonnegativity of $J$ we obtain for $0<\lambda<1$ that
	
	$$
		\lambda\bellman(q^1_0) + (1-\lambda)\bellman(q^2_0) \ge
		J\bigl(\lambda x^1 + (1-\lambda)x^2\bigr) \ge
		\bellman\bigl(\lambda q^1_0 + (1-\lambda)q^2_0\bigr).
	$$
Equality can be attained only if $J(x^2-x^1)=0$,
i.e., if $x^1\equiv x^2$.
	
\end{proof}

\begin{corollary}
The Bellman function $\bellman$ is continuous.
\end{corollary}

\begin{corollary}
	The mapping $q_0\mapsto \widehat u(\cdot,q_0)\in L_\infty(0,+\infty)$ taking the point $q_0$ to the optimal control $\widehat u(\cdot,q_0)$ on the trajectory through this point is  continuous in the weak $^*$ topology.
\end{corollary}

\begin{proof}
Consider a sequence $q_k$ which converges to the point $q_0$. Denote by $u_k(\cdot)=\widehat u(\cdot,q_k)$ the optimal control on the trajectory through the point $q_k$. Let us show that the sequence
$u_k(\cdot)$ converges in the weak $^*$ topology to the optimal control out of the point $q_0$. We may select from the sequence $u_k$ a weakly$^*$ converging subsequence $u_{k_m}$, because  the set $\mathbb U$ is weakly $^*$ compact. We then have $u_{k_m}\xrightarrow{w^*}u_0$. Let us prove that the trajectory
 $x_0(t) = {\cal K}(u_0(\cdot),q_0)(t)$ with control $u_0(\cdot)$ is optimal for the initial point $q_0$. Indeed, denote by $x_k(t)={\cal K}(u_k(\cdot),q_k)(t)$
the optimal trajectory for the point $q_k$. Then from the continuity of the function $\bellman$ and from the Fatou theorem we obtain

	\[
		\bellman(q_0) = \lim_{m\to\infty} \bellman(q_{k_m})
		=\lim_{m\to\infty} J(x_{k_m}) \ge J(x_0).
	\]

	\noindent Hence the trajectory $x_0(\cdot)$ is optimal. Therefore there exists a unique limit of the sequence $u_k$, namely the optimal control $u_0$ for the initial point $q_0$, i.e.,
the sequence $u_k$ converges in the weak $^*$ topology to $u_0$. 	
	
\end{proof}

\begin{corollary}
\label{corollary:map_q0_to_optimal_q_is_continuous}
	The mapping $q_0\mapsto \widehat q(\cdot,q_0)\in C[0,\tau]$ is continuous for any $\tau\ge 0$.
\end{corollary}

\begin{proof}

Let us show that the mapping  $q_0\mapsto \widehat x(\cdot,q_0)\in C[0,\tau]$ is continuous. Consider a converging sequence $q_k\to q_0$. Then we have $\widehat u(\cdot,q_k)\xrightarrow{w^*} \widehat u(\cdot,q_0)$. From formula (\ref{eq:xt_by_x0_y0_u}) it  follows immediately that
$\widehat x(t,q_k)\to \widehat x(t,q_0)$ for all $t$ and the convergence is uniform on $t\in[0,\tau]$. Similarly one obtains the proof of the convergence $\widehat y(\cdot,q_k)\to \widehat y(\cdot,q_0)$.
\end{proof}

\begin{lemma}
For any initial data $q_0$ there exists a time $T(q_0)$ such that
	 $\widehat q(t,q_0) = 0$ for all $t\ge T(q_0)$ and $T(q_0)\le C_2\max\{\sqrt{|x_0|},|y_0|\}$, where $C_2>0$ is some constant.
\end{lemma}

\begin{proof}

At first we shall prove that the optimal solutions
$\widehat q(t,q_0)$ tend to $0$ as $t\to+\infty$ due to the decrease of the function $\bellman$. Indeed, suppose the opposite: let $|\widehat q(t_k,q_0)|>\varepsilon$ for some
$\varepsilon>0$ and $t_k\to\infty$ as $k\to\infty$. Since the function $\bellman$ is continuous, we have $M=\min_{|q|=\varepsilon}\bellman(q)>0$. Therefore $\bellman(q)\ge M$ for all $|q|\ge\varepsilon$ due to convexity of the function
$\bellman$ (since $0$ is its absolute minimum). So,
$\bellman(\widehat q(t_k,q_0))\ge M>0$ for all $k$. This yields the claimed contradiction.

Now let us prove the assertion of the lemma itself. Denote

$$\rho(x,y) = \max\{\sqrt{|x|/a},|y|/b\},$$ 

\noindent where the constants
$a>0$ and $b>0$ will be chosen later. Let $\rho(x_0,y_0)=\delta$. Since the optimal trajectory tends to the origin, it
intersects the boundary of the ball $\rho(x,y)=\frac{1}{2}\delta$ in finite time. The remaining part of the proof will be devoted to the estimation of this time. Namely, our aim is to correlate the time of halving
$\rho(x,y)$ on the optimal trajectory and the value of the Bellman function.

We begin with the estimation from below of the integral of	
$\bigl<\widehat x(t),\widehat x(t)\bigr>$ along an optimal trajectory $\widehat x(t)$ on an interval $[0,\delta]$. We use the formula (\ref{eq:xt_by_x0_y0_u}): if $|x_0| =
a\delta^2$, then we have as $t\in[0,\delta]$	
	
	\[
		|\widehat x(t)| \ge \max\{0,a\delta^2 - b\delta t - \frac{1}{2} \delta^2\max\limits_{u\in\Omega}|u|\},
	\]

	\noindent and if $|y_0|=b\delta$, then
	
	\[
		|\widehat x(t)| \ge \max\{0,b\delta t - a\delta^2 - \frac{1}{2} \delta^2\max\limits_{u\in\Omega}|u|\}.
	\]

	\noindent Let us denote $c=\frac{1}{2}\max\limits_{u\in\Omega}|u|$, and put $a=5c$, $b=12c$, and $\tau=t/\delta\in[0,1]$. Then one of the two inequalities
	
	\[
		|\widehat x(t)| = |\widehat x(\tau\delta)| \ge\left[\begin{array}{ll}
			c\delta^2\max\{0,4 - 12\tau\},&\mbox{if}\ |x_0| = a\delta^2,\\
			c\delta^2\max\{0,12\tau - 6\},&\mbox{if}\ |y_0|=b\delta\\
		\end{array}\right.
	\]

\noindent is fulfilled. In any case either for $\tau\in[0,\frac{1}{3}]$, or for $\tau\in[\frac{1}{2},1]$ we obtain a non-degenerated lower bound on $|\widehat x(t)|$. Therefore for some constant $c_1>0$ we have

	\begin{equation}
	\label{eq:int_0_delta_ge_delta_5}
		\int_0^\delta \bigl<\widehat x(t),\widehat x(t)\bigr>dt =
		\delta\int_0^1 \bigl<\widehat x(\tau\delta),\widehat x(\tau\delta)\bigr>d\tau \ge
		c_1\delta^5.
	\end{equation}

Now let $\rho(q_0)=\rho(x_0,y_0)=\Delta$. Denote for brevity
$\rho(t) = \rho(\widehat x(t),\widehat y(t))$. Let
$t_{\frac{1}{2}}(q_0)>0$ be the minimal time instant such that
$\rho(t_{\frac{1}{2}}(q_0)) = \frac{1}{2}\Delta=\delta$. Calculate successively

	$$
		\begin{array}{llllll}
			t_0=0;             &t_1 = t_0+\tau_1;    & t_2 = t_1+\tau_2; & \ldots; & t_{n+1} = t_n+\tau_{n+1};& \ldots \\
			\tau_1 = \rho(t_0);& \tau_2 = \rho(t_1);& \tau_3 = \rho(t_2);& \ldots; & \tau_{n+1} = \rho(t_n);  & \ldots \\
		\end{array}
	$$
	
\noindent Let $N$ be a number such that $t_{N-1}\le t_\frac{1}{2}(q_0)$ and $t_N>t_\frac{1}{2}(q_0)$. The number $N$ is finite because $t_{n+1}-t_n\ge\delta$ for $t_n\le t_\frac{1}{2}(q_0)$. Since $\tau_n\ge\delta$ for $n\le N$, we obtain by virtue of
 inequality (\ref{eq:int_0_delta_ge_delta_5})
	
	$$
		\bellman(q_0) \ge \int_0^{t_N} \bigl<\widehat x(t,q_0), \widehat x(t,q_0)\bigr>\,dt \ge
		c_1\sum_{k=1}^N\tau_k^5 =
	$$
	$$
		=c_1\delta^5 \sum_{k=1}^N\left(\frac{\tau_k}{\delta}\right)^5\ge
		c_1\delta^5 \sum_{k=1}^N\frac{\tau_k}{\delta}\ge
		c_1\delta^4t_\frac{1}{2}(q_0).
	$$
	
	\noindent Therefore
	
	\[
		t_{\frac{1}{2}}(\Delta) \le c_2\, \frac{\beta(\Delta)}{\Delta^4} ,
	\]

	\noindent where $t_{\frac{1}{2}}(\Delta) = \sup_{\rho(q_0)=\Delta}t_{\frac{1}{2}}(q_0)$ and $\beta(\Delta)=\max_{\rho(q_0)=\Delta}\bellman(q_0)$.

	Now note that $t_{\frac{1}{2}}(\lambda\Delta) = \lambda t_{\frac{1}{2}}(\Delta)$ and $\beta(\lambda\Delta)=\lambda^5\beta(\Delta)$ for all $\lambda>0$ by virtue of the
symmetry subgroup $g$ of the group $\R_+$. Hence	
		
	\[
		T(q_0)\le (1+\frac{1}{2}+\frac{1}{4}+\ldots)t_{\frac{1}{2}}(\rho(q_0)) = c_3\rho(q_0),
	\]

	\noindent which completes the proof.
	
\end{proof}

In what follows let us denote by $T(q_0)$ the minimal time instant such that $\widehat q(t,q_0)=0$ for $t\ge T(q_0)$.

\begin{remark}
A lower bound on the time of hitting the origin of the form $T(q_0)\ge C_1\max\{\sqrt{|x_0|},|y_0|\}$ follows immediately from the consideration of the time optimal problem

	\[
		T\to\inf\ \ \mbox{as}\ \ \dot x=y,\ \dot y=u,\ u\in\widetilde\Omega,\quad x(0)=x_0,\,y(0)=y_0,\quad x(T)=y(T)=0,
	\]

	\noindent where $\Omega\subset\widetilde\Omega$, and
	$\widetilde\Omega$ is a square centered at the origin and with side length $2\diam\Omega$.
\end{remark}

\begin{corollary}
	The function $T(q_0)$ is continuous.
\end{corollary}

\begin{proof}
		The continuity of the function $T(q_0)$ at $q_0=0$ follows from the Squeeze theorem. Now consider the case $q_0\ne 0$. Consider a converging sequence $q_k\to q_0$. Denote
	
	\[
		T^- = \liminf_{k\to\infty} T(q_k),
		\qquad
		T^+ = \limsup_{k\to\infty} T(q_k).
	\]

\noindent Note that $T(q_0)\le T^-$. Indeed, if $t>T^-$, then due to Corollary \ref{corollary:map_q0_to_optimal_q_is_continuous} we have         $\widehat q(t,q_0)=0$.

Let us show that $T(q_0)\ge T^+$. According to Corollary
\ref{corollary:map_q0_to_optimal_q_is_continuous} we have

$$\widehat q(T(q_0),q_k) \to \widehat q(T(q_0),q_0)=0.$$ 

\noindent This implies that the
optimal trajectories with initial points $q_k$ appear at an arbitrarily small neighbourhood of the origin after the time $T(q_0)$. But at the origin the function $T(\cdot)$ is continuous. Hence $T\bigl(\widehat q(T(q_0),q_k)\bigr)\to 0$.
It remains to note that for any $k$ either $T(q_k)\le T(q_0)$, or $T(q_k) = T(q_0) + T\bigl(\widehat q(T(q_0),q_k)\bigr)$,
which completes the proof.

\end{proof}

\begin{corollary}
\label{conj:p_uniqueness}
For every $q_0$ the adjoint function $\widehat p(t,q_0)$ is unique.
\end{corollary}

\begin{proof}
Since $\widehat u(t,q_0) = 0$ for $t\ge T(q_0)$, we have that
$\bigl<\psi,u\bigr>$ attains its maximum at an interior point of
$\Omega$ (e.g., $0\in\Int\Omega$). Hence $\widehat \psi(t,q_0) = 0$ for $t\ge T(q_0)$, and consequently, $\widehat\phi(t,q_0)=0$ for $t\ge T(q_0)$. For $t<T(q_0)$ the function $\widehat p(t,q_0)$ is uniquely determined by

	\begin{equation}
	\label{eq:phi_psi_int_formula}
		\begin{array}{l}
			\widehat \phi(t,q_0) = - \int_t^{T(q_0)} \widehat x(\tau,q_0)\,d\tau,\\
			\widehat \psi(t,q_0) = \int_t^{T(q_0)} \widehat \phi(\tau,q_0)\,d\tau.\\
		\end{array}
	\end{equation}
\end{proof}

\begin{lemma}
The Bellman function is of class $C^1$ and
$\bellman'(q_0) = -\widehat p(0,q_0)$.
\end{lemma}

Here and in the sequel by $\bellman'(q_0)\in U^*$ we denote the differential of the Bellman function at the point $q_0$.

\begin{proof}

Let us start with the case $q_0\ne 0$. We shall prove that the differential of the Bellman function consists of a unique point. This will immediately yield $\bellman\in C^1$ outside of the origin. Consider problem (\ref{problem:model}) with the condition $q(0)=q_0$ replaced by the condition $q(0)\in l$, where $l$ is any supporting hyperplane to the set $\{q:\bellman(q)\leq\bellman(q_0)\}$ at the point $q_0$. By virtue of the strict convexity of the function $\bellman(q)$ the trajectory $\widehat q(t,q_0)$ is also the solution of the new problem for any choice of $l$. In accordance with the Pontryagin Maximum Principle there exists an adjoint function $p^l(t)$ satisfying equations
(\ref{eq:model_pmp_system}) and the transversality conditions $p^l(0)\perp l$. By a reasoning similar to the proof of Corollary \ref{conj:p_uniqueness} we obtain that $p^l(t) = \widehat p(t,q_0)$. Hence the vector
 $\widehat p(0,q_0)$ is orthogonal  to any hyperplane $l$. If
$\widehat p(0,q_0)\ne 0$, then the surface level of the function $\bellman$ is smooth and the differentiability of the Bellman function $\bellman$ follows from its homogeneity, $\bellman(g(\lambda) q) = \lambda^5\bellman(q)$ for all $\lambda>0$. In this case $\widehat p(0,q_0) = \lambda(q_0)\bellman'(q_0)$, where $\lambda(q_0)$ is the coefficient of proportionality. Let us show that $\widehat p(0,q_0)\ne 0$. Assume the contrary. Then $0\in\partial\bellman (q_0)$, since the sub-differential of a convex function is upper semi-continuous. But the Bellman function is strictly convex and has its minimum at $0$. This leads to a contradiction.

Now let us show that $\lambda(q_0) = -1$. In accordance with the Bellman equation	
	
	 	$$
		\frac{1}{2}\bigl<x_0,x_0\bigr> + \bigl<\bellman'_x(q_0),\widehat y(0,q_0)\bigr> +
		\bigl<\bellman'_y(q_0),\widehat u(0,q_0)\bigr> = 0.
	$$
	On the other hand, since
 $\dot{\mathcal H} = 0$, and $\widehat q(t,q_0)$ and
 $\widehat p(t,q_0)$ vanish for $t>T(q_0)$, we must have
 ${\mathcal H} = 0$. It follows that
	
	$$
		-\frac{1}{2}\bigl<x_0,x_0\bigr> + \bigl<\widehat \phi(0,q_0),\widehat y(0,q_0)\bigr>
		+ \bigl<\widehat \psi(0,q_0),\widehat u(0,q_0)\bigr> = 0.
	$$
	
	\noindent Hence $\lambda(q_0)=-1$ and $\bellman'(q_0) = -\widehat p(0,q_0)$ for $q_0\ne 0$.
	
Let us show the differentiability at $0$. Let $C=\max\{\bellman(x,y)| |x|^2+|y|^4=1\}$. Then from	
$\bellman(g(\lambda) q) = \lambda^5\bellman(q)$ for all
 $\lambda>0$ we obtain that $0\le\bellman(q) \le C (|x|^2+|y|^4)^\frac{5}{4}$. This completes the proof.
\end{proof}

\begin{lemma}
\label{lm:map_E_is_homeo}	
	The mapping $E:q_0\mapsto \widehat p(0,q_0) = -\bellman'(q_0)$ is a homeomorphism (i.e., a continuous bijection).
\end{lemma}

\begin{proof}

The mapping $E$ is continuous since $\bellman \in C^1$.
Injectivity follows from the strict convexity of the function $\bellman$. To prove surjectivity we consider the strictly convex set $B_1=\{q:\bellman(q)\le1\}$, which has a smooth boundary. For any direction $v\in U^*$ there exist exactly two distinct tangent hyperplanes to $\partial B_1$ which are perpendicular to $v$. Hence the differentials of the function $\bellman(q)$ at the contact points $q_1$ and $q_2$ are anti-parallel. Therefore in one of these points the differential of the function
$\bellman(q)$ is anti-parallel to $v$, say  $v=-\lambda\bellman'(q_1)$ for some $\lambda>0$. It follows that
$v=-\bellman'(g(\sqrt[5]{\lambda})q_1)$.

\end{proof}

\begin{remark}
By the definition of $T(q_0)$ we have $q(t,q_0)\ne0$ for all $t<T(q_0)$. Since the mapping $E$ is a homeomorphism and $E(0,0)=(0,0)$, we have $p(t,q_0)\ne0$ for all $t<T(q_0)$.
\end{remark}

\begin{corollary}
At the initial point the adjoint variables satisfy the bounds
	
	\[
		C_3\max\{\sqrt{|x_0|},|y_0|\}\le \max\{\sqrt[3]{|\phi_0|},\sqrt[4]{|\psi_0|}\}\le C_4\max\{\sqrt{|x_0|},|y_0|\}
	\]

	\noindent for some $C_3>0$ and $C_4>0$.
\end{corollary}

\begin{proof}
We use the fact that $E:q_0\mapsto \widehat p(0,q_0)$ is continuous.   Denote
	
	\[
		\begin{array}{l}
			C_3 = \min\limits_{\max\{\sqrt{|x_0|},|y_0|\}=1} \max\{\sqrt[3]{|\phi_0|},\sqrt[4]{|\psi_0|}\},\\
			C_4 = \max\limits_{\max\{\sqrt{|x_0|},|y_0|\}=1} \max\{\sqrt[3]{|\phi_0|},\sqrt[4]{|\psi_0|}\}.\\
		\end{array}
	\]

	\noindent Clearly $C_4\ge C_3\ge 0$. But $E^{-1}(0,0)=(0,0)$ by Lemma  \ref{lm:map_E_is_homeo}.
It remains to make use of the action of the group $g$:
	
	\[
		C_3\lambda\le \max\{\sqrt[3]{|\phi_0|},\sqrt[4]{|\psi_0|}\}\le C_4\lambda
	\]
	
	\noindent for $\max\{\sqrt{|x_0|},|y_0|\}=\lambda$, which completes the proof.
\end{proof}

It remains to prove that the mapping $E$ is Lipschitz. The proof is divided in two lemmas.

\begin{lemma}
For every point $q_0\in U$ there exist such $\delta>0$ and $C>0$ that for any $\Delta q=(\Delta x,\Delta y)$, $|\Delta q|<\delta$  the inequalities
	
	$$
		0\leq \bellman(q_0+\Delta q) - \bellman(q_0) - \bellman'(q_0)\Delta q) \le C|\Delta q|^2
	$$
	
	\noindent hold. Here the constants $C$ and $\delta$ can be chosen as continuous functions of $q_0$.
\end{lemma}

\begin{proof}
	For $t\in[0;T(q_0)]$, define $q(t)=(x(t),y(t))$ and $p(t)=(\phi(t),\psi(t))$ as the solutions of the ODE
	
	$$
		\left\{\begin{array}{ll}
			\dot y(t) = \widehat u(t,q_0), & y(0) = y_0+\Delta y,\\
			\dot x(t) = y(t), & x(0) = x_0+\Delta x,\\
			\dot \phi(t) = x(t), & \phi(0) = \widehat \phi(0,q_0),\\
			\dot \psi(t) = -\phi(t), & \psi(0) = \widehat \psi(0,q_0).\\
		\end{array}\right.
	$$
	
	\noindent For $t \geq T(q_0)$ the trajectory
$q(t)$ coincides with the optimal one with initial point $q(T(q_0))$, i.e., $q(t) = \widehat q\bigl(t-T(q_0),q(T(q_0))\bigr)$ (the trajectory $p(t)$ is not important
for $t>T(q_0)$). In this way, the trajectory $q(t)$ begins at $q_0+\Delta q$ and moves with the control that is optimal for the initial point $q_0$ until this control becomes identically zero, then the trajectory switches to the real optimal control. Integrating two times by parts,	we obtain
	
	\begin{eqnarray*}
\lefteqn{{\cal B}(q_0+\Delta q) - {\cal B}(q_0) \leq \frac12\int_0^{T(q_0)} \langle x(t) - \hat x(t,q_0),x(t) + \hat x(t,q_0) \rangle\,dt +} \\
&&+ {\cal B}(\Delta x+T(q_0)\Delta y,\Delta y)\\
&=& \frac12\langle \phi(t)+\hat\phi(t,q_0),\Delta x+t\Delta y \rangle|_0^{T(q_0)} + \frac12\langle \psi(t)+\hat\psi(t,q_0),\Delta y \rangle|_0^{T(q_0)} +\\
&& + {\cal B}(\Delta x+T(q_0)\Delta y,\Delta y)\\
&=& \frac12\langle \phi(T(q_0)),\Delta x+T(q_0)\Delta y \rangle - \langle \hat\phi(0,q_0),\Delta x \rangle + \frac12\langle \psi(T(q_0)),\Delta y \rangle - \\
&&-\langle \hat\psi(0,q_0),\Delta y \rangle + {\cal B}(\Delta x+T(q_0)\Delta y,\Delta y) \\
&=& \frac{T(q_0)}{2}||\Delta x||^2 + \frac{T^2(q_0)}{2}\langle \Delta y,\Delta x \rangle + \frac{T^3(q_0)}{6}||\Delta y||^2 + \langle {\cal B}'(q_0),\Delta q \rangle+ \\
&&+ {\cal B}(\Delta x+T(q_0)\Delta y,\Delta y).\\
\end{eqnarray*}

Here we used that
 $\phi(0) = \hat\phi(0,q_0) = -{\cal B}_x'(q_0)$, $\phi(T(q_0)) = T(q_0)\Delta x + \frac{T^2(q_0)}{2}\Delta y$, $\psi(0) = \hat\psi(0,q_0) = -{\cal B}_y'(q_0)$, $\psi(T(q_0)) = -\frac{T^2(q_0)}{2}\Delta x - \frac{T^3(q_0)}{6}\Delta y$, $\hat p(T(q_0),q_0) = 0$. Introduce the constant $C_5 = \max\{{\cal B}(q)\,|\,||x||^{\frac52}+||y||^5 = 1\}$. Then ${\cal B}(\Delta x+T(q_0)\Delta y,\Delta y) \leq C_5(||\Delta x+T(q_0)\Delta y||^{\frac52}+||\Delta y||^5)$.
We obtain the bound
\[ {\cal B}(q_0+\Delta q) - {\cal B}(q_0) - \langle {\cal B}'(q_0),\Delta q \rangle \leq C_6||\Delta q||^2 + C_5'(||\Delta x||^{\frac52}+||\Delta y||^{\frac52}+||\Delta y||^5),
\]
where the constants $C_5',C_6 > 0$, generally speaking, depend on
$T(q_0)$. The obtained inequality gives the upper bound claimed in the lemma. The convexity of the Bellman function
$\bellman(q)$ implies that this expression is nonnegative.
\end{proof}

The main difficulty in the remaining part of the proof, i.e., that the function $\bellman'(q)$ satisfies a Lipschitz condition, is that this function may not be twice differentiable (this is indeed the case if, say, $\Omega$ is a polyhedron). We shall conclude the proof of Theorem
\ref{thm:model_problem_bellman} with the following lemma.

\begin{lemma}
\label{lm:bilipschitzian_model}
	Let $g\in C^1(\R^n,\R)$ be an arbitrary function. If for some constants $0\le C'\le C$ and for all $q_1,q_2$ in some open convex domain in $\R^n$ we have the inequality
	
	$$
		C' |\Delta q^2| \le
		g(q_2)-g(q_1) - \bigl<g'(q_1),\Delta q\bigr>
		\le C|\Delta q^2|,
	$$

	\noindent where $\Delta q = q_2-q_1$, then the mapping $g':\R^n\to \R^{n*}$ satisfies a Lipschitz condition with Lipschitz constant $C$. If, moreover, $C'> 0$, then the mapping $g'$ is bi-Lipschitz on this domain.
	
\end{lemma}

\begin{proof}
Without loss of generality we can assume that the convex open domain is $\R^n$. We may also assume that a scalar product is defined in $\R^n$ and hence $g'(q)\in\R^n$. Let $\omega(q)\geq 0$ be a convolution kernel, more precisely, a nonnegative $C^{\infty}$ function with compact support $\mathrm{supp}\,\omega \subseteq\{|q|\leq 1\}$ and $\int \omega(q)\,dq = 1$. Denote
$\omega_k(q) = k^n\omega(kq)$ and $g_k=g*\omega_k$, where $*$ denotes the convolution. Then $g_k$ is $C^{\infty}$, $g_k' = g'*\omega_k$, and

	\begin{eqnarray*}
		\lefteqn{\Lambda_n(q_1,q_2) = g_k(q_2) - g_k(q_1) - g'_k(q_1)\Delta q =}&\\
		&=\int_{\R^n}\bigl(g(q_2-y) - g(q_1-y) - g'(q_1-y)\Delta q\bigr)\omega_k(y)\,dy.\\
	\end{eqnarray*}
	
	\noindent Consequently,
	
	$$
		C'|\Delta q|^2\le
		\Lambda_k(q_1,q_2) = \frac{1}{2}g''_k(q_1)[\Delta q;\Delta q] + o(|\Delta q|^2)
		\le C|\Delta q|^2.
	$$
	
	\noindent By substituting $\lambda\Delta q$ for
$\Delta q$ and by letting $\lambda \to 0$ we obtain that for all $k,q,\Delta q$ the inequality

	$$
		C'|\Delta q|^2 \le \frac{1}{2}g''_k(q)[\Delta q;\Delta q] \le C|\Delta q|^2.
	$$
	
	\noindent holds. Hence for all $k$ and $q$ we have
	$\|g''_k(q)\|\leq 2C$. It follows that
	
	$$
		|g_k'(q_2) - g_k'(q_1)| \leq \sup_{\theta\in[0,1]} |g_k''(q_1+\theta(q_2-q_1))| |\Delta q| \leq 2C|\Delta q|.
	$$
	
	\noindent By letting $k \to \infty$ (while $q_1$ and $q_2$ remain fixed) we obtain the desired result
	
	$$
		|g'(q_2) - g'(q_1)|\leq 2C|\Delta q|.
	$$
	
	Now let $C'>0$. We shall show that the mapping $g$ satisfies a bi-Lipschitz condition. We have
	
	$$
		g'_k(q_2) - g'_k(q_1) = \int_0^1 g''_k(q_1 + t\Delta q)\Delta q \,dt.
	$$
	
	\noindent Hence
	
	\begin{eqnarray*}
		\lefteqn{\left|\bigl<g'_k(q_2) - g'_k(q_1),\Delta q\bigr>\right| = \left|\int_0^1 g''_k(q_1+t\Delta q)[\Delta q,\Delta q]\,dt  \right| =}&\\
		&=\int_0^1 g''_k(q_1+t\Delta q)[\Delta q,\Delta q]\,dt \geq 2C'|\Delta q|^2.\\
	\end{eqnarray*}
	
	\noindent So, for any $k,q_2,q_1$ we get
		$$
		\left|\bigl<g'_k(q_2) - g'_k(q_1),\Delta q\bigr>\right| \geq 2C'|\Delta q|^2.
	$$
	
	\noindent By letting $k \to \infty$ (while $q_1$ and $q_2$ remain fixed) we obtain

	$$
		\left|\bigl<g'(q_2) - g'(q_1),\Delta q\bigr>\right| \geq 2C'|\Delta q|^2.
	$$
	\noindent Taking into account that $\left|g'(q_2) - g'(q_1)\right||\Delta q| \geq \left|\bigl<g'(q_2) - g'(q_1),\Delta q\bigr>\right|$, we obtain
	
	$$
		|g'(q_2) - g'(q_1)| \geq 2C'|\Delta q|,
	$$
	
	\noindent which means that the mapping $g$ is bi-Lipschitz if $C'>0$.
	
\end{proof}


This proves Theorem \ref{thm:model_problem_bellman}. As a first application of Theorem \ref{thm:model_problem_bellman} we obtain (here and in the sequel $M=\{(x,y)\}=U\oplus U$ is the phase space)

\begin{corollary}
\label{corollary:M_plus_is_optimal}

The set of all pairs $(q_0,E(q_0))$ in the extended phase space ${\cal M}=T^*M$ is a Lipschitz manifold $M_+$ which projects bijectively onto the plane $\{p=0\}$. All optimal trajectories belong to $M_+$ and reach the origin in finite time. On the other hand, $M_+$ consists of optimal trajectories only. There do not exist other trajectories of the Hamiltonian system (\ref{eq:model_pmp_system}) that reach the origin because otherwise they would be optimal due to the convexity of the problem.

\end{corollary}

Let us remark that it is not difficult to find trajectories of system (\ref{eq:model_pmp_system}) that come out of the origin. Indeed, the mapping $g(-1)$ preserves system
(\ref{eq:model_pmp_system}), but reverses the direction of movement. Hence the trajectories that come out of the origin form a Lipschitz manifold $M_-$,

\[
	M_-=g(-1)[M_+].
\]

\subsection{The quotient spaces \texorpdfstring{$M/g$}{M/g} and \texorpdfstring{${\cal M}/g$}{M/g}}

A paramount role in the investigation of the optimal synthesis (its existence has been proved in Theorem
\ref{thm:model_problem_bellman}) is played by the transition to the quotient space relative to the action $g$ of the group
 $\R_+$, and by the transfer of the optimal synthesis to this quotient space.

\begin{remark}
\label{rm:symmetries_commutate}
Since the above-mentioned actions of both groups $\R_+$ (for arbitrary triangles $\Omega$) and $S_3$ (for equilateral triangles only) take optimal trajectories into optimal ones, the mapping $E$ constructed in Theorem
\ref{thm:model_problem_bellman} commutes with these actions.
	
	\[
		E\bigl(x,y\bigr)=(\phi,\psi)\ \ \Rightarrow\ \
		\left\{\begin{array}{llcl}
			E\bigl((\lambda^2 x,\lambda y)\bigr)&=&(\lambda^3 \phi,\lambda^4 \psi)&\forall \lambda>0\\
			E(\alpha x,\alpha y) &=& (\alpha \phi, \alpha \psi)&\forall \alpha\in S_3
		\end{array}\right.
	\]
\end{remark}

After the factorization of the phase space $M=\R^4$ (without the origin) with respect to the action $g$ of the group $\R_+$
we obtain a three-dimensional sphere $S^3$. We can correctly transfer the optimal trajectories of problem
(\ref{problem:model}) to $S^3$ with preservation of the direction of movement on these trajectories. However, the velocity on the trajectories will be lost. After factorization of the sphere $\{M\setminus 0\}/g(\R_+)=S^3$ with respect to the action of the group $S_3$ one obtains a space similar to a lens space\footnote{More precisely, one obtains
the quotient space of the lens space $L(3;1,1)$ relative to the non-free action of $\Z_2$.}. One can also factorize the extending phase space ${\cal M}=T^*M=\R^8$ without the origin with respect to the action $g$ of the group $\R_+$. In this case one obtains the seven-dimensional sphere $S^7=\{{\cal M}\setminus 0\}/g(\R_+)$. We can correctly transfer the trajectories of the Hamiltonian system
(\ref{eq:model_pmp_system}) onto the sphere with preservation of the direction of movement.

Denote by $\pi$ the canonic projection  $M\setminus 0\to \{M\setminus 0\}/g(\R_+)$, and by $\widetilde\pi$ the canonic projection  ${\cal M}\setminus 0\to \{{\cal M}\setminus 0\}/g(\R_+)$. When there will be no risk of confusion we shall write $\pi$ instead of $\widetilde\pi$. By some abuse of notation, we shall write $M/g$ instead of $\{M\setminus 0\}/g(\R_+)$ and ${\cal M}/g$  instead of $\{{\cal M}\setminus 0\}/g(\R_+)$, always meaning that the origin is excluded before factorization.

\smallskip

Theorem \ref{thm:model_problem_bellman} allows one to transfer the optimal synthesis from the space $M=\{(x,y)\}$ onto the space $N=\{(\phi,\psi)\}$. Moreover, using the mapping $E$ we may identify $M$, $M_+$, and $N$ (or, after factorization $M/g$, $M_+/g$, and $N/g$). This turns out to be very useful while studying the switching surface $\cal S$ of the control, because the latter is simplest when described in the space of adjoint variables $N$. Since $E$ commutes with $g(\lambda)$ for $\lambda>0$, the spaces $M/g$, $M_+/g$, and $N/g$ can also be identified.

\subsection{One-dimensional Fuller problems in the interior}

In this subsection we consider an important special case, namely when the origin lies on one of the altitudes of the triangle
$\Omega$. In this case there exists a two-dimensional plane in $M$ such that

(i) the optimal trajectories of system (\ref{problem:model}) starting in this plane cannot leave it, and

(ii) the behaviour of optimal trajectories on this plane is equivalent to that for a one-dimensional Fuller problem with a non-symmetric interval of control.

\begin{lemma}
\label{lm:semi_singular}
   Let the affine line $A$ containing one of the altitudes of $\Omega$ pass through the origin. Then if one of the two following conditions is fulfilled:

   	\begin{enumerate}[(i)]
		\item the initial point lies in $A \oplus A$, i.e., $x_0,y_0\in A$,
		\item the initial values of the adjoint variables lie in $A$, i.e., $\phi^0, \psi^0\in A$, where $(\phi^0,\psi^0)=E(x_0,y_0)$,
	\end{enumerate}
			\noindent then the values of the variables on the optimal trajectory do not leave $A$ for all $t \geq 0$: $\widehat x(t,q_0)\in A$, $\widehat y(t,q_0)\in A$, $\widehat \phi(t,q_0)\in A$, $\widehat \psi(t,q_0)\in A$ and $\widehat u(t,q_0)\in A$ (recall that $q_0=(x_0,y_0)$).
\end{lemma}

\begin{proof}
Suppose condition (i) holds. Then by Theorem \ref{thm:model_problem_bellman} the optimal trajectory
$\widehat q(t,q_0)$ exists and is unique. Consider the trajectory $q(t)$ that is obtained by orthogonal projection of the optimal trajectory on $A \oplus A$.
The trajectory $q(t)$ is admissible since $A$ contains an altitude of $\Omega$. Since the length of a vector is greater than that of its orthogonal projection, we have
$J(q)\le J(\widehat q)$. It remains to remark that both $q(t)$ and $\widehat q(t,q_0)$ start from the same point. Consequently,  $\widehat q(t,q_0)\equiv q(t)$ by virtue of the uniqueness of the optimal trajectory (by Theorem
\ref{thm:model_problem_bellman}). Hence $\widehat q(t,q_0)\in A \oplus A$ for all $t$. Both adjoint variables $\widehat\phi(t,q_0)$ and $\widehat\psi(t,q_0)$ lie in $A$ due to
 (\ref{eq:phi_psi_int_formula}).

By restriction to the subspace  $A\oplus A\subset M$ we obtain a one-dimensional Fuller problem, its interval of control being the altitude of the triangle $\Omega$ contained in $A$. Hence
$E(A\oplus A) = A\oplus A$ (see \cite{ZelikinBorisov}, \S 3.5).

Suppose condition (ii) holds.  By Theorem \ref{thm:model_problem_bellman} one can find the unique initial values
$x_0$ and $y_0$ such that $E(x_0,y_0) = (\phi^0,\psi^0)$.
Therefore $x_0$ and $y_0$ must belong to $A$ (since the mapping $E$ is bijective and $E(A\oplus A) = A\oplus A$). Consequently, condition (i) holds and the claim of the lemma follows from the considerations above.
\end{proof}

Let us remark that under the conditions of Lemma \ref{lm:semi_singular} the singular control relative to the base $(ij)$ of the altitude contained in $A$ is active only on optimal trajectories that are contained in $A$.
Indeed, if the maximum of $\bigl<\psi,u\big>$ is attained simultaneously at all points of the edge $(ij)$, then the vector $\psi$ is perpendicular to the edge $(ij)$. For a singular control it is hence necessary that the vector $\psi$ is perpendicular to the edge $(ij)$ during some open interval of time $t\in(t_1,t_2)$. By differentiating $\psi(t)$ along the trajectories of system (\ref{eq:model_pmp_system}) one obtains that all vectors
$\phi(t)$, $x(t)$, $y(t)$ and $u(t)$ are perpendicular to the edge $(ij)$ for this interval.

\begin{defn}
We shall call an optimal trajectory {\it semi-singular} relative to the edge $(ij)$ of the triangle $\Omega$ if the trajectory belongs to the line $A$ containing the altitude of $\Omega$ which is perpendicular to the edge $(ij)$.
\end{defn}

\subsection{Barycentric coordinates for the case of an equilateral triangle}

The proof of the theorems on the chaotic nature of the optimal synthesis relies on key elements of the optimal synthesis of problem
(\ref{problem:model}) for the case when the triangle $\Omega$ is equilateral. Therefore, the main goal of this subsection will be
the determination of the key element of the optimal synthesis in this important special case. The triangle $\Omega$  will be supposed to be equilateral and centered on the origin (in what follows for the sake of brevity we shall call this case "the case of an equilateral triangle").

To describe explicitly some important examples of optimal trajectories we need a suitable coordinate system. Since the set $\Omega$ is an equilateral triangle centered on the origin, the most convenient way is to use the barycentric coordinate system. (Generally speaking, the sum of the coordinates must equal 1 in barycentric coordinates. Nevertheless, with a little abuse of notation we shall use the term "barycentric".)

\[
	\left\{\begin{array}{llll}
		x=(x^1,x^2,x^3), & x^1+x^2+x^3 = 0; & y=(y^1,y^2,y^3), & y^1+y^2+y^3 = 0;\\
		\phi=(\phi_1,\phi_2,\phi_3), & \phi_1+\phi_2+\phi_3 = 0; & \psi=(\psi_1,\psi_2,\psi_3), & \psi_1+\psi_2+\psi_3 = 0;\\
		u=(u^1,u^2,u^3), & u^1+u^2+u^3 = 0; &u_i\le1, i=1,2,3.&\\
	\end{array}\right.
\]

\noindent The vertices of the triangle $\Omega$ are the points
\footnote{We consider the case when the distance from the vertices to the origin equals $\sqrt{6}$. Any case where the triangle $\Omega$ has a different size can be reduced to that particular one
be an evident change of coordinates.} $(-2,1,1)$, $(1,-2,1)$, and $(1,1,-2)$. We shall call these the first, the second, and the third vertex of the triangle $\Omega$, respectively. In terms of barycentric coordinates, the condition of the maximum in (\ref{eq:model_pmp_system}) is rewritten in the form

\[
	u^i = -2, u^j=u^k=1,\mbox{ if } \psi_i<\min(\psi_j,\psi_k), \mbox{ where } \{i,j,k\} = \{1,2,3\},
\]

\noindent with the exception of the cases when $\psi_i=\psi_j\le\psi_k$.

The action of the discrete group $S_3$ on the barycentric coordinates consists in a permutation of the indices. Let us denote by $A_{ij}$ the line in $U$ that is symmetric relative to the transposition $(ij)\in S_3$. By Lemma \ref{lm:semi_singular} it follows that by restriction to each of the three lines $A_{12}$, $A_{13}$, and $A_{23}$ we obtain a one-dimensional Fuller problem with non-symmetric interval of control, namely

\[
	\begin{array}{l}
		\int\limits_0^{+\infty} x^2\,dt\to\min;\\
		\dot x=y;\ \ \dot y = u;\ \ u\in\bigl[-\sqrt{6},\frac{1}{2}\sqrt{6}\bigr].
	\end{array}
\]

Here $x$, $y$, and $u$ are one-dimensional. This problem is well-studied, its optimal synthesis has been designed and can be rewritten as an explicit expression of the coordinates, its Bellman function and the mapping $E$ have also been found explicitly (see \cite{ZelikinBorisov}). The geometrical properties of the optimal synthesis are as follows. On the plane $\R^2$ there exists a one-parametric family of trajectories that is taken onto itself by the action $g$ of the scale group $\R_+$. Each optimal trajectory reaches the origin in finite time. The optimal control has an infinite number of switchings from the left end of the interval to the right end and vice versa. After factorization of the phase space $\R^2$ (without the origin) with respect to the action of the scale group $g$
one gets the circle $S^1$, which at the same time is a periodic cycle $Z$. All optimal trajectories are taken to this cycle $Z\simeq S^1$. The cycle $Z$ is divided in two open intervals $Z^l$ and $Z^r$ by two switching points. The control on $Z^l$ is at the left end of the interval and on $Z^r$ on the right end. The countable number of control switchings on the optimal trajectories corresponds to a countable number of rotations along the cycle $Z$.

\smallskip

Therefore the mapping of the optimal synthesis of problem (\ref{problem:model}) onto the quotient space $M/g\simeq S^3$ takes all optimal trajectories from $A_{ij}$ onto the same trajectory $Z_{ij}\simeq S^1$. The trajectory $Z_{ij}$ is a periodic cycle which is divided by two switching points in two smooth intervals $Z_{ij}^s$ and $Z_{ij}^n$. On $Z_{ij}^n$ the control in the vertex $k$ of $\Omega$ is used, i.e., $u_i=u_j=1$, $u_k=-2$, $k\ne i,j$, and on $Z_{ij}^s$  the midpoint of the edge of $(ij)$ is used, $u_i=u_j=-\frac{1}{2}$, $u_k=1$. We call the first control {\it non-singular} and the second one {\it semi-singular}\footnote{The symbols $s$ and $n$ in the indices $Z_{ij}^s$ and $Z_{ij}^n$ stand for singular and non-singular} with respect to the edge $(ij)$ of the triangle $\Omega$.

So, on the quotient space $M/g$ three mutually linked cycles $Z_{ij}$ are defined. It is easy to check that their mutual linking numbers equal $1$ \cite{ZMHBasic}.

\subsection{Most important examples of periodic trajectories on \texorpdfstring{$M/g$}{M/g}}
\label{subsec:automodel_trajectories}

A very important role in constructing the optimal synthesis of problem (\ref{problem:model}) is played by the self-similar trajectories. We use this term for trajectories that are periodic up to the action $g$.

\begin{defn}
\label{defn:automodel}
We call a trajectory $(x(t),y(t),\phi(t),\psi(t))$ of system (\ref{eq:model_pmp_system}) {\it self-similar} if there exist a time instant $t_0>0$ and a number $\lambda_0>0$ such that

	\[
		g(\lambda_0)(x(t),y(t),\phi(t),\psi(t)) =
		\bigl(
			x(t_0+\lambda_0 t),y(t_0+\lambda_0 t),
			\phi(t_0+\lambda_0 t),\psi(t_0+\lambda_0 t)
		\bigr).
	\]
\end{defn}

\begin{remark}
\label{rm:lambda_0_less_1}
	If $\lambda_0<1$, then the self-similar trajectory reaches the origin in finite time $\frac{1}{1-\lambda_0}t_0$, hence by Theorem \ref{thm:model_problem_bellman} it lies in
	$M_+$ and is optimal. If $\lambda_0>1$, then the self-similar trajectory lies in $M_-$ and leaves the origin in finite time with an infinite number of switchings. If $\lambda_0=1$, then the self-similar trajectory is periodic and does not tend to the origin neither as $t\to+\infty$, nor as $t\to-\infty$. 	
\end{remark}

Since the action $g$ respects the optimal synthesis, we have that the map $g(\lambda)$ takes any self-similar trajectory with $\lambda_0<1$ to a self-similar trajectory with the same
$\lambda_0$. All trajectories of the corresponding family are projected onto the same periodic trajectory on the quotient space $M_+/g$. As it will be shown below, there exists a countable number of periodic trajectories on $M_+/g$, and consequently, there exists a countable number of families of self-similar trajectories.

We shall now describe some specific periodic trajectories on $M_+/g$. The proof of the chaotic nature of the optimal synthesis on $M_+/g$ will be based on this description. We begin with exact definitions of the Poincar\'e first return map and of the switching surface. Usually by switching surface one understands the set of points where the control on the trajectories of the Pontryagin Maximum Principle is discontinuous. But due to the existence of trajectories with semi-singular control at the midpoints of the edges of $\Omega$ we extend the definition.

\begin{defn}
\label{defn:switch_surface}
	By switching surface $\cal S$ of problem (\ref{problem:model}) will denote the set of points of the extended phase space ${\cal M}=T^*M$ where the maximum over $u$ of the scalar product $\bigl<\psi,u\bigr>$ in
(\ref{eq:model_pmp_system}) is attained at more than one point of the triangle $\Omega$. That means that	
	
	\[
		{\cal S}={\cal S}_{12}\cup{\cal S}_{23}\cup{\cal S}_{13} \subset {\cal M}.
	\]
	
	\noindent Here ${\cal S}_{ij}\subset{\cal M}$ is the set of points where $\argmax_{u\in\Omega}\bigl<\psi,u\bigr>$ contains the edge $(ij)$ of the triangle $\Omega$ (in this case the co-vector $\psi$ is perpendicular to the edge $(ij)$ of the triangle $\Omega$). By ${\cal S}_{123}$ we denote the intersection
	
	\[
		{\cal S}_{123} = {\cal S}_{12}\cap{\cal S}_{23}\cap{\cal S}_{13} = \{\psi=0\}\subset {\cal M}
	\]

	\noindent If the triangle $\Omega$ is equilateral, then in barycentric coordinates we have
	
	\[
		{\cal S}_{ij} = \{\psi_i=\psi_j\le\psi_k\},\quad k\ne i,j.
	\]	
\end{defn}

The set ${\cal S}\setminus {\cal S}_{123}$ is a smooth co-dimension 1 non-compact manifold without boundary. At any point ${\cal S}\setminus{\cal S}_{123}$ the velocity of system (\ref{eq:model_pmp_system}) suffers a tangent jump. More precisely, at any point of ${\cal S}\setminus{\cal S}_{123}$ the difference of the limits of the velocities of system (\ref{eq:model_pmp_system}) from both sides of
${\cal S}\setminus{\cal S}_{123}$ is tangent to ${\cal S}\setminus{\cal S}_{123}$. If the velocity at some point
${\cal S}\setminus{\cal S}_{123}$ is transversal to ${\cal S}\setminus{\cal S}_{123}$ from both sides, then the trajectory of system (\ref{eq:model_pmp_system}) which begins at that point does not intersect $\cal S$ during a finite interval of time. The Poincar\'e first return map $\Phi$ is defined on a subset of the manifold ${\cal S}\setminus{\cal S}_{123}$ and takes a point $q\in{\cal S}\setminus{\cal S}_{123}$ to the first intersection of the trajectory of the Hamiltonian system (\ref{eq:model_pmp_system}) starting at $q$ with the switching surface ${\cal S}$. Denote the image of $q$ by $\Phi(q)\in{\cal S}$. If the point $q$ lies on an optimal trajectory, then its image
$\Phi(q)$ lies on the same optimal trajectory. Taking into account that $M_+\cap {\cal S}_{123} = \{0\}$, we obtain $\Phi(({\cal S}\setminus{\cal S}_{123})\cap M_+)\subseteq ({\cal S}\setminus{\cal S}_{123})\cap M_+$. Let us remark that the mapping $\Phi$ is not defined for all points of $({\cal S}\setminus{\cal S}_{123})\cap M_+$.

\bigskip

Let us investigate the periodic optimal trajectories on $M_+/g$ for the case of an equilateral triangle. Clearly the trajectories $Z_{ij}$ are periodic on $M_+/g$. The remaining optimal trajectories intersect the switching surface ${\cal S}\setminus{\cal S}_{123}$ transversally. More precisely, let $q_0$ be an initial point and let $(x_1,y_1,\phi^1,\psi^1)$ be the intersection point of the optimal trajectory through $q_0$ with the switching surface ${\cal S}_{ij}$, taking place at some time instant $t_1 \in\bigl(0, T(q_0)\bigr)$. Suppose that at the intersection point not only the condition $\psi^1_i=\psi^1_j$, but also the condition $\phi^1_i=\phi^1_j$ holds, which is equivalent to the condition that the intersection is not transversal. Then by Lemma \ref{lm:semi_singular} the trajectory lies on $A_{ij}$ for $t\ge t_1$. Hence for any optimal trajectory only three cases are possible: (i) on the optimal trajectory $x(t),y(t),\phi(t),\psi(t)\in A_{ij}$ holds for all $t\ge 0$; (ii) the optimal trajectory intersects the switching surface transversally\footnote{Of course the velocity on the trajectories experiences a jump on the switching surface, but its limits from both sides of the switching surface are transversal to it.} for $0<t<T(q_0)$, and (iii) the optimal trajectory intersects the switching surface transversally until it reaches $A_{ij}$ for some $t_1 \in\bigl(0, T(q_0)\bigr)$ at some point $(x_1,y_1,\phi^1,\psi^1)$, and after that it remains on $A_{ij}$ for all $t\ge t_1$. We have proved the following lemma:

\begin{lemma}
\label{lm:periodic_transversal_S}
All periodic optimal trajectories on $M_+/g$ of system (\ref{eq:model_pmp_system}) with the exception of the semi-singular ones (see Lemma \ref{lm:semi_singular}) define periodic orbits of the Poincar\'e mapping $\Phi$ of the switching surface to itself. Moreover, any such trajectory intersects the switching surface transversally.
\end{lemma}

Let us remark that the arc $\gamma$ of the trajectory which joins $q$ with $\Phi(q)$ is a smooth curve, and the vectors $x(t)$, $y(t)$, $\phi(t)$, and $\psi(t)$ on this arc
$\gamma$ are polynomials in $t$ of second, first, third, and fourth degree, respectively.

Since the action $g$ respects the Hamiltonian system
(\ref{eq:model_pmp_system}), it commutes with the Poincar\'e mapping $\Phi$, and it induces a mapping of a subset of
$({\cal S}\cap M_+)/g$ to itself which with some abuse of notation will also be denoted by $\Phi$.

Consider a periodic trajectory (i.e., a cycle) on $M_+/g$ with $k$ intervals of smoothness, i.e., which intersects the quotient ${\cal S}/g$ of the switching surface successively at points $z_0$, $z_1$, ..., $z_{k-1}$, $z_k=z_0$. One can find such trajectories by considering the system of equations

\[
	\Phi(z_0) = z_1,\ldots,\Phi(z_{k-1})=z_0\quad \Leftrightarrow\quad \Phi^k(z_0)=z_0.
\]

\begin{figure}
	\centering
	\begin{subfigure}[b]{0.28\textwidth}
		\centering
		\includegraphics[width=0.6\textwidth]{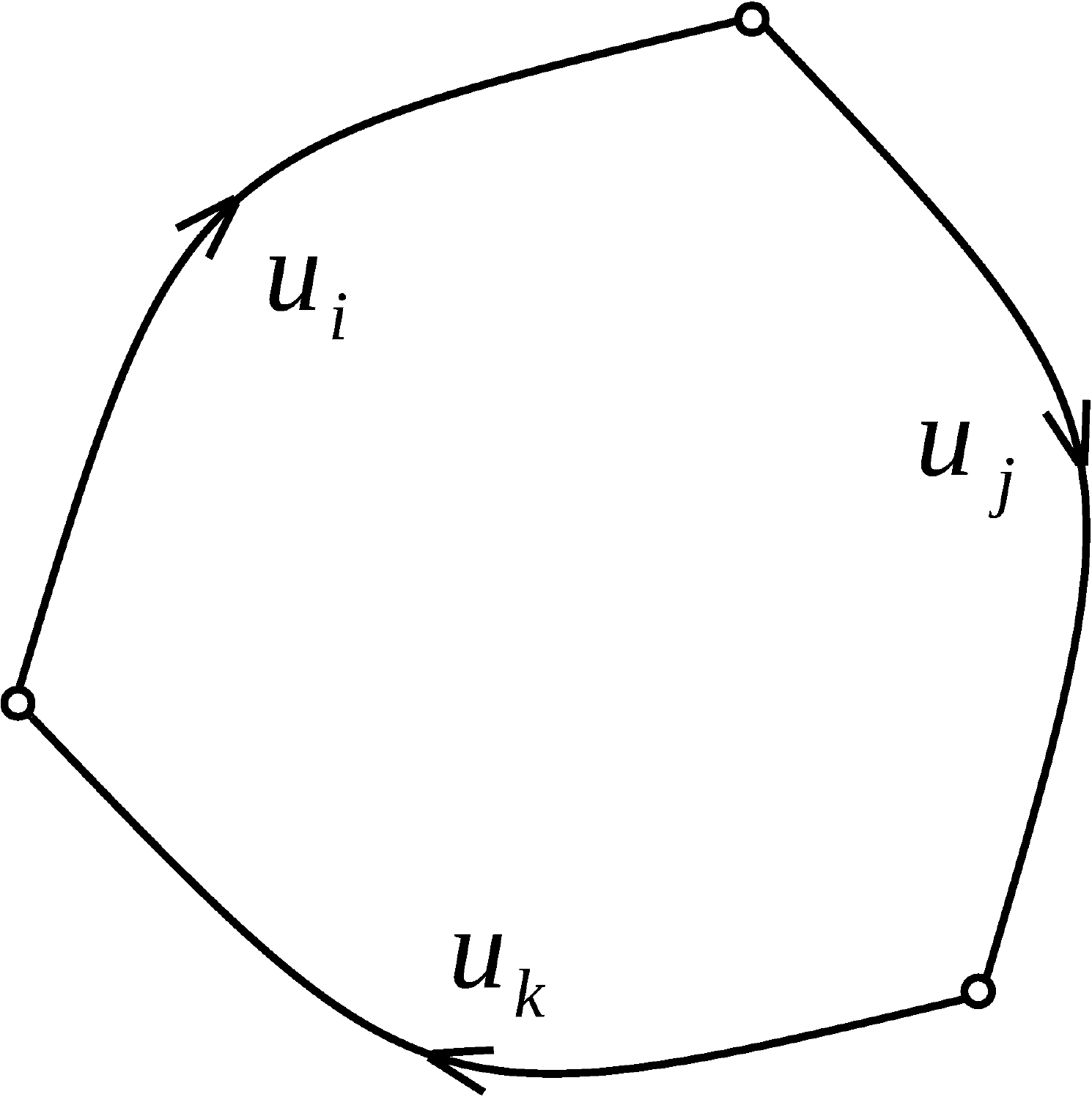}
		\caption{Three-link cycle $Z^\pm$.}
	\end{subfigure}
	\ 
	\begin{subfigure}[b]{0.28\textwidth}
		\centering
		\includegraphics[width=0.6\textwidth]{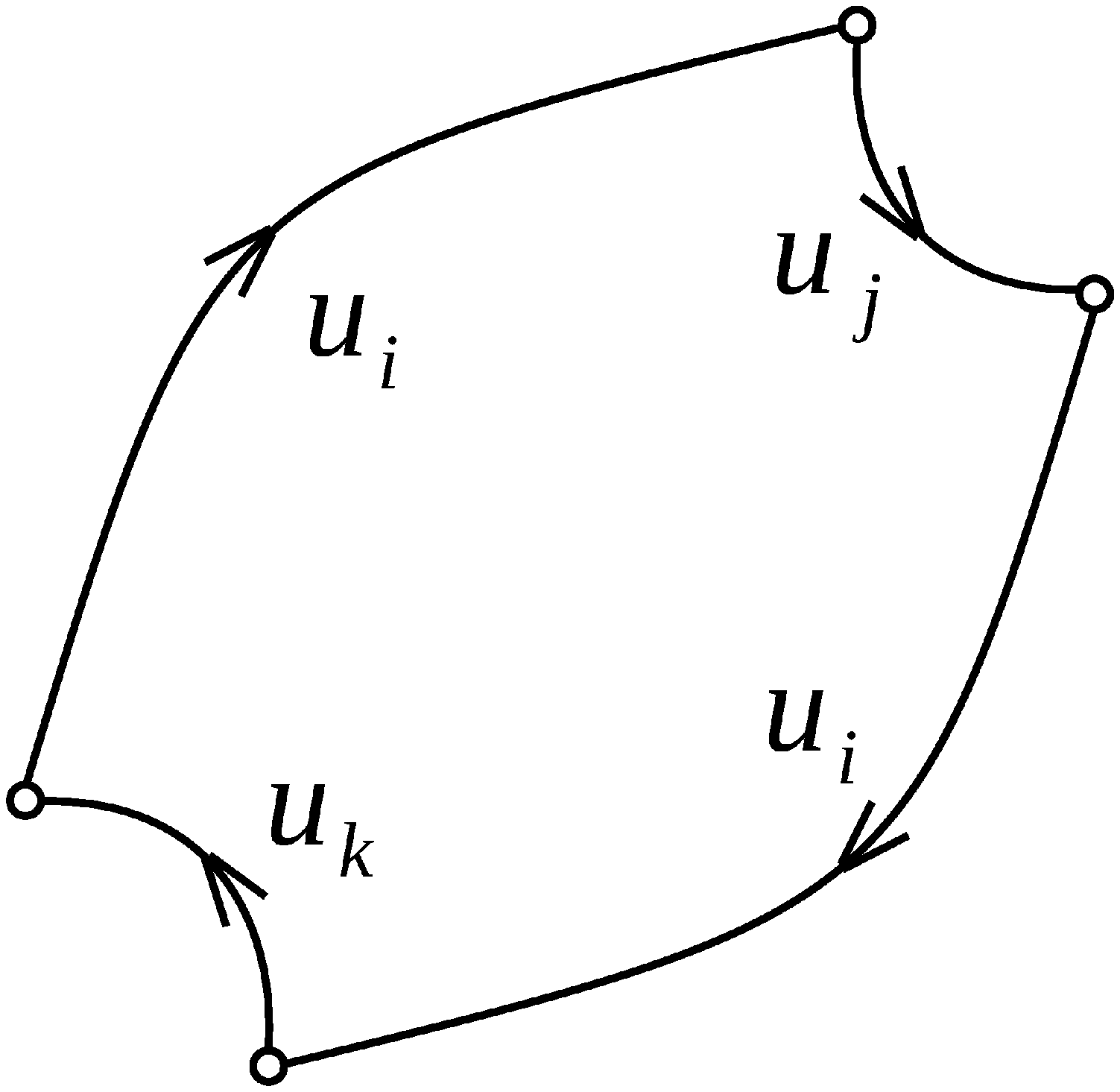}
		\caption{Four-link cycle $Q^i$}
	\end{subfigure}
	\ 
	\begin{subfigure}[b]{0.28\textwidth}
		\centering
		\includegraphics[width=0.6\textwidth]{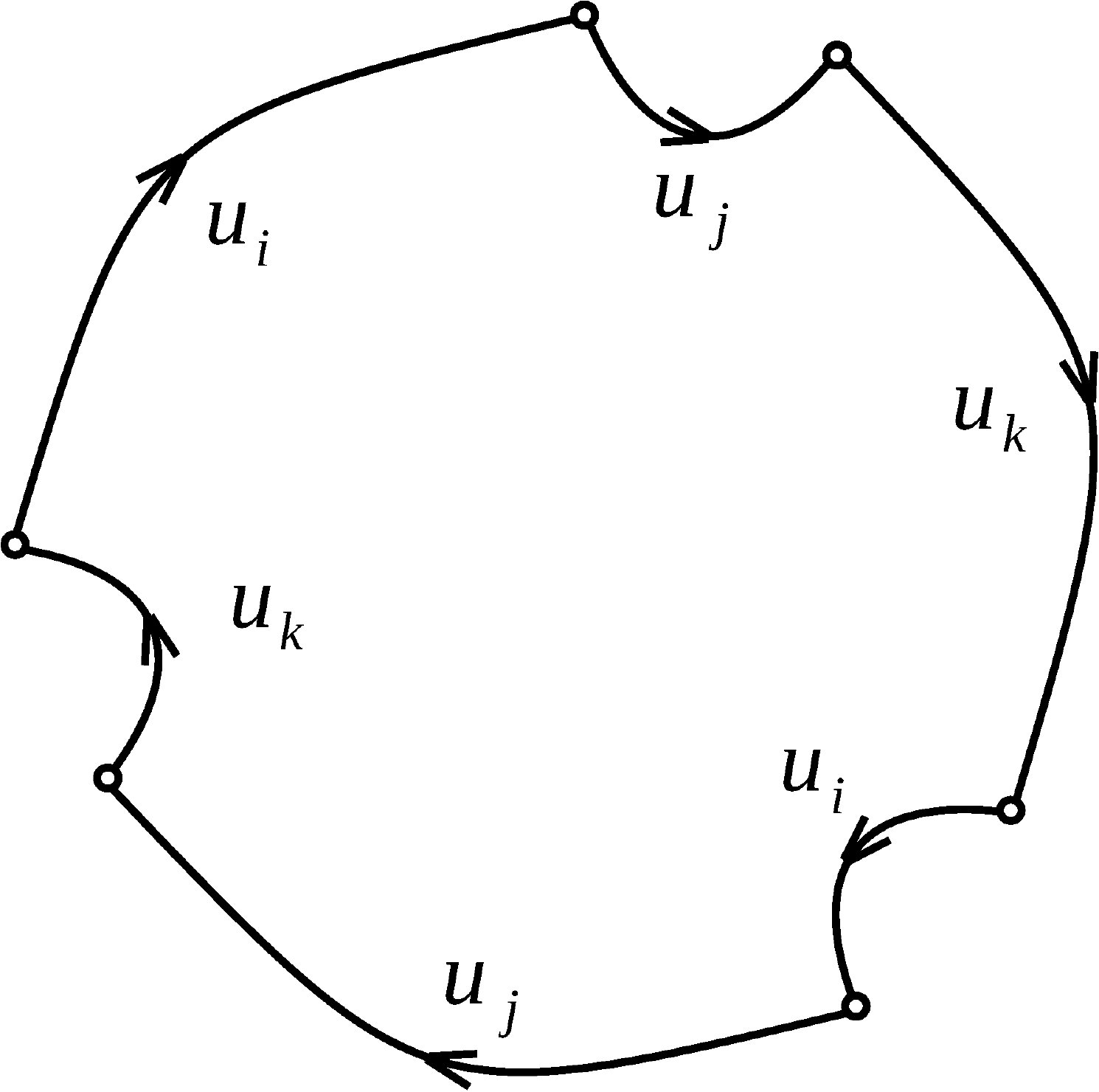}
		\caption{Six-link cycle $R^\pm$}
	\end{subfigure}
	\caption{Schematic depicting of the periodic trajectories $Z^\pm$, $Q^i$, and $R^\pm$ on the sphere $M_+/g$.}
	\label{fig:3_4_6_cycles}
\end{figure}

\begin{lemma}[\cite{ZMHBasic}]
\label{lm:model_probel_3_4_6_cycles}
   There exist the following periodic trajectories on $M_+/g$:
		\begin{enumerate}[(i)]
		\item two three-link cycles $Z^{\pm}$.
The control on each of these cycles runs cyclically through all three vertices of the triangle $\Omega$, clockwise $1\to2\to3\to1$	and counter-clockwise $1\to3\to2\to1$, respectively.			
        \item three four-link cycles $Q^i$, $i=1,2,3$.
The control on $Q^i$ alternates on the vertices of the triangle $\Omega$ in the following order:
$i\to j\to i\to k\to i$, where $j,k\ne i$.

		\item two six-link cycles $R^\pm$. The control on $R^+$ runs through all three vertices of the triangle $\Omega$ in the order $1\to2\to3\to1\to2\to3\to1$, and on
		$R^-$ in the reverse order.
	\end{enumerate}
\end{lemma}

\begin{proof}

We shall detail the proof from \cite{ZMHBasic} for item (i). It is rather difficult to directly solve the equation $\Phi^3(z)=z$, $z\in{\cal S}/g$, because the arising polynomials have too large degrees. To simplify the calculation we will use the discrete group of symmetries $S_3$ of the optimal synthesis, which was described in the subsection \ref{sec:symmetries}. Let
$\alpha=(123)\in S_3$. We assume that the sought cycle is invariant under the action of $\alpha$ and we shall find the solution of the corresponding simplified equation
	
		\begin{equation}
	\label{eq:automodel_alpha_Phi_z_eq_z}
		\bigl(\alpha\circ\Phi\bigr)(z)=z.
	\end{equation}
	
	\noindent The solution of the simplified equation will be evidently a solution of the initial equation $\Phi^3(z)=z$, because $\alpha$ and $\Phi$ commute.
		
Let some pre-image $\widetilde z\in\pi^{-1}(z)$ have the coordinates $\widetilde z = (x_0,y_0,\phi^0,\psi^0)$ and denote by
$\widetilde\gamma$ the trajectory of system (\ref{eq:model_pmp_system}) in $\cal M$ which passes through
$\widetilde z$ at $t = 0$, $\widetilde\gamma(0) = \widetilde z$.
Let us now describe the switching of the control. Suppose at the point $\widetilde z$ the control switches from the third vertex of the triangle $\Omega$ to the first one. This implies that
	
	\begin{equation}
	\label{eq:ineq_switch_point}
		\psi_1^0=\psi_3^0<\psi_2^0\quad\mbox{ and }\quad\phi_1^0>\phi_3^0.
	\end{equation}
	
	\noindent Since the control $u_0$ is constant on the arc from the point $z$ to the point $\Phi(z)$ and given by $u_0=(u^1_0,u^2_0,u^3_0) = (-2,1,1)$, we have that $\widetilde \gamma$ is given by	
		
	\begin{equation}
	\label{eq:x_y_phi_psi_ot_t}
		\left\{\begin{array}{l}
			y(t)    = y_0 + u_0t\\
			x(t)    = x_0 + y_0t + \frac{1}{2}u_0t^2\\
			\phi(t) = \phi^0 + x_0t + \frac{1}{2}y_0t^2 + \frac{1}{6}u_0t^3\\
			\psi(t) = \psi^0 - \phi^0t -\frac{1}{2}x_0t - \frac{1}{6}y_0t^2 - \frac{1}{24}u_0t^3\\
		\end{array}\right.
	\end{equation}
	
on this arc. Suppose the trajectory $\widetilde\gamma$ intersects the switching surface again at the time instant $t_0>0$, then $\alpha\bigr(\pi(\widetilde\gamma(t_0))\bigl) = z$, because $z$ satisfies equation (\ref{eq:automodel_alpha_Phi_z_eq_z}). In other words, there exists a number $\lambda_0>0$ such that
$\alpha(\widetilde\gamma(t_0)) = g(\lambda_0)\widetilde z$. 	
This last condition can be rewritten by virtue of (\ref{eq:x_y_phi_psi_ot_t}) as a system of algebraic equations that are linear in $x^0$, $y^0$, $\phi^0$, and
$\psi^0$, and polynomial in $t_0$ and $\lambda_0$. The solution of this system is defined only up to the action $g$. This follows from the fact that the point $\widetilde z\in\pi^{-1}(z)$ is defined only up to the action $g$ too.	 Hence, without loss of generality, one can set $t_0=1$. Solving the resulting system with respect to the linearly entering variables $x^0$, $y^0$, $\phi^0$, and $\psi^0$ and substituting the solution into the equality $\psi_1^0=\psi_3^0$ from (\ref{eq:ineq_switch_point}), we obtain a polynomial equation on $\lambda_0$:

	\[
		P_Z(\lambda_0) = \lambda_0^6 - 4\lambda_0^4 - 7\lambda_0^3 - 4 \lambda_0^2 + 1 = 0.
	\]
	
	\noindent By the Sturm method we obtain that the polynomial $P_Z(\lambda_0)$ has exactly one root in the interval $(0,1)$. This root yields a point $z$ satisfying inequalities (\ref{eq:ineq_switch_point}).
Hence on $M_+/g$ there exists a unique periodic trajectory $Z^+$ which is invariant with respect to the action of $\alpha$. The other trajectory $Z^-\in M_+/g$ is obtained from $Z^+$ by application of the reflection $(12)\in S_3$ or by replacement of the permutation $\alpha=(123)$ by $\alpha=(132)$.

	\smallskip

Items (ii) and (iii) are proved similarly to item (i). To obtain the four-link cycle on $M_+/g$ it is necessary to solve the equation $(\alpha\circ\Phi^2)(z)=z$, where $\alpha$
is one of the three transpositions $(12)$, $(23)$ or $(13)$ from $S_3$. The method leads to the polynomial equation

	\[
		\begin{array}{lll}
		P_Q(\lambda) =&
		1 + 4 \lambda + 60 \lambda^2 + 220 \lambda^3 - 607 \lambda^4 - 5080 \lambda^5 - 19\,700 \lambda^6 - 73\,944 \lambda^7 -\\
		&- 192\,258 \lambda^8 - 416\,272 \lambda^9 - 918\,956 \lambda^{10} - 1\,609\,184 \lambda^{11} - 2\,528\,300 \lambda^{12} -\\
		&- 4\,868\,880 \lambda^{13} - 5\,019\,696 \lambda^{14}
		- 10\,839\,184 \lambda^{15} - 8\,659\,545 \lambda^{16} -\\ &-18\,568\,404 \lambda^{17} - 12\,399\,696 \lambda^{18} - 27\,180\,572 \lambda^{19} - 14\,695\,579 \lambda^{20}\\
		& - 31\,988\,656 \lambda^{21} - 16\,556\,344 \lambda^{22} - 31\,988\,656 \lambda^{23} - 14\,695\,579 \lambda^{24} -\\
		&- 27\,180\,572 \lambda^{25} - 12\,399\,696 \lambda^{26} - 18\,568\,404 \lambda^{27} - 8\,659\,545 \lambda^{28} -\\
		&-10\,839\,184 \lambda^{29} - 5\,019\,696 \lambda^{30} - 4\,868\,880 \lambda^{31} - 2\,528\,300 \lambda^{32} - \\
		&-1\,609\,184 \lambda^{33} -918\,956 \lambda^{34} - 416\,272 \lambda^{35} - 192\,258 \lambda^{36} - 73\,944 \lambda^{37} -\\
		&-19\,700 \lambda^{38}- 5080 \lambda^{39} - 607 \lambda^{40} + 220 \lambda^{41} + 60 \lambda^{42} + 4 \lambda^{43} + \lambda^{44} = 0,\\
		\end{array}
	\]

	\noindent which also has a unique root in the interval
$(0,1)$. The coordinates of the points $z$ and $\Phi(z)$
satisfy inequalities which are similar to (\ref{eq:ineq_switch_point}), and consequently, on $M_+/g$
there exist three four-link periodic trajectories which are invariant with respect to a permutation of order 2.
	
To obtain the six-link cycles on $M_+/g$ it is necessary to solve the equation $(\alpha\circ\Phi^2)(z)=z$, where $\alpha$
is one of the cyclic permutations $(123)$, $(132)$ from $S_3$. As a result one obtains

	\[
		\begin{array}{ll}
			P_R(\lambda) =&
			\lambda^{20} - 12 \lambda^{19} + 30 \lambda^{18} + 66 \lambda^{17} - 117 \lambda^{16} - 504 \lambda^{15} - 207 \lambda^{14} +\\
			&+ 942 \lambda^{13} + 1271 \lambda^{12} - 390 \lambda^{11} - 1599 \lambda^{10} - 390 \lambda^9 + 1271 \lambda^8 + 942 \lambda^7 -\\
			&- 207 \lambda^6 - 504 \lambda^5 - 117 \lambda^4 + 66 \lambda^3 + 30 \lambda^2 - 12 \lambda + 1 = 0.\\
		\end{array}
	\]
	
	\noindent This polynomial also has a unique root in the interval $(0,1)$, and the coordinates of the points $z$ and
$\Phi(z)$ also satisfy inequalities which are similar to (\ref{eq:ineq_switch_point}). Consequently on $M_+/g$ we have two six-link periodic trajectories which are invariant with respect to a permutation of order 3.	
	\end{proof}

Note that the polynomials $P_Z$, $P_Q$ and $P_R$ are reciprocal. This is because the described method allows one to find self-similar trajectories on $\cal M$ disregarding the condition
$\lambda_0<1$. It remains to say that the mapping $g(-1)$
takes any self-similar trajectory on $\cal M$ to another self-similar trajectory, but it changes also the direction of the time flow. Hence $\lambda_0$ is replaced by $\frac{1}{\lambda_0}$, which also must be a root of the corresponding polynomial.

\subsection{Behaviour of optimal trajectories in the neighbourhood of periodic trajectories}
\label{subsec:periodic_2_4_6_fuller}

We shall show that almost all optimal trajectories of problem (\ref{problem:model}) are attracted in the backward time direction to one of the two three-link cycles $Z^\pm$ found in the preceding subsection, and in the forward time direction almost all optimal trajectories reach in finite time with a countable number of switches one of the three two-link cycles $Z_{ij}$. Here "almost all" means that on $M/g$ there exists a set of points $X$ of full Lebesgue measure which satisfy this property. The cycles $Q^i$ and $R^\pm$ do not belong to $X$. Moreover, there exists a countable number of periodic trajectories that do not belong to $X$. The set $(M/g)\setminus X$ has a non-integer dimension, and the behaviour of the trajectories in it is chaotic.

To obtain these results it is necessary to recall the results of \cite{ZMHBasic}. In this subsection we shall describe the results of \cite{ZMHBasic} along with the necessary explanations.

We begin with the behaviour of optimal trajectories in the vicinity of the cycles $Z^\pm$ and $Z_{ij}$. The cycles
$Z^\pm$ that have been found in the preceding subsection are repelling on the sphere $M/g$. More precisely, consider the Poincar\'e mapping $\Phi$  of the switching surface to itself. Each of the cycles $Z^\pm$ intersects the switching surface at three points $z^\pm_i$, $i=1,2,3$. Each such intersection point is a fixed point of the mapping $\Phi^3$. In the paper
\cite{ZMHBasic} the eigenvalues of the differential of the mapping $\Phi^3$ at the points $z^\pm_i$ have been calculated explicitly. These eigenvalues, naturally, do not depend on the choice of the point $z^\pm_i$ and all their absolute values are strictly greater than 1. Hence in the vicinity of $Z^\pm$ any trajectory will approach the corresponding cycle $Z^\pm$ in backward time. Let us remark that, generally speaking, this does not yet prove that almost all trajectories on
 $M/g$ tend in backward time to $Z^\pm$.

The behaviour of the optimal trajectories in the neighbourhood of the cycles $Z_{ij}$ has a different structure. As mentioned above, each two-link cycle consists of two parts
$Z_{ij}^n$ and $Z_{ij}^s$. The control on the interval
$Z_{ij}^n$ is non-singular and lies at the $k$ vertex of the triangle $\Omega$, $k\ne i,j$. The control on the interval
$Z_{ij}^s$ is semi-singular relative to the edge $(ij)$ and lies at its midpoint.

\smallskip

Consider a point $(x_0,y_0)\in M$ on $A_{ij}$, i.e., $x_0,y_0\in A_{ij}$. Hence by Lemma
\ref{lm:semi_singular} we have $\pi(x_0,y_0)\in Z_{ij}$. Suppose at first that the point $(x_0,y_0)$ is not a switching point and $\pi(x_0,y_0)$ lies on $Z_{ij}^n$. The control in $\pi(x_0,y_0)$ is the $k$ vertex of the triangle $\Omega$, $k\ne i,j$. Then for a point $(x_1,y_1)\in M$ sufficiently close to $(x_0,y_0)$,
the point $E(x_1,y_1)$ is sufficiently close to
$E(x_0,y_0)$ by Theorem \ref{thm:model_problem_bellman}. Consequently, the control on the optimal trajectory that emanates from $(x_1,y_1)$ coincides during some interval of time with the non-singular control on $Z_{ij}^n$.

Now suppose that $\pi(x_0,y_0)$ lies on $Z_{ij}^s$ and, consequently, the optimal control on $Z_{ij}^s$ is the midpoint of the edge $(ij)$ of the triangle $\Omega$. Denote
$(\phi^0,\psi^0)=E(x_0,y_0)$. Then $\psi^0_i=\psi^0_j>\psi^0_k$. If the point $(x_1,y_1)$ belongs to the neighbourhood of $(x_0,y_0)$, $(\phi^1,\psi^1)=E(x_1,y_1)$, and $\psi^1_i\ne\psi^1_j$, then due to the continuity of the mapping $E$, the control on the optimal trajectory passing through the point $(x_1,y_1)$ must be either in the $i$ or the $j$ vertex of the triangle $\Omega$.

Moreover, in  \cite{ZMHBasic} it has been shown that there exists a neighbourhood $V$ of the interval $Z_{ij}^s$ in $M/g$ with the following property. Let $(x_1,y_1)$ be an arbitrary point such that the image $\pi(x_1,y_1)$ lies in $V$. Then the optimal trajectory emanating from $(x_1,y_1)$ reaches $\pi^{-1}[Z_{ij}^s]$ in finite time without quitting the preimage $\pi^{-1}[V]$. Therefore the optimal control experiences a countable number of switchings between the $i$ and the $j$ vertex of the triangle $\Omega$. This process is similar to the chattering mode in the one-dimensional Fuller problem. The only difference is that the role of the interval of admissible controls is played by the edge $(ij)$ of the triangle
$\Omega$. A detailed proof and exact analytic
formulae can be found in \cite{ZMHBasic}.

\section{First theorem on the chaotic behaviour of trajectories in integral vortices}

In this section we prove Theorem \ref{thm:model_chaos_any_triangle} on the chaotic behaviour of trajectories in the integral vortex\footnote{By an {\it integral vortex} we mean the union of all trajectories of the dynamical system which pass through a given singular point.} of Hamiltonian system (\ref{eq:model_pmp_system}) in the case when the triangle $\Omega$ in model problem (\ref{problem:model}) is sufficiently close to an equilateral one. This theorem is the first in a series of theorems on the chaotic nature of trajectories in Hamiltonian systems with discontinuous right-hand side which are proven in this paper. Moreover, the main results on the structure of the optimal synthesis in problem (\ref{problem:model}) obtained in this section also serve as a starting point in the proof that such chaotic behaviour of trajectories in integral vortices is generic for high-dimensional Hamiltonian systems with discontinuous right-hand side. For the case of an equilateral triangle we shall prove Theorem \ref{thm:model_chaos_equilateral_triangle} in the subsequent sections, which is a generalization of Theorem \ref{thm:model_chaos_any_triangle} and which among other results furnishes bounds on the Hausdorff and Minkowski dimensions of the non-wandering set and the corresponding topological entropy.

\subsection{Statement of the first theorem on the chaotic behaviour of trajectories in the model problem}

Let us begin with the statement of the theorem on the chaotic nature of optimal trajectories in optimal control problem (\ref{problem:model}) and furnish all necessary explanations. The first three assertions of the theorem describe the set $\Xi \subset M_+$, which is the object of the theorem and is comprised of trajectories of Hamiltonian system (\ref{eq:model_pmp_system}). The last assertion describes the chaotic dynamics of the trajectories on this set. The set $\Xi$ is the analog of the set of non-wandering trajectories which is typical for integral vortices.

Every trajectory on $\Xi$ intersects the stratified manifold ${\cal S}={\cal S}_{12}\cup{\cal S}_{13}\cup{\cal S}_{23}$, which represents the set of discontinuities of the right-hand side of system (\ref{eq:model_pmp_system}), a countable number of times. In Theorem \ref{thm:model_chaos_any_triangle} we therefore describe the chaotic dynamics of these trajectories in terms of the intersections of the strata ${\cal S}_{ij}$. More precisely, the sequence of intersections of these strata can be encoded by elements of the space $\Sigma_{01}$ of bilaterally infinite words consisting of the letters 0 and 1, equipped with the standard direct product topology. The space $\Sigma_{01}$ is homeomorphic to Smale Horseshoe. Denote by $l:\Sigma_{01}\to\Sigma_{01}$ the topological Markov chain of the Bernoulli shift, i.e., let $l$ be the mapping that shifts every word on position to the left.

\begin{thm}
\label{thm:model_chaos_any_triangle}
There exists $\varepsilon>0$ such that if the angles of the triangle $\Omega$ in problem (\ref{problem:model}) are different from $\frac{\pi}{3}$ by no more than $\varepsilon$ and the distance from the center\footnote{As the center of $\Omega$ one may take, e.g., the orthocenter, the centroid, the center of the incircle or the circumcenter. This is possible because the triangle is close to equilateral and the distances between these points do not exceed $\varepsilon C\diam\Omega$ for some fixed $C > 0$.} of $\Omega$ to the origin does not exceed $\varepsilon\,\diam\Omega$, then there exists a subset $\Xi$ of the extended phase space ${\cal M}=T^*M=\{(x,y,\phi,\psi)\}$ of Hamiltonian system (\ref{eq:model_pmp_system}) with the following properties:
	
	\begin{enumerate}[(I)]
		\item For every point $z\in\Xi$ there exists a time instant $T(z)<\infty$ such that the trajectory $X(t,z)$ of Hamiltonian system (\ref{eq:model_pmp_system}) which passes through $z$ is well-defined and unique for all $t\in[-\infty,T(z)]$. Moreover, the trajectory $X(t,z)$ hits the origin in time $T(z)$, i.e., $X(T(z),z)=0$.
		
		\item The set $\Xi$ comprises trajectories of Hamiltonian system (\ref{eq:model_pmp_system}) and is invariant with respect to this system in the following sense. If $z\in \Xi$, then $X(t,z) \in \Xi$ for all $t\in[-\infty,T(z))$.
		
		\item The projection of the trajectory $X(t,z)$ on the phase space $M$, prolonged by 0 for $t>T(z)$, is optimal for all $z\in \Xi$ (i.e., $\Xi\subset M_+$). The trajectory $X(t,z)$ intersects the switching surface  $\cal S$ a countable, infinite number of times at time instants $\ldots<t_{-1}<t_0<t_1<t_2\ldots<T(z)$, i.e., $X(t_k,z)\in{\cal S}$, and where $t_0\le 0< t_1$, and $t_k\to T(z)$ as $k\to+\infty$ and $t_k\to-\infty$ as $k\to-\infty$.
		
		\item Consider the dynamical system defined by the map $\Phi:\Xi\cap{\cal S}\to\Xi\cap{\cal S}$, which takes a point $z\in\Xi$ on ${\cal S}$ to the next intersection point of the trajectory $X(t,z)$ with ${\cal S}$, i.e., $\Phi(z) = X(t_1,z)$. There exists an integer $n>0$, independent of the triangle $\Omega$, such that the map $\Phi^n$ is semi-conjugate to the topological Markov chain defined by the Bernoulli shift on the disjoint union of two copies of Smale Horseshoe. In other words, there exists a surjective continuous map $\Psi_{01}$ from $\Xi\cap{\cal S}$ to the space $\bigsqcup\limits^2 \Sigma_{01}$ such that the following diagram commutes:
		
		\begin{center}
		\begin{tikzpicture}[description/.style={fill=white,inner sep=2pt}]
		        \matrix (m) [matrix of math nodes, row sep=1em, column sep=2em, text height=1.5ex, text depth=0.25ex]
		        {    {\Xi\cap S}                  & {\Xi\cap S} \\
		             \bigsqcup\limits^2 \Sigma_{01} & \bigsqcup\limits^2 \Sigma_{01} \\ };
		        \path[->,font=\scriptsize]
		        (m-1-1) edge node[auto] {$\Phi^n$} (m-1-2)
		        (m-1-2) edge node[auto] {$\Psi_{01}$} (m-2-2)
		        (m-1-1) edge node[auto] {$\Psi_{01}$} (m-2-1)
		        (m-2-1) edge node[auto] {$l$} (m-2-2);
		\end{tikzpicture}
		\end{center}
		
		\noindent Here $l$ denotes the left shift on each copy of $\Sigma_{01}$.
		
	\end{enumerate}
\end{thm}

\subsection{Blowup of the singularity at the vertex of an integral vortex}
\label{subsec:blowing_procedure}

For the proof of Theorem \ref{thm:model_chaos_any_triangle} we need to perform a modified blowup of the origin. Form the topological viewpoint this means that we glue a sphere $S^7$ into the origin. For this construction $\Omega$ does not need to be an equilateral triangle. The set $\Omega$ can be an arbitrary convex compact set, the only condition is that $0\in\Int\Omega$, i.e., the conditions of Theorem \ref{thm:model_problem_bellman} have to be satisfied.

Let us note, however, that the optimal trajectories of problem (\ref{problem:model}) reach the origin in finite time. Hence after the blowup the velocity vector field of (\ref{eq:model_pmp_system}) will degenerate as we approach the sphere which was glued in.

\begin{defn}
\label{defn:blowing}
	The blowup of the singularity at the origin will be performed by a mapping

	\[
		\blowing:(x,y,\phi,\psi) \mapsto (\mu,\widetilde x,\widetilde y,\widetilde \phi,\widetilde \psi),
	\]

	\noindent where $\mu\in \R_+$, and $\widetilde x\in\R^2$, $\widetilde y\in\R^2$, $\widetilde \phi\in\R^2$, and $\widetilde \psi\in\R^2$ lie on the manifold

	\begin{equation}
	\label{eq:blowing_spheroid}
		{\cal C}_0=\bigl\{|\widetilde y|^{24}+|\widetilde x|^{12}+|\widetilde \phi|^8+|\widetilde \psi|^6 = 1\bigr\}\subset \R^8.
	\end{equation}

	\noindent The mapping $\blowing$ is given by the formulae

	\begin{equation}
	\label{eq:blowing_y_x_phi_psi}
		\widetilde y = y/\mu,\quad\widetilde x = x/\mu^2,\quad
		\widetilde \phi = \phi/\mu^3\ \ \mbox{and}\ \ \widetilde \psi = \psi/\mu^4,
	\end{equation}

	\noindent where

	\begin{equation}
	\label{eq:blowing_model_mu}
		\mu = \bigl(|y|^{24}+|x|^{12}+|\phi|^8+|\psi|^6\bigr)^\frac{1}{24}.
	\end{equation}
\end{defn}

Let us carry over the action $g$ of the group $\R_+$ to the cylinder given by the coordinates $(\mu,\widetilde x,\widetilde y,\widetilde \phi,\widetilde \psi)$ by virtue of $\blowing$, such that the map $\blowing$ is equivariant with respect to the action $g$,

\[
	\blowing \circ g(\lambda) \eqdef g(\lambda) \circ \blowing \quad \Longrightarrow \quad
	g(\lambda)\bigl(\mu,\widetilde x,\widetilde y,\widetilde \phi,\widetilde \psi\bigr) =
	(\lambda\mu,\widetilde x,\widetilde y,\widetilde \phi,\widetilde \psi)\quad \forall\lambda>0.
\]

The blowup of the singularity of system (\ref{eq:model_pmp_system}) at the origin has been performed by relations (\ref{eq:blowing_y_x_phi_psi},\ref{eq:blowing_model_mu}) precisely because the action $g$ of the group $\R_+$ can then be written in a very simple form.

\begin{defn}
\label{defn:cylinder_C_model_problem}
	Let us denote the cylinder ${\cal C}_0\times\{\mu\in\R\}$ over ${\cal C}_0$ by ${\cal C}$. The manifold ${\cal C}_0$ itself will be identified with the zero section, ${\cal C}_0={\cal C}\cap\{\mu=0\}$. The switching surfaces ${\cal S}_{ij}$ shall be prolonged on ${\cal C}_0$ in a natural way. By "vertical direction" we shall assume the direction of the tangent vector $\frac{\partial}{\partial\mu}$.
\end{defn}

\begin{lemma}
\label{lm:blowing_is_diffeo}
	The blowup map $\blowing$ is a diffeomorphism from ${\cal M}\setminus0$ onto ${\cal C}\cap\{\mu>0\}$.
\end{lemma}

\begin{proof}
	From relations (\ref{eq:blowing_y_x_phi_psi},\ref{eq:blowing_model_mu}) it easily follows that the map $\blowing$ is well-defined on ${\cal M}\setminus 0$ and maps ${\cal M}\setminus 0$ bijectively onto ${\cal C}\cap\{\mu>0\}$. Moreover, the map $\blowing$ is smooth on ${\cal M}\setminus 0$.
	
	Let us show that the differential $d\blowing$ is non-degenerate on the sphere given by $\{|y|^{24}+|x|^{12}+|\phi|^8+|\psi|^6=1\} = \blowing^{-1}[\{\mu=1\}]$. Regularity of $d\blowing$ at all other points will then follows from the equivariance of $\blowing$ with respect to the action $g$. The restriction of $\blowing$ on the spheroid $\blowing^{-1}[\{\mu=1\}]$ is by virtue of (\ref{eq:blowing_y_x_phi_psi}) a diffeomorphism. Hence the restriction $d\blowing|_{T_z\blowing^{-1}[\{\mu=1\}]}$ of the differential is non-degenerate at every point $z\in\blowing^{-1}[\{\mu=1\}]$. On the other hand, for $z\in{\cal M}\setminus 0$ we have
	
	\[
		\frac{\displaystyle d}{\displaystyle d\lambda} \bigl(\blowing(g(\lambda)z)\bigr)\big|_{\lambda=1} =
		\mu\frac{\displaystyle \partial}{\displaystyle \partial\mu}\big|_{\blowing(z)} \in T_{\blowing(z)}{\cal C}.
	\]
	
	\noindent If $z\in\blowing^{-1}[\{\mu=1\}]$, then the tangent vector $\frac{\displaystyle \partial}{\displaystyle \partial\mu}\big|_{\blowing(z)}$ does not lie in the tangent space $T({\cal C}\cap\{\mu=1\})$. Hence the map  $d\blowing|_z$ is surjective and therefore non-degenerate.
\end{proof}
Let us remark that the map $\blowing^{-1}$ is formally defined only on ${\cal C}\cap\{\mu>0\}$. However, we may define it on ${\cal C}\cap\{\mu<0\}$ by the same relations

\[
	\blowing^{-1}:(\mu,\widetilde x,\widetilde y,\widetilde \phi,\widetilde \psi) \mapsto
	(x,y,\phi,\psi),\mbox{ where }y=\mu \widetilde y,\ x=\mu^2\widetilde x,\ \phi=\mu^3\widetilde\phi,\ \psi=\mu^4\widetilde\psi.
\]

\noindent Then $\blowing^{-1}$ becomes a two-fold covering map over ${\cal M}\setminus 0$.

If the set $\Omega$ is a triangle, then a switching surface $\widetilde{\cal S}$ is defined on the cylinder ${\cal C}$. Namely, denote by $\widetilde {\cal S}_{ij}$ the closure of the set of all points in ${\cal C}$ which are taken to ${\cal S}_{ij}$ by the map $\blowing^{-1}$. In other words, the set $\widetilde{\cal S}_{ij}$ consists of those points $(\mu,\widetilde x,\widetilde y,\widetilde \phi,\widetilde \psi)$ such that $\argmax_{u\in\Omega}\bigl<\widetilde\psi,u\bigr>$ contains the edge $(ij)$ of the triangle $\Omega$. Let us define also $\widetilde{\cal S}_{123}=\widetilde{\cal S}_{12}\cap\widetilde{\cal S}_{23}\cap\widetilde{\cal S}_{13}$ and $\widetilde{\cal S}=\widetilde{\cal S}_{12}\cup\widetilde{\cal S}_{23}\cup\widetilde{\cal S}_{13}$. For ease of notation we will henceforth omit the tilde over $\widetilde{\cal S}_{ij}$, $\widetilde{\cal S}_{123}$, and $\widetilde{\cal S}$, since this will never lead to any confusion.

\subsection{Reparametrization of time}
\label{subsec:reparametrize_time}

The discontinuous vector field on the right-hand side of (\ref{eq:model_pmp_system}) can be written as follows in the coordinates $(\mu,\widetilde x,\widetilde y,\widetilde \phi,\widetilde \psi)$:

\begin{equation}
\label{eq:blowing_model_hamilton_vector_field}
	\left\{
		\begin{array}{l}
			\dot\mu = \Upsilon(\widetilde x,\widetilde y,\widetilde \phi,\widetilde \psi,u) =
				\frac{1}{24}\bigl(
					24|\widetilde y|^{22}\bigl<\widetilde y,u\bigr> +
					12|\widetilde x|^{10}\bigl<\widetilde x,\widetilde y\bigr> +
					8 |\widetilde \phi|^6\bigl<\widetilde \phi,\widetilde x\bigr> -
					6 |\widetilde \psi|^4\bigl<\widetilde \psi,\widetilde \phi\bigr>
				\bigr)\\
			\dot{\widetilde \psi} =
				\frac{1}{\mu} \bigl(-\widetilde\phi - 4\Upsilon\widetilde \psi\bigr);\\
			\dot{\widetilde \phi} =
				\frac{1}{\mu} \bigl(\widetilde x - 3\Upsilon\widetilde \phi\bigr);\\
			\dot{\widetilde x} =
				\frac{1}{\mu} \bigl(\widetilde y - 2\Upsilon\widetilde x\bigr);\\
			\dot{\widetilde y} =
				\frac{1}{\mu} \bigl(u - \Upsilon\widetilde y\bigr);\\
			\bigl<\widetilde\psi,u\bigr> \to \max\limits_{u\in\Omega}.
		\end{array}
	\right.
\end{equation}
Here the triangle $\Omega$ does not need to be equilateral.

\noindent The solutions of the ODEs (\ref{eq:blowing_model_hamilton_vector_field}) and (\ref{eq:model_pmp_system}) will be understood in the classical sense of Filippov (see \cite{Filippov}) as solutions of a system of ODEs with discontinuous right-hand side.

Denote the vector field ${\cal C}\cap\{\mu>0\}$ on the right-hand side of (\ref{eq:blowing_model_hamilton_vector_field}) by $\xi$. Formally the vector field $\xi$ is defined for $\mu>0$, but we prolong it by the same formulas onto the lower half ${\cal C}\cap\{\mu<0\}$ of the cylinder. Then $\blowing^{-1}$ will take the vector field $\xi$ to the vector field of system (\ref{eq:model_pmp_system}) for $\mu>0$ as well as for $\mu<0$.

Note that for $\mu\to 0$ the field $\xi$ grows as $\frac{1}{\mu}$. However, the field $\mu\xi$ can already be prolonged onto the zero section ${\cal C}_0={\cal C}\cap\{\mu=0\}$ of the cylinder ${\cal C}$ at all points where the covector $\widetilde\psi$ is not orthogonal to any of the edges of the triangle $\Omega$. The integral curves of the vector field $\mu\xi$ either do not intersect the zero section ${\cal C}_0$ or are contained in it. This is because we have $\dot\mu = \mu \Upsilon$ along the vector field $\mu\xi$, where the vector field $\Upsilon$ does not depend on $\mu$. Moreover, the components of the vector field $\mu\xi$ which correspond to the coordinates $(\widetilde x,\widetilde y,\widetilde \phi,\widetilde \psi)$ do not depend on $\mu$. Hence every integral curve of the vector field $\mu\xi$ which lies in ${\cal C}_0$ can be lifted in a unique manner to ${\cal C}\cap\{\mu\ne 0\}$ if the initial value of $\mu$ is given. On the other hand, every trajectory lying in ${\cal C}\cap\{\mu\ne 0\}$ projects to a trajectory in ${\cal C}_0$.

Hence the integral curves of the vector fields $\xi$ and $\mu\xi$ coincide on ${\cal C}\cap\{\mu\ne0\}$, but the velocity on the trajectories is different. Let us denote the time parameter for the movement on the trajectories of the vector field $\mu\xi$ by $s$. Then the parameters $s$ and $t$ are related by

\[
	ds=\frac{1}{\mu}dt.
\]

The time parameter $s$ allows to characterize the optimal trajectories by their limiting behaviour as they approach ${\cal C}_0$ and is hence of advantage in the subsequent considerations. This behaviour is described in the following two lemmas. In the first lemma we prove that the function $\mu(s)$ decays exponentially on any optimal trajectory, and in the second lemma we give a sufficient condition of optimality of a trajectory in terms of the decay rate of the quantity $\mu(s)$ on this trajectory.

\begin{lemma}
\label{lm:trajectory_exp_decrease}
	Consider the image of an optimal trajectory 
	
	$$\bigl(\widehat x(t,q_0),\widehat y(t,q_0),\widehat\phi(t,q_0),\widehat\psi(t,q_0)\bigr)$$
	
	\noindent on ${\cal C}\cap \{\mu>0\}$ for $t<T(q_0)$. Fix time instants $t_0=t(s_0)<T(q_0)$ and $t_1=t(s_1)<T(q_0)$. Then there exist positive constants $\gamma_1$ and $\gamma_2$, independent of the trajectory and of the choice of $t_0$ and $t_1$, such that
	
	\[
		D_1\mu_0 e^{-\gamma_1 (s_1-s_0)} \le \mu_1 \le D_2\mu_0 e^{-\gamma_2 (s_1-s_0)},
	\]

	\noindent where $\mu_k=\mu\bigl(\widehat x(t_k,q_0),\widehat y(t_k,q_0),\widehat\phi(t_k,q_0),\widehat\psi(t_k,q_0)\bigr)$, $k=0,1$, and $D_1=\frac{1}{D_2} = \frac{\gamma_2}{\gamma_1}$.
\end{lemma}

\begin{proof}
	Denote short-hand 
	
	$$\mu(t) = \mu \bigl(\widehat x(t,q_0),\widehat y(t,q_0),\widehat\phi(t,q_0),\widehat\psi(t,q_0)\bigr).$$ 
	
	\noindent The bounds obtained in Theorem \ref{thm:model_problem_bellman} immediately yield the existence of  $\gamma_1>0$ and $\gamma_2>0$ such that for every initial point $q_0$ we have
	
	\begin{equation}
	\label{eq:mu_estimate_via_t}
		\gamma_2\bigl(T(q_0) - t\bigr) \le \mu(t) \le \gamma_1\bigl(T(q_0) - t\bigr)
	\end{equation}

	\noindent on the optimal trajectory emanating from $q_0$. Hence
	
	\[
		s_1-s_0 = \int\limits_{t_0}^{t_1} \frac{1}{\mu(t)}\,dt\ \
		\begin{array}{l}
			\le \frac{1}{\gamma_2} \ln \frac{T(q_0)-t_0}{T(q_0)-t_1};
			\vspace{0.2cm}\\
			\ge \frac{1}{\gamma_1} \ln \frac{T(q_0)-t_0}{T(q_0)-t_1}.\\
		\end{array}
	\]

	\noindent It follows that
	
	\[
		e^{-\gamma_1(s_1-s_0)} \le \frac{T(q_0)-t_1}{T(q_0)-t_0}\le e^{-\gamma_2(s_1-s_0)}.
	\]

	The proof is concluded by application of bounds (\ref{eq:mu_estimate_via_t}).
\end{proof}

\begin{corollary}
	On any optimal trajectory we have $\mu\to +0$ and $s\to+\infty$ as $t\to T(q_0)-0$.
\end{corollary}

\begin{lemma}
\label{lm:trajectory_vanish_to_C0}
	Let $\widetilde z(s)$, $s\in\R$, be a trajectory of the vector field $\mu\xi$ on the upper half ${\cal C}\cap\{\mu>0\}$ of the cylinder ${\cal C}$. If $\mu(s)\to 0$ as $s\to+\infty$ and
	
	\[
		T=\int_0^{+\infty}\mu(s)ds<\infty,
	\]
	
	\noindent then on the trajectory $\widetilde z(s)$ we have $s(t)\to+\infty$ as $t\to T-0$ and $s(t)\to-\infty$ as $t\to-\infty$. Moreover, the trajectory $\blowing^{-1}\bigl(\widetilde z(s(t))\bigr)$ lies in $M_+$ and is optimal, given we prolong it by zero for $t\ge T$.
\end{lemma}

\begin{proof}
	Since $dt=\mu ds$, we have
	
	\[
		t=\int_0^s \mu(\sigma)d\sigma
	\]

	\noindent on the trajectory $\widetilde z(s)$. By positivity of $\mu(s)$ we have that $t\to T-0$ if and only if $s\to+\infty$.
	
	Now consider the trajectory $z(t)=\blowing^{-1}\bigl(\widetilde z(s(t))\bigr)$, which is defined for $t<T$. By virtue of (\ref{eq:blowing_model_mu}) we get $x,y,\phi,\psi\to0$ as $t\to T-0$. Since the trajectory $z(t)$ is a trajectory of Hamiltonian system (\ref{eq:model_pmp_system}) we have by virtue of Corollary \ref{corollary:M_plus_is_optimal} that $z(t)$ is optimal if prolonged by zero for $t\ge T$.
	
	The time when the origin is reached by an optimal trajectory can by Theorem \ref{thm:model_problem_bellman} be bounded from above by a function of the distance to the origin. Hence there exists a neighbourhood of 0 in $\cal M$ such that the trajectory $z(t)$ stays outside of this neighbourhood for all $t<0$. Hence the function $\mu(s)$ is bounded away from zero for $s<0$. It follows that $\int_0^{-\infty} \mu(s)d s=-\infty$, i.e., $t(s)\to-\infty$ as $s\to-\infty$.
\end{proof}

\begin{remark}
	Note that the section ${\cal C}_0$ can be identified with the quotient space $({\cal M}\setminus 0)/g$, and the trajectories of the vector field $\mu\xi$ can be interpreted as images of the trajectories of system (\ref{eq:model_pmp_system}) under the natural projection $({\cal M}\setminus 0)\to ({\cal M}\setminus 0)/g$. However, the concrete representation of $({\cal M}\setminus 0)/g$ as the section ${\cal C}_0$ has some advantages. It yields a system of ODEs, defined by the vector field $\mu\xi$, on the whole cylinder ${\cal C}$, it yields the uniform convergence of optimal trajectories to ${\cal C}_0$, and it open the possibility to define optimality by virtue of the preceding lemmas.
\end{remark}

\subsection{Robustness of the self-similar trajectories}

In this subsection we investigate what happens to Hamiltonian system (\ref{eq:model_pmp_system}) if the triangle $\Omega$ is replaced by a nearby triangle $\Omega'$.

Let the triangle $\Omega'$ be such that each vertex of $\Omega'$ is located in an $\varepsilon$-neighbourhood of the corresponding vertex of $\Omega$, where $\varepsilon > 0$ is sufficiently small. Then $0\in\Int\Omega'$ and  $\Omega'$ fulfills the requirements of problem (\ref{problem:model}). All objects which were introduced above and which are related to the triangle $\Omega'$ will be marked by a prime. E.g., the vector field (\ref{eq:blowing_model_hamilton_vector_field}) corresponding to the triangle $\Omega'$ will be denoted by $\xi'$.

Since for all $i\ne j$ the edges $(ij)$ and $(ij)'$ of the triangles $\Omega$ and $\Omega'$ are close to each other, the switching hyperplanes ${\cal S}_{ij}$ and ${\cal S}'_{ij}$ will be situated at an angle $\alpha_{ij}$ in extended phase space ${\cal M}$ which is close to zero. In order to define closeness of mappings defined on ${\cal S}_{ij}$ and ${\cal S}'_{ij}$ in the sequel, we will identify the hyperplanes ${\cal S}_{ij}$ and ${\cal S}'_{ij}$ by virtue of the map $\mathrm{pr}$ given by

\[
	\mathrm{pr}(x,y,\phi,\psi) = (x,y,\phi,O_{\alpha_{ij}}\psi)\in{\cal S}'_{ij},\mbox{ where }(x,y,\phi,\psi)\in{\cal S}_{ij},
\]

\noindent and $O_{\alpha_{ij}}\in O(2,\R)$ denotes a rotation of the plane by the angle $\alpha_{ij}$. The map $\mathrm{pr}$ commutes with the action $g$ of the group $\R_+$ and can be continued to the intersection ${\cal S}_{ij}\cap{\cal C}_0$ of the switching hyperplane with the zero section in a smooth manner.

\begin{lemma}
\label{lm:Poincare_transversal_robust}
	Consider a trajectory on ${\cal C}_0$ of the vector field $\mu\xi$ for the triangle $\Omega$. Suppose the trajectory intersects the switching surfaces ${\cal S}_{ij}\cap{\cal C}_0$ and ${\cal S}_{jk}\cap{\cal C}_0$ transversally\footnote{The jump of the vector field $\mu\xi$ at the point $A$ ($B$) on the hypersurface ${\cal S}_{ij}$ (${\cal S}_{jk}$) is tangent to this hypersurface. Therefore it is irrelevant for the definition of the transversality of the intersection from which side we take the limit of the velocity vector as we approach the switching surface.} at points $A$ and $B$, respectively, and assume $A,B\notin {\cal S}_{123}$. Here we may have $k=i$, but we assume $j\ne i,k$. Then for all $r\in\N$ and $\alpha>0$ there exists $\varepsilon>0$ such that if each vertex of some triangle $\Omega'$ is located in the $\varepsilon$-neighbourhood of the corresponding vertex of the triangle $\Omega$, then the Poincar\'e return maps $\Phi$ and $\Phi'$ defined by the vector fields $\mu\xi$ and $\mu\xi'$ are $\alpha$-close diffeomorphisms of neighbourhoods of $A$ and $A'=\mathrm{pr}\,A$ in the $C^r$ metric.
\end{lemma}

\begin{proof}
	In order to work in some neighbourhoods of ${\cal S}_{ij}$ and ${\cal S}_{jk}$ we have to extend the vector field $\mu\xi$ in a smooth manner onto these neighbourhoods. This can be accomplished by relaxing the last condition in  (\ref{eq:blowing_model_hamilton_vector_field}) and setting the control equal to the vertex $j$ of $\Omega$ throughout the neighbourhoods. We shall denote the resulting velocity vector field by $\mu\widetilde\xi$. Performing a similar operation with the system corresponding to the triangle $\Omega'$ we obtain a vector field $\mu\widetilde\xi'$.
	
	Since the switching surfaces are closed and the vector field $\mu\xi$ is transversal to ${\cal S}_{ij}$ and ${\cal S}_{jk}$ at $A$ and $B$, we have that every trajectory which emanates from a point on ${\cal S}_{ij}$ sufficiently close to $A$ will intersect ${\cal S}$ in the vicinity of $B$ for the first time. Hence the Poincar\'e return map will not change if we replace the vector field $\mu\xi$ by the vector field $\mu\widetilde\xi$. The relevant arcs of the trajectories, which come out of a neighbourhood of $A$ on ${\cal S}_{ij}$ and lead to a neighbourhood of $B$ on ${\cal S}_{jk}$ without control switchings, are left unchanged. A similar statement holds true if the vector field $\mu\xi'$ is replaced by the vector field $\mu\widetilde\xi'$.
	
	Since the vector fields $\mu\widetilde\xi$ and $\mu\widetilde\xi'$ are smooth, the assertion of the lemma follows immediately from the transversality of the vector field $\mu\widetilde\xi$ to the switching surfaces ${\cal S}_{ij}$ and ${\cal S}_{jk}$ at the points $A$ and $B$, respectively.
\end{proof}

Each trajectory $z(t)$ of system (\ref{eq:model_pmp_system}) can be put in correspondence to its image $\pi(z(t))$ on ${\cal C}_0={\cal M}/g$ under the canonical projection $\pi:{\cal M}\to{\cal M}/g$. This image can be found explicitly as follows. We send the trajectory $z(t)$ to the cylinder ${\cal C}$ by virtue of the map $\blowing$ and discard the coordinate $\mu$. Under this operation the self-similar trajectories transform to periodic trajectories of the vector field $\mu\xi$ on ${\cal C}_0$.

In this subsection we study the robust periodic trajectories on ${\cal C}_0$. We shall apply the obtained results to the six-link cycles $R^\pm$ which have been found in Lemma \ref{lm:model_probel_3_4_6_cycles}.

\begin{defn}
\label{defn:robust_automodel}
	We shall call an optimal self-similar trajectory $z(t)$ of system (\ref{eq:model_pmp_system}) {\it robust} if it is not semi-singular, does not intersect ${\cal S}_{123}$, and its image $\pi(z(t))$ on ${\cal C}_0$ is a robust periodic trajectory. Here the latter means that (i) $\pi(z(t))$ consists of a finite number of smooth pieces (links), and (ii) the differential of the Poincar\'e return map at the switching points of $\pi(z(t))$ does not have eigenvalues equal to 1, apart from the trivial eigenvalue 1 corresponding to the velocity vector along the trajectory.
\end{defn}

The definition above is well-defined, in the sense that the Poincar\'e return map along such a periodic trajectory is a smooth diffeomorphism by virtue of Lemmas \ref{lm:periodic_transversal_S} and \ref{lm:Poincare_transversal_robust}. It turns out that the robust self-similar optimal trajectories are preserved under small changes of the triangle $\Omega$.

\begin{lemma}
\label{lm:robust_automodel}
	Let $z(t)\in{\cal M}$ be a robust optimal self-similar trajectory of problem (\ref{problem:model}). Then for all $r\in\N$ and $\alpha>0$ there exists a number $\varepsilon>0$ such that if every vertex of some triangle $\Omega'$ is located in the $\varepsilon$-neighbourhood of the corresponding vertex of the triangle $\Omega$, then in problem (\ref{problem:model}) defined for the triangle $\Omega'$ there exists a robust optimal self-similar trajectory $z'(t)$ with the following properties. The periodic trajectories $\pi(z(t))$ and $\pi(z'(t))$ are $\alpha$-close in the $C^0$-metric, and the Poincar\'e return maps on the cylinder ${\cal C}$ defined for systems (\ref{eq:model_pmp_system}) with triangles $\Omega$ and $\Omega'$ are $\alpha$-close in the $C^r$-metric in the neighbourhood of the switching points of the trajectories $\pi(z(t))$ and $\pi(z'(t))$, respectively.
\end{lemma}

\begin{proof}
	Consider the periodic trajectory $\widetilde z(s)$ of the vector field $\mu\xi$ that is obtained as the vertical projection of the image $\blowing(z(t))$ on ${\cal C}_0$. Let $s_0$ be a time instant such that $\widetilde z(s)$ is smooth in a neighbourhood of $s=s_0$, and let $L\ni \widetilde z(s_0)$ be a small smooth piece of hypersurface of codimension 1 in ${\cal C}$ which is transversal to $\widetilde z(s)$. Since $\dot{\widetilde z}(s)\in T_{\widetilde z(s)}{\cal C}_0$, we have that the hypersurface $L$ is also transversal to ${\cal C}_0$. Denote by $\widehat\Phi$ the Poincar\'e return map $\widehat\Phi:L\to L$ along the trajectories of the vector field $\mu\xi$. Then the point $\widetilde z(s_0)$ is a fixed point of the map $\widehat\Phi$.
	
	Let now $s_1<\ldots<s_m$ be the points of non-smoothness (the switching points) on the trajectory $\widetilde z(s)$, labeled in consecutive manner and such that $s_1$ is the first point of non-smoothness after $s_0$. Denote by $L_i$ the intersection of a small $\delta$-neighbourhood of the point $\widetilde z(s_i)$ with the corresponding hypersurface of discontinuity ${\cal S}_{ij}$ of the vector field $\mu\xi$. Then the Poincar\'e return map $\widehat\Phi:L\to L$ can be written as a composition of the maps
	
	\[
		\widehat\Phi: L \xrightarrow{\Phi_0} L_1
		\xrightarrow{\Phi_1} L_2 \xrightarrow{\Phi_2} \ldots
		\xrightarrow{\Phi_{m-1}} L_m \xrightarrow{\Phi_m} L.
	\]

	\noindent It follows by virtue of Lemma \ref{lm:Poincare_transversal_robust} that $\widehat\Phi$ is a smooth diffeomorphism.
	
	Let $\widetilde \Phi$ be the restriction of $\widehat\Phi$ on ${\cal C}_0$. By virtue of the definition of robustness of the self-similar trajectory $z(t)$ the differential $d\widetilde \Phi$ is robust at the point $\widetilde z(s_0)$, i.e., it has no non-trivial eigenvalues equal to 1. The map $\widehat\Phi$ can be expressed explicitly in terms of $\widetilde \Phi$ by
	
	\begin{equation}
	\label{eq:blowed_Poincare_map}
		\widehat\Phi(\mu,\widetilde x,\widetilde y,\widetilde \phi,\widetilde \psi) =
		(\lambda_0\mu,\widetilde\Phi(\widetilde x,\widetilde y,\widetilde \phi,\widetilde \psi)),
	\end{equation}
	
	\noindent where $\lambda_0$ comes from the definition \ref{defn:automodel} of self-similarity. Hence the differential $d\widehat\Phi|_{z(s_0)}$ contains the additional eigenvalue $\lambda_0$ as compared to $d\widetilde\Phi|_{z(s_0)}$. By optimality of $z(t)$ we have $0<\lambda_0<1$ (see Remark \ref{rm:lambda_0_less_1}). Hence $d\widehat\Phi|_{z(s_0)}$ fulfills the conditions of robustness.
	
	\smallskip
	
	Let us now describe how the map $\widehat\Phi$ changes if the triangle $\Omega$ is replaced by a close-by triangle $\Omega'$. Consider the vector field $\mu\xi'$ on ${\cal C}$ which is obtained from system (\ref{eq:model_pmp_system}) with the triangle $\Omega'$. Since the triangles $\Omega$ and $\Omega'$ are close, we immediately obtain that the maps $\Phi_k$ and $\Phi_k'$ are close in the $C^r$-metric. Hence also the maps $\widehat\Phi$ and $\widehat\Phi'$ are close. Since the point $z(s_0)\in L$ is a robust fixed point of the map $\widehat\Phi$, it will be preserved by a small perturbation of $\widehat\Phi$ in the $C^1$-metric (see e.g., \cite{Katok}, Proposition 1.1.4).
\end{proof}

\begin{figure}
	\centering
	\includegraphics[width=0.33\textwidth]{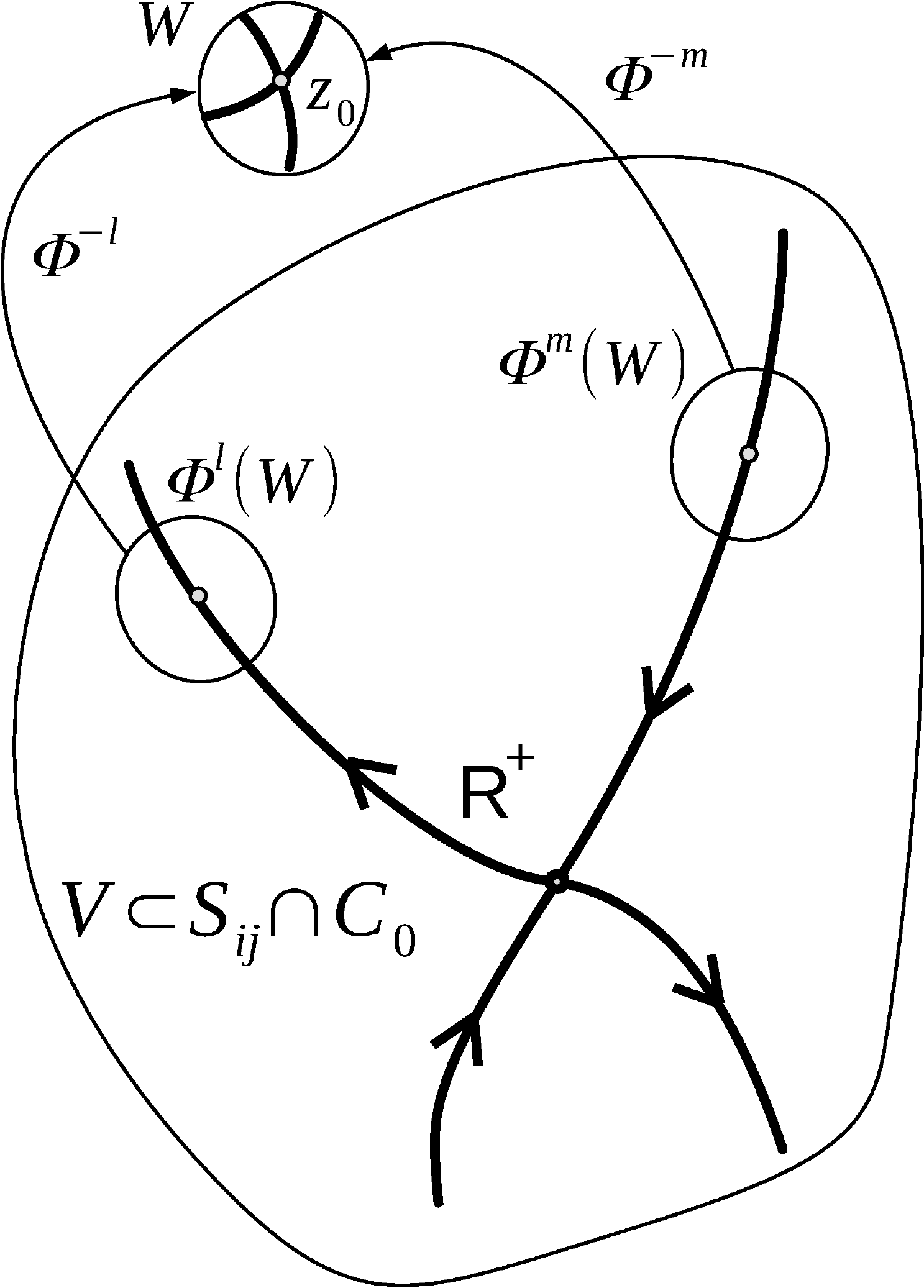}
	\caption{Schematic representation of the results of Lemma \ref{lm:homoclinic_point_any_triangle}.}
	\label{fig:homoclinic_point}
\end{figure}

Let us finally note that if the triangle $\Omega$ is contracted homothetically by a factor of $\lambda>0$, then the trajectories of system (\ref{eq:model_pmp_system}) will be left unchanged. What changes is the velocity on them, which will change by a factor of $\lambda$. A rotation of the triangle $\Omega$ about the origin leads to a rotation of the whole optimal synthesis by the same angle. Hence the assertions of the preceding lemma can be reformulated in terms of closeness of the angles and the centers of the triangles $\Omega$ and $\Omega'$, as it was done for the equilateral triangle in the formulation of Theorem \ref{thm:model_chaos_any_triangle}.

\subsection{Homoclinic orbit on the zero section of the cylinder \texorpdfstring{${\cal C}$}{C}}

In this subsection we construct a homoclinic orbit of the blown-up system (\ref{eq:blowing_model_hamilton_vector_field}) for the case of an equilateral triangle $\Omega$ and show that it is stable under small perturbations of the triangle $\Omega$.

Let hence $\Omega$ be an equilateral triangle with center in the origin. In the paper \cite{ZMHBasic} it has been shown that the self-similar trajectories $Z^\pm$, $Q^i$, and $R^\pm$ from Lemma \ref{lm:model_probel_3_4_6_cycles} are robust (see \cite{ZMHBasic}, Propositions 6.3, 6.4, and 6.5). Moreover, it has been shown that the periodic orbit $Z^\pm$ is repelling (more precisely, all nontrivial eigenvalues of the differential $d\widetilde\Phi$ are real and have modulus strictly greater than 1), and the periodic orbits $Q^i$ and $R^\pm$ are hyperbolic (more precisely, exactly one non-trivial eigenvalue of $d\widetilde\Phi$ has modulus strictly smaller than 1, all others have modulus strictly greater than 1).  Hence if the triangle $\Omega$ is close to equilateral and its center is close to the origin, then by virtue of Lemma
\ref{lm:robust_automodel} there still exist periodic orbits of system (\ref{eq:model_pmp_system}) on ${\cal C}_0$ which are close to the orbits $Z^\pm$, $Q^i$, and $R^\pm$. We shall keep the notations for these orbits. The Poincar\'e return maps along these periodic orbits are also close to those for the case of an equilateral triangle.

It follows that on ${\cal S}\cap{\cal C}_0$ there exists a 1-dimensional smooth stable manifold and a 5-dimensional smooth unstable manifold of the Poincar\'e return map $\widetilde \Phi:{\cal S}\to{\cal S}$ in the neighbourhood of each point of the intersections $Q^i\cap{\cal S}$ and $R^\pm\cap{\cal S}$. Indeed, by virtue of Lemma \ref{lm:Poincare_transversal_robust} the map $\widetilde\Phi:{\cal S}\to{\cal S}$ is smooth in the neighbourhood of the intersections $Q^i\cap{\cal S}$ and $R^\pm\cap{\cal S}$. But in general the map $\widetilde\Phi$ is discontinuous on $\cal S$. Moreover, if we extend the 1-dimensional stable manifold further away from the intersections $Q^i\cap{\cal S}$ and $R^\pm\cap{\cal S}$, it will stringently encounter the discontinuity set of the map $\widetilde\Phi$. Nevertheless, we have the following lemma.

\begin{lemma}
\label{lm:homoclinic_point_any_triangle}
	Assume that the triangle $\Omega$ is sufficiently close to an equilateral triangle with center at the origin. Then there exists a homoclinic point  $z_0\in {\cal C}_0$ on the switching surface $\cal S$ such that the iterations $\widetilde\Phi^n(z_0)$ tend to the six-link periodic orbit $R^+\cap{\cal S}$ as $n\to\pm\infty$ (see Fig.~\ref{fig:homoclinic_point}). Moreover, for every sufficiently small $\varepsilon$-neighbourhood $V\subset {\cal C}_0$ of $R^+\cap{\cal S}$ there exists a $\delta$-neighbourhood $W\subset {\cal C}_0 \cap{\cal S}$ of $z_0$ and integers $m>0$ and $l<0$ such that (i) the images ${\widetilde\Phi}^m(W)$ and ${\widetilde\Phi}^l(W)$ are contained in one connection component of $V$; (ii) the restrictions of the maps ${\widetilde\Phi}^m$ and ${\widetilde\Phi}^l$ to $W$ are diffeomorphisms; and (iii) the image under the map ${\widetilde\Phi}^{-m}$ of the stable manifold of the periodic orbit $R^+\cap{\cal S}$ in the ${\widetilde\Phi}^m(W)$ and the image under the map ${\widetilde\Phi}^{-l}$ of the unstable manifold in ${\widetilde\Phi}^l(W)$ intersect at the homoclinic point $z_0$, and this intersection is transversal.
	
	A similar statement holds for the periodic orbit $R^-\cap{\cal S}$.
\end{lemma}

\begin{proof}
	It is sufficient to prove the existence of a transversal homoclinic point $z_0$ for the case of an equilateral triangle $\Omega$ and to insure that the trajectory on which $z_0$ is situated is bounded away from the set ${\cal S}_{123}$ of discontinuity of the Poincar\'e return map. Both conditions have been verified by a numerical simulation of system (\ref{eq:blowing_model_hamilton_vector_field}) for the case of an equilateral triangle, as we shall detail below.

	Denote the 5-dimensional unstable manifold of the map $\tilde\Phi^6$ in the neighbourhood of some switching point $\tilde z(s_0)$ on the periodic orbit $R^+$ by $H_5 \subset {\cal C}_0 \cap {\cal S}$. Let $H_5^0$ be the connection component of the set $H_5 \setminus {\cal S}_{123}$ which contains the point $\tilde z(s_0)$. Then the restriction $\tilde\Phi^{-6}|_{H_5^0}$ is smooth, $\tilde\Phi^{-6}[H_5^0] \subset H_5^0$, and $\tilde z(s_0) \in H_5^0$ is the unique fixed point and attractor of the map $\tilde\Phi^{-6}$.

	Let $\sigma_0,\sigma_1,\dots,\sigma_6 = \sigma_0$ be the 1-dimensional stable manifolds of the map $\tilde\Phi^6$ emanating from the points $\tilde z(s_0),\tilde z(s_1),\dots,\tilde z(s_6)=\tilde z(s_0)$ of the periodic orbit $R_+$, respectively. Note that the map $\tilde\Phi$ takes the curve $\sigma_k$ to the curve $\sigma_{k+1}$ and simultaneously contracts the image curve towards the switching point $\tilde z(s_{k+1})$. The set of trajectories of system (\ref{eq:blowing_model_hamilton_vector_field}) which pass through the curves $\sigma_k$ form a cylindrical 2-dimensional surface $\Sigma$, which contains the periodic trajectory $R^+$. The curves $\sigma_k$ are a subset of the intersection $\Sigma \cap {\cal S}$.

	Let us consider the behaviour of the curve $\sigma_k$ when it encounters the discontinuity hypersurface ${\cal S}_{123} \subset {\cal S}$ of the map $\tilde\Phi$. Recall that ${\cal S}_{123}$ is the intersection of the three strata ${\cal S}_{12},{\cal S}_{13},{\cal S}_{23}$ of the switching surface ${\cal S}$. Assume that the trajectories of system (\ref{eq:blowing_model_hamilton_vector_field}) intersect the switching surface ${\cal S}$ transversally in the neighbourhood of some point $z \in {\cal S}_{123}$. Then these trajectories experience one switching on one side with respect to ${\cal S}_{123}$ and two consecutive switchings on the other side. This situation is schematically depicted in Fig.~\ref{Fig3Strata}. Hence the curve  $\sigma_k$ either splits into two branches or two branches merge into one for each encounter of ${\cal S}_{123}$. The resulting net of curves still lies on the cylindrical surface $\Sigma$ and equals the intersection $\Sigma \cap {\cal S}$. Let us remark that the preimages of the branching points under the map $\tilde\Phi$ are points of non-smoothness of the curves making up the set $\Sigma \cap {\cal S}$.

	\begin{figure}
		\begin{center}
		\begin{subfigure}[t]{0.45\textwidth}
			\begin{center}
			\includegraphics[width=\textwidth]{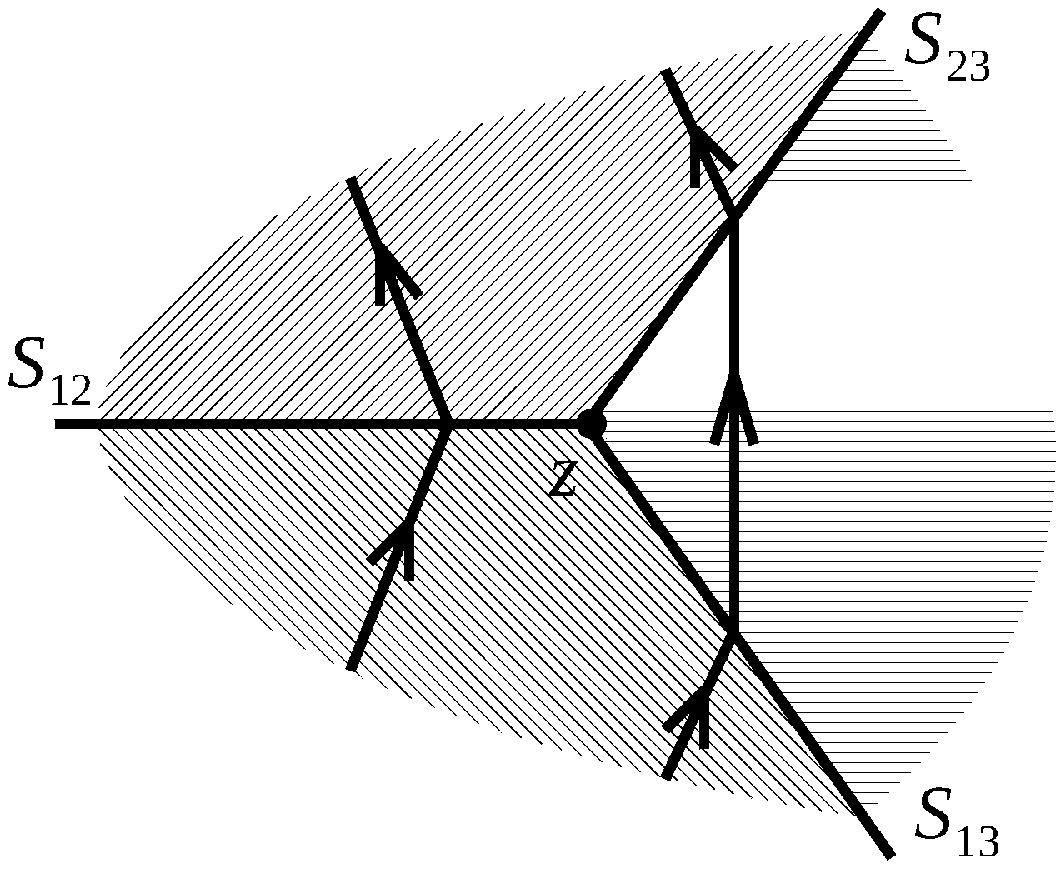}
			\end{center}
			\caption{Behaviour of trajectories in the vicinity of ${\cal S}_{123}$}
			\label{Fig3Strata}
		\end{subfigure}
		\ \ \
		\begin{subfigure}[t]{0.45\textwidth}
			\begin{center}
			\includegraphics[width=\textwidth]{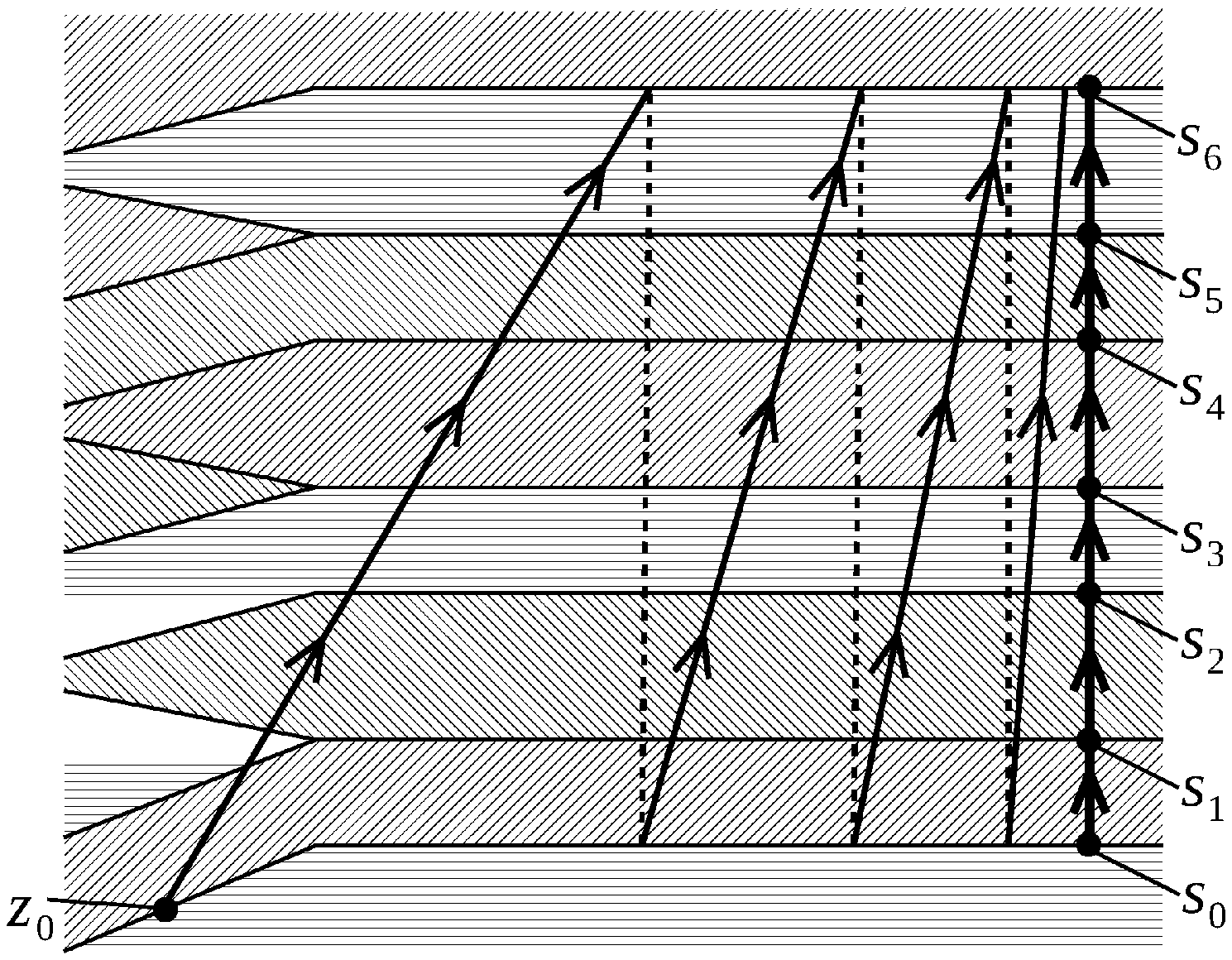}
			\end{center}
			\caption{Cylindrical surface $\Sigma$, curves $\sigma_k$, homoclinic point $z_0$ and trajectory of system (\ref{eq:blowing_model_hamilton_vector_field}) passing through $z_0$}
			\label{Fig7Switch}
		\end{subfigure}
		\end{center}
		\caption{Schematic representation of the curves $\sigma_k$.}
	\end{figure}

	If we track the curves $\sigma_k$ in the direction pointing away from the switching points $\tilde z(s_k)$ of the periodic orbit $R^+$, then we see that the curves $\sigma_k$ with odd index $k$ pass through a branching point, and the curves $\sigma_k$ with even index $k$ encounter a point of non-smoothness, which is the preimage of the above-mentioned branching points\footnote{Whether the curves with odd indices encounter branching points and the curves with even indices non-smoothness points or vice versa depends of course on the indexation of the switching points on $R_+$ and has no intrinsic meaning. The point is that the branching points and the points of non-smoothness alternate.}. After passage through the point of non-smoothness the curve $\sigma_0$ transversally intersects the unstable manifold $H_5^0$, thus defining the sought homoclinic point $z_0$. This situation is schmatically depicted in Fig.~\ref{Fig7Switch}. On this figure we see the cylindrical surface $\Sigma$. In order to obtain the cylindrical topology, one should identify the top and the bottom of the picture. The hatchings encode the control used in system (\ref{eq:blowing_model_hamilton_vector_field}), taking values in the three vertices of the equilateral triangle $\Omega$. The periodic orbit $R^+$ is represented by the bold vertical line, the curves $\sigma_k$ by the horizontal semi-bold lines. The thin line represents the trajectory passing through the point $z_0$, the dashed lines symbolize the identification of the intersection points of this trajectory with the curves $\sigma_0$ and $\sigma_6$, respectively. The periodicity of the whole picture follows from the invariance of the periodic orbit $R^+$ with respect to a permutation of order 3 of the symmetry group $S_3$.

	From the above it follows that we can set the numbers $m,l$, e.g., to $m = 7$ and $l = 0$. Moreover, it is clear that the trajectory passing through $z_0$ is bounded away from the discontinuity surface ${\cal S}_{123}$.

	The proof for the periodic orbit $R^-$ follows similar lines.

\end{proof}

Let us remark that the neighbourhood $W$ in Lemma \ref{lm:homoclinic_point_any_triangle} can be replaced by an arbitrarily small neighbourhood that is contained in $W$.

\subsection{Conclusion of the proof of the first theorem on the chaotic behaviour}

It is well-known that a homoclinic point generates a Smale Horseshoe in the neighbourhood of a periodic point. In our case the Poincar\'e return map is formally non-smooth, but Lemma \ref{lm:homoclinic_point_any_triangle} allows to ignore the discontinuities of the return map and to directly use  the classical theorem stating that a transversal homoclinic point generates a Smale Horseshoe in an arbitrarily small neighbourhood of the periodic point.

\begin{proof}[Proof of Theorem \ref{thm:model_chaos_any_triangle}]
	We shall start with the construction of the set $\Xi$. For this we need to lift the Smale Horseshoe in the neighbourhood of the 6-point periodic orbit $R_+\cap{\cal S}$ of the map $\widetilde\Phi$ from the zero section ${\cal C}_0$ to the whole cylinder $\cal C$ and to show that all trajectories in this lift tend to ${\cal C}_0$ and satisfy the conditions of Lemma \ref{lm:trajectory_vanish_to_C0}.

	Consider the map $\Phi^6$, defined on ${\cal S}\subset {\cal C}$ in the neighbourhood of the points of $R_+\cap{\cal S}$. Since the cycle $R_+$ is obtained from a self-similar trajectory with $\lambda_0<1$, we get by virtue of (\ref{eq:blowed_Poincare_map}) that $d\Phi^6(\frac{\partial}{\partial \mu})=\lambda_0\frac{\partial}{\partial\mu}$ at the points of $R_+\cap{\cal S}$. Let us fix one of the six points of $R_+\cap{\cal S}$ and consider its neighbourhood in which the map $\Phi^6$ is contracting in the vertical direction. This neighbourhood can be chosen as a cylinder over some neighbourhood in ${\cal C}_0$, since by virtue of (\ref{eq:blowed_Poincare_map}) the contraction condition does not depend on $\mu$.
	
	Let hence $V_0\subset V$ be a neighbourhood of one of the six points of $R_+\cap{\cal S}$ in ${\cal C}_0$ such that the map $\Phi^6$ is contracting by a factor of $\widetilde\lambda_0<1$ on $V_0\times\{\mu\in\R\}$ in the vertical direction. By virtue of Lemma \ref{lm:homoclinic_point_any_triangle} on $V_0$ there exists a Smale Horseshoe $\Lambda_0$ for some iteration $(\widetilde\Phi^6)^{N'}$ (see, e.g., \cite{Katok}, Theorem 6.5.5). Set $N=6N'$.
	
	The Smale Horseshoe $\Lambda_0$ is partitioned in two subsets as follows. Define two subsets $W_0$ and $W_1$ of the neighbourhood $V$. The set $W_0$ is defined as $\Phi^l(W)$, where the neighbourhood $W\ni z_0$ and the power $l$ are taken from Lemma \ref{lm:homoclinic_point_any_triangle} (see Fig.~\ref{fig:homoclinic_point}). Hence for every point $z\in W_0$ the relation $\Phi^{m-l}(z)\in V$ holds, and the map $\Phi^{m-l}|_{W_0}$ is a diffeomorphism. The set $W_1$ consists of those points $z\in V$ which do not leave this neighbourhood of $R_+\cap{\cal S}$ under the iterations $\Phi^k$, $0\le k\le m-l$, i.e., $W_1=\bigcap_{k=0}^{m-l}\Phi^{-k}(V)$. The map $\Phi^{m-l}|_{W_1}$ is a diffeomorphism. Clearly $W_0\cap W_1=\emptyset$. Hence the Smale Horseshoe $\Lambda_0$ consists of two disjoint subsets, $\Lambda_0 = (\Lambda_0\cap W_0)\sqcup(\Lambda_0\cap W_1)$.
	
	Let $z\in (W_0\sqcup W_1)\cap\Phi^{-N}(W_0\sqcup W_1)$. Consider the trajectory of the vector field $\mu\xi$, i.e., of system (\ref{eq:blowing_model_hamilton_vector_field}), emanating from $z$. Denote by $S(z)$ the time (in the parametrization by $s$) which the trajectory needs in order to cross $N$ consecutive times the strata ${\cal S}_{ij}$ of the switching surface. Since the restrictions of the map $\Phi^{m-l}$ to $W_0$ and $W_1$ are diffeomorphisms and the trajectories of the vector field $\mu\xi$ in $(W_0\sqcup W_1)\cap\Phi^{-N}(W_0\sqcup W_1)$ intersect the switching surface ${\cal S}$ transversally, we have the bounds
	
	\begin{equation}
	\label{eq:Smin_Sz_Smax}
		0<S_{\min} \le S(z) \le S_{\max}
	\end{equation}
for some constants $S_{\min}$ and $S_{\max}$.

Note that for every point of the Smale Horseshoe $\Lambda_0$ the trajectory of the vector field $\mu\xi$ through this point exists and is unique for all $s\in(-\infty,+\infty)$. Indeed, every power $\Phi^n$ of the Poincar\'e return map is well-defined on the points of $\Lambda_0$, since  $\Lambda_0\subset V_0\cap\Phi^{-N}(V_0)$ and $\Phi^N(\Lambda_0)=\Lambda_0$. Moreover, the trajectory of the vector field $\mu\xi$ emanating from an arbitrary point of $\Lambda_0$ intersects the switching surface only transversally and is hence unique. Its existence for all $s\in(-\infty,+\infty)$ follows from the bounds (\ref{eq:Smin_Sz_Smax}) on the transition times.
	
	To every point $z\in\Lambda_0$ of the Smale Horseshoe we put in correspondence an infinite sequence $\Psi(z)\in\Sigma_{01}$ consisting of 0 and 1. We define the position $j\in\Z$ of $\Psi(z)$ to be equal to 0 if all iterations $\Phi^k(z)$ for $k$ between $jN$ and $(j+1)N$ lie in $V_0$, and we set it equal to 1 if $\Phi^k(z)$ is in the neighbourhood of the homoclinic point $z_0$ for some $k\in(jN,(J+1)N)$. Note that the iterations $\Phi^{jN}(z)$ lie in $V_0$ in any case.
	
	The restriction of $\Phi^{N}$ onto the cylinder over $\Lambda_0\times\{\mu\in\R\}$ in $\cal C$ is in general not contracting in the vertical direction, regardless of the fact that $\Phi^6$ is contracting in the neighbourhood of $R_+\cap{\cal S}$. This happens because during the $N$ iterations of the map $\Phi$ the points of the Smale Horseshoe $\Lambda_0$ may leave the neighbourhood of $R_+\cap{\cal S}$ and visit the neighbourhood of the homoclinic point $z_0$.

	Let $\lambda_{\max}$ be the maximal factor of expansion of the map $\Phi^N$ in the vertical direction for the points in the cylinder over $\Lambda_0$,
	
	$$
		\lambda_{\max} = \max\limits_{z\in\Lambda_0\times{\R\setminus 0}}\frac{\mu(\widetilde\Phi^N(z))}{\mu(z)}.
	$$
	
	\noindent The maximum is well-defined, because $\Lambda_0$ is compact, and the continuous function $\frac{\mu(\Phi^{N}(z))}{\mu(z)}$ does not change if we replace $z$ by $g(\lambda)z$ for any $\lambda\ne 0$.
	
	It follows that if the position $j$ of the sequence $\Psi(z)$ equals 0, then the map $\Phi^N$ is contracting in the vertical direction by a factor of $\widetilde\lambda_0^{N/6}$ in the neighbourhood of the point $\Phi^{jN}(z)$, because the iterations $\Phi^k(z)$ do not leave the neighbourhood of $R_+\cap{\cal S}$ for $k\in[jN,(j+1)N]$. However, if the position $j$ of the sequence $\Psi(z)$ equals 1, then the map $\Phi^N$ extends the vertical direction by a factor of at most $\lambda_{\max}$ in the neighbourhood of $\Phi^{jN}(z)$.
	
	Let us choose $K\in\N$, $K>2$ such that $\widetilde\lambda_0^{KN/6}\lambda_{\max}<1$. Consider the Smale Horseshoe $\Lambda_1\subset\Lambda_0$ consisting of those points $z\in\Lambda_0$ for which the sequence $\Psi(z)$ has zeros at all position indices which are not divisible by $K$. At the positions whose indices are divisible by $K$ we can still have 0 as well as 1. Then the restriction of the map $\Phi^{KN}$ on  $\Lambda_1\times\{\mu\in\R\}$ contracts the vertical direction by a factor of at least $\lambda_1=\widetilde\lambda_0^{KN/6}\lambda_{\max}<1$. Set the integer $n$ from item (IV) of Theorem \ref{thm:model_chaos_any_triangle} equal to $KN$.
	
	Define the set $\Xi_1$ as follows. Consider the upper half of the cylinder $\Lambda_1\times\R_+$ and pass a trajectory of the vector field $\mu\xi$ through every point of the obtained set for $s\in(-\infty,+\infty)$. Transport the obtained set to the original extended phase space $\cal M$ by virtue of the map $\blowing^{-1}$.
	
	In the neighbourhood of the trajectory $R_-$ there also exists a Smale Horseshoe $\Lambda_2$ for the map $\Phi^n$. Without restriction of generality we can take equal powers of the map $\Phi$ for the Horseshoes $\Lambda_1$ and $\Lambda_2$. Define a set $\Xi_2$ in a similar way as we defined $\Xi_1$. These two sets are disjoint. Indeed, for every point $z\in\Lambda_1$ the majority of the iterations $\Phi^k(z)$ lie in the neighbourhood of $R_+\cap{\cal S}$, namely at least $(K-1)N$ out of every $KN$, where $K>2$. A similar statement holds for the Horseshoe $\Lambda_2$ in the neighbourhood of $R_-\cap{\cal S}$, and we immediately obtain $\Xi_1\cap\Xi_2=\emptyset$.
	
	Define $\Xi=\Xi_1\sqcup\Xi_2$.
	
	The time $T(y)$ in item (I) of Theorem \ref{thm:model_chaos_any_triangle} exists by virtue of Lemma \ref{lm:trajectory_vanish_to_C0}. Indeed, the map $\Phi^n:\Lambda_1\times\{\mu>0\}\to\Lambda_1\times\{\mu>0\}$ contracts the coordinate $\mu(z)$ of every point $z\in \Lambda_1\times\{\mu>0\}$ by a factor of at least $\lambda_1=\widetilde{\lambda_0}^{(K-1)N/6} \lambda_{\max}<1$. The parameter $s$ on the trajectory $X(z,s(t))$ grows by an amount of $KS(z)\in[KS_{\min},KS_{\max}]$. Hence the parameter $\mu(z(s))$ exponentially decays on the trajectory emanating from $z$. Therefore the conditions of Lemma \ref{lm:trajectory_vanish_to_C0} are satisfied. Uniqueness follows from the definition of the Poincar\'e return map.
	
	Items (II) and (III) of Theorem \ref{thm:model_chaos_any_triangle} follow from the construction of the set $\Xi$, optimality also follows from Lemma \ref{lm:trajectory_vanish_to_C0}. In order to prove item (IV) we define the sequence $\Psi_{01}(z)\in\Sigma_{01}$ for $z\in \Xi_1$ as the subsequence of $\Psi(\pi(z))$ consisting of the elements at positions with indices divisible by $K$, $\Psi_{01}(z)_j = \Psi(z)_{jK}$ for all $j\in\Z$. For $z\in\Xi_2$ the sequence $\Psi_{01}(z)$ lies in the second copy of the Smale Horseshoe $\Sigma_{01}$ and is defined in a similar way.
\end{proof}

Theorem \ref{thm:model_chaos_any_triangle} proven above allows to find the elements of the dynamical system defined by the Bernoulli shift $l:\Sigma_{01}\to\Sigma_{01}$ with positive entropy in the original model problem (\ref{problem:model}). For instance, we have the following corollary.

\begin{corollary}
	If the triangle $\Omega$ satisfies the conditions of Theorem \ref{thm:model_chaos_any_triangle}, then there exists a countable, infinite number of different 1-parametric families of self-similar trajectories in the optimal synthesis of problem (\ref{problem:model}).
\end{corollary}

\begin{proof}
	Consider an arbitrary periodic trajectory of the Bernoulli shift $l$. Clearly there exists an infinite number of such trajectories. The preimage of an arbitrary such trajectory under the mapping $\Psi_{01}$ yields a 1-parametric (with respect to the action $g$ of the group $\R_+$) family of self-similar trajectories.
\end{proof}

\section{Topological properties of the Poincar\'e return map}
\label{sec:topological_properties_of_Poincare_map}

In Theorem \ref{thm:model_chaos_any_triangle} we obtained a semi-local result on the existence of chaos in the optimal synthesis of problem (\ref{problem:model}) with an arbitrary triangle $\Omega$, given it is sufficiently close to an equilateral triangle centered at the origin. For the equilateral triangle the optimal synthesis can be described exactly. In Theorem \ref{thm:model_chaos_equilateral_triangle} we describe the non-wandering set of trajectories and the graph of the topological Markov chain conjugated with the Poincar\'e return map $\Phi$. We find bounds on the Hausdorff and Minkowski dimensions of the non-wandering set of trajectories and the corresponding topological entropy. In order to obtain these results, we start with the consideration of the global topological properties of the return map $\Phi$ of the switching surface ${\cal S}$ on itself.

In this section we assume that the triangle $\Omega$ is equilateral and centered on the origin, if not we state it explicitly.

\subsection{Topological structure of the switching surface}

The investigation of the structure of the optimal synthesis in the model problem (\ref{problem:model}) with equilateral triangle $\Omega$ is based on a thorough study of the Poincar\'e return map $\Phi$. As was stated earlier, the map $\Phi$ is defined on some subset of the switching surface. In this subsection we first give an exact topological description of the switching surface itself and of the subset of points on which the map $\Phi$ is defined.

\begin{figure}
	\centering
	\includegraphics[width=0.5\textwidth]{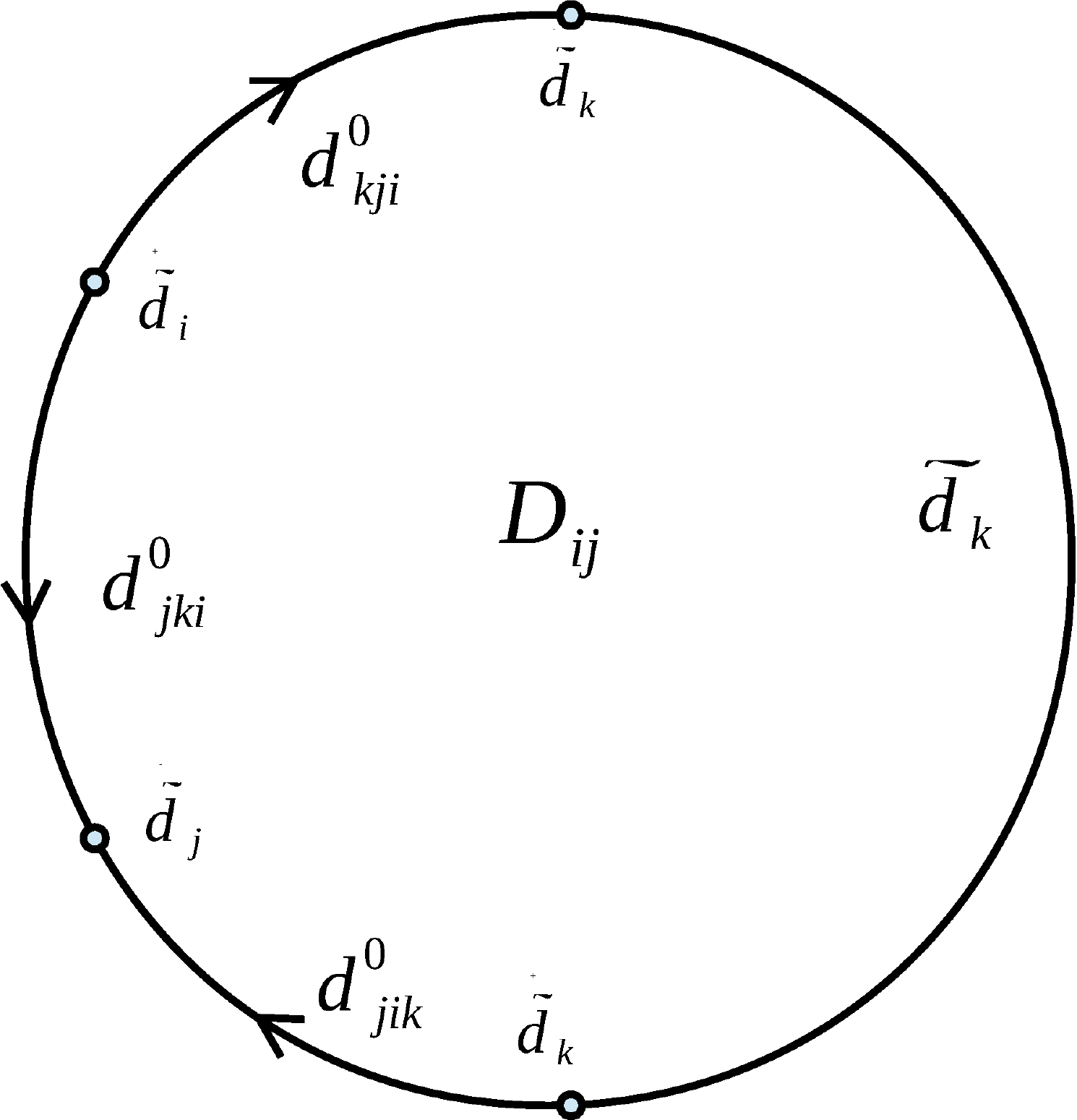}
	\caption{Schematic representation of the disc $D_{ij}$ of the switching surface $\pi({\cal S})$.}
	\label{fig:switch_serface}
\end{figure}

The switching surface $\cal S$ can be described by means of the variables $\psi_i$. Hence the intersection $M_+\cap{\cal S}$ is most easily described in the space of adjoint variables $N=\{(\phi,\psi)\}$. We shall identify this space with $M$ and $M_+$ by virtue of the map $E$. Let $\varepsilon:N/g\to N$ be a right inverse of the canonical projection $\pi:N\setminus 0\to N/g$. Let us for convenience define the embedding $\varepsilon$ by

\[
	\bigl(\varepsilon\circ\pi\bigr)(\phi,\psi) = (\mu^{-3}\phi,\mu^{-4}\psi)\mbox{, where }\mu = \sqrt[12]{|\phi|^4 + |\psi|^3}.
\]

\noindent We have chosen this way of defining $\varepsilon$ because then $\varepsilon$ commutes with $g$, and the projection $\pi$ can be imagined as mapping the point $(\phi_0,\psi_0)\in N\setminus 0$ to the intersection of its orbit $g(\R_+)(\phi_0,\psi_0)$ with the spheroid $\varepsilon(N/g) = \{(\phi,\psi):|\phi|^4 + |\psi|^3 = 1\}$. The space $N/g$ is endowed with the structure of a smooth Riemannian manifold, the map $\pi$ becomes a smooth surjection, and $\varepsilon$ is a smooth embedding.

The intersection ${\cal S}\cap M_+$ consists of three 2-dimensional half-spheres $D_k$, defined by the relations $\psi_i=\psi_j\le\psi_k$. These half-spheres intersect each other in the circle given by $\psi=0$,

\[
	D_k=\varepsilon(N/g)\cap\{(\phi,\psi)\in N: \psi_i=\psi_j\le \psi_k\}
\]

\noindent and

\[
	D_1\cap D_2\cap D_3 = \varepsilon(N/g)\cap\{(\phi,\psi)\in N:\psi=0\}.
\]

\noindent Each half-sphere $D_k$ is in turn divided into two discs, $D_k=D_{ij}\cup D_{ji}$, by the plane $\phi_i=\phi_j$, where

\[
	D_{ij} = \varepsilon(N/g)\cap\{(\phi,\psi)\in N: \psi_i=\psi_j\ge \psi_k\mbox{ and }\phi_i\le\phi_j\}.
\]
Define also
\[
	\tilde d_k = D_{ij}\cap D_{ji} = D_k\cap\{\phi_i=\phi_j\}.
\]

\noindent On every optimal trajectory which passes through a point in $D_{ij} \setminus \widetilde d_k$ the control switches from the vertex $i$ of the triangle $\Omega$ to the vertex $j$.

The circle $\varepsilon(N/g)\cap\{\psi=0\}$ is divided into 6 arcs (see Fig.~\ref{fig:switch_serface})

\[
	d_{ijk}^0=\varepsilon(N/g)\cap\{(\phi,\psi):\psi=0\mbox{ and }\phi_i\le\phi_j\le\phi_k\}.
\]

\noindent The arcs are separated by the points

\[
	\tilde d_k^+=\varepsilon(N/g)\cap\{\phi_k>\phi_i=\phi_j\}\mbox{ and }
	\tilde d_k^-=\varepsilon(N/g)\cap\{\phi_k<\phi_i=\phi_j\}.
\]

Thus the image $\pi({\cal S})$ of the switching surface on the quotient space $N/g$ is a 2-dimensional finite CW-complex with cells $D_{ij}$, $d_{ijk}^0$, $\tilde d_k$, and $d_k^\pm$. It is not hard to see that the surface $\pi({\cal S})$ is homotopically equivalent to the wedge sum of two 2-dimensional spheres. Let us choose orientations on the arcs $d_{ijk}^0$ as a function of the parity of the permutation $(ijk)$ as is shown on Fig.~\ref{fig:switch_serface}. The orientations will be useful when factoring with respect to the discrete symmetry group $S_3$.

The Poincar\'e return map $\Phi$ is then defined at all points of the image $\pi({\cal S})$ except those on $\widetilde{d_k}$.


Since the return map $\Phi$ commutes with the action of the group $S_3$, we may restrict our consideration to one of the discs $D_{ij}$. The action of $S_3$ on $\pi({\cal S})$ respects the structure of the cell complex and can be described as follows:

\begin{enumerate}
	\item The group $S_3$ acts freely and transitively on the 0-dimensional cells $d_i^+$ and $d_i^-$.
	\item The group $S_3$ acts freely and transitively on the 1-dimensional cells $d_{ijk}^0$ and preserves their orientation.
	\item A normal subgroup of $S_3$ acts freely and transitively on the 1-dimensional cells $\widetilde{d_k}$. The stabilizer of the cell $\widetilde{d_k}$ is the subgroup generated by the transposition $(ij)\in S_3$.
	\item The group $S_3$ acts freely and transitively on the 2-dimensional cells $D_{ij}$.
\end{enumerate}

\subsection{Factorization with respect to the action of the group \texorpdfstring{$S_3$}{S3}}

Before we start with the description of the properties of the map $\Phi$, let us first introduce some conventions which will facilitate the text and make it more comprehensible. The image $\Phi(z)$ of an arbitrary point $z\in D_{ij}$ lies either in $D_{jk}$ or in $D_{ji}$. By application of the corresponding permutation $\sigma\in S_3$ we may return it to the surface $D_{ij}$, $(\sigma\circ\Phi)(z)\in D_{ij}$. In this way the trajectory $\{\Phi^n(z)\}_{n \in \mathbb Z}$ of an arbitrary point $z\in D_{ij}$ can be represented on $D_{ij}$ only. In this section, if we say that the image of a point $z\in D_{ij}$ under the map $\Phi$ is some point $z'\in D_{ij}$, we always mean that the corresponding permutation $\sigma\in S_3$ has been applied. By some abuse of notation we shall nonetheless write $\Phi(z)=z'$.

We could also pass directly to the quotient space ${\cal S}/S_3$. The Poincar\'e return map $\Phi$ can be correctly carried over to ${\cal S}/S_3$, since it commutes with the action of $S_3$. However, the way described above, albeit equivalent, seems preferable to us, because it leaves the topology of the disc $D_{ij}$ unchanged.

We shall perform all computations in the quotient space $N/g$, identifying it with $M/g$ and $M_+/g$ by virtue of the map $E$ according to Theorem \ref{thm:model_problem_bellman}. Then the disc $D_{ij}\subset (N/g)$ can be seen as a subset of $N$, given explicitly by the system of inequalities

\[
	\left\{\begin{array}{l}
		|\psi|^2+|\phi|^2=1;\\
		\psi_i=\psi_j\le \psi_k;\\
		\phi_i\le\phi_j;\\
		\phi_1+\phi_2+\phi_3=0;\\
		\psi_1+\psi_2+\psi_3=0.\\
	\end{array}\right.
\]

Note that in the preceding subsection we used the normalization $|\psi|^3+|\phi|^4=1$ which has the advantage that it is compatible with the action $g$ of the group $\R_+$. But it will be more convenient for us to perform the numerical calculations with the normalization $|\psi|^2+|\phi|^2=1$.

On the disc $D_{ij}$ we introduce the coordinates

\[
	\left\{\begin{array}{l}
		\chi_1 = \frac{1}{2}(\phi_i-\phi_j) - \sqrt{3}\psi_i,\\
		\chi_2 = \sqrt{\frac{3}{2}}\phi_k.\\
	\end{array}\right.
\]

Then the disc $D_{ij}$ is mapped to the oblate disc $2\chi_1^2+\chi_2^2\le 1$, and the variables $\phi$ and $\psi$ can be recovered from the coordinates $\chi_1,\chi_2$ by the formulas

\[
	\left\{\begin{array}{ll}
	\phi_i= \frac{1}{2}\bigl(-\sqrt{2/3}\chi_2-\chi_1+\sqrt{1-\chi_1^2-\chi_2^2}\bigl);&
	\psi_i= -\frac{\chi_1+\sqrt{1-\chi_1^2-\chi_2^2}}{2\sqrt{3}}.\\
	\phi_j=\frac{1}{2}\bigl(-\sqrt{2/3}\chi_2-\chi_1+\sqrt{1-\chi_1^2-\chi_2^2}\bigr); &
	\psi_j= -\frac{\chi_1+\sqrt{1-\chi_1^2-\chi_2^2}}{2\sqrt{3}};\\
	\phi_k= \sqrt{\frac{2}{3}}\chi_2; &
	\psi_k= \frac{\chi_1+\sqrt{1-\chi_1^2-\chi_2^2}}{\sqrt{3}}.\\
	\end{array}\right.
\]

\smallskip


\subsection{Known pieces of the synthesis}

For definiteness, we shall consider the map $\Phi$ as acting on the disc $D_{31}$. In this subsection we describe the subsets of $D_{31}$ which correspond to pieces of the optimal synthesis described in the paper \cite{ZMHBasic}.

The arc $\tilde d_2$ borders a domain $C$ such that the trajectory emanating in the forward time direction from each switching point in $C$, with switching from control 3 to control 1, experiences an infinite number of switchings from control 1 to control 3 and back and then hits the semi-singular submanifold corresponding to the edge $[A_1 A_3]$ of the triangle $\Omega$. This domain $C$ is bounded by the arc $\tilde d_2$, part of the arc $d^0_{132}$, and some curve $c$ linking the point $\tilde d_2^-$ to the arc $d^0_{132}$. On Fig.~\ref{first_reg} the domain $C$ is painted in green. The Poincar\'e return map takes the domain $C$ onto a subset $\tilde C \subset C$. This map leaves the arc $\tilde d_2$ point-wise fixed, and the curve $c$ is taken to a subarc $\tilde c$ joining the point $\tilde d_2^-$. The image $\tilde C$ is painted in dark green on Fig.~\ref{first_reg}. The curve separating $\tilde C$ from its complement $C \setminus \tilde C$ joins the point $\tilde d_2^+$. This curve is the image under the Poincar\'e map $\Phi$ of the intersection of the arc $d^0_{132}$ with the boundary of $C$.

\begin{figure}[H]
\centering
\includegraphics[width=\textwidth]{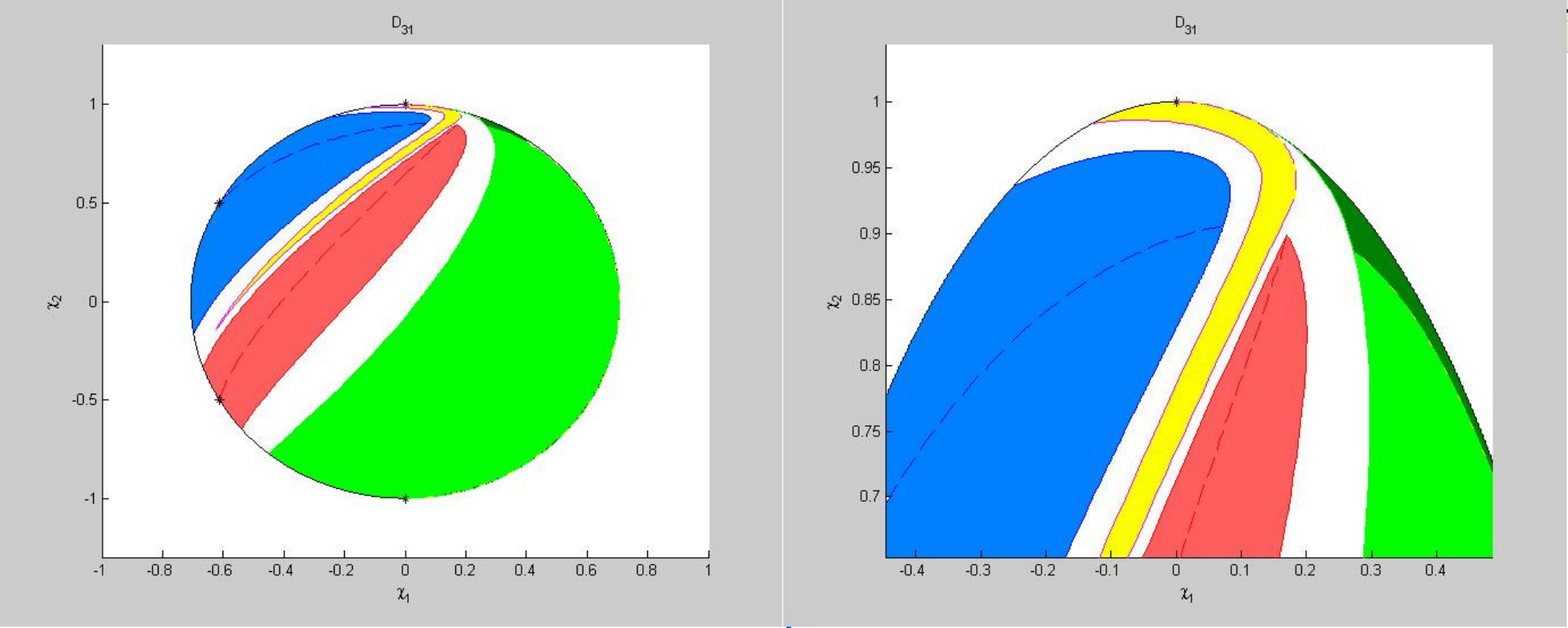}
\caption{Known regions}
\label{first_reg}
\end{figure}

The inverse $\Phi^{-1}$ of the Poincar\'e return map takes the complement $C \setminus \tilde C$, painted in light green on Fig.~\ref{first_reg}, consecutively to the domains painted in blue, red, and yellow. Recall that the Poincar\'e map in the original space makes no difference between the arcs $d^0_{132},d^0_{123}$. Thus $\Phi^{-1}$ takes the boundary of the domain $C \setminus \tilde C$ as follows to the boundary of the blue domain. The curve separating the domain $\tilde C$ from the complement $C \setminus \tilde C$ is taken to a sub-arc of the arc $d^0_{123}$. The sub-arc of $d^0_{132}$ maps to a sub-arc of $d^0_{213}$, and the intersection of the boundary of $C \setminus \tilde C$ with $D_{31}$ is mapped to the intersection of the boundary of the blue domain with $D_{31}$. 

The restriction of the inverse Poincar\'e map $\Phi^{-1}$ on the blue domain, having the red domain as its image, or the restriction of the map $\Phi$ on the red domain, having the blue domain as its image, is discontinuous in the topology of $D_{31}$. Namely, the intersections of the arcs $d^0_{132},d^0_{123}$ with the boundary of the red domain are glued together and are taken by $\Phi$ to the dark-blue dashed line, and the intersections of the arcs $d^0_{213},d^0_{123}$ with the boundary of the blue domain are glued together and taken by $\Phi^{-1}$ to the dark-red dashed line.

The map $\Phi^{-1}$ takes the boundary of the red domain as follows to the boundary of the yellow domain. The sub-arc of $d^0_{132}$ is taken to the sub-arc of $d^0_{213}$, and the sub-arc of $d^0_{123}$ is taken to the dashed part of the boundary of the yellow region.

Note also that in an arbitrarily small neighbourhood of the point $\tilde d_2^-$ there exist points which do not belong neither to the yellow nor to the green domain. From the results in paper \cite{ZMHBasic} it follows that all such points, given they are sufficiently close to $\tilde d_2^-$, are taken to the yellow domain by a finite power of the map $\Phi$. These mappings correspond to switchings from control 3 to control 1 and back. For each given point the number of switchings until the hit of the yellow region is finite, but it grows unbounded as the initial point in question approaches the boundary of the domain $C$.

This yields a complete description of the behaviour of the Poincar\'e return map in some neighbourhood of the closure of the arc $\tilde d_2$. Let us establish the correspondence with the domains on the quotient of the switching surface with respect to the Fuller group $g$ which are depicted on Fig.~4 of the paper \cite{ZMHBasic}. The stratum $C_{ij}^{ij}$ from Fig.~4 corresponds to the green domain, part of the yellow domain and other points in a sufficiently small neighbourhood of the point $\tilde d_2^-$.

\subsection{Global structure of the Poincar\'e map}

Let us now consider the image of the neighbourhood of the arcs $d^0_{132},d^0_{123}$ under the map $\Phi$. Recall that this map glues these neighbourhoods together along these arcs. The image of these arcs is given by some curve $\sigma_b$ joining the points $\tilde d_2^+$ and $\tilde d_3^+$. The dark-blue dashed line and the boundary between the dark green and the light green domains are parts of this curve. The curve $\sigma_b$ partitions the disc $D_{31}$ in two sub-domains. Let us denote the sub-domain containing the dark green region by $T_b$ and the sub-domain containing the light green region by $R_b$.

Consider the image of the neighbourhood of the arcs $d^0_{213},d^0_{123}$ under the inverse map $\Phi^{-1}$. This map glues these neighbourhoods together along these arcs and maps the arcs to some curve $\sigma_f$ joining the points $\tilde d_1^-$ and $\tilde d_2^-$. The dark red dashed line and the dashed part of the boundary of the yellow region are parts of this curve. The curve $\sigma_f$ partitions the disc $D_{31}$ in two sub-domains. Denote the sub-domain containing the green region by $T_f$, and the sub-domain containing the blue region by $R_f$.

Thus the map $\Phi$ is continuous on the sub-domains $T_f,R_f$ and takes them bijectively to the sub-domains $T_b,R_b$, respectively. If we extend the map $\Phi$ continuously to the boundary of $T_f$, then it maps in the following way to the boundary of the sub-domain $T_b$. The arc $\tilde d_2$ remains fixed, the arc $d^0_{132}$ maps to the curve $\sigma_b$, and the curve $\sigma_f$ maps to the arc $d^0_{213}$. If we extend the map $\Phi$ continuously to the boundary of $R_f$, this boundary is mapped as follows to the boundary of the sub-domain $R_b$. The arc $d^0_{213}$ is mapped to the arc $d^0_{132}$, the arc $d^0_{123}$ is mapped to the curve $\sigma_b$, and the curve $\sigma_f$ is mapped to the arc $d^0_{123}$. The Poincar\'e map $\Phi$ is schematically depicted in Fig.~\ref{Poincare_skizze}.

\begin{figure}[H]
\centering
\includegraphics[width=\textwidth]{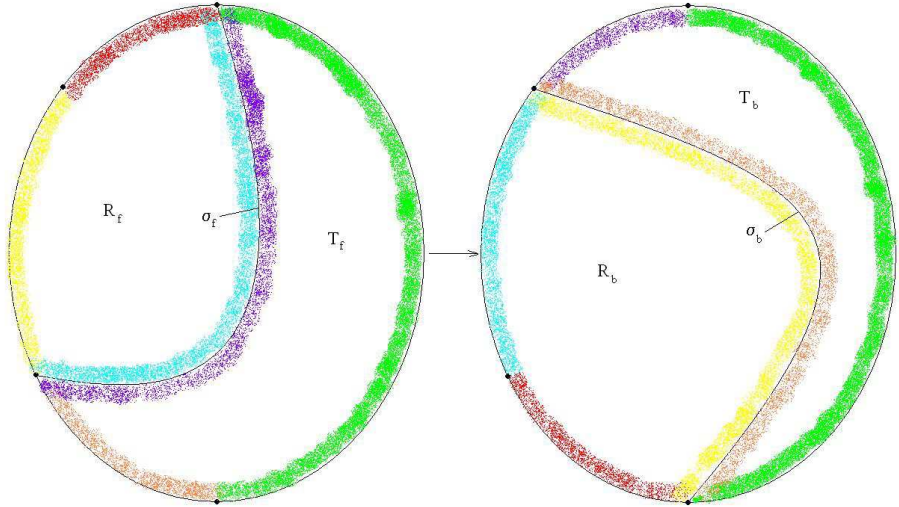}
\caption{Global structure of the Poincar\'e map}
\label{Poincare_skizze}
\end{figure}


The curves $\sigma_f$ and $\sigma_b$ intersect each other in one point. The coordinates of this point are solutions of a certain system of algebraic equations, because this point is taken to the yellow region by three-fold application of the map $\Phi$, and the variables $x,y$ can be written as explicit algebraic functions of the variables $\phi,\psi$ in the yellow region. Numerical calculations of the coordinates of the intersection point yield $\chi_1 \approx 0.174146349178$, $\chi_2 \approx 0.907326683652$. 

\subsection{Transitions of rotational type and transpositional type}

Consider the Poincar\'e return map $\Phi$ taking some switching point from control $i$ to control $j$ to some switching point from control $j$ to control $k$. Introduce the following definition.

If it happens that $i = k$, then we will call the corresponding transition {\it transpositional}. If on the other hand $i \not= k$, and hence $\{i,j,k\} = \{1,2,3\}$, then we call the transition {\it rotational}. Let us denote the types of transitions by $T$ (transposition) and $R$ (rotation), respectively.

The advantage of describing the sequence of controls on a given trajectory in the terms above is that it is invariant under the action of the permutation group $S_3$ which shuffles the vertices of the set of admissible controls $\Omega$.

Note that the Poincar\'e mapping from the sub-domain $T_f$ to $T_b$ induces a transpositional transition, while the Poincar\'e map from the sub-domain $R_f$ to $R_b$ induces a rotational transition. The indices $b$ and $f$ stand for {\it backward} and {\it forward}, respectively. A trajectory with a switching point, e.g., in the intersection $R_f \cap T_b$, will experience a rotational transition in the forward time direction and a transpositional transition in the backward time direction. Note also that every point in the green region $C$ experiences during its evolution in forward time only transpositional transitions, until it hits a semi-singular trajectory in finite time, but after an infinite number of switchings.

\subsection{An attractor in the backward time direction}

In \cite{ZMHBasic} the existence on $D_{31}$ of a repelling fixed point $F_R$ of the Poincar\'e map was proven. This fixed point corresponds to the three-link periodic cycles in the quotient space with respect to the Fuller group $g$ (\cite{ZMHBasic}, Theorem 3). Further, an hyperbolic periodic orbit of the map $\Phi$ was found which consists of two points $F_R^1,F_R^2$ and corresponds to the six-link periodic cycles from \cite{ZMHBasic}, Theorem 5. On Fig.~\ref{region1} the point $F_R$ and the pair of points $F_R^1,F_R^2$ are marked by blue and red crosses, respectively. 

\begin{figure}[H]
\centering
\includegraphics[width=\textwidth]{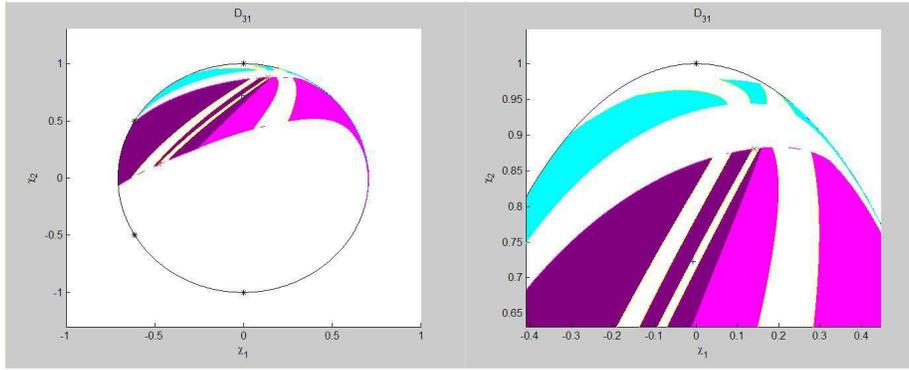}
\caption{Regions I and Ib}
\label{region1}
\end{figure}

The self-similar trajectories corresponding to the periodic orbits enumerated above experience a cyclic sequence of all three controls. Hence the corresponding transitions are of rotational type, and all three points $F_R,F_R^1,F_R^2$ lie in the intersection of the sub-domains $R_f,R_b$.

The point $F_R$ is repelling in the forward time direction and hence attracting in the backward time direction. Hence there exists a neighbourhood of the point $F_R$ such that every point of this neighbourhood tends to $F_R$ under repeated applications of the inverse Poincar\'e map $\Phi^{-1}$ and experience exclusively rotational transitions on the way. Denote the largest such neighbourhood by I. Clearly the region I is a subset of the domain $R_b$.

Numerical computations revealed that one of the branches of the stable submanifold of the hyperbolic periodic orbit $(F_R^1,F_R^2)$ lies in region I. It follows that both points $F_R^1,F_R^2$ lie on the boundary of region I, and region I is actually bounded by the unstable submanifolds of the periodic orbit $(F_R^1,F_R^2)$. Hence region I is nothing but the intersection of the strip lying between the unstable manifolds of the periodic orbit $(F_R^1,F_R^2)$ and the sub-domain $R_b$. The remaining parts of the boundary of region I consists of parts of the arc $d^0_{123}$ and the curve $\sigma_b$.

The intersections of region I with the blue, yellow, red, and green domain is painted in magenta on Fig.~\ref{region1}. Denote the image of region I under the inverse Poincar\'e map by Ic. Since the fixed point $F_R$ is attracting in backward time, the region Ic is a subset of region I. The intersection of region Ic with the blue, yellow, and red domains are painted in dark magenta on Fig.~\ref{region1}. Denote the complement of region Ic in region I by Ia. The intersection of region Ia with the red and green domain is painted in light magenta on Fig.~\ref{region1}. The boundary between the regions Ia and Ic is part of the curve $\sigma_f$. When region Ic is mapped by $\Phi$ on its image, this boundary arc is taken to a part of the boundary of region I, namely a part of the arc $d^0_{123}$. This latter sub-arc of $d^0_{123}$ is also part of the boundary of region Ic. The map $\Phi$ takes it on the boundary of region I, namely on a part of the curve $\sigma_b$.

It follows from the preceding that the image of region I under the map $\Phi$ is the union of region I with the image of region Ia. Let us denote this image $\Phi[Ia]$ by Ib. The intersections of region Ib with the blue, part of the yellow, and green domains is painted in cyan on Fig.~\ref{region1}. When region Ia is mapped to its image by the Poincar\'e map $\Phi$, then the sub-arc of the curve $\sigma_f$ on the boundary of region Ia is taken to a sub-arc of $d^0_{213}$, and the sub-arc of $\sigma_b$ is taken in the interior of the dark green region. Note that the intersection of the curve $\sigma_f$ with region Ib lies entirely on the boundary of the yellow region, but the boundary of the yellow region does not contain the intersection point of the curves $\sigma_f$, $\sigma_b$.

\subsection{Intermediate regions}

In the preceding subsection we established that the points lying in the blue and red domain are taken to the green domain by one or two applications of the Poincar\'e map $\Phi$, respectively, from where they join the semi-singular trajectory after an infinite number of transitions of transpositional type. On the other hand, the inverse map $\Phi^{-1}$ takes the points of region Ib to region I, where they further to the point $F_R$, experiencing transitions exclusively of rotational type. In this subsection we consider the remaining points of the disc $D_{31}$.

Every point of $D_{31}$ which does not lie in the closure of the green, blue, or red domain or one of the regions I, Ib, must lie in the closure of one of the following four regions.

Region II is bounded by a sub-arc of $d^0_{132}$, part of the boundary of the red domain, part of the boundary of region Ic, and part of the boundary of the blue domain. Region III is bounded by part of the boundary of region I, part of the boundary of the green domain, part of the boundary of region Ib, and part of the boundary of the blue domain. Region IV is bounded by a sub-arc of $d^0_{123}$, part of the boundary of the green domain, part of the boundary of region Ia, and part of the boundary of the red domain. Region V is bounded by part of the boundary of region Ib, part of the boundary of the green domain, and a sub-arc of $d^0_{213}$. The location of these regions is schematically depicted in Fig.~\ref{regions}, see also Fig.~\ref{regions2345}. 

\begin{figure}[H]
\centering
\includegraphics[width=\textwidth]{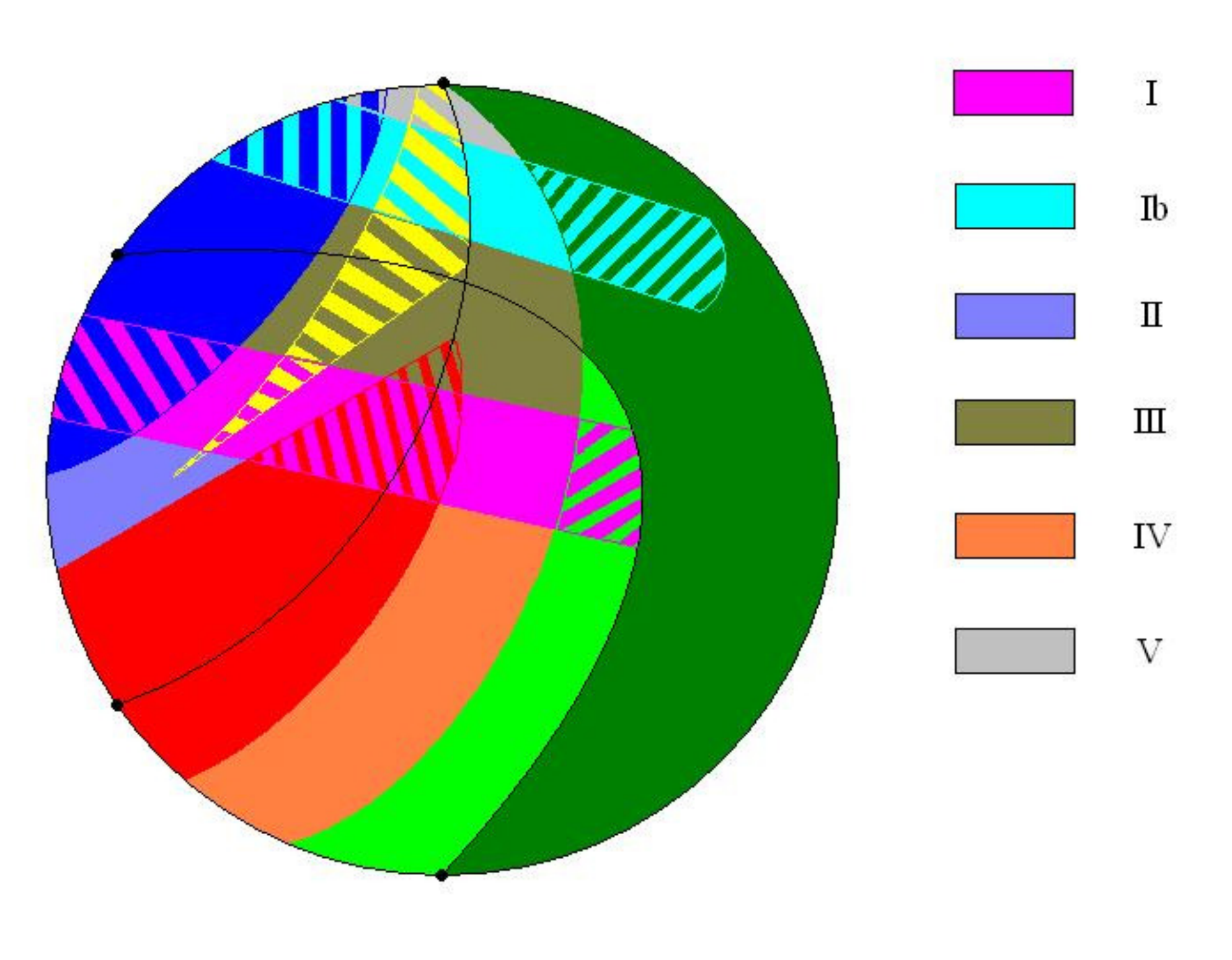}
\caption{Schematic location of the regions}
\label{regions}
\end{figure}

\begin{figure}[H]
\centering
\includegraphics[width=\textwidth]{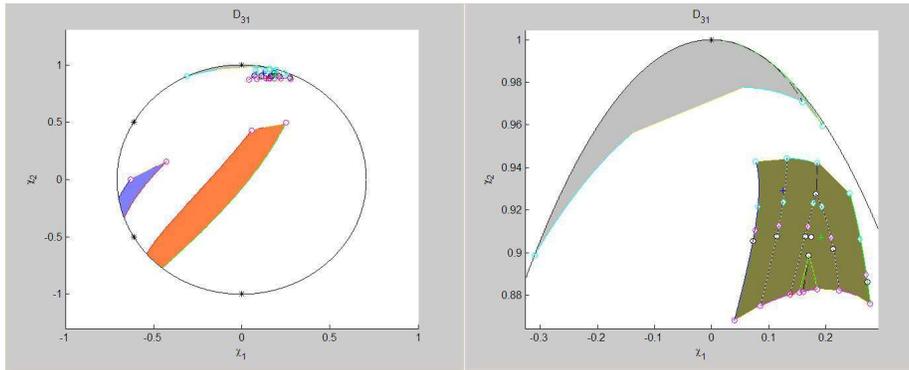}
\caption{Regions II,III,IV,V}
\label{regions2345}
\end{figure}

Let us now consider the dynamics on these intermediate regions which is induced by the Poincar\'e map. Note that both curves $\sigma_f$ and $\sigma_b$ partition region III in two parts, respectively. Denote the subsets delimited by the curve $\sigma_f$ by IIIl and IIIr, respectively, and the subsets delimited by the curve $\sigma_b$ by IIIu and IIId, respectively. Here region IIIl borders the blue domain, IIIr the green domain, IIIu borders region Ib, and IIId borders region I. Note also that the curve $\sigma_f$ divides region V in two parts, which we denote by Vl and Vr. Here region Vl borders the arc $d^0_{213}$, and Vr the green domain.

Then the Poincar\'e map $\Phi$ maps region II bijectively onto region IIId, here the part of the boundary of the blue domain is taken to a part of the boundary of the green domain, the part of the boundary of the red domain is taken to the part of the boundary of the blue domain, the part of the boundary of region Ic is taken to the part of the boundary of region I, and the sub-arc of $d^0_{132}$ to a part of the curve $\sigma_b$. The corresponding transition is of rotational type.

The Poincar\'e map maps also region IV bijectively onto region IIIu, here the part of the boundary of the light green domain is taken to the part of the boundary of the dark green domain, the part of the boundary of the red domain is taken to the part of the boundary of the blue domain, the part of the boundary of region Ia is taken to the part of the boundary of region Ib, and the sub-arc of $d^0_{123}$ to a part of the curve $\sigma_b$. The corresponding transition is of transpositional type.

From Fig.~\ref{Poincare_skizze} it is clear that the image of region Vl is bounded by the arcs $d^0_{132}$ and $d^0_{123}$. Here the intersection of the boundaries of regions Ib and Vl is taken to some curve joining these two arcs. This curve cannot intersect region I, since the image of region I under the inverse map $\Phi^{-1}$ does not intersect region Vl. The endpoints of this curve lie in the blue and yellow domains, respectively, hence the endpoints of the image of this curve lie in the green and red domain, respectively. Therefore $\Phi$ takes region Vl to the union of some subsets of the red and green domains and a subset of region IV. The corresponding transition is of rotational type.

Now we shall consider how $\Phi$ acts on region Vr. Consider the image of region V under the inverse map $\Phi^{-1}$. From the preceding it follows that $\Phi^{-1}[V]$ is bounded by the boundary of the green domain, the curve $\sigma_f$, and the boundary of region Ia. In particular, this image contains the region Vr. Hence $\Phi$ takes region Vr to a subset of region V. The corresponding transition is of transpositional type.

From the preceding paragraph it follows that region IIIr is contained in the preimage $\Phi^{-1}[V]$ of region V. Hence region IIIr is taken by $\Phi$ to a subset of region V. The corresponding transition is of transpositional type.

Consider now the image of region IIIl under $\Phi$. The intersection of the boundaries of region IIIl and the blue domain is taken to part of the boundary of the green domain. The intersection of the boundaries of regions IIIl and Ic is taken to part of the boundary of region I. The part of the curve $\sigma_f$ bordering region IIIl is taken to a sub-arc of $d^0_{123}$. The endpoints of this part of the curve $\sigma_f$ lie in the blue and yellow domain, respectively. Hence the endpoints of the sub-arc of $d^0_{123}$ lie in the blue and red domain, respectively. Therefore the image $\Phi[IIIl]$ is a union of region II and subsets of the red and blue domains and region IV. The corresponding transition is of rotational type.

The images of regions II, IIIl, IIIr, IV, Vl, and Vr under the map $\Phi$ are schematically depicted in Fig.~\ref{images2345}. On Fig.~\ref{images2345} on the right it can be seen that region IV is divided by the images of regions III and V into three strips IVa,IVb,IVc. Here we define region IVa as the intersection of region IV with the image of region III, region IVc as the intersection of region IV with the image of region V, and region IVb as the complement of IVa and IVc in region IV.

\begin{figure}[H]
\centering
\includegraphics[width=\textwidth]{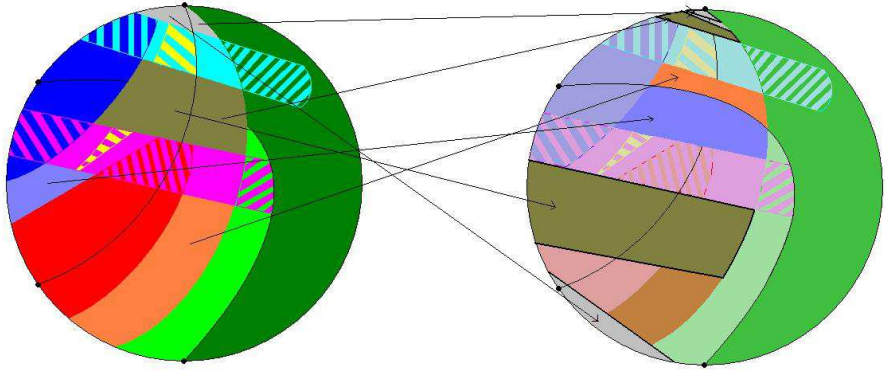}
\caption{Images under $\Phi$ of regions II,III,IV,V}
\label{images2345}
\end{figure}

Let us now consider the preimages of regions II-V under the map $\Phi$. The inverse map $\Phi^{-1}$ takes region II to a subset of region IIIl. Here the sub-arc of $d^0_{123}$ bordering region II is taken to an arc on the curve  $\sigma_f$. The parts of the boundaries of the blue and red domains are taken to parts of the boundaries of the red and yellow domains, respectively, and the part of the boundary of region Ic to another part of the boundary of region Ic on the other side. The corresponding transition is of transpositional type.

Consider the preimage of region IV under $\Phi$. The sub-arc of $d^0_{132}$ bordering region II is mapped by $\Phi^{-1}$ to a sub-arc of $d^0_{213}$, the parts of the boundaries of the green and red domains are taken to parts of the boundaries of the blue and yellow domains, respectively, and the part of the boundary of region Ia is mapped to a part of the boundary of region Ic. Hence the preimage of region IV is the union of subsets of regions IIIl, Vl, and Ib. Denote this subset of region Vl by Va. The corresponding transition is of rotational type.

The preimages of regions IIIu and IIId are the regions IV and II, respectively. The corresponding transitions are of transpositional and rotational types, respectively.

Finally, the preimage of region V has been described above. It is the union of the regions IIIr and Vr and a subset of region Ib. The corresponding transition is of transpositional type.

The preimages of regions II, IIId, IIIu, IV, and V under the map $\Phi$ are depicted schematically in Fig.~\ref{preimages2345}.

\begin{figure}[H]
\centering
\includegraphics[width=\textwidth]{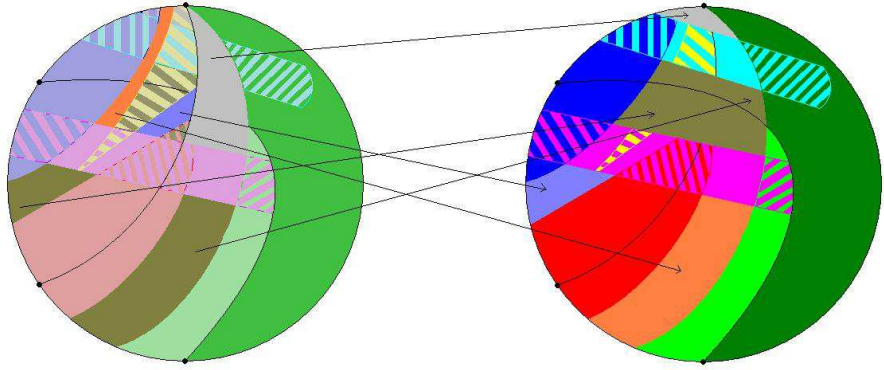}
\caption{Images under $\Phi^{-1}$ of the regions II,III,IV,V}
\label{preimages2345}
\end{figure}

The dynamics of the Poincar\'e map on the regions II-V can be described by the oriented graph presented in Fig.~\ref{graph2345}. The outgoing edges denote transitions into points of the red, blue, or green domains, where they are attracted by the semi-singular trajectory. The incoming edges denote transitions from points in the regions I or Ib. Under iterations of the inverse map $\Phi^{-1}$ these points are attracted by the fixed point $F_R$. The letters t,r denote the type of transition, transpositional or rotational, respectively.

\begin{figure}[H]
\centering
\begin{tikzpicture}[->,>=stealth',shorten >=1pt,auto,node distance=3.5cm,thick,main node/.style={font=\sffamily\LARGE\bfseries}, minimum size=4mm]

	\node[main node] (II) {II};
	\node[main node] (IV) [below right of = II] {IV};
	\node[main node] (III) [above right of = IV] {III};
	\node[main node] (V) [below right of = III] {V};
	\node[main node] (nl) [left of = IV] {};
	\node[main node] (vr) [right of = III] {};
	\node[main node] (nr) [right of = V] {};

	\path[every node/.style={font=\sffamily\large}]
        (II) edge [bend left=10] node [above] {r} (III)
	(III) edge [bend left=10] node [below] {r} (II)
	(III) edge node [below] {r} ([xshift=-2ex]vr.west)
	([xshift=2ex]nl.east) edge node [below] {r} (IV)
	(V) edge node [below] {r} (IV)
	(V) edge node [below] {r} ([xshift=-2ex]nr.west)
	(IV) edge [bend left=10] node [above left] {t} (III)
	(III) edge [bend left=10] node [below right] {r} (IV)
	(III) edge node [below left] {t} (V)
	([xshift=2ex,yshift=-2ex]vr.south) edge node [below right] {t} (V);
	\path[->,every loop/.style={looseness=20},every node/.style={font=\sffamily\large}]
	(V) edge [loop below] node [below] {t} (V)
        ;
\end{tikzpicture}
\caption{Transition graph}
\label{graph2345}
\end{figure}
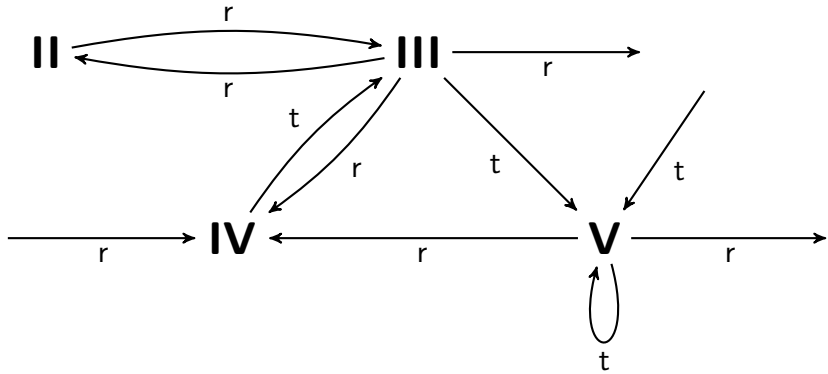

\subsection{Resolution of the dynamics on region V}

Consider the intersection of region V with the images of regions III and V on the one hand, and with the preimages of regions IV and V, on the other hand. These intersections are schematically depicted on the right-hand side of Fig.~\ref{region5}. On the left-hand side of Fig.~\ref{region5} one sees the exact locations of the points which are denoted by dots on the right-hand side.
\begin{figure}[H]
\centering
\includegraphics[width=\textwidth]{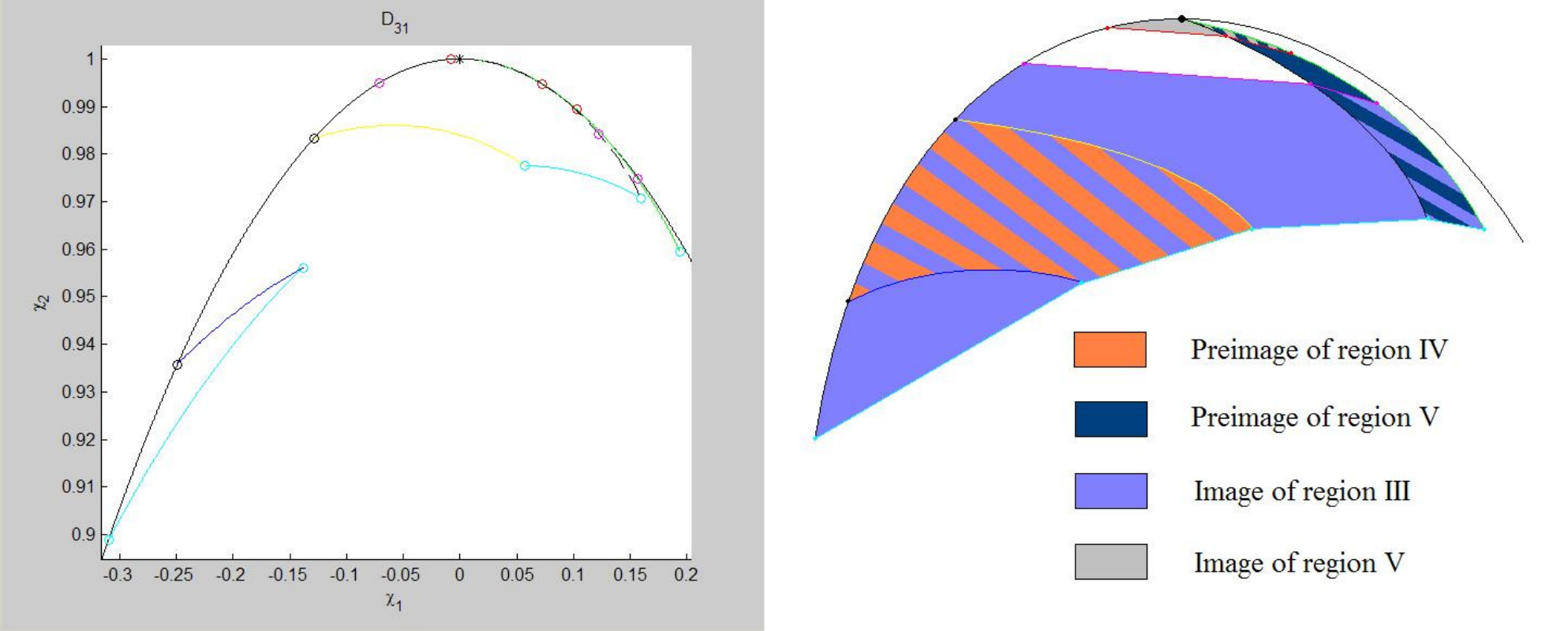}
\caption{Dynamics on region V}
\label{region5}
\end{figure}

From Fig.~\ref{region5} it becomes clear that region Va is a subset of the intersection of the image of region III with region V. Hence every point which comes into region V from outside of region III can never return to any of the regions II, III, and IV. Such a point either lies in the intersection of region V with the blue or the yellow domain, or hits the yellow domain after a finite number of applications of the map $\Phi$, where the corresponding transitions are of transpositional type only. From the yellow domain it gets after three applications of $\Phi$ to the green domain, from where it is attracted by the semi-singular trajectory.

The graph on Fig.~\ref{graph2345} implies that every trajectory that is not eventually attracted in forward time direction by the semi-singular trajectory, and in backward time direction by the fixed point $F_R$, has to pass an infinite number of times through the closure of region III. Hence when studying these trajectories, we can restrict our attention to the mapping of the closure of region III on itself.

The graph on Fig.~\ref{graph2345} implies that this mapping can proceed in three different ways. Either through two transitions of rotational type passing region II, or through two transitions, of transpositional and rotational type, respectively, passing region IV, or finally through three transitions, of transpositional, rotational, and again transpositional type, respectively, passing regions V and IV. Let us denote these possibilities of the mapping of the closure of region III to itself by $A,B,C$, respectively. Hence every trajectory which is not eventually attracted in forward time direction by the semi-singular trajectory, and in backward time direction by the fixed point $F_R$, can be represented by a bilaterally infinite sequence of symbols $A,B,C$. In the sequel we shall argue in terms of transitions of type $A,B,C$ instead of transitions of transpositional and rotational types. Here the former are actually composed by the latter.

Consider a point in the closure of region III. When applying recursively the map $\Phi$ to this point, it will pass a number of times $n_f$ through the closure of region III. Here $n_f$ may be finite or infinite. The corresponding transitions will be labeled by symbols $i_1,i_2,\dots,i_{n_f} \in \{A,B,C\}$. When applying the inverse map $\Phi^{-1}$ to this point, it will pass a number of times $n_b$ through the closure of region III, where $n_b$ also may be finite or infinite. The corresponding transitions will be labeled by symbols $j_1,j_2,\dots,j_{n_b} \in \{A,B,C\}$. We then associate to the initial point the two-sided sequence $j_{n_b}j_{n_b-1}\dots j_1.i_1i_2\dots i_{n_f}$. Here the period separates transitions which lie in the future (with respect to the time flow defined by the iterations of $\Phi$) from transitions which lie in the past.

On the other hand, let us associate to the two-sided sequence 

$$j_{n_b}j_{n_b-1}\dots j_1.i_1i_2\dots i_{n_f}$$

\noindent the subset of points in the closure of region III which experience the transitions $i_1,i_2,\dots,i_{n_f} \in \{A,B,C\}$ under iterations of the map $\Phi$, and the transitions $j_1,j_2,\dots,j_{n_b}$ under the iterations of the inverse map $\Phi^{-1}$. Thus, e.g., the set $A.B$ is the intersection of the closure of region III with the preimage of transition $B$ and the image of transition $A$.

The set $A.$ then coincides with the closure of region IIId, the set $B.$ is the image of the region IVc, and the set $C.$ is the image of region IVa. On the other hand, the set $.A$ is the closure of the preimage of region II, the set $.B$ the closure of the intersection of the preimage of region IV with region III, and the set $.C$ is the preimage of the region Va. The sets $.A,.B,.C$ are depicted on the left-hand side of Fig.~\ref{region3_im}, the sets  $A.,B.,C.$ on its right-hand side. The transitions $A,B,C$ map the former to the latter, respectively.

\begin{figure}[H]
\centering
\includegraphics[width=\textwidth]{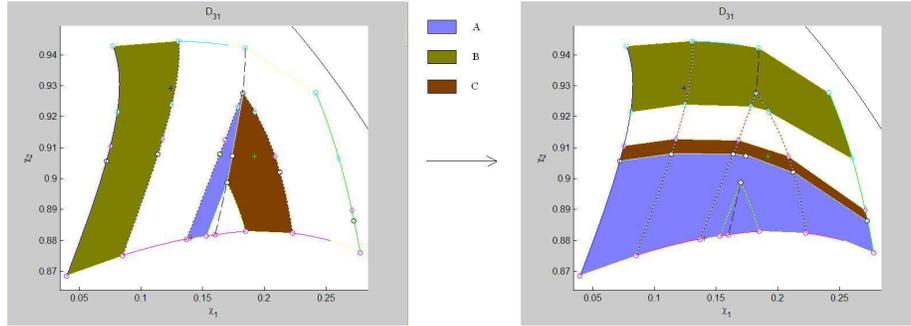}
\caption{Dynamics on region III}
\label{region3_im}
\end{figure}

Note that the sets $.A,.C$ intersect each other in some arc of the curve $\sigma_f$, which is at the same time part of their boundaries. Both sets border region I and the red domain. The set $.A$ borders also the yellow domain, and the set $.C$ borders the preimage of the intersection of the yellow domain with region V. On the other hand, the set $.B$ borders the regions I and Ib and the blue and yellow domains. The set $B.$ borders region Ib, the blue and green domains, and the image of region IVb. The set $C.$ borders the set $A.$ along the curve $\sigma_b$, and both sets border the blue and green domains. The set  $C.$ borders also the image of region IVb, while $A.$ border region I.

Note that all three sets on the left-hand side of Fig.~\ref{region3_im} intersect all three sets on the right-hand side. Hence any type of transition can follow any type. Moreover, each of the transitions $A,B,C$ features a fixed point. On Fig.~\ref{region3_im} these fixed points are represented by a red, blue, and green cross, respectively. The fixed points of $A$ and $B$ correspond to the six-link and the four-link cycle, respectively, which have been found in the paper \cite{ZMHBasic}, and the fixed point of $C$ corresponds to a nine-link cycle (which has been discovered in 2002).

\subsection{Resolution of the dynamics of transitions of type \texorpdfstring{$B$}{B}}
\label{subsec:resolve_B}

From Fig.~\ref{dynamicsB1} it is clear that the sets $.AA$, $.AC$, $.CA$, $.CC$ have an empty intersection with the set $B.$. Hence the sequence of symbols corresponding to a trajectory of the system cannot contain any subsequence of the form $Bij$, where $i,j \in \{A,C\}$. By induction it follows that if the sequence somewhere contains two consecutive symbols of the set $\{A,C\}$, then the symbol $B$ can occur only to the right of this pair.

\begin{figure}[H]
\centering
\includegraphics[width=\textwidth]{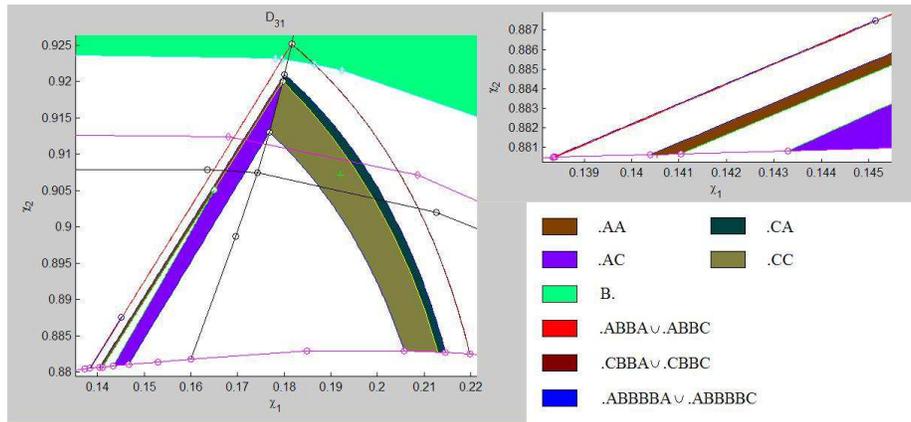}
\caption{Dynamics with participation of transition B}
\label{dynamicsB1}
\end{figure}

The sets $.BA$, $.BC$ are subsets of $B.$, and the sets $.BBA,.BBC$ intersect all three sets $A.$, $B.$, $C.$. The sets $.BBBA$, $.BBBC$ are also subsets of $B.$, and the sets $.BBBBA$, $.BBBBC$ are subsets of $A.$ (see Fig.~\ref{dynamicsB2}). The sets $.ABBBBA$, $.ABBBBC$ are also subsets of $A.$, whereas $.ABBA$, $.ABBC$, $.CBBA$, $.CBBC$ have a non-empty intersection with the set $B.$ (see Fig.~\ref{dynamicsB1}). 

\begin{figure}[H]
\centering
\includegraphics[width=0.85\textwidth]{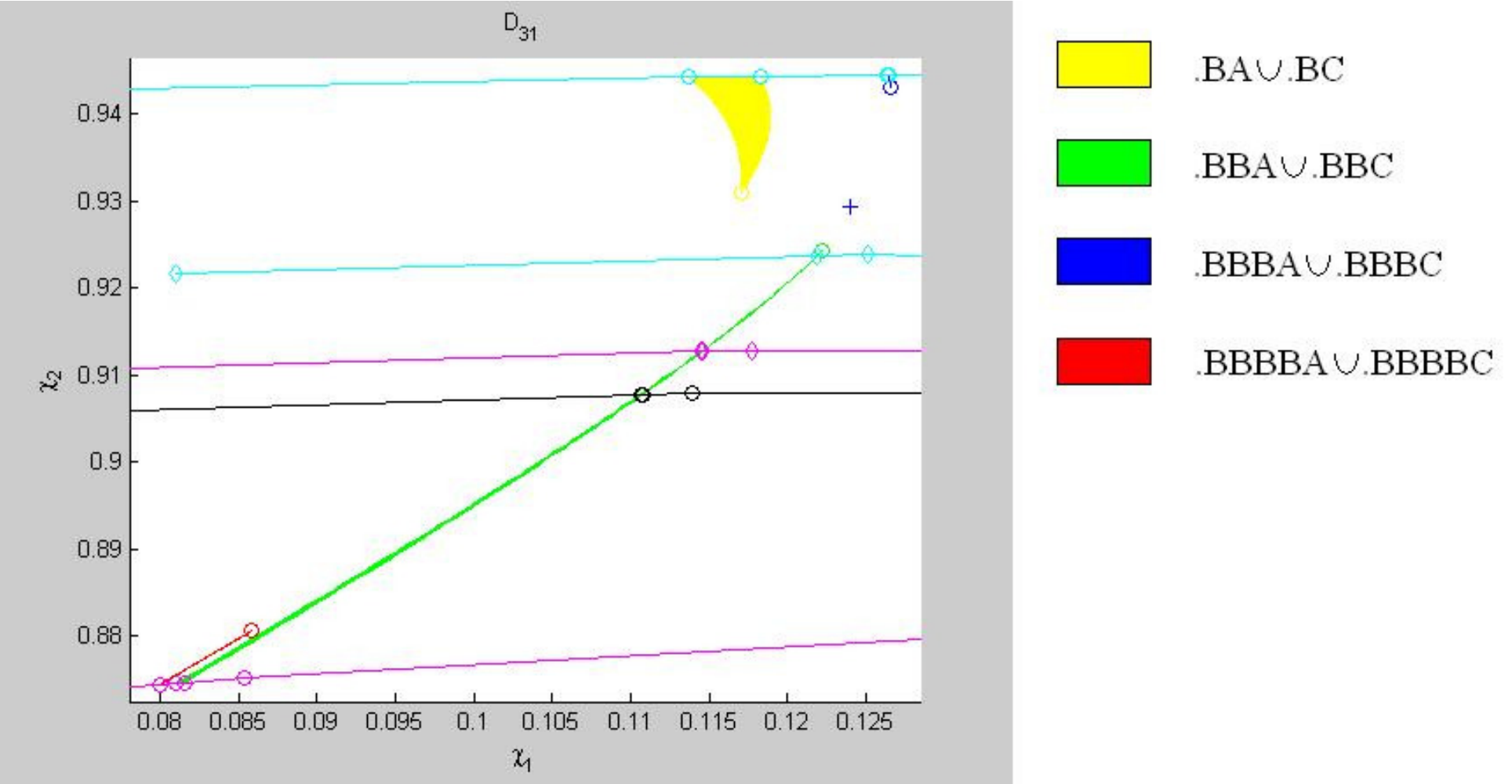}
\caption{Dynamics with participation of transition B}
\label{dynamicsB2}
\end{figure}

The sets $.BABBA,.BABBC,.BCBBA,.BCBBC$, which are painted in blue on Fig.~\ref{dynamicsB3}, are subsets of the union $.BA \cup .BC$ and hence also subsets of $B.$. However, the sets $.BBABBA,.BBABBC,.BBCBBA,.BBCBBC$, which are painted in red on Fig.~\ref{dynamicsB3}, are already only subsets of $A.$, and the sets $.ABBABBA,.ABBABBC,.ABBCBBA,.ABBCBBC$, which are painted in green on Fig.~\ref{dynamicsB3}, do not intersect $B.$. 

\begin{figure}[H]
\centering
\includegraphics[width=0.6\textwidth]{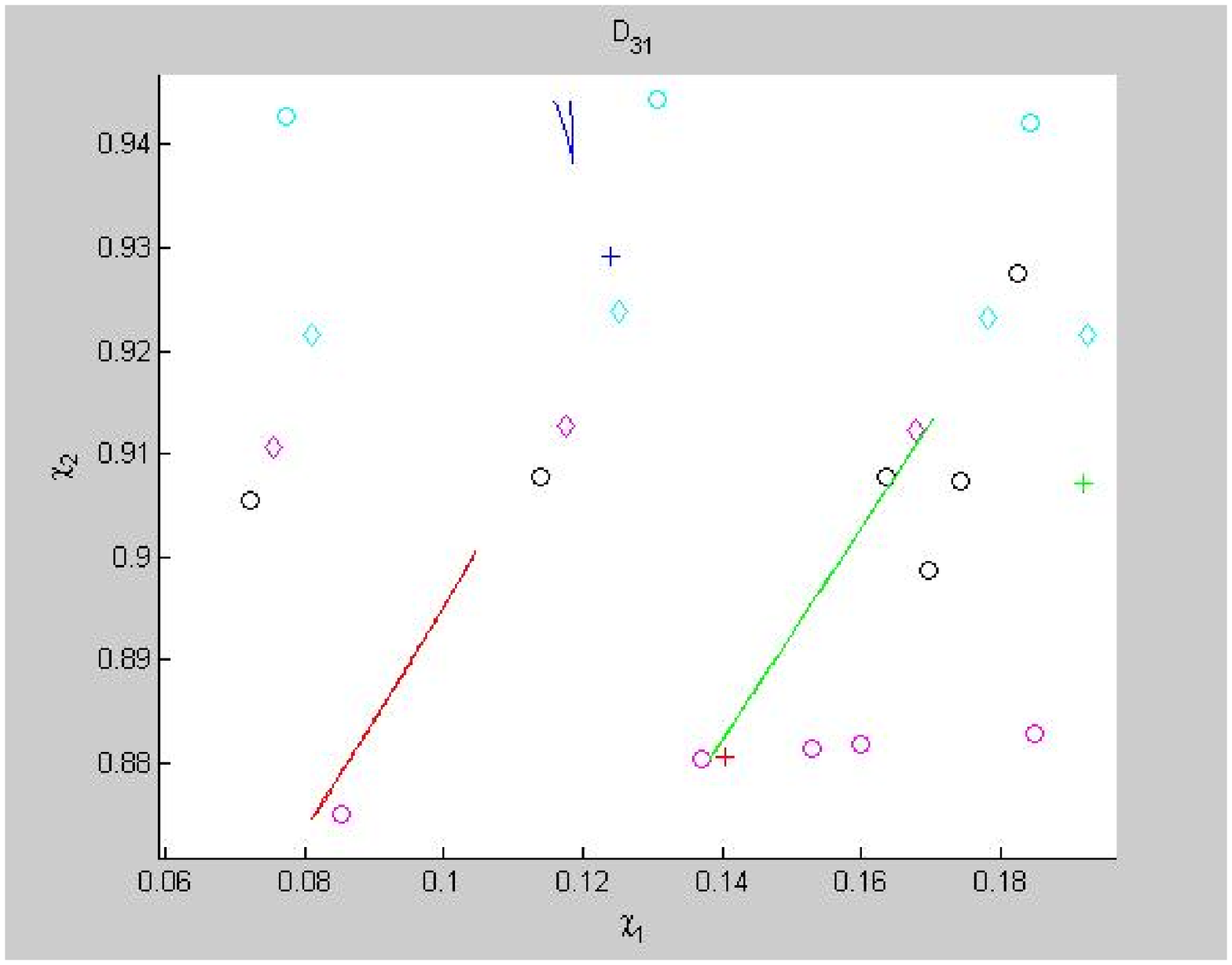}
\caption{Dynamics with participation of transition B}
\label{dynamicsB3}
\end{figure}

From the preceding it follows that if the sequence of symbols contains a subsequence of the form $Bi$, where $i \in \{A,C\}$, then this subsequence can be continued to the left only by one of $\dots AABBBBi$, $\dots ABBjBBi$, $\dots BBi$, where $j \in \{A,C\}$ and the dots denote a subsequence of arbitrary length, but consisting only of symbols in $\{A,C\}$.

In particular, a sequence which is infinite on both sides either (i) consists entirely of symbols $B$, in which case it corresponds to the fixed point of the transition $B$, i.e., the four-link periodic cycle, or (ii) it contains a subsequence consisting entirely of symbols $B$ and stretching out infinitely far to the right, a subsequence not containing any symbol $B$ and stretching out infinitely far to the left, and between these two infinite subsequences maximally six additional symbols, or (iii) it does not contain the symbol $B$ at all. Therefore, if the corresponding trajectory of the dynamics on region III is not the fixed point of transition $B$, then it hits in backward time direction the union of the sets $A.A$, $A.C$, $C.A$, $C.C$ and remains there indefinitely.

\subsection{Dynamics of the transitions of types \texorpdfstring{$A$}{A} and \texorpdfstring{$C$}{C}}
\label{subsec:AC_dynamics}

We shall now consider the dynamics of the transitions $A,C$ on the union of the sets $A.A$, $A.C$, $C.A$, $C.C$, or more precisely, on the closed region IIIb defined as the union of the sets $A.A$, $A.C$, $C.A$, $C.C$ and the closure of the intersection of the red domain and region III. Hence region IIIb borders region I, the yellow domain, the image of region IVb, and the preimage of the intersection of the yellow domain with region V. The region IIIb is depicted on the left-hand side of Fig.~\ref{region3bi} and on the right-hand side of Fig.~\ref{region3bpi}.

\begin{figure}[H]
\centering
\includegraphics[width=\textwidth]{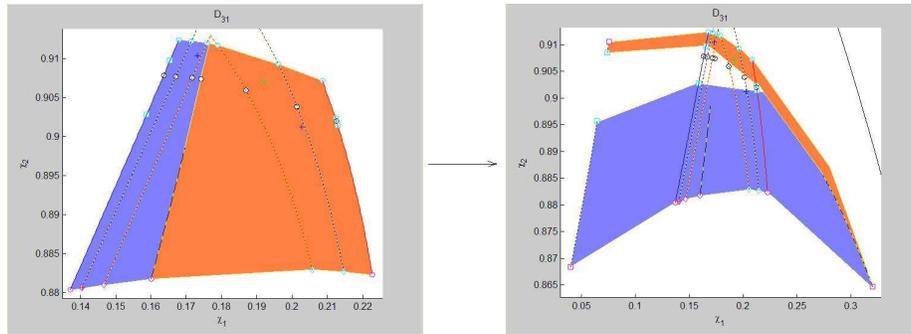}
\caption{Image of region IIIb under $F$}
\label{region3bi}
\end{figure}
\begin{figure}[H]
\centering
\includegraphics[width=\textwidth]{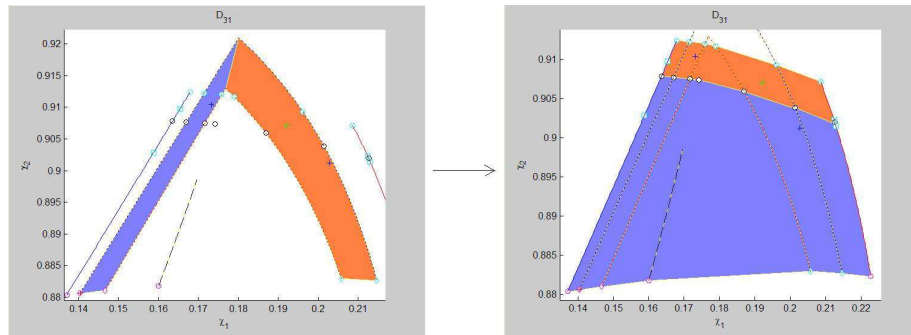}
\caption{Image of region IIIb under $F^{-1}$}
\label{region3bpi}
\end{figure}

Let us define mappings $F,F^{-1}: \mbox{IIIb} \to D_{31}$ as follows. Let $\chi$ be a point in the strip between the yellow domain and the curve $\sigma_f$. We then define its image $F(\chi)$ as the image of $\chi$ under the transition $A$, i.e., under the power $\Phi^2$ of the Poincar\'e map. Both applications of $\Phi$ induce transitions of rotational type. Let now $\chi$ be a point in the strip between the preimage of the yellow domain and the curve $\sigma_f$. We then define its image $F(\chi)$ as the image of $\chi$ under the transition $C$, i.e., under the power $\Phi^3$ of the Poincar\'e map. The three applications of $\Phi$ induce transitions of transpositional, rotational, and again transpositional type, respectively. The map $F$ can be extended continuously on the intersection of region IIIb with the curve $\sigma_f$. Let $\chi$ be a point in the strip between region I and the curve $\sigma_b$. We define its image $F^{-1}(\chi)$ under the inverse map $F^{-1}$ by the preimage of $\chi$ under the transition $A$, i.e., as $\Phi^{-2}(\chi)$. Here both transitions generated by the application of $\Phi^{-1}$ have rotational type. Let finally $\chi$ be a point in the strip between the image of region IVb and the curve $\sigma_b$. We define $F^{-1}(\chi)$ as the preimage of $\chi$ under the transition $C$, i.e., as $\Phi^{-3}(\chi)$. Here the three transitions generated by the application of $\Phi^{-1}$ have transpositional, rotational, and again transpositional type, respectively. The map $F^{-1}$ can be extended continuously on the intersection of region IIIb with the curve $\sigma_b$.

The map $F$ takes an arc on the curve $\sigma_f$ to an arc on the curve $\sigma_b$, and $F^{-1}$ vice versa, an arc on $\sigma_b$ to an arc on $\sigma_f$. The images of these arcs do not intersect region IIIb, and the intersection of region IIIb with $F[IIIb]$ as well as the intersection of region IIIb with $F^{-1}[IIIb]$ both consist of two connection components, respectively. Thus the topology of the map of region IIIb on itself is equivalent to the topology of the map in the classical Smale Horseshoe. The images of region IIIb under the maps $F,F^{-1}$ are depicted in Fig.~\ref{region3bi}, Fig.~\ref{region3bpi}.

On Fig.~\ref{region3bi}, \ref{region3bpi} the fixed points of $F$ are represented by red and green crosses. These two points are the fixed points of the transitions $A$ and $C$, respectively, and have been already depicted in Fig.~\ref{region3_im}. The blue crosses in Fig.~\ref{region3bi}, \ref{region3bpi} denote a periodic orbit of the map $F$ consisting of two points, namely the fixed points of the compositions $A \circ C$ and $C \circ A$. This orbit corresponds to a 15-link periodic cycle on the quotient space $M_+/g$, which is invariant with respect to a permutation of order 3 in $S_3$. All three periodic orbits are of hyperbolic type. The stable eigenvalues are approximately equal to 

$$0.789389405489713,\ -0.195427788394708,\ -0.064093706436160,$$

\noindent respectively, and the unstable eigenvalues to 

$$25.416902415438734,\ -2.553017991174260,\ -73.29685207454,$$ 

\noindent respectively, for the fixed points of the mappings $A,C,A\circ C$.

Let us remark that the map $F$ is continuously differentiable and even algebraic on a dense subset of region IIIb. This subset is the set of points which under repeated application of the Poincar\'e map $\Phi$ eventually hit the interior of the green domain, without leaving the interior of the disc $D_{31}$ on their way. Therefore any such point can be obtained by repeated application of the inverse map $\Phi^{-1}$ to some point in the green, or equivalently, blue, domain.

\newcommand{\dx}{\Delta X}
\newcommand{\dy}{\Delta Y}
\newcommand{\du}{\Delta u}
\newcommand{\dv}{\Delta v}

\newcommand{\nw}[1]{\mathrm{NW}(#1)}

\newcommand{\uxv}{\overline{u_x}}
\newcommand{\uxn}{\underline{u_x}}
\newcommand{\uyv}{\overline{u_y}}

\newcommand{\vyv}{\overline{v_y}}
\newcommand{\vyn}{\underline{v_y}}
\newcommand{\vxv}{\overline{v_x}}

\newcommand{\xuv}{\overline{x_u}}
\newcommand{\xun}{\underline{x_u}}
\newcommand{\xvv}{\overline{x_v}}

\newcommand{\yvv}{\overline{y_v}}
\newcommand{\yvn}{\underline{y_v}}
\newcommand{\yuv}{\overline{y_u}}

\section{Chaotic dynamics of the Poincar\'e map}

In this section we prove the main theorem of the first part of this paper, on the semi-conjugateness of the dynamics of the optimal synthesis in model problem (\ref{problem:model}) with equilateral triangle with a certain Markov chain. We find bounds on the Hausdorff and Minkowski dimensions of the non-wandering set and compute the topological entropy.

\subsection{Bilipschitz property of the Poincar\'e return map}

The main tool used in the investigation of the optimal synthesis of problem (\ref{problem:model}) in the previous sections was Theorem \ref{thm:model_problem_bellman} on the passage to the adjoint variables, or more precisely, the map $E$ allowing this passage. Therefore a proof of the bilipschitz property of the map $\Phi$ needs to commence with an investigation of this property for the map $E$.

In Theorem \ref{thm:model_problem_bellman} we have shown that in a broad class of optimal control problems, including model problem (\ref{problem:model}), the map $E$ is locally Lipschitz and bijective. In general, the inverse map $E^{-1}$ does not need to be even locally Lipschitz for the general problem (\ref{problem:model}).

In the next lemma we show that the map $E$ in model problem (\ref{problem:model}) is locally bilipschitz in the neighbourhood of any point not lying on the submanifolds $A_{ij}$.

\begin{lemma}
\label{lm:bilipschitzian_model_E}
	Consider a point $q_0=(x_0, y_0)$ which does not lie on any of the planes $A_{ij}$, i.e., $x_0\notin\cup A_{ij}$ or $y_0\notin\cup A_{ij}$. Then there exists a neighbourhood of $q_0$ on which the map $E$ satisfies a bilipschitz property.
\end{lemma}

\begin{proof}
	The proof consists in obtaining a lower bound on the difference $\bellman(q_1)-\bellman(q_0) - \bellman'(q_0)\Delta q$ and an application of Lemma \ref{lm:bilipschitzian_model}. Here we use the usual notations $q_i=(x_i,y_i)$ and $\Delta q=(\Delta x,\Delta y)=q_1-q_0$.
	
	Let $X_0(t)$ and $X_1(t)$ be two arbitrary trajectories emanating from points $x_0$ and $x_1$, respectively. For an arbitrary number $\lambda\in[0,1]$ we define $X_\lambda(t) = (1-\lambda)X_0(t) + \lambda X_1(t)$. Then the functional $J$ defined in (\ref{problem:model}) satisfies the relation (see Lemma \ref{Bellman_convex})
	
	\[
		(1-\lambda)J(X_0) + \lambda J(X_1) - J(X_\lambda) = \lambda(1-\lambda)J(\Delta X),
	\]
	
	\noindent where $\Delta X(t) = X_1(t) - X_0(t)$. If the trajectories $X_0$ and $X_1$ are optimal and $\dot X_i(0) = y_i$, $i=1,2$, then
	
	\[
		(1-\lambda)\bellman(q_0) + \lambda\bellman(q_1) - \bellman(q_\lambda) \ge
		\lambda(1-\lambda)J(\Delta X),
	\]
	
	\noindent where $q_\lambda=(1-\lambda)q_0 + \lambda q_1$. Since both sides of the inequality coincide at $\lambda=0$, we have after differentiation with respect to $\lambda$ at $\lambda=0$ that
	
	\[
		\bellman(q_1) - \bellman(q_0) - \bellman'(q_0)\Delta q \ge J(\Delta X).
	\]
	
	It rests to estimate $J(\Delta X)$ from below. The simplest way to do this, namely $J(\Delta X)\ge\bellman(\Delta q)$ yields an insufficient result of an order which is too high, namely $\bellman(\Delta q)\ge A \Delta x^{\frac{5}{2}} + B\Delta y^5$. In general, one cannot have a better bound on the Bellman function at an arbitrary point.
		
	However, if the point $q_0$ does not lie on $A_{ij}$, we can have a better bound on the functional $J(\Delta X)$. Assume first that $q_0\notin \cup (A_{ij}\times A_{ij})$ is not a switching point. This means that in a sufficiently small neighbourhood $V$ of the point $q_0$ there exists $\tau>0$ such that for time instants $t < \tau$ all optimal trajectories emanating from points $q_1\in V$ are steered by the same control $u_0$ residing in one of the vertices of the triangle $\Omega$. Indeed, the lift of the optimal trajectory $(\widehat x(t),\widehat y(t),\widehat \phi(t),\widehat \psi(t))$ emanating from $\bigl(q_0,E(q_0)\bigr)$ intersects the switching surface $\cal S$ in some point $(x',y',\phi',\psi')\ne\bigl(q_0,E(q_0)\bigr)$ at some time instant $t_0>0$ for the first time. Since the surface $\cal S$ is closed, we have that for every $\varepsilon>0$ there exists $\delta>0$ such that the trajectory of Hamiltonian system (\ref{eq:model_pmp_system}) emanates from a $\delta$-neighbourhood of the point $\bigl(q_0,E(q_0)\bigr)$, then it intersects ${\cal S}$ for the first time at a time instant not smaller than $t_0-\varepsilon$. Now note that by virtue of Theorem \ref{thm:model_problem_bellman} the map $E$ is continuous, hence if $q_1$ is close to $q_0$, then also $(q_1,E(q_1))$ is close to $(q_0,E(q_0))$.
	
	For all $t\in[0,\tau]$ we have the equation $\Delta X(t) = \Delta x + t\Delta y$, because the control on the optimal trajectories emanating from the points $q_0$ and $q_1$ coincides on this time interval, given the point $q_1$ lies in a neighbourhood of $q_0$. Hence there exists $C'>0$ such that
	
	\[
		\bellman(q_1) - \bellman(q_0) - \bellman'(q_0)\Delta q\ge
		J(\Delta X) \ge \frac{1}{2}\int_0^\tau |\Delta x + t\Delta y)|^2\,dt \ge C'(|\Delta x|^2 + |\Delta y|^2).
	\]
	
	Application of Lemma \ref{lm:bilipschitzian_model} then yields the proof of the assertion of Lemma \ref{lm:bilipschitzian_model_E} for all points $q_0$ not lying on a switching surface. Suppose, on the contrary, that the point $q_0$ lies on a switching surface. In order to get a lower bound on $J(\Delta X)$ we have to move on all optimal trajectories emanating from a neighbourhood of $q_0$ by a small time interval $t_0>0$ and to repeat the reasoning above.
\end{proof}

Hence\footnote{On the sets $M$, $N$, $M_+$ the structure of a metric space is defined in a natural way by the ambient vector space ${\cal M}=T^*M=M\times N$. However, the set $M_+$ is not a smooth submanifold of ${\cal M}$.},

\begin{enumerate}
	\item The map $E:M\to N$ is a locally Lipschitz homeomorphism, which is bilipschitz at all points outside of $\cup (A_{ij}\times A_{ij})$;
	\item The map $q\mapsto (q,E(q))$ is a locally bilipschitz homeomorphism from $M$ to $M_+$;
	\item The map $(q,E(q))\mapsto E(q)$ is a locally Lipschitz homeomorphism from $M_+$ to $N$, which is bilipschitz at all points outside of $\cup (A_{ij}\times A_{ij})$.
\end{enumerate}

All these maps retain the enumerated properties after factorization with respect to the action $g$.

Let us now consider the bilipschitz property of the Poincar\'e map $\Phi$. Let $x\in({\cal S}\cap M_+)/g$ be a point which lies neither on $\bigcup A_{ij}^{\times 4}/g$ nor on ${\cal S}_{123}/g$. By virtue of Lemma  \ref{lm:semi_singular} the image $\Phi(x)\in ({\cal S}\cap M_+)/g$ is then well-defined. Let us assume that $\Phi(x)\notin {\cal S}_{123}/g$ and $\Phi(x)\notin \bigcup A_{ij}^{\times 4}/g$. Then the projection of the trajectory of Hamiltonian system (\ref{eq:model_pmp_system}) through $x$ intersects the switching surface transversally both in the point $x$ and in the point $\Phi(x)$. The map $\Phi:{\cal S}/g\to {\cal S}/g$ is hence well-defined on a neighbourhood of $x$ in ${\cal S}/g$ and is a local homeomorphism of a neighbourhood of $x$ in ${\cal S}/g$ onto its image. We then get immediately that the restriction of $\Phi$ on $({\cal S}\cap M_+)/g$ is bilipschitz in a neighbourhood of $x$. By virtue of the bilipschitz property of the map $E$ we obtain the following lemma.

\begin{lemma}
\label{lm:bilipschitzian_model_Phi}
	The Poincar\'e return map $\Phi:({\cal S}\cap M_+)/g\to ({\cal S}\cap M_+)/g$ is bilipschitz in the neighbourhood of every point $x$ such that neither $x$ nor $\Phi(x)$ do not lie on ${\cal S}_{123}/g$ or $\bigcup A_{ij}^{\times 4}/g$. The bilipschitz property is conserved if we consider the map $\Phi$ as acting in the spaces $M/g$ or $N/g$ by identification by virtue of the map $E$.
\end{lemma}

\begin{corollary}
	The map $F:A.A\cup A.C \cup C.A \cap C.C\to A.A\cup A.C \cup C.A \cap C.C$, which was defined in Subsection \ref{subsec:AC_dynamics}, is bilipschitz.
\end{corollary}



\subsection{Conditions of Lipschitz hyperbolicity}

\begin{figure}
\centering
\begin{tikzpicture}[->,>=stealth',shorten >=1pt,auto,node
		distance=3.4cm,thick,main node/.style={circle,fill=blue!10,
		draw,font=\sffamily\Large\bfseries}, minimum size=4mm]

		\node[main node] (11) {$A.A$};
		\node[main node] (12) [right of=11] {$A.C$};
		\node[main node] (21) [below of=11] {$C.A$};
		\node[main node] (22) [right of=21] {$C.C$};

		\path[->,font=\scriptsize]
		(11) edge (12)
		(11) edge[loop above] (11)
		(12) edge (22)
		(12) edge[bend right=5] (21)
		(21) edge (11)
		(21) edge[bend right=5] (12)
		(22) edge (21)
		(22) edge[loop below] (22);
\end{tikzpicture}
\caption{Graph $\widetilde \Gamma$}\label{fig:graph_AA_CC}
\end{figure}
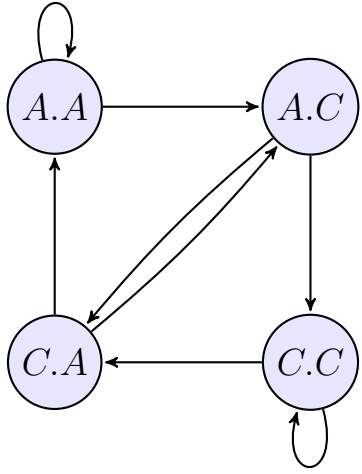

Let us study the dynamics of the map $F$ in detail. As was shown in the previous section, we may choose disjoint rectangular neighbourhoods $B_1$, $B_2$, $B_3$, $B_4$ of the sets $A.A$, $A.C$, $C.A$, $C.C$ and corresponding coordinates $(X_i,Y_i)$ on $B_i$, $i=1,2,3,4$. Denote by $F_{ij}$ the restriction of $F$ on the subset of those points in $B_i$ which are mapped to $B_j$ by one iteration of the mapping $F$, i.e., $F_{ij}=F|_{B_i\cap F^{-1}(B_j)}$. The maps $F_{ij}$ are defined in accordance with the graph $\widetilde \Gamma$ on Fig.~\ref{fig:graph_AA_CC}: the map $F_{ij}$ is defined if and only if there exists an edge linking the corresponding vertices in the graph $\widetilde\Gamma$.

\begin{table}
	\centering
	\begin{tabular}{|c|cc|c|cc|c|}
		\hline
		$F$	& $\uxn$&$\uxv$ &$\uyv$	&$\vyn$	&$\vyv$	&$\vxv$\\
		\hline
		$F_{11}$& 21	& 25.6	& 0.57	& 0.739	& 0.792	& 0.15\\
		$F_{31}$& 20.3	& 22.4	& 0.37	& 0.4	& 0.48	& 0.14\\
		$F_{12}$& 22.5	& 34	& 3.1	& 0.72	& 0.785	& 0.5\\
		$F_{32}$& 21	& 25	& 1.9	& 0.4	& 0.48	& 0.07\\
		$F_{23}$& 3.08	& 3.28	& 0.06	& 0.07	& 0.18	& 0.06\\
		$F_{43}$& 3.11	& 3.28	& 0.09	& 0.176	& 0.019	& 0.0591\\
		$F_{24}$& 2.35	& 2.9	& 0.75	& 0.1	& 0.22	& 0.1\\
		$F_{44}$& 2.3	& 2.9	& 0.8	& 0.175	& 0.21	& 0.1\\
		\hline
	\end{tabular}
	\begin{tabular}{|c|cc|c|cc|c|}
		\hline
		$F^{-1}$&$\yvn$	&$\yvv$	&$\yuv$	&$\xun$	&$\xuv$	&$\xvv$\\
		\hline
		$F_{11}^{-1}$	& 1.26	& 1.36	& 0.098	& 0.039	& 0.049	&0.37\\
		$F_{31}^{-1}$	& 2.09	& 2.51	& 0.16	& 0.0444	& 0.049	&0.045\\
		$F_{12}^{-1}$	& 1.342	& 1.384	& 0.0195	& 0.029	& 0.045	&0.18\\
		$F_{32}^{-1}$	& 2.1	& 2.5	& 0.0061	& 0.04	& 0.048	&0.2\\
		$F_{23}^{-1}$	& 5.6	& 14.1	& 0.25	& 0.305	& 0.323	&0.171\\
		$F_{43}^{-1}$	& 5.3	& 5.66	& 0.0094	& 0.303	& 0.323	&0.151\\
		$F_{24}^{-1}$	& 5	& 14.4	& 0.4	& 0.37	& 0.394	&1.41\\
		$F_{44}^{-1}$	& 4.6	& 5.6	& 0.21	& 0.364	& 0.381	&1.55\\
		\hline
	\end{tabular}
	\caption{Constants $\uxv$, $\uxn$ etc. for the maps $F$ and $F^{-1}$.}
	\label{table:uv_xy}
\end{table}

\begin{table}
	\begin{center}
	\begin{tabular}{|c|cccc|}
		\hline
			& $B_1$	& $B_2$	& $B_3$	& $B_4$\\
		\hline
		$c$	&0.00743	&0.0225	&0.022	& 0.0488\\
		$d$	&0.146	&0.321	&0.103	& 0.374\\
		\hline
	\end{tabular}
	\end{center}
	\caption{Constants $c_i$ and $d_i$ for the maps $F$ and $F^{-1}$ on the rectangles $B_i$.}
	\label{table:cd}
\end{table}

\begin{table}
	\begin{center}
	\begin{tabular}{|c|cccccc|}
		\hline
		$F$	&$\sigma_1$&$\sigma_2$&$\sigma_3$&$\sigma_4$&$\sigma_5$&$\sigma_6$\\
		\hline
		$\lambda^-_\sigma$	&0.797	&0.249	&0.0748	&0.0238	&0.0201	&0.0883\\
		$\lambda'^-_\sigma$	&0.734	&0.137	&0.0178	&0.00462	&0.0304	&0.0271\\
		$\alpha^-_\sigma$		&0.739	&0.701	&0.643	&0.695	&0.674	&0.673\\
		\hline
	\end{tabular}
	\end{center}
	\caption{Constants $\lambda^-_\sigma$, $\lambda'^-_\sigma$, and $\alpha^-$ for the map $F$.}
	\label{table:cicles_F}
\end{table}

\begin{table}
	\begin{center}
	\begin{tabular}{|c|cccccc|}
		\hline
		$F^{-1}$&$\sigma_1$&$\sigma_2$&
		$\sigma_3$&$\sigma_4$&$\sigma_5$&$\sigma_6$\\
		\hline
		$\lambda^+_\sigma$	&0.0484	&0.46	&0.000785&0.00829&0.000403&0.0162\\
		$\lambda'^+_\sigma$	&0.0386	&0.288	&0.000341&0.00294&0.0000864&0.0117\\
		$\alpha^+_\sigma$		&0.931	&0.625	&0.895	&0.822	&0.835&0.927\\
		\hline
	\end{tabular}
	\end{center}
	\caption{Constants $\lambda^+_\sigma$, $\lambda'^+_\sigma$, and $\alpha^+$ for the map $F^{-1}$.}
	\label{table:cicles_F_obr}
\end{table} 

Denote by $U_{ij}(X_i,Y_i)$ and $V_{ij}(X_i,Y_j)$ the first and the second coordinate of the image $F_{ij}(X_i,Y_i)$ in the rectangle $B_j$, respectively, such that $F_{ij}(X_i,Y_i) = \bigl(U_{ij}(X_i,Y_i),V_{ij}(X_i,Y_j)\bigr)$. According to the previous subsection the map $F_{ij}$ satisfies the inequalities

\[
    \begin{array}{c}
        \left\{\begin{array}{ccccc}
            \uxn^{ij} |\dx| & \leq & \bigl|U_{ij}(X_i+\dx,Y_i)-U_{ij}(X_i,Y_i)\bigl| & \leq & \uxv^{ij} |\dx|,\\
            &&\bigl|U_{ij}(X_i,Y_i+\dy)-U_{ij}(X_i,Y_i)\bigl| & \leq & \uyv^{ij} |\dy|,\\
        \end{array}\right. \bigskip\\
        \left\{\begin{array}{ccccc}
            \vyn^{ij} |\dy| & \leq & \bigl|V_{ij}(X_i,Y_i+\dy) - V_{ij}(X_i,Y_i)) & \leq & \vyv^{ij} |\dy|,\\
            &&\bigl|V_{ij}(X_i+\dx,Y_i) - V_{ij}(X_i,Y_i)\bigr| & \leq & \vxv^{ij} |\dy|,\\
        \end{array}\right.
    \end{array}
\]

\noindent at every point $(X_i,Y_i)\in B_i\cap F^{-1}(B_j)$  and for every feasible variation $(\Delta X,\Delta Y)$, where $\uxv^{ij}$, $\uxn^{ij}$ etc. denote positive constants whose numerical values are given on the left-hand side of Table \ref{table:uv_xy}\footnote{The numerical values in the tables are rounded in the right direction, i.e., such that the inequalities are not violated.}. Moreover, the inverse mappings $F_{ij}^{-1}$ satisfy the same sort of inequalities with other constants $\xuv^{ij}$, $\xun^{ij}$ etc. whose numerical values are given on the right-hand side of Table \ref{table:uv_xy}.

Note that the values of the constants $\uxv$, $\uxn$ etc. depend on the choice of coordinates $X_i$ and $Y_i$. The values given in Table \ref{table:uv_xy} are valid for coordinates $X_i$ and $Y_i$ explicitly defined by

\begin{eqnarray*}
X_1 = X_3 &=& 3636.746266824261\chi_1 -
 4481.163862258486\chi_2 +\\
 &+&3855.150222649458\chi_1^2 -9508.946041171987\chi_1\chi_2+\\
 &+&5858.268250499643\chi_2^2 + 1403.048095297728\chi_1^3 - \\
 &-&5062.144152379995\chi_1^2\chi_2 +
 6220.813348898283\chi_1\chi_2^2 - \\
 &-&2553.628003178778\chi_2^3, \\
\end{eqnarray*}
\begin{eqnarray*} 
X_2 = X_4 &=& 15(-434.858450049792\chi_1 -
 1178.182552029670\chi_2 - \\
 &-&120.018213318585\chi_1^2 + 1034.238125839267\chi_1\chi_2 +\\
 &+&1197.949813776455\chi_2^2 + 294.581582485282\chi_1^3 -\\
 &-&56.750230435401\chi_1^2\chi_2 -
 570.219924579354\chi_1\chi_2^2 - \\
 &-&401.871126962896\chi_2^3), \\
\end{eqnarray*}
\begin{eqnarray*} 
Y_1 = Y_2 &=& 3.239086633898\chi_1 -
 908.305710474105\chi_2 - \\
 &-&19.706205900010\chi_1^2 -2.212072200027\chi_1\chi_2+\\
 &+&1018.022649246405\chi_2^2 + 6.183323553959\chi_1^3 + \\
 &+&19.428876703623\chi_1^2\chi_2 - 1.286781359343\chi_1\chi_2^2 -\\ 
 &-&379.719994048579\chi_2^3, \\
\end{eqnarray*}
\begin{eqnarray*} 
Y_3 = Y_4 &=& 7714.46037676461\chi_1 +
 57227.20693518666\chi_2 - \\
 &-&1394.23893421762\chi_1^2 -16419.64632698288\chi_1\chi_2 -\\
 &-&61359.76468135373\chi_2^2 + 68.70577470503\chi_1^3 +\\
 &+&1493.59693767978\chi_1^2\chi_2 + 8734.94321423633\chi_1\chi_2^2 +\\
 &+&21935.18377642208\chi_2^3.
\end{eqnarray*}

Since the rectangles $B_i$ define a pre-Markov partition of the set\footnote{For the exact definition of pre-Markov partition, see  \cite{LokutLipsChaos}.} ${\cal B}=\sqcup_i B_i$ with respect to the maps $F_{ij}$, we can apply the apparatus developed in \cite{LokutLipsChaos} for the investigation of Lipschitz hyperbolic dynamical systems. The main condition for the applicability of this technique for the map $F$ is the condition of Lipschitz hyperbolicity:

\begin{defn}
\label{defn:lips_hyperbolic}
	We say that a dynamical system on a pre-Markov partition ${\cal B}=\sqcup_i B_i$ satisfies the conditions of {\it Lipschitz hyperbolicity} if the following conditions hold:

\begin{enumerate}[(I)]
	\item There exist constants $c_i\ge 0$, $i=1,2,3,4$, such that the inequalities

	\[
		\uxn^{ij}-c_i\;\uyv^{ij}>0\mbox{ and }\frac{\displaystyle c_i\;\vyv^{ij} + \vxv^{ij}}{\displaystyle \uxn^{ij}-c_i\;\uyv^{ij}} \leq c_j \mbox{ for each directed edge }(ij)\mbox{ in }
		\widetilde \Gamma
	\]
hold;

	\item The constants $c_i\ge 0$ can be chosen such that for every elementary cycle\footnote{A cycle $(i_0,\ldots i_N,i_0)$ in a directed graph $\widetilde\Gamma$ will be called elementary if it contains each vertex of the graph at most once.} $\sigma=(i_1 i_2 \ldots i_k i_1)$, $i_r\in\{1,2,3,4\}$ the condition

	\begin{equation}
	\label{eq:cotr_along_simple_cicle}
		\lambda_\sigma^-=
		\lambda_{i_1i_2}^-\lambda_{i_2i_3}^-\ldots\lambda_{i_ki_1}^-<1,
	\end{equation}

	holds, where

	\begin{equation}
	\label{eq:lambda_ij}
		\lambda_{ij}^-=\frac{\displaystyle\uxn^{ij}\vyv^{ij} + \uyv^{ij}\vxv^{ij}}{\displaystyle \uxn^{ij} - c_i\uyv^{ij}}.
	\end{equation}
\end{enumerate}

\end{defn}

The map $F$ satisfies the conditions of Lipschitz hyperbolicity on ${\cal B}=\sqcup_i B_i$ with the constants $c_i$ given in Table \ref{table:cd}. Indeed, condition (I) can be verified by direct substitution, and condition (II) can be verified by substitution into all elementary cycles. There are six elementary cycles in the graph $\widetilde\Gamma$, namely 

\begin{eqnarray*}
\sigma_1=(A.A;A.A),\ \sigma_2=(C.C;C.C),\ \sigma_3=(A.A;A.C;C.A;A.A),\\
\sigma_4=(C.C;C.A;A.C;C.C),\ \sigma_5=(A.A;A.C;C.C;C.A;A.A),\\
\end{eqnarray*}

\noindent and $\sigma_6=(A.C;C.A;A.C)$). The numerical values of the products (\ref{eq:cotr_along_simple_cicle}) for the map $F$ are given in Table \ref{table:cicles_F}.

In order to obtain bounds on the dimensions we shall need also the constants $\alpha^-_\sigma$ and ${\lambda^-_\sigma}=
{\lambda_{i_1i_2}^-}'{\lambda_{i_2i_3}^-}'\ldots{\lambda_{i_ki_1}^-}'$, where

\begin{equation}
\label{eq:lambda_ij_prime_and_alpha}
        {\lambda_{ij}^-}' = \frac{\displaystyle\uxn^{ij}\;\vyn^{ij} - \uyv^{ij}\;\vxv^{ij}}{\displaystyle \uxn^{ij} + c_i\uyv^{ij}} \leq
        \lambda_{ij}^-
        \qquad\mbox{and}\qquad
        \alpha^-_\sigma = \frac{\log \lambda_{\sigma}^-}{{\log \lambda_{\sigma}^-}'}.
\end{equation}

\noindent The numerical values of the constants $\alpha^-_\sigma$ and $\lambda'^-_\sigma$ are also given in Table \ref{table:cicles_F}.

The Lipschitz hyperbolicity of $F$ implies that ${\cal B}=\sqcup_i B_i$ is a pre-Markov partition also for the inverse map $F^{-1}$, see \cite{LokutLipsChaos}. Note that in the general case it is not true that the Lipschitz hyperbolicity of $F$ implies that the inverse map $F^{-1}$ satisfies conditions $(I)$ and $(II)$ with some constants $d_i\ge 0$. However, in our case the constants $d_i$ for the map $F^{-1}$ indeed exist, their numerical values are given in Table \ref{table:cd}.

Thus we have proven the following lemma.

\begin{lemma}
\label{lm:F_lips_hyperbolic}
	The maps $F$ and $F^{-1}$ act on a pre-Markov partition ${\cal B}=\sqcup_i B_i$ and satisfy the condition of Lipschitz hyperbolicity on it with positive constants $c_i$ and $d_i$, $i=1,2,3,4$, respectively, whose numerical values are given in Table\footnote{The values of the constants $c_i$ and $d_i$ have been found on a computer with the help of a not very sophisticated optimization procedure. We shall not detail it here, because the reader may verify the validity of the constants $c_i$ and $d_i$ by direct substitution.} \ref{table:cd}.
\end{lemma}

Note that for the map $F^{-1}$ the constants ${\lambda^+_\sigma}'=
{\lambda_{i_1i_2}^+}'{\lambda_{i_2i_3}^+}'\ldots{\lambda_{i_ki_1}^+}'$, ${\lambda^+_\sigma}=
{\lambda_{i_1i_2}^+}{\lambda_{i_2i_3}^+}\ldots{\lambda_{i_ki_1}^+}$, and $\alpha^+_\sigma$, which are given in Table \ref{table:cicles_F_obr},  have been computed by the formulas

\begin{equation}
\label{eq:lambda_alpha_plus}
	{\lambda_{ji}^+}' = \frac{\displaystyle\yvn^{ji}\;\xun^{ji} - \yuv^{ji}\;\xvv^{ji}}{\displaystyle \yvn^{ji} + d_j\xvv^{ji}}\quad \leq \quad
        \lambda_{ji}^+ =    \frac{\displaystyle\yvn^{ji}\;\xuv^{ji} + \yuv^{ji}\;\xvv^{ji}}{\displaystyle \yvn^{ji} - d_j\xvv^{ji}},
        \ 
        \alpha^+_\sigma = \frac{\log \lambda_{\sigma}^+}{{\log \lambda_{\sigma}^+}'}.
\end{equation}

\subsection{Conjugation with the topological Markov chain}

We want to prove that the Lipschitz hyperbolic dynamic system $F:{\cal B}\to{\cal B}$ is semi-conjugated with the two-sided topological Markov chain $\Sigma_{\widetilde\Gamma}$. For this we use the following theorem from \cite{LokutLipsChaos}. Here $S(F)$ denotes the set of points $x\in{\cal B}$ whose images $F^n(x)$ are in ${\cal B}$ for all $n\in\Z$.

\begin{thm}[\cite{LokutLipsChaos}, Theorem 1]
\label{thm:lips_hyperbolic}
	Suppose that the Lipschitz dynamical system $F:{\cal B}\to {\cal B}$ possesses a pre-Markov partition $B=\sqcup_{i=1}^N B_i$ with an oriented graph $\Gamma$. Suppose also that the map $F$ is invertible\footnote{The condition below actually guarantees the invertibility of $F$.}, and that

	\[
	        \uxn^{ij}\vyn^{ij} - \uyv^{ij}\vxv^{ij} > 0\mbox{ for each directed edge }(ij)\mbox{ of the graph }\Gamma,
	\]

\noindent and that $F$ and $F^{-1}$ satisfy the conditions of Lipschitz hyperbolicity\footnote{According to Definition \ref{defn:lips_hyperbolic}.} with constants $c_i$ and $d_i$, respectively, and that $c_i d_i<1$ for $1\leq i\leq N$. Then there exists a homeomorphism $\Psi_\Gamma:\Sigma_\Gamma\to S(F)$ which semi-conjugates $F$ with the left Markov shift $l$, $\Psi_\Gamma\circ l = F\circ \Psi_\Gamma$. Moreover, the non-wandering set $\nw{F}$ is a subset of $S(F)$, and if the graph $\Gamma$ is strongly connected, then $\nw{F}$ and $S(F)$ coincide.
\end{thm}

As a direct consequence of this theorem we obtain the following lemma on the structure of the non-wandering set of points of the map $F$ under consideration.

\begin{lemma}
\label{lm:F_conj_markov}
	The Lipschitz dynamical system $F:{\cal B}\to {\cal B}$ is semi-conjugated with the left shift $l$ on the topological Markov chain $\Sigma_{\widetilde\Gamma}$ of bilaterally infinite paths on the graph $\widetilde\Gamma$ (see Fig.~\ref{fig:graph_AA_CC}), i.e., there exists a continuous embedding $\Psi_{\widetilde\Gamma}:\Sigma_{\widetilde\Gamma}\to{\cal B}$ such that $F\circ\Psi_{\widetilde\Gamma} = \Psi_{\widetilde\Gamma}\circ l$. Moreover, $S(F)=\nw{F}=\Psi_{\widetilde\Gamma}(\Sigma_{\widetilde\Gamma})$.
\end{lemma}

\begin{proof}
	It is trivial to check the conditions of the theorem for the map $F$ under consideration, since the constants $c_i$ and $d_i$ have already been calculated (see Lemma \ref{lm:F_lips_hyperbolic}). Strong connectedness of the graph $\widetilde\Gamma$ is easily seen.
\end{proof}

\subsection{Bounds on the dimensions}

By virtue of Lemma \ref{lm:bilipschitzian_model_E} the map $E$ is bilipschitz in a neighbourhood of the points of $\cal B$. Since bilipschitzian mappings leave the dimensions of sets invariant \cite{Falconer}, we may consider the Hausdorff or Minkowski dimension of the non-wandering set $\nw{F}$.

In order to obtain bounds on the Hausdorff and Minkowski dimensions of the non-wandering set $\nw{F}$ we use another theorem from \cite{LokutLipsChaos}. For the application of this theorem it is sufficient to know the values of the Lipschitz constants $\uxv$, $\uxn$ etc. and the constants $c_i$ and $d_i$.

\begin{thm}[\cite{LokutLipsChaos}, Theorem 5]
    Suppose that the Lipschitz dynamical system $F:{\cal B}\to {\cal B}$ possesses a pre-Markov partition $B=\sqcup_{i=1}^N B_i$ with an oriented graph $\Gamma$. Suppose also that the map $F$ is invertible, and that

	\[
	        \uxn^{ij}\vyn^{ij} - \uyv^{ij}\vxv^{ij} > 0\mbox{ for each directed edge }(ij)\mbox{ of the graph }\Gamma,
	\]

\noindent and that $F$ and $F^{-1}$ satisfy the conditions of Lipschitz hyperbolicity with constants $c_i$ and $d_i$, respectively, and that $c_i d_i<1$ for $1\leq i\leq N$. If the graph $\Gamma$ is strongly connected, then

    $$
        (s'_- + s'_+)\alpha\leq \dim_H\nw{F} \leq \overline\dim_B\nw{F} \leq s_- + s_+.
    $$

    \noindent The constants $s_-$, $s'_-$, $s_+$, and $s'_+$ are chosen such that the spectral radius of the corresponding matrix equals 1,

    $$
        \rho\bigl((\Lambda_-)_{s_-}\bigr)=1,\ \rho\bigl((\Lambda'_-)_{s'_-}\bigr)=1,\
        \rho\bigl((\Lambda_+)_{s_+}\bigr)=1,\ \rho\bigl((\Lambda'_+)_{s'_+}\bigr)=1,\
    $$

    \noindent where $\Lambda_\pm=(\lambda_{ij}^{\pm})$, $\Lambda'_\pm=({\lambda_{ij}^{\pm}}')$, and\footnote{The matrix $A_s$ is obtained by taking element-wise the $s$-th power of the nonnegative matrix $A$. If $s=0$, then $A_0$ is a (0,1)-matrix with ones indicating the non-zero elements of the original matrix $A$.} $A_s=(a_{ij}^s)$ for $A=(a_{ij})$, $a_{ij}\ge 0$, and $\alpha = \min_{\sigma}\left\{\min\{\alpha^-_\sigma,\alpha^+_\sigma\} \right\}$, where the minimum is over all elementary cycles  $\sigma$ in the graph $\Gamma$. The values of $\alpha^\pm_\sigma$ are defined in (\ref{eq:lambda_ij_prime_and_alpha}) and (\ref{eq:lambda_alpha_plus}).

    Here for $(ij)$ a directed edge of $\Gamma$ the numbers $\lambda_{ij}^-$ and ${\lambda_{ij}^-}'$ are given by the formulas (\ref{eq:lambda_ij}) and (\ref{eq:lambda_ij_prime_and_alpha}), and ${\lambda_{ji}^+}$ and ${\lambda_{ji}^+}'$ by the formula (\ref{eq:lambda_alpha_plus}). If $(ij)$ is not an edge of $\Gamma$, then $\lambda_{ij}^-={\lambda_{ij}^-}'=\lambda_{ji}^+={\lambda_{ji}^+}'=0$.
\end{thm}

The numerical values of the constants $s_\pm$, $s'_\pm$, and $\alpha$ for the map $F$ under consideration have been calculated with the aid of a computer and are given by

\begin{equation}
\label{eq:s_pm}
	s_+ = 0.408; \quad s'_+ = 0.327; \quad s_- = 0.876; \quad s'_- = 0.593.
\end{equation}

Note that solving the equation $\rho(\Lambda_s)=1$ with respect to $s$ seems to be a nontrivial task at the first glance. However, the following result on the monotonicity of the function $\rho(\Lambda_s)$ with respect to $s$ guarantees the existence and uniqueness of the solution and indicates a simple numerical solution method.

\begin{prop}[\cite{LokutLipsChaos}]
    Let $\Lambda=(\lambda_{ij})$ be an irreducible $N\times N$ matrix with nonnegative coefficients. If the coefficients $\lambda_{ij}$ satisfy condition (\ref{eq:cotr_along_simple_cicle}) on the contraction along elementary cycles, then the function $\rho(\Lambda_s)$ is continuous and strictly monotonically decreasing with $s$. Moreover, $\rho(\Lambda_0)>1$ and $\rho(\Lambda_s)\to 0$ as $s\to+\infty$.
\end{prop}

The numerical values of the coefficients $\alpha_\sigma^{\pm}$ for $F$ and $F^{-1}$ along all elementary cycles are given in Tables \ref{table:cicles_F} and \ref{table:cicles_F_obr}, respectively. Therefore
\[
	\alpha = \min_{\sigma}\left\{\min\{\alpha^-_\sigma,\alpha^+_\sigma\} \right\} = 0.625\,.
\]

\begin{lemma}
\label{lm:nw_F_dims}
	The Hausdorff and Minkowski dimensions of the non-wandering set $\nw{F}$ satisfy the bounds

	\begin{equation}
	\label{eq:dim_estim_nw_F}
		0.575 \leq \dim_H\nw{F} \leq \overline\dim_B\nw{F} \leq 1.284\,.
	\end{equation}
\end{lemma}

\subsection{Unilateral Markov chain}

In the second part of this paper we will study general Hamiltonian systems with discontinuous right-hand side. For this we will need the description of the set of points in the quotient with respect to the action $g$ which are not attracted to the semi-singular cycles $Z_{ij}$ in the forward time direction, regardless of the asymptotic behaviour in the backward time direction. This will be the subject of this subsection.

Obviously the points of the intersection $Q^k\cap D_{31}$, $k=1,3$, are not attracted to $Z_{13}$ when evolving in the forward time direction under iterations of the Poincar\'e map $\Phi$, because by virtue of Lemma \ref{lm:model_probel_3_4_6_cycles} we have $(ij)\Phi^2(Q^k\cap D_{ij})=Q^k\cap D_{ij}$, $k=i,j$. Moreover, in \cite{ZMHBasic} it was shown that the differential of the Poincar\'e map in the ambient quotient space, $\Phi:{\cal S}/g \to {\cal S}/g$, has exactly one eigenvalue with modulus strictly smaller than 1. By virtue of the Hadamard-Perron theorem \cite{Katok} on ${\cal S}/g$ there exists a smooth immersed one-dimensional stable submanifold ${\cal Q}^k$, consisting of points which are attracted to $Q^k\cap D_{ij}$, $k=i,j$, under the iterations of the map $\Phi$. If we pass to the extended phase space ${\cal M}=T^*M$ we obtain that the trajectories of Hamiltonian system (\ref{eq:model_pmp_system}) emanating from the preimage $\pi^{-1}({\cal Q}^k)$ hit the origin in finite time and are thus optimal by Theorem \ref{thm:model_problem_bellman}. Hence on $D_{31}$ there exist two one-dimensional Lipschitz\footnote{There is no smooth structure on $D_{31}$ which is preserved by the map $E$. Therefore we consider only the Lipschitzian structure on $D_{31}$, which is left invariant by the map $E$, see Theorem \ref{thm:model_problem_bellman} and Lemma \ref{lm:bilipschitzian_model_E}.} immersed submanifolds ${\cal Q}^k\supset Q^k\cap D_{13}$, $k=1,3$, such that for every point $z\in{\cal Q}^k$ the images $\Phi^n(z)$ are not attracted to the semi-singular cycles $Z_{ij}$ as $n\to\infty$.

Besides ${\cal Q}^k$, $k=1,3$, on $D_{31}$ there exists a set $\widetilde{\cal Y}$ of points such that their images do not tend to $Z_{ij}$ neither. The set $\widetilde {\cal Y}$ is fractal, it consists of the union of sets $i_0.i_1i_2...$, where each $i_k$ denotes one of the symbols $A$ or $C$. The structure of these sets is described in Lemma 4 of \cite{LokutLipsChaos}. Each such set is a Lipschitz curve from the functions space $\Lip_{d_{j}}(Y_{j}\to X_{j})$ on the corresponding rectangle $B_{j}\supset i_0.i_1$. Consider the natural map $\Psi_{\widetilde\Gamma}^+$ from $\widetilde{\cal Y}$ to the unilateral Markov chain $\Sigma_{\widetilde\Gamma}^+$ of paths on the graph $\widetilde\Gamma$ which are infinite to the right. The preimage of a point in $\Sigma_{\widetilde\Gamma}^+$ is the corresponding Lipschitz curve $i_0.i_1i_2...$. It is clear that by construction the map $\Psi_{\widetilde\Gamma}^+$ semi-conjugates the left shift $l$ on $\Sigma_{\widetilde\Gamma}^+$ with the map $F$, i.e., $l\circ\Psi_{\widetilde\Gamma}^+=\Psi_{\widetilde\Gamma}^+\circ F$. The map $\Psi_{\widetilde\Gamma}^+$ is continuous by virtue of Remark 1 of \cite{LokutLipsChaos}. Bounds on the dimension of the set $\widetilde{\cal Y}$ have also been calculated in \cite{LokutLipsChaos}\footnote{Unfortunately, in \cite{LokutLipsChaos} this bound has not been singled out into a separate proposition. For details we refer the reader to the proof of Theorem 5. More precisely, above bounds on the dimensions $\dim_H \widetilde{\cal Y}$ and $\overline\dim_B \widetilde{\cal Y}$ are exactly the bounds on the dimensions of the sets $Q_i^\pm$ in \cite{LokutLipsChaos}, which have been obtained in the proof of Theorem 5.},

\[
	1+\alpha_+s'_+ \le \dim_H \widetilde{\cal Y} \le \overline\dim_B \widetilde{\cal Y} \le 1 + s_+,
\]

\noindent where $\alpha_+ = \min_{\sigma}\frac{\log \lambda_{\sigma}^+}{{\log \lambda_{\sigma}^+}'}$.

We have proven the following lemma.

\begin{lemma}
	\begin{enumerate}
		\item On the disc $D_{31}$ there exists a one-dimensional immersed Lipschitz submanifold ${\cal Q}^2$, consisting of points which under the iterations of the Poincar\'e map $\Phi$ tend to the intersection points of the four-link cycle $Q^2$ with ${\cal S}/g$.
		\item On the disc $D_{31}$ there exists a set $\widetilde{\cal Y}\subset{\cal B}$ with the following properties. (i) For every point $z\in\widetilde{\cal Y}$ the images $F^n(z)$ do not tend to the semi-singular cycles $Z_{ij}$ as $n\to+\infty$. (ii) The restriction of $F$ on $\widetilde{\cal Y}$ is by the continuous map $\Psi_{\widetilde\Gamma}^+$ semi-conjugated to the left shift on the unilateral Markov chain $\Sigma_{\widetilde\Gamma}^+$. Here the preimage of any point of $\Sigma_{\widetilde\Gamma}^+$ is a one-dimensional Lipschitz curve in $\cal B$. (iii) The dimensions of the set $\widetilde{\cal Y}$ satisfy the bounds
		
		\[
			1.204 \le \dim_H \widetilde{\cal Y} \le \overline\dim_B \widetilde{\cal Y} \le 1.408.
		\]
	\end{enumerate}
\label{lm:oneside_chain}
\end{lemma}

\begin{figure}
    \centering
    \begin{tikzpicture}[->,>=stealth',shorten >=1pt,auto,node
    distance=3.4cm,thick,main node/.style={circle,fill=blue!10,
    draw,font=\sffamily\Large\bfseries}, minimum size=10mm]

    \node[main node] (1) {$\cal A$};
    \node[main node] (2) [right of=1] {$\cal B$};
    \node[main node] (3) [below right of=2] {$\cal C$};
    \node[main node] (4) [above right of=2] {$\cal D$};

    \path[every node/.style={font=\sffamily\small}]
        (1) edge [bend right] node [below] {r} (2)
        (2) edge [bend right] node [above] {r} (1)
        (2) edge [bend right] node [below] {t} (3)
        (3) edge [bend right] node [right] {r} (4)
        (4) edge [bend right] node [above] {t} (2);
    \end{tikzpicture}

    \caption{The graph $\widehat\Gamma$. The letters $r$ and $t$ are irrelevant for the definition of the graph $\widehat\Gamma$. They are necessary for the construction of the graph $\Gamma$ in Definition~\ref{defn:graph_Gamma}.}
\label{fig:graph_hatted_Gamma}
\end{figure}
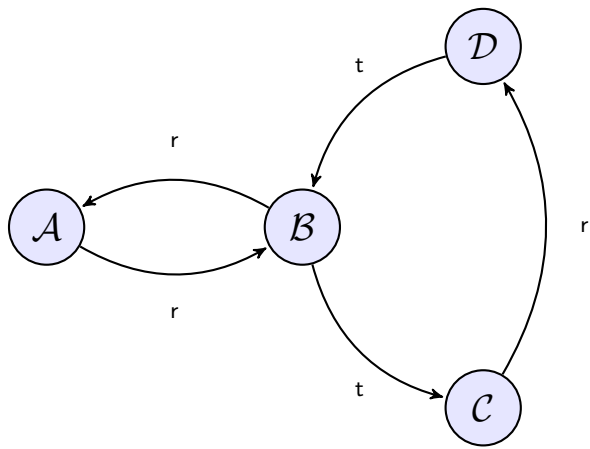

\subsection{The fractal nature of the Poincar\'e map}

The map $F$ plays of course an auxiliary role. Our main goal is the investigation of the Poincar\'e return map $\Phi$. In order to formulate the main theorem on the map $\Phi$ we shall need the following graph $\Gamma$.

\begin{defn}
\label{defn:graph_Gamma}
	The graph $\widehat\Gamma$ depicted on Fig.~\ref{fig:graph_hatted_Gamma} is a prototype of the directed graph $\Gamma$. The graph $\Gamma$ can be obtain from $\widehat\Gamma$ as follows. The vertex set of $\Gamma$ is a direct product of the vertex set of $\widehat\Gamma$ with the set of ordered pairs $(ij)$, where $i,j \in \{1,2,3\},\; (i\ne j)$. Hence the graph $\Gamma$ has 24 vertices $A_{12},A_{13}, ... D_{32}$. In $\Gamma$ there exists a directed edge from the vertex $A_{ij}$ to the vertex $B_{i'j'}$ if and only if $j=i'$ and $i\ne j'$, i.e., $\{i,j=i',j'\}=\{1,2,3\}$. This rule applies to all edges in $\widehat\Gamma$ which are marked by the letter $r$. There exists a directed edge from the vertex $A_{ij}$ to the vertex $C_{i'j'}$ if and only if $i'=j$ and $j'=i$. This rule applies to all edges in $\widehat\Gamma$ which are marked by the letter $t$. There are no other edges in $\Gamma$.
\end{defn}

Note that the graph $\Gamma$ consists of two disjoint and mutually isomorphic components $\Gamma_\pm$, which are strongly connected and each possess one component in the spectral decomposition\footnote{A graph possesses one component in the spectral decomposition if the greatest common divisor of the lengths of its elementary cycles equals 1 \cite{Katok}.}.

\begin{thm}
\label{thm:poincare_map_chaotic_structure_equilateral_triangle}
	The dynamics of the Poincar\'e map $\Phi:({\cal S}\cap M_+)/g\to({\cal S}\cap M_+)/g$ can be described as follows:
	
	\begin{enumerate}
		\item The non-wandering set ${\cal Y}=\nw{\Phi}\subset({\cal S}\cap M_+)/g$ of the map $\Phi$ is homeomorphic to the topological chain $\Sigma_\Gamma$ of bilaterally infinite paths on the graph $\Gamma$. The homeomorphism is realized by a mapping $\widehat\Psi_\Gamma$ which conjugates $\Phi$ and the left shift $l$ on $\Sigma_\Gamma$, i.e., the following diagram commutes:
		
		\begin{center}
		\begin{tikzpicture}[description/.style={fill=white,inner sep=2pt}]
		        \matrix (m) [matrix of math nodes, row sep=1em, column sep=2em, text height=1.5ex, text depth=0.25ex]
		        {    {\cal Y} & {\cal Y} \\
		             \Sigma_\Gamma & \Sigma_\Gamma \\ };
		        \path[->,font=\scriptsize]
		        (m-1-1) edge node[auto] {$\Phi$} (m-1-2)
		        (m-1-2) edge node[auto] {$\widehat\Psi_\Gamma$} (m-2-2)
		        (m-1-1) edge node[auto] {$\widehat\Psi_\Gamma$} (m-2-1)
		        (m-2-1) edge node[auto] {$l$} (m-2-2);
		\end{tikzpicture}
		\end{center}
		
		\noindent The dimensions of the set ${\cal Y}$ satisfy the bounds

		\[
			0.575 \le \dim_H {\cal Y} \le \overline\dim_B {\cal Y} \le 1.284.
		\]

		\item The set of points $z\in({\cal S}\cap M_+)/g$ which are not attracted to one of the semi-singular cycles $Z_{ij}$ under the iterations of $\Phi$ contains (i) the stable manifolds ${\cal Q}^i$ of $\Phi$ emanating from the hyperbolic points $Q^i\cap ({\cal S}/g)$; (ii) a set ${\cal Y}^+\subset({\cal S}\cap M_+)/g$ which is invariant under $\Phi$, i.e., $\Phi({\cal Y}^+)\subset{\cal Y}^+$. The dynamics of $\Phi$ on ${\cal Y}^+$ is semi-conjugated to the left shift $l$ on the topological Markov chain $\Sigma_\Gamma^+$ of paths on the graph $\Gamma$ which are infinite to the right. The conjugation is realized by some map $\widehat\Psi_\Gamma^+:{\cal Y}^+\to \Sigma_\Gamma^+$, i.e., the following diagram commutes:
		
		\begin{center}
		\begin{tikzpicture}[description/.style={fill=white,inner sep=2pt}]
		        \matrix (m) [matrix of math nodes, row sep=1em, column sep=2em, text height=1.5ex, text depth=0.25ex]
		        {    {\cal Y}^+ & {\cal Y}^+ \\
		             \Sigma_\Gamma^+ & \Sigma_\Gamma^+ \\ };
		        \path[->,font=\scriptsize]
		        (m-1-1) edge node[auto] {$\Phi$} (m-1-2)
		        (m-1-2) edge node[auto] {$\widehat\Psi_\Gamma^+$} (m-2-2)
		        (m-1-1) edge node[auto] {$\widehat\Psi_\Gamma^+$} (m-2-1)
		        (m-2-1) edge node[auto] {$l$} (m-2-2);
		\end{tikzpicture}
		\end{center}
		
		\noindent The map $\Psi_\Gamma^+$ is continuous, surjective, and the preimage of any point is a Lipschitz submanifold homeomorphic to a line segment. The dimensions of the set ${\cal Y}^+$ satisfy the bounds
		
		\[
			1.204 \le \dim_H {\cal Y}^+ \le \overline\dim_B {\cal Y}^+ \le 1.408.
		\]

	\end{enumerate}
\end{thm}

Note that item (ii) of Theorem \ref{thm:poincare_map_chaotic_structure_equilateral_triangle} does not describe all points $z\in({\cal S}\cap M_+)/g$ which do not tend to one of the semi-singular cycles $Z_{ij}$ under the iteration of the map $\Phi$. For instance, Theorem \ref{thm:poincare_map_chaotic_structure_equilateral_triangle} does not mention the three-link cycles.

\begin{proof}[Proof of Theorem \ref{thm:poincare_map_chaotic_structure_equilateral_triangle}]

	We start with the proof of item 1. Denote by $S(\Phi)$ the set of points $z\in{\cal S}/g$ such that the images $\Phi^k(z)$ do not tend to any of the three semi-singular cycles $Z_{ij}$ in the forward time direction, i.e., as  $k\to+\infty$, and do not tend to any of the two three-link cycles $Z^\pm$ in backward time direction, i.e., as $k\to-\infty$.

	In order to describe the set $S(\Phi)$ and the dynamical system $\Phi:S(\Phi)\to S(\Phi)$ we consider the intersection of the set $S(\Phi)$ with the disc $D_{31}$. Since the map $F$ is given by $\Phi^2$ on $A.A\cup C.A$ and by $\Phi^3$ on $A.C\cup C.C$, the set $S(F)$ contains the sets $\nw{F}$, $\Phi[\nw{F}]$, and $\Phi^2[\nw{F}]$. Applying all permutations from $S_3$ to $\nw{F}$ we obtain that

	\[
		S(\Phi) \supset S_1(\Phi) =
		\bigsqcup_{\sigma\in S_3}\sigma\Bigl(\nw{F}\cup \Phi\bigl[\nw{F}\bigr] \cup \Phi^2\bigl[\nw{F}\bigr]\Bigr).
	\]

	In Subsections \ref{subsec:resolve_B} and \ref{subsec:AC_dynamics} it has been shown that the intersection of the difference $S(\Phi)\setminus S_1(\Phi)$ with the closure of region III on the disc $D_{31}$ consists of points $z$ corresponding to a sequence $(\ldots i_{-1}i_0.i_1\ldots)$, $i_k\in\{A,B,C\}$, ending with an infinite tail of symbols $B$ only. Any such point $z$ lies on the stable manifold ${\cal Q}^2$ and its images under $\Phi$ tend to the four-link cycle $Q^2$ in the forward direction of time. The other points of the intersection $(S(\Phi)\setminus S_1(\Phi)) \cap D_{31}$ can be obtained by application of $\Phi$ or $\Phi^2$ to the points $z$ described previously and a corresponding permutation in $S_3$. Hence we obtain

	\[
		S(\Phi)\setminus S_1(\Phi) = \tilde S(\Phi)\subset \bigsqcup_i {\cal Q}^i,
	\]

	\noindent and therefore the points of the set $\tilde S(\Phi)$ are not non-wandering points of the map $\Phi$. The points in the set $({\cal S}/g)\setminus S(\Phi)$ are also not non-wandering points of $\Phi$, since they are attracted to one of the three-link cycles $Z^\pm$ under the iterations of $\Phi^{-1}$, or are attracted to some point on one of the semi-singular segments $Z_{ij}^s$ under the iterations of $\Phi$. Therefore

	\[
		\nw{\Phi} = S_1(\Phi) =
		\bigsqcup_{\sigma\in S_3}\sigma\Bigl(\nw{F}\cup \Phi\bigl(\nw{F}\bigr) \cup \Phi^2\bigl(\nw{F}\bigr)\Bigr).
	\]

	Taking into account the results of Lemmas \ref{lm:F_conj_markov} and \ref{lm:nw_F_dims}, it remains to explain how the graph $\Gamma$ arises and to construct the map $\widehat\Psi_\Gamma$ in order to complete the proof of item 1.
	
	First note that the topological Markov chain $\Sigma_{\widetilde\Gamma}$ is identical to the chain $\Sigma_{AC}$ of bilaterally infinite words with the two-letter alphabet $\{A,C\}$. Recall that up to the application of a permutation returning the point to the disc $D_{31}$ we have the identity $F=\Phi^3$ on the set $A.C\cup C.C$ and the identity $F=\Phi^2$ on $A.A\cup C.A$. Let us consider the map $\widehat \Phi:z\mapsto \sigma(\Phi(z))$ on the set $\nw{\Phi}\cap D_{31}$, where $\sigma\in S_3$ takes $\Phi(z)$ to $D_{31}$. Then the map $\Psi_{\tilde\Gamma}$ from Lemma \ref{lm:F_conj_markov} naturally gives rise to a homeomorphism $\widehat\Psi_{\widehat\Gamma}$ conjugating the dynamical system $\widehat\Psi:\nw{\Phi}\cap D_{31}\to \nw{\Phi}\cap D_{31}$ with the left shift on the topological Markov chain $\Sigma_{\widehat\Gamma}$. Here the directed graph $\widehat\Gamma$ has four vertices and is depicted on Fig.~\ref{fig:graph_hatted_Gamma}. In order to return to the original map $\Phi$ we have to take into account all discs $D_{ij}$ and to track in which of the six discs each image $\Phi^k(z)$ is located. This can be easily accomplished by utilizing the letters $r$ and $t$ which label the edges of the graph on Fig.~\ref{fig:graph_hatted_Gamma} (recall that $A=rr$ and $C=trt$). For the disc $D_{31}$ the letter $r$ corresponds to the permutation $(312)\in S_3$, and the letter $t$ to the transposition $(13)\in S_3$. This finally yields that the restriction of the map $\Phi$ to the non-wandering set $\nw{\Phi}$ is conjugated to the left shift on the chain $\Sigma_\Gamma$. The first part of the theorem has been proven.
	
	\smallskip
	
	Item (i) of the second part of the theorem is trivial. The proof of item (ii) can be based on Lemma \ref{lm:oneside_chain} and is identical to the proof of the first part, given we set, e.g.,
	
	\[
		{\cal Y}^+ = \bigsqcup_{\sigma\in S_3}\sigma\Bigl(
		\widetilde{\cal Y}\cup \Phi\bigl(\widetilde{\cal Y}\bigr) \cup \Phi^2\bigl(\widetilde{\cal Y}\bigr)
		\Bigr).
	\]
	
	\noindent
\end{proof}

\begin{remark}
	The set ${\cal Y}^+$ can be chosen in many ways. One may replace it, e.g., by $\Phi^k({\cal Y}^+)$ for any $k\in\Z$. However, working with the union $\bigcup_{k\in\Z}\Phi^k({\cal Y}^+)$ encounters some difficulties. Firstly, its Minkowski dimension may be strictly greater than the Minkowski dimension of ${\cal Y}^+$. Moreover, the preimage of any point of $\Sigma_\Gamma^+$ in $\bigcup_{k\in\Z}\Phi^k({\cal Y}^+)$ is no more a Lipschitz submanifold, but rather an immersed Lipschitz submanifold. This second difficulty is similar to the mingling of the stable and unstable manifolds of some fixed point of a hyperbolic map if we move farther away from the fixed point.
\end{remark}

\subsection{Main theorem on the chaotic nature of the optimal synthesis in the model problem}

We are now ready to formulate the main theorem of the first part of this paper, describing the dynamics of the optimal trajectories in the original optimal control problem (\ref{problem:model}). The first three items of the theorem describe the considered sets ${\cal X}$ and ${\cal X}^+$, consisting of trajectories of Hamiltonian system (\ref{eq:model_pmp_system}), and the remaining items describe the chaotic dynamics on these sets. Since every trajectory on the sets ${\cal X}$ and ${\cal X}^+$ intersects the stratified manifold ${\cal S}={\cal S}_{12}\cup{\cal S}_{13}\cup{\cal S}_{23}$ of discontinuity of the right-hand side of system (\ref{eq:model_pmp_system}), the chaotic dynamics of these trajectories can best be described in terms of the sequence of intersections of the strata ${\cal S}_{ij}$.

\begin{thm}
\label{thm:model_chaos_equilateral_triangle}
	In the extended phase space ${\cal M}=T^*M$ of problem (\ref{problem:model}) there exist two sets ${\cal X}$ and ${\cal X}^+$ with the following properties.
	
	\begin{enumerate}[(I)]
		\item For every point $z\in{\cal X}\cup {\cal X}^+$ there exists a time $T(z)<\infty$ such that the trajectory $X(t,z)$ of Hamiltonian system (\ref{eq:model_pmp_system}) through $z$ exists and is unique for all $t\in(-\infty,T(z)]$. The trajectory $X(t,z)$ hits the origin at time $T(z)$, $X(T(z),z)=0$.
		
		\item The sets $\cal X$ and ${\cal X}^+$ are invariant with respect to the flow of Hamiltonian system (\ref{eq:model_pmp_system}) in the following sense. If $z\in {\cal X}^+$, then $X(t,z)$ lies in ${\cal X}^+$ for all $t\in[0,T(z))$. If $z\in {\cal X}$, then $X(t,z)$ lies in ${\cal X}$ for all $t\in(-\infty,T(z))$.
		
		\item The projection of the trajectory $X(t,z)$ on phase space $M$, extended by $0$ for $t>T(z)$, is optimal for all $z\in {\cal X}\cup{\cal X}^+$, i.e., ${\cal X}\cup{\cal X}^+\subset M_+$. Moreover, the trajectory $X(t,z)$ intersects the switching surface $\cal S$ a countable number of times at time instants $\ldots<t_{-1}<t_0\le 0 <t_1<t_2\ldots<T(z)$, $X(t_k,z)\in{\cal S}$, where $t_k\to T(z)$ as $k\to+\infty$ and $t_k\to-\infty$ as $k\to-\infty$.
		
		\item The dynamical system $\Phi:{\cal X}\cap {\cal S}\to{\cal X}\cap{\cal S}$ which takes a point $z\in{\cal X\cap\cal S}$ to the next intersection point of the trajectory $X(t,z)$ with the switching surface ${\cal S}$, $\Phi(z) = X(t_1,z)$, is semi-conjugated with the topological Markov chain of bilaterally infinite paths on the graph $\Gamma$ by means of some map $\Psi_\Gamma$:

		\begin{center}
		\begin{tikzpicture}[description/.style={fill=white,inner sep=2pt}]
		        \matrix (m) [matrix of math nodes, row sep=1em, column sep=2em, text height=1.5ex, text depth=0.25ex]
		        {    {\cal X\cap S} & {\cal X\cap S} \\
		             \Sigma_\Gamma & \Sigma_\Gamma \\ };
		        \path[->,font=\scriptsize]
		        (m-1-1) edge node[auto] {$\Phi$} (m-1-2)
		        (m-1-2) edge node[auto] {$\Psi_\Gamma$} (m-2-2)
		        (m-1-1) edge node[auto] {$\Psi_\Gamma$} (m-2-1)
		        (m-2-1) edge node[auto] {$l$} (m-2-2);
		\end{tikzpicture}
		\end{center}
		
		\noindent The map $\Psi_\Gamma$ is continuous and surjective, and the preimage of any point $\sigma\in\Sigma_\Gamma^+$ is a one-dimensional smooth manifold, diffeomorphic to an open interval. Denote by $.V$ the set of all paths on $\Gamma$ which commence at some vertex $V$. Then the preimages $\Psi_\Gamma^{-1}[.A_{ij}]$, $\Psi_\Gamma^{-1}[.B_{ij}]$, $\Psi_\Gamma^{-1}[.C_{ij}]$, and $\Psi_\Gamma^{-1}[.D_{ij}]$ lie on the stratum ${\cal S}_{ij}$.

		\item The dynamical system $\Phi:{\cal X}^+\cap {\cal S}\to{\cal X}^+\cap{\cal S}$ which takes a point $z\in{\cal X_+\cap\cal S}$ to the next intersection point of the trajectory $X(t,z)$ with the switching surface ${\cal S}$, $\Phi(z) = X(t_1,z)$, is semi-conjugated with the left shift on the topological Markov chain $\Sigma_\Gamma^+$ of paths on $\Gamma$ which are infinite to the right (see Definition \ref{defn:graph_Gamma}). The conjugation is realized by some map $\Psi_\Gamma^+$, i.e., the following diagram is commutative:
		
		\begin{center}
		\begin{tikzpicture}[description/.style={fill=white,inner sep=2pt}]
		        \matrix (m) [matrix of math nodes, row sep=1em, column sep=2em, text height=1.5ex, text depth=0.25ex]
		        {    {\cal X_+\cap S} & {\cal X_+\cap S} \\
		             \Sigma_\Gamma^+ & \Sigma_\Gamma^+ \\ };
		        \path[->,font=\scriptsize]
		        (m-1-1) edge node[auto] {$\Phi$} (m-1-2)
		        (m-1-2) edge node[auto] {$\Psi_\Gamma^+$} (m-2-2)
		        (m-1-1) edge node[auto] {$\Psi_\Gamma^+$} (m-2-1)
		        (m-2-1) edge node[auto] {$l$} (m-2-2);
		\end{tikzpicture}
		\end{center}
		
		\noindent The map $\Psi_\Gamma^+$ is continuous and surjective, and the preimage of any point $\sigma\in\Sigma_\Gamma^+$ is a Lipschitz submanifold, whose relative interior is homeomorphic to a two-dimensional open disc. The preimages $(\Psi_\Gamma^+)^{-1}[.A_{ij}]$, $(\Psi_\Gamma^+)^{-1}[.B_{ij}]$, $(\Psi_\Gamma^+)^{-1}[.C_{ij}]$, and $(\Psi_\Gamma^+)^{-1}[.D_{ij}]$ lie on the stratum ${\cal S}_{ij}$.
		
		\item The Hausdorff and Minkowski dimensions of ${\cal X}$ and ${\cal X}^+$ satisfy the bounds
		
		\[
			3.204 \le \dim_H {\cal X}^+ \le \overline\dim_B {\cal X}^+ \le 3.408,
		\]
		\[
			2.575 \le \dim_H {\cal X} \le \overline\dim_B {\cal X} \le 3.284.
		\]

		\item The topological entropy of the left shift $l$ on $\Sigma_\Gamma$ and $\Sigma_\Gamma^+$ equals
		
		\[
		        h_{\mathrm{top}}(l) =
			\log_2\left(
			\sqrt[3]{\frac{1}{2} + \frac{\sqrt{69}}{18}} +
			\sqrt[3]{\frac{1}{2} - \frac{\sqrt{69}}{18}}\,
			\right)
			\approx 0.4057.
		\]

		
		\item Similar assertions hold for trajectories of Hamiltonian system (\ref{eq:model_pmp_system}) which leave the origin, except that they are no more optimal in problem (\ref{problem:model}).

	\end{enumerate}
\end{thm}

\begin{proof}

	In order to construct the sets ${\cal X}$ and ${\cal X}^+$ we use the sets ${\cal Y}$ and ${\cal Y}^+$, respectively, which have been constructed in Theorem \ref{thm:poincare_map_chaotic_structure_equilateral_triangle}. There are two differences between these pairs of sets. Firstly, the latter pair lies in quotient space $M_+/g$, while the former lies in extended phase space $\cal M$. Secondly, the sets ${\cal Y}$ and ${\cal Y}^+$ lie on the switching surface ${\cal S}/g$ and do not contain the optimal trajectories itself. Let us first construct intermediate sets $\cal Z$ and $\cal Z_+$. For every point in ${\cal Y}$ or ${\cal Y}^+$ we place the projection of the trajectory of Hamiltonian system (\ref{eq:model_pmp_system}) through this point into the set ${\cal Z}$ or ${\cal Z}_+$, respectively. The set $\cal Z$ shall contain the whole trajectory, and the set ${\cal Z}_+$ only half of it, corresponding to the forward time direction. Note that ${\cal Z}\cap ({\cal S}/g) = {\cal Y}$ and ${\cal Z}_+\cap ({\cal S}/g) = {\cal Y}^+$, since the sets ${\cal Y}$ and ${\cal Y}^+$ are invariant with respect to the map $\Phi$, and therefore no new points on ${\cal S}$ will appear. The sets ${\cal X}$ and ${\cal X}^+$ are then simply defined as the preimages of $\cal Z$ and $\cal Z_+$ under the projection $\pi:{\cal M}\to{\cal M}/g$. Since the sets ${\cal Y}$ and ${\cal Y}^+$ lie in $M_+/g$, we have that every trajectory through a point in $\pi^{-1}({\cal Y})$ or $\pi^{-1}({\cal Y}^+)$ is optimal and reaches the origin in finite time. This proves items (I-III).
	
	The semi-conjugating maps $\Psi_\Gamma$ and $\Psi_\Gamma^+$ are both constructed by the same procedure. First we apply the projection $\pi$ onto the factor space ${\cal M}/g$, and then the corresponding map $\widehat\Psi_\Gamma$ or $\widehat\Psi_\Gamma^+$ from Theorem \ref{thm:poincare_map_chaotic_structure_equilateral_triangle}:
	
	\[
		\Psi_\Gamma = \widehat\Psi_\Gamma\circ\pi\quad\mbox{ and }\quad\Psi_\Gamma^+ = \widehat\Psi_\Gamma^+\circ\pi.
	\]

	Let us first study the properties of the map $\Psi_\Gamma$. Clearly it is continuous as a composition of continuous maps. It is surjective as a composition of a surjective and a bijective map. In order to avoid confusion, we shall now denote the Poincar\'e map on the original switching surface ${\cal S}$ by $\Phi^\uparrow$ and on the quotient ${\cal S}/g$ by $\Phi^\downarrow$. Then the commutativity of the diagram from item (IV) follows from
	
	\[
		\Psi_\Gamma\circ\Phi^\uparrow = \widehat\Psi_\Gamma\circ\pi\circ\Phi^\uparrow =
		\widehat\Psi_\Gamma\circ\Phi^\downarrow\circ\pi = l\circ \widehat\Psi_\Gamma\circ\pi=
		l\circ\Psi_\Gamma.
	\]

	\noindent Here the second relation holds because Hamiltonian system (\ref{eq:model_pmp_system}) respects the action $g$, and hence this action is respected by the Poincar\'e map. The third relation holds by virtue of Theorem \ref{thm:poincare_map_chaotic_structure_equilateral_triangle}. The preimage $\pi^{-1}(z)$ of any point $z\in M_+/g$ is given by a smooth one-dimensional curve, namely the orbit of the action $g$ of the group $\R_+$. Since the map $\widehat\Psi_\Gamma$ is bijective, we have that the preimage $\Psi_\Gamma^{-1}(\sigma)$ is a smooth curve for every $\sigma\in\Sigma_\Gamma$. This proves item (IV).
	
	The proof of item (V) for the mapping $\Psi_\Gamma^+$ is almost identical to the proof of item (IV). The only difference appears in the consideration of the preimage $(\Psi_\Gamma^+)^{-1}(\sigma)$ of points $\sigma\in\Sigma_\Gamma^+$. In Theorem \ref{thm:poincare_map_chaotic_structure_equilateral_triangle} it was shown that $(\widehat\Psi_\Gamma^+)^{-1}(\sigma)$ is a Lipschitz curve in $({\cal S}\cap M_+)/g$. By virtue of the structure of the projection $\pi^{-1}$ we obtain that $(\Psi_\Gamma^+)^{-1}(\sigma)$ is a two-dimensional Lipschitz submanifold, homeomorphic to the direct product of a segment and an interval.

	\smallskip
	
	In order to obtain the bounds in item (VI) we shall show that
	
	\[
		\dim_H{\cal X} \ge 1 + \dim_H {\cal Z}\ge 2 + \dim_H {\cal Y};\qquad
		\overline\dim_B{\cal X} \le 1 + \overline\dim_B {\cal Z}\le 2 + \overline\dim_B {\cal Y};
	\]
	\[
		\dim_H{\cal X}^+ \ge 1 + \dim_H {\cal Z}_+\ge 2 + \dim_H {\cal Y}^+;\ 
		\overline\dim_B{\cal X}^+ \le 1 + \overline\dim_B {\cal Z}_+\le 2 + \overline\dim_B {\cal Y}^+.
	\]
	
	\noindent These bounds are obtained by application of the formulas for the Hausdorff and Minkowski dimensions of direct products of sets, see \cite{Falconer}, Chapter 7.
	
	\smallskip
	
	Let us now compute the topological entropies of the left shifts on $\Sigma_\Gamma$ and $\Sigma_\Gamma^+$. Since the graph $\Gamma$ is not connected and consists of two disjoint copies of some graph $\Gamma'$, we may replace the graph $\Gamma$ by $\Gamma'$ for the calculus of the topological entropy. The graph $\Gamma'$ can be obtained from $\Gamma$ by the identification of the vertices ${\cal A}_{ij}$ with ${\cal A}_{ji}$, ${\cal B}_{ij}$ with ${\cal B}_{ji}$ etc. Denote by $A_{01}$ the $0,1$-matrix generated by the graph $\Gamma'$. The matrix $A_{01}$ is irreducible, because $\Gamma'$ is strongly connected and the greatest common divisor of the length of its cycles is 1. By virtue of the Perron-Frobenius theorem the matrix $A_{01}$ then has a unique positive eigenvalue $\lambda_{\max}>0$. The sought topological entropies then coincide and equal $\log_2(\lambda_{\max})$ \cite{Katok}.
	
	Let $p$ be the eigenvector of $A_{01}$ corresponding to the eigenvalue $\lambda_{\max}$. Since both $\lambda_{\max}$ and $p$ are unique (up to a homothety), $p$ must be invariant with respect to all coordinate permutations which are induced by the action of the group $S_3$ on the graph $\Gamma'$. Hence if the vertex ${\cal A}_{12}$ corresponds to an entry $a$ of the vector $p$, then all other vertices of the form ${\cal A}_{ij}$ also correspond to an entry $a$ in $p$. The vertices ${\cal B}_{ij}$ and ${\cal C}_{ij}$ then correspond to the entry $\lambda_{\max} a$, and the vertices ${\cal D}_{ij}$ to the entry $\lambda_{\max}^2 a$. We then obtain the algebraic relation $a = \lambda_{\max}^2 a + \lambda_{\max}^3 a$. Therefore
	
	\[
		\lambda_{\max}^3 + \lambda_{\max}^2 - 1 = 0.
	\]

	\noindent The value of $\lambda_{\max}$ claimed in the theorem can then be obtained by the Cardano formula.
	
	\smallskip
	
	The last item of the theorem is proven by application of the inversion $g(-1)$ to Hamiltonian system (\ref{eq:model_pmp_system}).
\end{proof}

\begin{remark}
	The Hausdorff and Minkowski dimensions of the sets $\cal X$ and ${\cal X}^+$ do not change if we project the sets on the phase space $M=\{(x,y)\}=\R^4$, because the map $E$ is locally Lipschitz and hence the projection  $M_+=\mathrm{graph}\,E\to M$ is locally bilipschitz.
\end{remark}

\begin{remark}
	The set $\Xi$ defined in Theorem \ref{thm:model_chaos_any_triangle} for the case of an equilateral triangle $\Omega$ is a subset of $\cal X$.
\end{remark}

\section{Chaos on finite time intervals in Hamiltonian systems with discontinuous right-hand side}
\label{sec:mercedes}

In this section we prove and formulate two theorems on the chaotic behaviour of trajectories in integral vortices in Hamiltonian systems with discontinuous right-hand side in generic position. The proofs of these theorems are based on the structure of optimal synthesis in model problem (\ref{problem:model}), more precisely on Theorems \ref{thm:model_chaos_any_triangle} and \ref{thm:model_chaos_equilateral_triangle}.

\subsection{Hamiltonian systems with discontinuous right-hand side}

Consider a smooth symplectic manifold ${\cal M}^{2n}$ of dimension $2n$. Suppose that a $(2n-1)$-dimensional stratified submanifold ${\cal S}^H \subset {\cal M}$ partitions ${\cal M}$ into a finite number of open domains $\Omega_1,\dots,\Omega_k$, $({\cal M}=\overline{\bigcup \Omega_i})$. Consider a smooth Hamiltonian $H(q,p): {\cal M} \to  \R$ such that its restriction $H_i = H|_{\Omega_i}$ to any of the domains $\Omega_i$ is a smooth function, which is $C^\infty$-extendable to a neighbourhood of the closure $\overline{\Omega_i}$. Consider an open set ${\cal U}\subset {\cal M}$. Suppose that the set ${\cal U}$ contains parts of only three $(2n-1)$-dimensional strata ${\cal S}^H_{ij} \subset {\cal S}^H, \,(i,j=1,2,3)$, such that ${\cal S}^H_{ij}$ separates the domains $\Omega_i$ and $\Omega_j$. Let the submanifolds ${\cal S}^H_{ij}$ join in a stratum ${\cal S}^H_{123}$ of dimension $(2n-2)$, as depicted on Fig.~\ref{fig:mercedes}.

\begin{figure}
	\centering
	\includegraphics[width=0.3\textwidth]{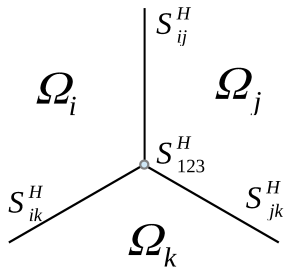}
	\caption{Arrangement of the strata ${\cal S}^H_{ij}$, ${\cal S}^H_{123}$ and the domains $\Omega_i$}
	\label{fig:mercedes}
\end{figure}


The main tool in the study of the generic behaviour of solutions of differential equations is the investigation of their typical singularities. For systems of ordinary differential equations with continuous right-hand side this program was at large realized in the work of Poincar\'e. But in optimal control theory the key role is played by Hamiltonian systems with discontinuous right-hand side and a tangential jump in the velocity vector, which are generated by the application of the Pontryagin Maximum Principle.

The solutions of piece-wise smooth Hamiltonian systems which enter into a singular point $x_0$ of order 2, lying on a switching surface ${\cal S}$ of codimension 1, have been studied in the work of I. Kupka and Zelikin \& Borisov \cite{Kupka}, \cite{ZelikinBorisov}. There it was shown that while approaching the point $x_0$, these solutions experience a bi-constant ratio chattering in generic position. This means that they intersect the switching surface ${\cal S}$ at time instants $t_1,t_2,\dots,t_n,\dots$ forming an asymptotically bi-constant ratio sequence,

\[
	\lim_{n\to \infty}\frac{t_{2n+2}-t_{2n+1}}{t_{2n+1}-t_{2n}}=a, \quad
	\lim_{n\to \infty}\frac{t_{2n+3}-t_{2n+2}}{t_{2n+2}-t_{2n+1}}=b.
\]

In this section we study the situation when the singular point $x_0$ of order 2 lies on a stratum of codimension 2. In this case a qualitatively new phenomenon arises, which is typical for problems of optimal synthesis and for Hamiltonian systems with discontinuous right-hand side. Besides the trajectories entering into $x_0$ with usual chattering there exists a Cantor set ${\cal X}_H(x_0)$ of trajectories, the dynamics of whose intersection with the strata ${\cal S}^H_{ij}$ can be described by a topological Markov chain $\Sigma_{\Gamma}^+$. Here $\Gamma$ is a certain oriented graph (see Def.~\ref{defn:graph_Gamma}). More precisely, $\Sigma_{\Gamma}^+$ is a quotient of the system describing the dynamics of the intersection of the trajectories in ${\cal X}_H(x_0)$ with ${\cal S}^H_{ij}$. The intersection of the set ${\cal X}_H(x_0)$ of trajectories with a surface transversal to them is a multi-dimensional Cantor set similar to Smale's Horseshoe. The topological Markov chain $\Sigma_{\Gamma}^+$ is homeomorphic to the standard version of Smale's Horseshoe \cite{Katok} as a topological space, but not conjugated to it as a dynamical system.

\subsection{Formulation of theorems on chaos in Hamiltonian systems with discontinuous right-hand side}

Before formulating the main theorems, we shall formally define the configuration depicted on Fig.~\ref{fig:mercedes}. We shall demand that the following inequalities be satisfied, where $H$ may be replaced by $-H$:

\begin{equation}
\label{eq:mercedes}
    H_i(x) > \max(H_j(x),H_k(x))\ \ \forall x\in\Omega_i,
\end{equation}

\noindent for all mutually different $i$, $j$, and $k$ from $\{1,2,3\}$. Set

\begin{equation}
\label{eq:defn_G_F_0_F_r}
	\begin{array}{c}
	3G=3F_0=H_1+H_2+H_3,\\
	3F_1=H_2-H_3,\; 3F_2=H_3-H_1,\;3F_3=H_1-H_2.\\
	\end{array}
\end{equation}

\noindent Clearly the $H_i$ are uniquely determined by $G$, $F_1$, and $F_2$ and vice versa. Here $F_3\equiv -F_1-F_2$ has been introduced for brevity of notations.

\begin{defn}
\label{defn:strange_point}
	A point $x_0 \in {\cal S}^H_{123}$, $H_1(x_0)=H_2(x_0) = H_3(x_0)$ will be called {\bf strange}\footnote{We hope that this notation will be associated by the reader with strange attractors.} if the following conditions are satisfied at $x_0$:

	\begin{enumerate}[(i)]
		\item In some neighbourhood of $x_0$ condition (\ref{eq:mercedes}) is satisfied, where $H$ is possibly replaced by $-H$.

		\item The functions $F_r$, $(\ad F_i)F_r$, $(\ad F_j)(\ad F_i)F_r$, and $(\ad F_k)(\ad F_j)(\ad F_i)F_r$ vanish at $x_0$. Here the index $r$ runs through the set $\{1,2,3\}$, and the indices $i,j,k$ through $\{0,1,2,3\}$. The set of the differentials of these functions at $x_0$ has the maximal rank\footnote{This conditions may be weakened.} allowed by the anti-commutativity and linearity relations, the Jacobi identities, and the relation $F_1+F_2+F_3\equiv0$. In other words, these differentials are in general position.

		\item The bilinear form

		\[
			B_{rr'} =  \ad F_r(\ad G)^3F_{r'}|_{x_0}, \; r,r'=1,2,3
		\]

		\noindent has maximal rank 2. Moreover, it is symmetric and negative semi-definite if condition (\ref{eq:mercedes}) holds for $H$, and positive semi-definite if this condition holds for $-H$. All other commutators of 5-th order of the functions $G=F_0$ and $F_r$, $r=1,2,3$ which are independent of the aforementioned ones vanish at the point $x_0$.
	\end{enumerate}
\end{defn}

The behaviour of the trajectories in the integral vortex of trajectories joining $x_0$ is described in the two Theorems \ref{thm:main_chaos_hamilton_equilateral_triangle} and \ref{thm:main_chaos_hamilton_any_triangle} below. In the first of these theorems a much more detailed description of the chaotic behaviour of the trajectories is given than in the second one, but the bilinear form $B_{rr'}$ has to satisfy an additional requirement, which for the model problem (\ref{problem:model}) is equivalent to the condition that the triangle of controls is equilateral, as in Theorem \ref{thm:model_chaos_equilateral_triangle}. In the second one of these theorems this condition is replaced by the condition of closeness to an equilateral triangle, as in Theorem \ref{thm:model_chaos_any_triangle}.

\begin{thm}
\label{thm:main_chaos_hamilton_equilateral_triangle}
	Suppose the point $x_0$ of a Hamiltonian system with piece-wise smooth Hamiltonian $H$ is strange, as in Definition \ref{defn:strange_point}, and the form $B_{rr'}$ is proportional to the bilinear form given by the matrix

	\begin{equation}
	\label{eq:matrix_1_1_2}
		\left(\begin{array}{ccc}
			-1   & 1/2  & 1/2  \\
			1/2  &  -1  & 1/2  \\
			1/2  & 1/2  &  -1  \\
		\end{array}\right )
	\end{equation}

	\noindent with a positive proportionality factor if condition (\ref{eq:mercedes}) is satisfied for $H$, and with negative proportionality factor if condition (\ref{eq:mercedes}) is satisfied for $-H$.
		
	Then there exists a set ${\cal X}_H(x_0)$ of points in the neighbourhood of $x_0$ such that

	\begin{enumerate}[(I)]
		\item For every point $z\in{\cal X}_H(x_0)$ there exists a time $T(z)<\infty$ such that the trajectory $X(t,z)$ through $z$ exists and is unique for all $t\in[0;T(z)]$. Moreover, the trajectory $X(t,z)$ hits $x_0$ at time $T(z)$, $X(T(z),z)=x_0$.
		
		\item The set ${\cal X}_H(x_0)$ consists of trajectories of the Hamiltonian system with Hamiltonian $H$ and is invariant with respect to this system in the following sense: if $z\in {\cal X}_H(x_0)$, then $X(t,z)$ lies in ${\cal X}_H(x_0)$ for all $t\in[0;T(z))$. Moreover, in the interval $t\in(0;T(z))$ the trajectory $X(t,z)$ intersects the switching surface ${\cal S}^H$ an infinite, countable number of times at the moments $t_1<t_2<\ldots$, $X(t_k,z)\in{\cal S}^H$, where $t_k\to T(z)$ as $k\to+\infty$.

		\item The dynamical system $\Phi_H:{\cal X}_H(x_0)\cap {\cal S}^H \to {\cal X}_H(x_0)\cap {\cal S}^H$ which takes the point $z\in{\cal X}_H(x_0)\cap {\cal S}^H$ to the next intersection point of the trajectory $X(t,z)$ with ${\cal S}^H$, i.e., $\Phi_H(z) = X(t_1,z)$, is by means of a map $\Psi^H_\Gamma$ semi-conjugated to the unilateral topological Markov chain $\Sigma_\Gamma^+$ of paths on the fixed graph $\Gamma$ (see Definition \ref{defn:graph_Gamma}) which are infinite to the right, where $\Gamma$ is independent of $x_0$ and $H$,

		\begin{center}
		\begin{tikzpicture}[description/.style={fill=white,inner sep=2pt}]
			\matrix (m) [matrix of math nodes, row sep=1em, column sep=2em, text height=1.5ex, text depth=0.25ex]
			{    {\cal X}_H(x_0)\cap {\cal S}^H & {\cal X}_H(x_0)\cap {\cal S}^H \\
				\Sigma_\Gamma^+ & \Sigma_\Gamma^+ \\ };
			\path[->,font=\scriptsize]
			(m-1-1) edge node[auto] {$\Phi_H$} (m-1-2)
			(m-1-2) edge node[auto] {$\Psi^H_\Gamma$} (m-2-2)
			(m-1-1) edge node[auto] {$\Psi^H_\Gamma$} (m-2-1)
			(m-2-1) edge node[auto] {$l$} (m-2-2);
		\end{tikzpicture}
		\end{center}
		
		\noindent Here $l$ is the Bernoulli left shift and $\Psi^H_\Gamma$ is a continuous surjective map. The preimage $(\Psi^H_\Gamma)^{-1}(\sigma)$ of any point $\sigma\in\Sigma_\Gamma^+$ is homeomorphic to an open unit disc $D^2$, and the diameter of $\Phi_H^k\bigl((\Psi^H_\Gamma)^{-1}(\sigma) \bigr)$ tends to 0 as $k\to+\infty$. Denote by $(.V)$ the set of paths in $\Gamma$ which start at the vertex $V$. Then the preimages of the sets $(\Psi^H_\Gamma)^{-1}(.A_{ij})$, $(\Psi^H_\Gamma)^{-1}(.B_{ij})$, $(\Psi^H_\Gamma)^{-1}(.C_{ij})$, and $(\Psi^H_\Gamma)^{-1}(.D_{ij})$ lie on the stratum ${\cal S}^H_{ij}$ with the same indices.

		\item If\footnote{Unfortunately an error has slipped into the papers \cite{HLZ_Chaos,HLZ_Chaos_DAN} and the condition $dG(x_0)=0$ has been forgotten.} $dG(x_0)=0$, then the Hausdorff and Minkowski dimensions of the set ${\cal X}_H(x_0)$ do not depend on $x_0$ and $H$ (of course, the point $x_0$ has to satisfy the conditions of the theorem) and coincide with the dimensions of the set ${\cal X}$ given in Theorem \ref{thm:model_chaos_equilateral_triangle}. Hence the bounds

		\begin{equation}
			3.204762 \leq \dim_H {\cal X}_H(x_0) \leq \overline{\dim_B} {\cal X}_H(x_0) \leq 3.407495
			\label{eq:dimH_dimB}
		\end{equation}
\noindent are valid for these dimensions.

		\item The topological entropy of the Bernoulli shift $l$ is given by

		\[
			h_{\mathrm{top}}(l) =
			\log_2\left(
			\sqrt[3]{\frac{1}{2} + \frac{\sqrt{69}}{18}} +
			\sqrt[3]{\frac{1}{2} - \frac{\sqrt{69}}{18}}\,\right)
			\approx 0.4057.
		\]

		\item Similar statements are valid for the trajectories leaving the point $x_0$ if the direction of time flow is reversed.
	\end{enumerate}
\end{thm}

\bigskip

The set ${\cal X}_H(x_0)$ is the direct analogue of the set ${\cal X}^+$ from Theorem \ref{thm:model_chaos_equilateral_triangle}. Formally Hamiltonian system (\ref{eq:model_pmp_system}) of model problem  (\ref{problem:model}) does not satisfy the conditions of Theorem (\ref{problem:model}), however, because most of the Poisson brackets in the definition of the strange point vanish identically. Therefore the condition of linear independence of their gradients at the point does not hold. We could formulate Theorem \ref{thm:main_chaos_hamilton_equilateral_triangle} in a way such that these cases satisfy its conditions, but then we would lose the structural stability of the phenomenon, see Remark \ref{rm:structural_stability_equilateral_triangle}).

\begin{remark}
	Using methods for the computation of the topological entropy it is not hard to calculate the Perry measure $\mu_P$ on $\Sigma_\Gamma^+$ which satisfies the variational principle and maximizes the metric entropy:

	\[
		\max_{\mu} h_\mu(l) =h_{\mu_P}(l) = h_{top}(l).
	\]
\end{remark}

\begin{thm}
\label{thm:main_chaos_hamilton_any_triangle}
	Let $H$ be a piece-wise smooth Hamiltonian and $x_0$ be a strange point of the corresponding of a Hamiltonian system in the sense of Definition \ref{defn:strange_point}. Suppose there exists a number $\lambda_H(x_0)$, positive if condition (\ref{eq:mercedes}) holds for $H$ and negative if condition (\ref{eq:mercedes}) holds for $-H$, such that the form $\lambda_H(x_0) B_{rr'}$ is close enough\footnote{Here "close enough" means that there exists a neighbourhood of matrix (\ref{eq:matrix_1_1_2}) in the manifold of symmetric matrices of rank 2, independent of $x_0$ and $H$, and a number $\lambda_H(x_0)$ such that the form $\lambda_H(x_0)B_{rr'}$ is an element of this neighbourhood.} to the bilinear form with matrix (\ref{eq:matrix_1_1_2}).

	Then in the neighbourhood of $x_0$ there exists a set of points $\Xi_H(x_0)$ with the following properties:
	
	\begin{enumerate}[(I)]
		\item For every point $z\in\Xi_H(x_0)$ there exists a time $T(z)<\infty$ such that the trajectory $X(t,z)$ passing through $z$ exists and is unique for $t\in[0,T(z)]$. Moreover, the trajectory $X(t,z)$ hits the point $x_0$ at time $T(z)$, $X(T(z),z)=x_0$.
		
		\item The set $\Xi_H(x_0)$ consists of trajectories of the Hamiltonian system with Hamiltonian $H$ and is invariant with respect to this system in the following sense. If $z\in \Xi_H(x_0)$, then $X(t,z)$ lies in $\Xi_H(x_0)$ for all $t\in[0,T(z))$. Moreover, in the time interval $t\in(0,T(z))$ the trajectory $X(t,z)$ intersects the switching surface ${\cal S}^H$ a countable, infinite number of times at time instants $t_1<t_2<\ldots$, $X(t_k,z)\in{\cal S}^H$, where $t_k\to T(z)$ as $k\to+\infty$.
		
		\item Consider the dynamical system $\Phi_H:\Xi_H(x_0)\cap{\cal S}^H\to\Xi_H(x_0)\cap{\cal S}^H$ taking a point $z\in\Xi_H(x_0)$ on ${\cal S}^H$ to the next intersection point of the trajectory $X(t,z)$ with ${\cal S}^H$, i.e., $\Phi_H(z) = X(t_1,z)$. There exists a natural number $n>0$, independent of the Hamiltonian system and the point $x_0$ and coinciding with the number $n$ from \ref{thm:model_chaos_any_triangle}), such that the map $\Phi_H^n$ is semi-conjugated with the the topological Markov chain of the Bernoulli shift on the disjoint union of two copies of the space $\Sigma_{01}^+$ of unilaterally infinite words with letters 0 and 1. In other words, there exists a continuous surjective map $\Psi^H_{01}$ from $\Xi_H(x_0)\cap{\cal S}^H$ to the space $\bigsqcup\limits^2 \Sigma_{01}^+$ such that the following diagram commutes:

		\begin{center}
		\begin{tikzpicture}[description/.style={fill=white,inner sep=2pt}]
		        \matrix (m) [matrix of math nodes, row sep=1em, column sep=2em, text height=1.5ex, text depth=0.25ex]
		        {    {\Xi_H(x_0)\cap {\cal S}^H}                  & {\Xi_H(x_0)\cap {\cal S}^H} \\
		             \bigsqcup\limits^2 \Sigma_{01}^+ & \bigsqcup\limits^2 \Sigma_{01}^+ \\ };
		        \path[->,font=\scriptsize]
		        (m-1-1) edge node[auto] {$\Phi_H^n$} (m-1-2)
		        (m-1-2) edge node[auto] {$\Psi^H_{01}$} (m-2-2)
		        (m-1-1) edge node[auto] {$\Psi^H_{01}$} (m-2-1)
		        (m-2-1) edge node[auto] {$l$} (m-2-2);
		\end{tikzpicture}
		\end{center}
		
		\noindent Here $l$ denotes the left shift on each of the copies of $\Sigma_{01}^+$. The preimage $(\Psi^H_{01})^{-1}(\sigma)$ of every point $\sigma\in\Sigma_{01}^+$ is homeomorphic to the open two-dimensional disc $D^2$, and the diameter of $\Phi_H^k\bigl((\Psi^H_{01})^{-1}(\sigma) \bigr)$ tends to 0 as $k\to+\infty$.

		\item Besides the Cantor set $\Xi_H(x_0)$ of trajectories which has been described in (I)--(IV) above, there are two other sets of chattering trajectories which enter into the point $x_0$, namely

		\begin{enumerate}[(a)]
			\item There exist two 1-parametric families of "three-chain" trajectories $R_{123}$ and $R_{132}$ which hit the point $x_0$ in finite time with chattering, i.e., with a countable number of consecutive intersections of the strata $S_{12}^H$, $S_{23}^H$, and $S_{31}^H$. The strata are intersected in the described order for $R_{123}$ and in the reverse order for $R_{132}$.

			\item There exist three 2-parametric families of "four-chain" trajectories $Q_1$, $Q_2$, and $Q_3$. Each trajectory from $Q_i$ intersects consecutively and cyclically the strata ${\cal S}^H_{ij}$, $S_{ik}^H$, $S_{ik}^H$, and ${\cal S}^H_{ij}$ a countable number of times, where $i\neq j\neq k$.
		\end{enumerate}

		\item Similar statements are valid for the trajectories leaving the point $x_0$ if the direction of time flow is reversed.
		
	\end{enumerate}
\end{thm}

\begin{remark}
	If the condition for an equilateral triangle from Theorem \ref{thm:main_chaos_hamilton_equilateral_triangle} is satisfied, then the set $\Xi_H(x_0)$ is a subset of ${\cal X}_H(x_0)$.
\end{remark}

\subsection{Descending system of Poisson brackets}

The proofs of Theorems \ref{thm:main_chaos_hamilton_equilateral_triangle} and \ref{thm:main_chaos_hamilton_any_triangle} are similar and based on key facts about the chaotic behaviour of optimal synthesis in the model problem, which have been obtained in the proofs of Theorems \ref{thm:model_chaos_equilateral_triangle} and \ref{thm:model_chaos_any_triangle}, respectively. We wish to find the optimal synthesis of model problem (\ref{problem:model}), which was described in Theorems \ref{thm:main_chaos_hamilton_equilateral_triangle} and \ref{thm:main_chaos_hamilton_any_triangle} and which sits in a generic Hamiltonian system with discontinuous right-hand side. To this end we use the so-called descending system of Poisson brackets. This allows us to perform an effective investigation of the behaviour of trajectories in integral vortices. For example, the theorem on the Hamiltonian nature of the flow of singular trajectories (see \cite{LokutHSF}) or the theorem on the structure of the Lagrangian manifold in the neighbourhood of a singular trajectory of first order for extremal problems which are holonomous with respect to a control taking values in a polyhedron (see \cite{LokutMnogogr}), have been proven with the method of the descending system of Poisson brackets.

Let us describe the structure of the descending system of Poisson brackets in detail. Instead of the Hamiltonians $H_i$ we shall use the Hamiltonians $F_i$ from (\ref{eq:defn_G_F_0_F_r}). Keep in mind that $F_1+F_2+F_3\equiv 0$. The Hamiltonian $H$ can be written as follows, where we assume a summation over the set $\{1,2,3\}$ over repeating upper and lower indices.

\[
	H = G +  \sum_{r=1}^3 u^rF_r= G + u^rF_r,
\]

\noindent where $u\in \R^3$, $u^1+u^2+u^3=0$, and the function $u(x)$ is chosen as follows: if $x\in\Omega_k$ and $(ijk)$ is an even permutation, i.e., $(ijk)\in A_3$ then

\[
	u^i(x) = -1, u^j(x)=1, u^k(x)=0.
\]

\noindent Thus $u(x)$ is the corresponding vertex of the triangle

\begin{equation}
\label{eq:triangle_omega_for_hamiltonian_system}
	\Omega = \conv\bigl\{(-1,1,0),(0,-1,1),(1,0,-1)\bigr\}\subset
	U=\{u\in\R^3:u_1+u_2+u_3=0\}.
\end{equation}

Since the Hamiltonian $H$ is not smooth, we define the trajectory of the Hamiltonian system of ODEs with discontinuous right-hand side according to Filippov \cite{Filippov}, as follows. Let $\omega$ be the symplectic form on ${\cal M}$ and $i_\omega$ the canonical isomorphism $i_\omega(z):T^*_z{\cal M} \to T_z{\cal M}$ defined by the form $\omega$. Then an absolutely continuous curve $z(t)$ is a trajectory of the Hamiltonian system of ODEs with Hamiltonian $H$ if and only if for almost all $t$ the differential inclusion

\[
	\dot z(t) \in \overline{\conv}\bigl\{\mbox{accumulation points of}\ i_\omega dH(\widetilde z)\ \mbox{as}\ \ \widetilde z\to z(t)\bigr\}\subset T_{x(t)}{\cal M}
\]

\noindent holds. Equivalently,

\[
	\dot z(t) = i_\omega\Bigl(dG\bigl(z(t)\bigr) + u^rdF_r\bigl(z(t)\bigr)\Bigr),\quad\mbox{for some}\quad u\in\argmax_{u\in\Omega}\{u^rF_r(z(t))\}.
\]

\begin{remark}
	Note that we do not pre-define a scalar product on the space $U=\{u:u_1+u_2+u_3=0\}\subset \R^3$. The sum $u^rF_r$ is to be understood as the contraction of a vector and a co-vector, where $u(x)\in U$ and $F(x)\in U^*$.
\end{remark}

We now come to the construction of the descending system of Poisson brackets for the Hamiltonian $H$. To this end we single out a set of ODEs along an arbitrary trajectory of the system and partition it into rows. We shall write the symbol $\rceil k\lceil$ to indicate that the corresponding ODE belongs to row $k$. In the first row of the descending system we write the three ODEs

\[
	\rceil 1\lceil\ \ \ddt F_i = \{G,F_i\} + \{F_r,F_i\}u^r, \quad i=1,2,3.
\]

\noindent The second row contains two types of equations,

\[
	\begin{array}{lll}
		\rceil2\lceil&\ddt \{G,F_i\} = \{G,\{G,F_i\}\} + \{F_r,\{G,F_i\}\}u^r,& i=1,2,3;\\
		\rceil2\lceil&\ddt \{F_i,F_j\} = \{G,\{F_i,F_i\}\} + \{F_r,\{F_i,F_j\}\}u^r,& i,j=1,2,3.\\
	\end{array}
\]

\noindent In general, the $m$-th row contains the equations

\[
	\rceil m\lceil\ \ \ \ddt {\cal K}^m = \{G,{\cal K}^m\} + \{F_r,{\cal K}^m\}u^r,
\]

\noindent where ${\cal K}^m = \{K_m,\{K_{m-1},\ldots\{K_2,K_1\}\ldots\}\}$, $K_1=F_i$, $i=1,2,3$, and the remaining symbols $K_j$ denote $G$ as well as $F_i$. The equations in row $(m+1)$ can be obtained by differentiating the right-hand side of the equations in the $m$-th row with respect to $t$. For example, we have

\[
	\rceil m+1\lceil\ \ \ \ \ddt\{G,{\cal K}^m\} = \{G,\{G,{\cal K}^m\}\} + \{F_r,\{G,{\cal K}^m\}\}u^r
\]

\noindent and

\[
	\rceil m+1\lceil\ \ \ \ \ddt\{F_i,{\cal K}^m\} = \{G,\{F_i,{\cal K}^m\}\} + \{F_r,\{F_i,{\cal K}^m\}\}u^r,\
	i=1,2,3.
\]

\noindent The descending system is written out up to the row with number $2h$, where $h$ is the order of the singularity at the point $x_0$ (see \cite{HLZOrderTorus}). In our case we have $h=2$, since the control explicitly appears at $x_0$ when differentiating the fourth time.

Therefore in the $m$-th row we have equations on the time derivative of brackets of order $m$, and the right-hand sides of these equations are affine in $u$ with brackets of order $m+1$ as coefficients.

We remark that the construction of the descending system is similar to the derivation of singular trajectories. In the latter case, the equalities $F_i(x(t))\equiv0$ are differentiated along the trajectory of the Hamiltonian system with Hamiltonian $H=G + F_ru^r$. From our point of view the descending system, albeit it contains a large number of equations, is more favorable than direct differentiation, because it does not produce explicit derivatives of the control like $\dot u,\ddot u,\ldots$, which are usually inconvenient to work with.

\begin{defn}
\label{defn:main_poisson|_brackets}
	In the descending system we shall call the brackets $F_r$, $(\ad G)F_r$, $(\ad G)^2 F_r$, $(\ad G)^3F_r$, $r=1,2,3$ the primary brackets. The other brackets of order not exceeding four shall be called secondary\footnote{Note that the function $\dot G$ does not belong to the descending system and is hence neither a primary nor a secondary bracket.}.
\end{defn}

The reduction of the Hamiltonian system with discontinuous right-hand side to the Hamiltonian system generated by the Pontryagin maximum principle when applied to model problem (\ref{problem:model}) is based on the fact that in the neighbourhood of a strange point the secondary brackets in any given row $m$ have a larger order of smallness than the primary brackets, and hence do not influence the qualitative behaviour of the system.

\begin{lemma}
\label{lm:descending_system_general_nongeneral_bracket_order}
	Consider an arbitrary trajectory $x(t)$ which enters a strange point $x_0=x(0)$ at $t=T$. Then there exists a constant $c>0$ such that if ${\cal K}^m$ is a primary bracket of order $m\le 4$ in the descending system, then for all $t$ close to $T$ we have
	
	\[
		|{\cal K}^m\bigl(x(t)\bigr)| \le c (T-t)^{5-m},
	\]

	\noindent and if ${\cal K}^m$ is a secondary Poisson bracket of order $m\le 4$, then

	\[
		|{\cal K}^m\bigl(x(t)\bigr)| \le c (T-t)^{6-m}.
	\]
\end{lemma}

\begin{proof}

	We proof the lemma by reverse induction over $m$. The main tool is the following formula, which is valid for every bracket ${\cal K}^m$, $m\le 4$ from the descending system for $t\in[0,T]$:

	\begin{equation}
	\label{eq:K_int}
		{\cal K}^m =
		\int_T^t \Bigl(\{G,{\cal K}^m\}(x(\tau)) + u^r(\tau)\{F_r,{\cal K}^m\}(x(\tau))\Bigr)\;d\tau.
	\end{equation}

	\noindent This follows because ${\cal K}^m(x(T))=0$ for $m\le 4$.
	
	Let us now prove the base of the induction. Consider the last row of the descending system. If ${\cal K}^4$ is a secondary bracket, then the integrand in (\ref{eq:K_int}) vanishes at $t = T$ according to the definition of the strange point $x_0$. Hence for small $T-t$ we have for some $b_4>0$ that

	\[
		|{\cal K}^4(x(t))|\le b_4 (T-t)^2,\quad\mbox{ if }{\cal K}^4\mbox{ is a secondary bracket}.
	\]

	\noindent Consider a primary bracket ${\cal K}^4$ in the last row of the descending system. For this bracket the integrand in (\ref{eq:K_int}) does not need to vanish at $t=T$. Hence the most which we can guarantee is
	
	\[
		|{\cal K}^4(x(t))|\le b_4 (T-t),\quad\mbox{ if }{\cal K}^4\mbox{ is a primary bracket}.
	\]

	The induction step is very similar to the induction base. Suppose $m<4$. If ${\cal K}^m$ is a secondary bracket, then the integrand in (\ref{eq:K_int}) contains only secondary brackets. Hence we obtain from the assumption of the induction for $m + 1$ that for some $b_m>0$ and for $t<T$ close to $T$

	\[
		|{\cal K}^m(x(t))|\le b_m (T-t)^{6-m},\quad\mbox{ if }{\cal K}^m\mbox{ is a secondary bracket}.
	\]

	\noindent For a primary bracket ${\cal K}^m$ we obtain the following. The first summand in the integrand in (\ref{eq:K_int}) is a primary bracket in row $(m+1)$, and the second summand is a secondary bracket and is majorized by the first one for $t<T$ close to $T$. Hence for small $T-t\ge0$ we have

	\[
		|{\cal K}^m(x(t))|\le b_m (T-t)^{5-m},\quad\mbox{ if }{\cal K}^m\mbox{ is a primary bracket}.
	\]
	
	\noindent This completes the proof.
\end{proof}

\begin{remark}
\label{rm:main_bracket_ineq}
	At the end of the proofs of Theorems \ref{thm:main_chaos_hamilton_equilateral_triangle} and \ref{thm:main_chaos_hamilton_any_triangle} we show that on the trajectories in the integral vortex of the strange point $x_0$ which are of interest to us we have the following lower bound for the primary Poisson brackets:

	\[
		c'' (T-t)\le
		\max_{m}\Big|{\cal K}^m\bigl(x(t)\bigr)\Big|^{\frac{1}{5-m}} \le c' (T-t)
	\]

	\noindent for some constants $c',c''>0$.
\end{remark}

Of course bounds similar to those obtained in Lemma \ref{lm:descending_system_general_nongeneral_bracket_order} are valid also for trajectories leaving the strange point. We have to replace $T-t$ by the time which has elapsed since the trajectory left the strange point $x_0$.

\subsection{Blow-up of the singularity at the strange point}
\label{subsec:blowing_of_singularity}

In this section we describe the blow-up of the singularity at the strange point in the case when the form $B_{rr'}$ is merely non-degenerate. We show that on the zero section, i.e., on the sphere which will be glued in in place of the strange point, we obtain a smooth sub-manifold which contains exactly the trajectories of Hamiltonian system (\ref{eq:model_pmp_system}) which is obtained by application of the PMP to model problem (\ref{problem:model}) with triangle $\Omega$. Here the form $B_{rr'}$ generates the scalar product on the plane given by $u^1+u^2+u^3=0$ which defines the functional $J$ in model problem (\ref{problem:model}). Whether the triangle $\Omega$ is equilateral is hence determined by the form $B_{rr'}$.

\begin{remark}
\label{rm:B_as_dot_product}
	In the sequel we shall assume without loss of generality that condition (\ref{eq:mercedes}) holds for $H$. In this case the form $B_{rr'}$ is negative definite on $U$ according to Definition \ref{defn:strange_point}. We then use the form $-B_{rr'}$ for the definition of a scalar product on $U$.
\end{remark}

Let us introduce new local coordinates in the neighbourhood of the point $x_0$. Consider the set of all possible brackets ${\cal K}^m$, $m\le 4$ from the descending system of Poisson brackets, all in all $255$ brackets\footnote{The number 255 is obtained as follows. We have 3 brackets in the first row of the descending system, namely $F_1$, $F_2$, and $F_3$. The bracket $G$ does not participate in the differentiation in the construction of the descending system. There are $3\cdot 4=12$ brackets in row two, namely $\{F_i,F_j\}$ and $\{G,F_i\}$. Continuing, we obtain $255 = 3+ 3\cdot 4 + 3\cdot4^2 + 3\cdot 4^3$ in total.}. The set ${\cal K} = \{{\cal K}^m,m\le 4\}$ of functions is linearly dependent, because of the Jacobi identity, the anti-symmetry of the Poisson bracket, and due to the relation $F_1+F_2+F_3\equiv0$. Moreover, some brackets vanish identically. However, according to Definition \ref{defn:strange_point} of strange point the gradients $d{\cal K}^m$ form a set of maximally possible rank at $x_0$. Therefore the set ${\cal K}$ defines a map of a neighbourhood of the point $x_0$ into a linear subspace of $\R^{255}$ of dimension\footnote{The dimension $31$ can be obtained as follows. The dimension of a free nilpotent Lie algebra of order 4 with three generators $G$, $F_1$, and $F_2$ equals $32$. However, we excluded $G$ from the set of local coordinates, and the dimension is reduced by 1.} $31$. Let us extend the system ${\cal K}$ to a full set of local coordinates in a neighbourhood of $x_0$ by a collection of arbitrary smooth functions $w(x)=(w_1(x),...,w_{(\dim{\cal M}-31)}(x))\in\R^{\dim{\cal M}-31}$. We only demand that the differential $dw$ at $x_0$ has maximal rank and that $w(x_0) = 0$. On the space $\R^{\dim{\cal M}-31}$ we define the standard Euclidean norm.

In the coordinates $({\cal K},w)$ the Hamiltonian system of ODEs with Hamiltonian $H$ can be written as

\begin{equation}
\label{eq:general_system_K_w}
	\left\{\begin{array}{l}
		\ddt{\cal K}^m = \{G,{\cal K}^m\} + \{F_r,{\cal K}^m\}u^r,\quad m\le 4;\vspace{0.3cm}\\
		\ddt w = \alpha({\cal K},w) + \beta_r({\cal K},w)u^r;\vspace{0.3cm}\\
		F_ru^r \to \max\limits_{u\in\Omega}.\\
	\end{array}\right.
\end{equation}

\noindent Here $\alpha$ and $\beta_r$ are some smooth functions. Note also that the brackets ${\cal K}^5$ of fifth order are smooth functions of the local coordinates $({\cal K},w)$ in a neighbourhood of $x_0$.

The blow-up of the singularity will be performed similar to the blow-up of the origin in the model problem as in Section \ref{subsec:blowing_procedure}. To this end we define the action $g_H(x_0)$ of the group $\R\setminus\{0\}$ on the space $\R^{\dim {\cal M}}=\{(K,w)\}$, as follows.

\begin{defn}
	Let $\lambda$ be an arbitrary non-zero number. Define the action $g_H(x_0)(\lambda)$ by

	\[
		g_H(x_0)(\lambda):
		({\cal K}^1,{\cal K}^2,{\cal K}^3,{\cal K}^4,w)\mapsto
		(\lambda^4{\cal K}^1,\lambda^3{\cal K}^2,\lambda^2{\cal K}^3,\lambda{\cal K}^4,\lambda w).
	\]
\end{defn}

Of course the action $g_H(x_0)$ does not preserve the Hamiltonian system with Hamiltonian $H$. But the principal part of the system can be isolated by a blow-up of the singularity at $x_0$ by means of the action $g_H(x_0)$. To this end we introduce notations for the local coordinates in the neighbourhood of the strange point $x_0$. Here and in the sequel we shall set $F_0 = G$ and
\[ {\cal K}^1_r = F_r,\quad {\cal K}^2_{ir} = \{F_i,{\cal K}^1_r\},\quad {\cal K}^3_{jir} = \{F_j,{\cal K}^2_{ir}\},\quad {\cal K}^4_{kjir} = \{F_k,{\cal K}^3_{jir}\},
\]
where the index $r$ runs through $\{1,2,3\}$, and the other indices $i,j$, and $k$ can additionally take the value 0. The local coordinates will be $w,{\cal K}^1_r,{\cal K}^2_{ir},{\cal K}^3_{jir}$, and ${\cal K}^4_{kjir}$.

Taking Remark \ref{rm:B_as_dot_product} into account, we set

\begin{equation}
\label{eq:blowing_general_mu}
	\begin{array}{rl}
	\mu({\cal K}) = \Bigl(&
	(-B^{rr'}{\cal K}^1_r{\cal K}^1_{r'})^3 +
	(-B^{rr'}{\cal K}^2_{0r}{\cal K}^2_{0r'})^4 +\\
	&+(-B^{rr'}{\cal K}^3_{00r}{\cal K}^3_{00r'})^6 +
	(-B^{rr'}{\cal K}^4_{000r}{\cal K}^4_{000r'})^{12}
	\Bigr)^{1/24},
	\end{array}
\end{equation}

\noindent where the matrix $B^{rr'}$ denotes the matrix of the form on $U^*$ which is the inverse of the negative definite form $B_{rr'}$ on $U$ from Definition \ref{defn:strange_point} of strange points. We shall also set

\begin{equation}
\label{eq:blowing_general_K_w}
	(\widetilde{\cal K},\widetilde w) = g_H(x_0)\left(1/\mu({\cal K})\right)\Bigl({\cal K},w\Bigr).\\
\end{equation}

\noindent The variables $\widetilde{\cal K} = \{\widetilde{\cal K}^m,m\le 4\}$ and $\widetilde w$ lie on the spherical cylinder

\[
	{\cal C}^H_0 =
	\Bigl\{(\widetilde{\cal K},\widetilde w):\ \ \mu(\widetilde{\cal K}) = 1\Bigr\}.
\]

\noindent The cylinder ${\cal C}^H$ will be constructed as in Definition \ref{defn:cylinder_C_model_problem} in Section \ref{subsec:blowing_procedure}.

Denote the map $({\cal K},w) \mapsto (\mu,\widetilde{\cal K},\widetilde w)$ by $\blowingBig$. Topologically this blow-up procedure is equivalent to the removal of the submanifold on which the primary Poisson brackets vanish and gluing in the spherical cylinder ${\cal C}^H_0$ in place of the strange point $x_0$.

In the sequel we denote the division of a Poisson bracket by a corresponding power of $\mu$ by a wide tilde. For example, we have

\[
	\widetilde{\cal K}^m = {\cal K}^m/\mu^{5-m}\quad\mbox{or}\quad
	\widetilde{\{G,{\cal K}_m\}} = \{G,{\cal K}_m\}/\mu^{4-m}.
\]

\noindent For brackets up to 4-th order we obtain a local coordinate on ${\cal C}^H$, and the brackets of 5-th order do not change under the tilde operation.

Let us now carry over the Hamiltonian system of ODEs to the space ${\cal C}^H$, as we have done in Section \ref{subsec:reparametrize_time}. By differentiating (\ref{eq:blowing_general_mu}) and (\ref{eq:blowing_general_K_w}) along an arbitrary trajectory we obtain

\begin{equation}
\label{eq:blowing_general_hamilton_vector_field}
	\left\{\begin{array}{lll}
	\ddt\mu &=& \Upsilon(\mu,\widetilde{\cal K},\widetilde w, u);\\
	\ddt\widetilde{\cal K}^m &=& \frac{\displaystyle 1}{\displaystyle\mu}\left(
		\widetilde{\{G,{\cal K}^m\}} + \widetilde{\{F_r,{\cal K}^m\}}u^r - (5-m)\Upsilon\widetilde{\cal K}^m
	\right);\\
	\ddt \widetilde w &=& \frac{\displaystyle1}{\displaystyle\mu}\bigl(\alpha + \beta_ru^r - \Upsilon\widetilde w\bigr);\vspace{0.2cm}\\
	\widetilde F_ru^r &\to& \max\limits_{u\in\Omega};\\
	\end{array}\right.
\end{equation}

\noindent where

\begin{equation}
\label{eq:Upsilon_formula}
	\begin{array}{r}
	24\Upsilon(\mu,\widetilde{\cal K},\widetilde w, u) =
	-6\bigl(B^{rr'}\widetilde{\cal K}^1_r\widetilde{\cal K}^1_{r'}\bigr)^2 B^{rr'}\widetilde{\cal K}^1_r\bigl(\widetilde{\{G,{\cal K}^1_{r'}\}} + \widetilde{\{F_{r''},\widetilde{\cal K}^1_{r'}\}}u^{r''}\bigr)\ +\\
	+8(B^{rr'}\widetilde{\cal K}^2_{0r}\widetilde{\cal K}^2_{0r'})^3 B^{rr'}\widetilde{\cal K}^2_{0r}\bigl(\widetilde{\{G,{\cal K}^2_{0r'}\}} + \widetilde{\{F_{r''},{\cal K}^2_{0r'}\}}u^{r''}\bigr)\ +\\
	+12(B^{rr'}\widetilde{\cal K}^3_{00r}\widetilde{\cal K}^3_{00r'})^5 B^{rr'}\widetilde{\cal K}^3_{00r}\bigl(\widetilde{\{G,{\cal K}^3_{00r'}\}} + \widetilde{\{F_{r''},{\cal K}^3_{00r'}\}}u^{r''}\bigr)\ + \\
	+24(B^{rr'}\widetilde{\cal K}^4_{000r}\widetilde{\cal K}^4_{000r'})^{11} B^{rr'}\widetilde{\cal K}^4_{000r}\bigl(\widetilde{\{G,{\cal K}^4_{000r'}\}} + \widetilde{\{F_{r''},{\cal K}^4_{000r'}\}}u^{r''}\bigr).\ \ \\
	\end{array}
\end{equation}

As in Section \ref{subsec:reparametrize_time}, let us denote the vector field on ${\cal C}^H\cap\{\mu>0\}$ defined by the right-hand side of (\ref{eq:blowing_general_hamilton_vector_field}) by $\xi_H(x_0)$. Formally the field $\xi_H(x_0)$ is defined in a neighbourhood of ${\cal C}^H_0$ for $\mu>0$ only, but we extend it to the lower half of the neighbourhood, i.e., to ${\cal C}^H\cap\{\mu<0\}$, by the same formula. More percisely, the map

\[
	\blowing^{-1}:(\mu,\widetilde {\cal K}, \widetilde w) \mapsto
	({\cal K}, w),\quad\mbox{where}\quad ({\cal K}, w) = g_H(x_0)(\mu)\bigl(\widetilde {\cal K}, \widetilde w\bigr),
\]

\noindent is a 2-fold covering of a neighbourhood of ${\cal C}^H_0$ in ${\cal C}^H\cap\{\mu\ne0\}$ over its preimage in the pointed neighbourhood\footnote{Note that the image of a neighbourhood of ${\cal C}^H_0$ in ${\cal C}^H\cap\{\mu\ne0\}$ under the mapping $\blowing^{-1}$ does not contain the pointed neighbourhood of the point $x_0$.} of $x_0$. It allows us to carry over the values of the functions $\alpha$ and $\beta$ and the Poisson brackets of 5-th order on the lower half ${\cal C}^H\cap\{\mu<0\}$ of the cylinder. In this case the map $\blowing^{-1}$ takes the vector field of system (\ref{eq:blowing_general_hamilton_vector_field}) to the field of the Hamiltonian $H$ for $\mu>0$ as well as for $\mu<0$.

Note that the field $\xi_H(x_0)$ grows like $\frac{1}{\mu}$ as $\mu\to 0$. However, the field $\mu\xi_H(x_0)$ can already be continuously extended on the section ${\cal C}^H_0={\cal C}^H\cap\{\mu=0\}$ of the cylinder at all points except those in the set ${\cal S}^H=\bigl\{\widetilde K^1_r=\widetilde K^1_{r'}, r\ne r'\bigr\}$. We shall hence assume in the sequel that the field $\mu\xi_H(x_0)$ is defined on ${\cal C}^H_0\setminus{\cal S}^H$ by means of the natural extension by continuity. Surely we have $\mu\xi_H(x_0)\ne0$ on ${\cal C}_H^0$, regardless of the fact that $\mu=0$ on ${\cal C}_H^0$. The trajectories of the field $\mu\xi_H(x_0)$ either do not intersect the section ${\cal C}^H_0$ or lie on this section, because along the trajectories of the field $\mu\xi_H(x_0)$ we have $\dot\mu = \mu \Upsilon$. Moreover, on ${\cal C}^H_0$ those components of the field $\mu\xi_H(x_0)$ which correspond to the variables $(\widetilde{\cal K},\widetilde w)$ are independent of $\mu$.

Therefore the integral curves of the fields $\xi_H(x_0)$ and $\mu\xi_H(x_0)$ coincide on ${\cal C}^H\cap\{\mu\ne0\}$, but the velocity of movement along them is different. If we denote by $s$ the time parameter along the trajectories of the field $\mu\xi_H(x_0)$, then $s$ and $t$ are related by

\[
	ds=\frac{1}{\mu}dt.
\]

\noindent System (\ref{eq:blowing_general_hamilton_vector_field}) can then be rewritten in the form
\begin{equation}
\label{eq:blowing_general_hamilton_vector_field_s}
	\left\{\begin{array}{lll}
	\dds\mu &=& \mu\Upsilon(\mu,\widetilde{\cal K},\widetilde w, u);\\
	\dds\widetilde{\cal K}^m &=&
		\widetilde{\{G,{\cal K}^m\}} + \widetilde{\{F_r,{\cal K}^m\}}u^r - (5-m)\Upsilon\widetilde{\cal K}^m;\\
	\dds \widetilde w &=& \alpha + \beta_ru^r - \Upsilon\widetilde w;\vspace{0.2cm}\\
	\widetilde F_ru^r &\to& \max\limits_{u\in\Omega}.\\
	\end{array}\right.
\end{equation}

\subsection{Model optimal control problem on the zero section \texorpdfstring{${\cal C}^H_0$}{C0}}
\label{subsec:model_problem_inside}

Consider an arbitrary trajectory $x(t)$ which enters the strange point $x_0=x(0)$ as $t\to -0$. In this case we have $\mu(x(t))\le c'|t|\to 0$ as $t\to-0$ by virtue of Lemma \ref{lm:descending_system_general_nongeneral_bracket_order}. Therefore if $x(t)$ lies in the set $\blowingBig^{-1}({\cal C}^H\setminus{\cal C}^H_0)$ for small $t<0$, then the image of the trajectory $\blowingBig\bigl(x(t)\bigr)$ on ${\cal C}^H\cap\{\mu>0\}$ will tend to ${\cal C}^H_0$ as $t\to-0$.

Of interest for us are those trajectories $x(t)$ entering $x_0=x(0)$, for which at least one of the primary Poisson brackets ${\cal K}^m(x(t))$ has the largest possible order $|t|^{5-m}$ as $t\to-0$. In this case $\mu(x(t))$ has order $|t|$ as $t\to-0$. Namely, by virtue of Lemma \ref{lm:descending_system_general_nongeneral_bracket_order} we have $c''|t|\le\mu(x(t))\le c'|t|$ for small $t<0$. Therefore the trajectory $x(t)$ does not leave the set $\blowingBig^{-1}({\cal C}^H\setminus{\cal C}^H_0)$ for small $t<0$, and its image $\blowingBig(x(t))$ on ${\cal C}^H\cap\{\mu>0\}$ tends to ${\cal C}^H_0$. In fact, a stronger assertion holds. For $t\to-0$ the image $\blowingBig(x(t))$ tends to the submanifold ${\cal D}_0^H\subset{\cal C}^H_0$ on which all secondary Poisson brackets vanish,

\begin{eqnarray*}
	{\cal D}_0^H=\Bigl\{ \mu=0\ \ \mbox{and}\ \ \widetilde{\cal K}^m=0,\ \ \mbox{where}\ \ \widetilde{\cal K}^m\ \mbox{runs through all}\\
	\mbox{secondary Poisson brackets}\Bigr\} \subset {\cal C}^H_0.\\
\end{eqnarray*}

\noindent Moreover, every trajectory of the field $\mu\xi_H(x_0)$ that commences in ${\cal D}_0^H$ will never leave this submanifold. Indeed, on  ${\cal D}_0^H$ we have on trajectories of the vector field $\mu\xi_H(x_0)$ that $\dds\mu=0$ and $\dds\widetilde{\cal K}^m = 0$, where $\widetilde{\cal K}^m$ stands for an arbitrary secondary Poisson bracket in the descending system.

System (\ref{eq:blowing_general_hamilton_vector_field_s}) of ODEs has the same form on ${\cal D}_0^H$ as in the blown-up system (\ref{eq:blowing_model_hamilton_vector_field}) of the PMP in model problem (\ref{problem:model}). Indeed, let us rename the primary brackets of the descending system by

\begin{equation}
\label{eq:phi_psi_x_y_K_main}
	\widetilde\psi_r = \widetilde {\cal K}^1_r,\ \
	\widetilde\phi_r = -\widetilde {\cal K}^2_{0r},\ \
	\widetilde x^r   =  B^{rr'}\widetilde {\cal K}^3_{00r'},\ \
	\widetilde y^r   =  B^{rr'}\widetilde {\cal K}^4_{000r'},
\end{equation}

\noindent where $\widetilde\phi,\widetilde\psi\in U^*$ and $\widetilde x,\widetilde y\in U$. Then we obtain the following system of ODEs on ${\cal D}_0^H$

\begin{equation}
\label{eq:ODE_on_D0}
	\left\{\begin{array}{llrll}
		\dds{\widetilde \psi_r} &=&
			-\widetilde\phi_r &-& 4\Upsilon\widetilde \psi_r;\\
		\dds{\widetilde \phi_r} &=&
			-B_{rr'}\widetilde x^{r'} &-& 3\Upsilon\widetilde \phi_r;\\
		\dds{\widetilde x^r} &=&
			\widetilde y^r &-& 2\Upsilon\widetilde x^r;\\
		\dds{\widetilde y^r} &=&
			u^r &-& \Upsilon\widetilde y^r;\\
		\dds{\widetilde w} &=& \alpha(x_0) + \beta_r(x_0)u^r &-& \Upsilon\widetilde w;\\
		\widetilde\psi_ru^r &\to& \max\limits_{u\in\Omega}.\qquad\qquad\\
	\end{array}\right.
\end{equation}

\noindent If we set the secondary brackets in (\ref{eq:Upsilon_formula}) to zero, on ${\cal D}_0^H$ we obtain

\begin{eqnarray*}
	24\Upsilon|_{{\cal D}_0^H} = &-&
			6 \bigl(-B^{rr'}\widetilde \psi_r\widetilde \psi_{r'}\bigr)^2\bigl(-B^{rr'}\widetilde \psi_r\widetilde \phi_{r'}\bigr) +
			\ 8 \bigl(-B^{rr'}\widetilde \phi_r\widetilde \phi_{r'}\bigr)^3\bigl(\widetilde \phi_r\widetilde x^r\bigr)\ \ +\\
			&+&12\bigl(-B_{rr'}\widetilde x^r\widetilde x^{r'}\bigr)^{5} \bigl(-B_{rr'}\widetilde x^r\widetilde y^{r'}\bigr) +
			24\bigl(-B_{rr'}\widetilde y^r\widetilde y^{r'}\bigr)^{11}\bigl(-B_{rr'}\widetilde y^r u^{r'}\bigr).\\
\end{eqnarray*}

We see that if we neglect the variable $\widetilde w$, then the system of ODEs on ${\cal D}_0^H$ which is generated by the original Hamiltonian system with Hamiltonian $H$ coincides after blow-up with Hamiltonian system (\ref{eq:blowing_model_hamilton_vector_field}) of the Pontryagin maximum principle for model problem (\ref{problem:model}), given we define the scalar product on $U$ by means of the form $-B_{rr'}$. The system of ODEs which we obtain on ${\cal C}^H_0$ is in some sense a "nilpotentization" of the original Hamiltonian system in the neighbourhood of the strange point $x_0$.

Let us remark also that the property of the triangle $\Omega$ of being equilateral is determined by the form $-B_{rr'}$. Namely, the triangle $\Omega$ is equilateral if the form $B_{rr'}$ is as described in Theorem  \ref{thm:main_chaos_hamilton_equilateral_triangle}. Indeed, the vertices of the triangle $\Omega$ given by (\ref{eq:triangle_omega_for_hamiltonian_system}) have to form an equilateral triangle with respect to the form $-B_{rr'}$ with center in the origin. Hence the Gram matrix of the vertices of $\Omega$ has to be proportional to matrix (\ref{eq:matrix_1_1_2}) with a negative proportionality coefficient, whose modulus equals the side-length of the triangle $\Omega$.

\subsection{Nilpotentization in the neighbourhood of the strange point}

Below we shall show that the behaviour of the original Hamiltonian system is described by its behaviour on the zero section ${\cal C}^H_0$ of the cylinder $\cal C$. The proof is based on a thorough consideration of the Poincar\'e map in system (\ref{eq:blowing_general_hamilton_vector_field_s}) in the neighbourhood of points of the strange sets $\Xi$ and ${\cal X}$ from Theorems \ref{thm:main_chaos_hamilton_any_triangle} and \ref{thm:main_chaos_hamilton_equilateral_triangle}. As an auxiliary system of ODEs in the neighbourhood of $x_0$ we consider the system obtained from system (\ref{eq:blowing_general_hamilton_vector_field_s}) by "nilpotentization" with respect to $\mu$ by means of the blow-up map $\blowing$.

The nilpotentization procedure in the neighbourhood of the strange point by means of the blow-up map is performed as follows. We transform the right-hand side of system (\ref{eq:blowing_general_hamilton_vector_field_s}), which was obtained by the blow-up map $\blowing$, as follows. The differential equations on the variables $\widetilde{\cal K}$ and $\widetilde w$ are extended without modification from the zero section ${\cal C}^H_0$ ot the whole cylinder $\cal C$. The right-hand side of the differential equation on $\mu$ is linearized with respect to $\mu$ in the neighbourhood of the zero section. This means terms of order $o(\mu)$ are discarded. As a result we obtain the system

\begin{equation}
\label{eq:blowing_general_linearized_hamilton_vector_field}
	\left\{\begin{array}{lll}
	\dds\mu &=& \mu\,\Upsilon(0,\widetilde{\cal K},\widetilde w, u);\\
	\dds\widetilde{\cal K}^1_r &=& \widetilde{\cal K}^2_{0r} + \widetilde{\cal K}^2_{r'r}u^{r'} - 4 \widetilde{\cal K}^1_r\Upsilon,\\
	\dds\widetilde{\cal K}^2_{ir} &=& \widetilde{\cal K}^3_{0ir} + \widetilde{\cal K}^3_{r'ir}u^{r'} - 3 \widetilde{\cal K}^2_{ir}\Upsilon,\\
	\dds\widetilde{\cal K}^3_{jir} &=& \widetilde{\cal K}^4_{0jir} + \widetilde{\cal K}^4_{r'jir}u^{r'} - 2 \widetilde{\cal K}^3_{jir}\Upsilon,\\
	\dds\widetilde{\cal K}^4_{000r} &=&  B_{rr'}u^{r'} - \widetilde{\cal K}^4_{000r}\Upsilon,\\
	\dds\widetilde{\cal K}^4_{kjir} &=& - \widetilde{\cal K}^4_{kjir}\Upsilon,\quad\ \mbox{ if}\ {\cal K}^4_{kjir}\mbox{ -- secondary bracket}\\
	\dds \widetilde w &=& \alpha(x_0) + \beta_r(x_0)u^r - \Upsilon(0,\widetilde{\cal K},\widetilde w, u)\widetilde w;\vspace{0.2cm}\\
	\widetilde {\cal K}^1_ru^r &\to& \max\limits_{u\in\Omega},\\
	\end{array}\right.
\end{equation}

\noindent where here and below the summation over the upper and lower index $r'$ runs over $r'=1,2,3$. Let us rewrite this system in the original coordinates ${\cal K}$ and $w$, which have been introduced before blow-up. Applying the inverse map $\blowing^{-1}$ and transforming back the time parameters $s$ and $t$ we obtain

\begin{equation}
\label{eq:general_linearized_system_K_w}
\left\{\begin{array}{l}
		\ddt{\cal K}^1_r = {\cal K}^2_{0r} + {\cal K}^2_{r'r}u^{r'},\\
		\ddt{\cal K}^2_{ir} = {\cal K}^3_{0ir} + {\cal K}^3_{r'ir}u^{r'},\\
		\ddt{\cal K}^3_{jir} = {\cal K}^4_{0jir} + {\cal K}^4_{r'jir}u^{r'},\\
		\ddt{\cal K}^4_{000r} =  B_{rr'}u^{r'},\\
		\ddt{\cal K}^4_{kjir} = 0,\quad\ \mbox{ if}\ {\cal K}^4_{kjir}\mbox{ is a secondary bracket}\\
		\ddt w = \alpha(x_0) + \beta_r(x_0)u^r;\vspace{0.3cm}\\
		{\cal K}^1_ru^r \to \max\limits_{u\in\Omega}.\\
	\end{array}\right.
\end{equation}

\noindent Hence the "nilpotentization" by means of the blow-up map $\blowing$ with respect to $\mu$ in the neighbourhood of the strange point leads to system (\ref{eq:general_linearized_system_K_w}). This system can be obtained from system (\ref{eq:general_system_K_w}) by replacing the functions $\alpha$ and $\beta_r$ and the brackets of 5-th order by their values at the point $x_0$.

As mentioned before, system (\ref{eq:blowing_general_linearized_hamilton_vector_field}) coincides on the zero section ${\cal C}^H_0$ with system (\ref{eq:blowing_general_hamilton_vector_field}). The differential equations on $\widetilde{\cal K}$ and $\widetilde w$ are independent of $\mu$, and the right-hand side of the equation on $\mu$ is obtained by linearization of (\ref{eq:blowing_general_hamilton_vector_field}) with respect to $\mu$ in the neighbourhood of ${\cal C}^H_0={\cal C}^H\cap\{\mu=0\}$.

Let us note the following important facts about system (\ref{eq:blowing_general_linearized_hamilton_vector_field}):

\begin{enumerate}
	\item The right-hand side of the differential equations on $\widetilde{\cal K}$ is independent of $\widetilde w$ and $\mu$. After the coordinates $\widetilde{\cal K}$ on the trajectory of system (\ref{eq:blowing_general_linearized_hamilton_vector_field}) have been found, the variables $\ln\mu$ and $\widetilde w$ can be calculated by direct integration. 
	\item If the control is constant, then the evolution of the secondary Poisson brackets is uncoupled from the evolution of the primary Poisson brackets. The primary Poisson brackets influence only the time instant of the control switching. 
\end{enumerate}

Hence every trajectory of the original Hamiltonian system with Hamiltonian $H$ which lies in ${\cal D}_0^H$ can be found also in the system determined by model problem (\ref{problem:model}). Indeed, consider an arbitrary trajectory $\bigl(\widehat x(t),\widehat y(t),\widehat \phi(t),\widehat \psi(t),\widehat u(t)\bigr)$ of the PMP of model problem (\ref{problem:model}) lying in $M_+$ (see Corollary \ref{corollary:M_plus_is_optimal}), i.e., entering the origin. By virtue of Theorem \ref{thm:model_problem_bellman} this trajectory enters the origin at a time instant $t=T\bigl(x(0),y(0)\bigr)$. By virtue of (\ref{eq:phi_psi_x_y_K_main}) this trajectory determines the primary Poisson brackets ${\cal K}$ as functions of time $t$ for $t<T\bigl(x(0),y(0)\bigr)$. In order to obtain a trajectory of system (\ref{eq:general_linearized_system_K_w}) we still need to determine the secondary Poisson brackets ${\cal K}$ and the variables $w$ as functions of time. To this end we set the secondary Poisson brackets identically zero and determine the function $w(t)$ from the last equation of system (\ref{eq:general_linearized_system_K_w}),

\[
	w(t) = w_0 + \alpha(x_0)\Bigl(t-T\bigl(x(0),y(0)\bigr)\Bigr) + \beta_r(x_0)y^r(t).
\]

\noindent Since we are interested in trajectories which enter the strange point $x_0$, we also set $w_0=w(0)=0$. The constructed trajectory satisfies the ODEs of the "nilpotentized" system (\ref{eq:general_linearized_system_K_w}). Hence we have constructed a map $\Pi_+$ from $M_+\subset\{(x,y,\phi,\psi)\}$ into the space $\{({\cal K},w)\}$ by means of the formulas

\[
	\begin{array}{ll}
		\Pi_+:&
		\left\{\begin{array}{l}
			{\cal K}^1_r = \psi_r,\ \
				{\cal K}^2_{0r} = - \phi_r,\ \
				{\cal K}^3_{00r}= B_{r'r}x^{r'},\ \
				{\cal K}^4_{000r}= B_{r'r}y^{r'};\\
			{\cal K} = 0\quad\mbox{for all secondary brackets}\ {\cal K};\\
			w = \beta_r(x_0)y^r - \alpha(x_0)T\bigl(x(0),y(0)\bigr).
		\end{array}\right.
	\end{array}
\]

\begin{lemma}
\label{lm:Pi_plus}
	The map $\Pi_+:M_+\to\{({\cal K},w)\}$ constructed above is an equivariant injection, i.e., it is compatible with the action of the group $\R_+$,
	
	\[
		\Pi_+ \circ g(\lambda) = g_H(x_0)(\lambda) \circ \Pi_+.
	\]
	
	\noindent Moreover, the map $\Pi_+$ is continuous on $M_+$ and takes the trajectories of Hamiltonian system (\ref{eq:model_pmp_system}) of model problem (\ref{problem:model}) on $M_+$ to trajectories of the "nilpotentized" system (\ref{eq:general_linearized_system_K_w}).
	
	If in addition $dG(x_0)=0$, then the map $\Pi_+$ is locally Lipschitz.
\end{lemma}

\begin{proof}
	By construction the map $\Pi_+$ is equivariant and injective. Continuity of $\Pi_+$ on $M_+\setminus\{0\}$ follows from Theorem \ref{thm:model_problem_bellman}. In this theorem it is proven that the function $T(\cdot)$ is continuous. Also by construction the map $\Pi_+$ takes the trajectories of system (\ref{eq:model_pmp_system}) to the trajectories of system (\ref{eq:general_linearized_system_K_w}).
	
	If $dG(x_0)=0$, then $x(t)=x_0$ is a trajectory of the original Hamiltonian $H=G+F_ru^r$ for $u=0$. Hence ${\cal K}(t)=0$, $w(t)=0$ is a trajectory of system (\ref{eq:general_system_K_w}) for $u=0$. Then the condition $\alpha(x_0)=0$ is necessary and the map $\Pi_+$ does not contain a summand with $T(x(0),y(0))$. It is hence locally Lipschitz by virtue of Theorem \ref{thm:model_problem_bellman}.
\end{proof}

Thus we can find every trajectory of system (\ref{eq:model_pmp_system}) in the "nilpotentized" system (\ref{eq:general_linearized_system_K_w}). However, in the original system (\ref{eq:general_system_K_w}) these trajectories may not exist. However, it turns out that the trajectories from the sets ${\cal X}$ (Theorem \ref{thm:model_chaos_equilateral_triangle}) and $\Xi$ (Theorem \ref{thm:model_chaos_any_triangle}) can still be found in the original Hamiltonian system. To this end we have first to project the sets $\Pi_+({\cal X})$ and $\Pi_+(\Xi)$ on ${\cal C}^H_0$, and then to lift each trajectory from the zero section ${\cal C}^H_0$ by means of the Hadamard-Perron theorem. 
\bigskip

By the equivariance of the map $\Pi_+$ the map $\Pi_+/g:M_+/g \to {\cal C}^H_0$ is well-defined. Moreover, the image of $M_+/g$ lies in ${\cal D}_0^H$. Since the vector field $\mu\xi_H(x_0)$ of the blown-up "nilpotentized" system (\ref{eq:blowing_general_linearized_hamilton_vector_field}) and the vector field $\mu\xi$ of the blown-up model system (\ref{eq:blowing_model_hamilton_vector_field}) are also preserved by the action of the group $\R\setminus\{0\}$, we have that every trajectory of the field $\mu\xi$ on $M_+/g$ is taken to some trajectory of the field $\mu\xi_H(x_0)$.

Let us now transfer the sets ${\cal X}$ and $\Xi$ from the model optimal control problem to the zero section ${\cal C}^H_0$. Define

\[
	{\cal X}^0_H(x_0) = (\Pi_+/g)\bigl({\cal X}/g\bigr)\subset {\cal D}_0^H;
	\quad\mbox{ and }\quad
	\Xi^0_H(x_0) = (\Pi_+/g)\bigl(\Xi/g\bigr)\subset {\cal D}_0^H.
\]

Note that the set ${\cal X}_H(x_0)$ is defined for the Hamiltonian system with Hamiltonian $H$ only if the triangle $\Omega$ is equilateral with respect to the form $B_{rr'}$. In other words, it is defined only if the form $B_{rr'}$ has the form described in the conditions of Theorem \ref{thm:main_chaos_hamilton_equilateral_triangle}.

\subsection{Poincar\'e map in the Hamiltonian system}

In the previous section we have constructed sets ${\cal X}^0_H(x_0)$ and $\Xi^0_H(x_0)$ on the submanifold ${\cal D}_0^H\subset{\cal C}^H_0$ for a generic Hamiltonian system with Hamiltonian $H$. However, the inverse map $\blowingBig^{-1}$ will take the section ${\cal C}^H_0$ entirely to the strange point $x_0$. In order to find the corresponding trajectories in the original Hamiltonian system, we have to apply the Hadamard-Perron theorem to each (hyperbolic) trajectory of the sets $\Xi_H(x_0)$ and ${\cal X}_H(x_0)$ on ${\cal D}_0^H$. We will show that their stable submanifolds do not anymore lie in the zero section ${\cal C}^H_0$. The union of these stable submanifolds will yield the sought sets ${\cal X}_H(x_0)$ and $\Xi_H(x_0)$. The unstable submanifolds lie in ${\cal C}^H_0$ and have no analog in the original Hamiltonian system.

We shall apply the Hadamard-Perron theorem to the Poincar\'e map in the neighbourhood of an arbitrary trajectory $\gamma$ from ${\cal X}^0_H(x_0)$ or $\Xi^0_H(x_0)$. Every trajectory $\bigl(\widehat x(t),\widehat y(t),\widehat \phi(t),\widehat \psi(t),\widehat u(t)\bigr)$ from the set $\Xi$ or $\cal X$ in model problem (\ref{problem:model}) intersects the switching surface only transversally. Hence in the neighbourhood of any of its points a smooth Poincar\'e map is defined. In this section we transfer the structure of this map from the model problem to the trajectory $\gamma$, first in the "nilpotentized" system  (\ref{eq:blowing_general_linearized_hamilton_vector_field}), and then also in the original system (\ref{eq:blowing_general_hamilton_vector_field_s}).

We are interested in the differential of the Poincar\'e map $\overline\Phi_H$ from the switching surface to itself in the "nilpotentized" system (\ref{eq:blowing_general_linearized_hamilton_vector_field}) in the neighbourhood of an arbitrary trajectory in ${\cal D}_0^H$. The Poincar\'e map in the original Hamiltonian system will be denoted by $\Phi_H$, without the overbar.

In order to construct the map $\overline\Phi_H$ in the neighbourhood of the switching point $z_0=\gamma(0)$ on the trajectory $\gamma(t)$ from ${\cal X}^0_H(x_0)$ or $\Xi^0_H(x_0)$, we use the following simplifying consideration. Since the trajectory $\gamma(t)$ intersects the switching surface ${\cal S}^H$ transversally consecutively at the points $z_0$ and $z_1=\overline\Phi_H(z_0)=\gamma(\tau)$, the control is also constant between the intersection points with ${\cal S}^H$ for any trajectory in the neighbourhood of $\gamma(t)$, and coincides with the control on $\gamma(t)$. Hence we may assume that the control in the "nilpotentized" system (\ref{eq:blowing_general_linearized_hamilton_vector_field}) is constant and compute $d\overline\Phi_H$ as a composition of the linearization of the evolution of system (\ref{eq:blowing_general_linearized_hamilton_vector_field}) for the fixed time interval $\tau$ with the projection onto the tangent plane to ${\cal S}^H$ along the vector field of system (\ref{eq:blowing_general_linearized_hamilton_vector_field}). Let us denote the evolution map by $\Phi_H(\tau)$ and compute it.

The equation of variations for system (\ref{eq:blowing_general_linearized_hamilton_vector_field}) in the neighbourhood of $\gamma$ is of the form

\[
	\ddt \Phi_H(t) = M\Phi_H(t).
\]

\noindent Since the trajectory $\gamma$ lies in ${\cal D}_0^H$, not all primary brackets vanish on $\gamma$. Since the function $\Upsilon$ is independent of $\mu$ and the secondary Poisson brackets $\widetilde K$, the matrix $M$ must have the following block structure. (Here and in the sequel we shall write $\Upsilon(t) = \Upsilon(\gamma(t))$ for brevity. The identity matrix of any appropriate size is denoted by $I$.)

\begin{tikzpicture}[description/.style={fill=white,inner sep=1pt}]
	\matrix (m) [matrix of math nodes, row sep=1em, column sep=0.35em, text height=1.5ex]
	{
		\Upsilon(t)& 0 & 0 & 0 & 0 & 0 & 0 & 0 & 0\\
		0   & * & * & * & * & * & * & * & 0\\
		0   & * & * & * & * & * & * & * & 0\\
		0   & * & * & * & * & * & * & * & 0\\
		0   & * & * & * & * & * & * & * & 0\\
		0   & 0 & 0 & 0 & 0 &-3\Upsilon(t)I& * & * & 0\\
		0   & 0 & 0 & 0 & 0 & 0 &-2\Upsilon(t)I& * & 0\\
		0   & 0 & 0 & 0 & 0 & 0 & 0 &-\Upsilon(t)I& 0\\
		0   & * & * & * & * & * & * & * &-\Upsilon(t)I\\
	};
	\draw[decorate, decoration = brace] (m-1-1.south west) -- node [left] {$\mu$} (m-1-1.north west);
	\draw[decorate, decoration = brace] (m-2-1.south west) -- node [left] {\tiny Primary brackets $\widetilde{\cal K}^1_r$} (m-2-1.north west);
	\draw[decorate, decoration = brace] (m-3-1.south west) -- node [left] {\tiny Primary brackets $\widetilde{\cal K}^2_{0r}$} (m-3-1.north west);
	\draw[decorate, decoration = brace] (m-4-1.south west) -- node [left] {\tiny Primary brackets $\widetilde{\cal K}^3_{00r}$} (m-4-1.north west);
	\draw[decorate, decoration = brace] (m-5-1.south west) -- node [left] {\tiny Primary brackets $\widetilde{\cal K}^4_{000r}$} (m-5-1.north west);
	\draw[decorate, decoration = brace] (m-6-1.south west) -- node [left] {\tiny Secondary brackets $\widetilde{\cal K}^2$} (m-6-1.north west);
	\draw[decorate, decoration = brace] (m-7-1.south west) -- node [left] {\tiny Secondary brackets $\widetilde{\cal K}^3$} (m-7-1.north west);
	\draw[decorate, decoration = brace] (m-8-1.south west) -- node [left] {\tiny Secondary brackets $\widetilde{\cal K}^4$} (m-8-1.north west);
	\draw[decorate, decoration = brace] (m-9-1.south west) -- node [left] {$\widetilde w$} (m-9-1.north west);
	\draw[decorate, decoration = brace] (m-1-1.north west) -- node [above] {$\mu$} (m-1-1.north east);
	\draw[decorate, decoration = brace] (m-1-2.north west) -- node [above] {Primary brackets $\widetilde{\cal K}$} (m-1-5.north east);
	\draw[decorate, decoration = brace] (m-1-6.north west) -- node [above] {Secondary brackets $\widetilde{\cal K}$} (m-1-8.north east);
	\draw[decorate, decoration = brace] (m-1-9.north west) -- node [above] {$\widetilde w$} (m-1-9.north east);
\end{tikzpicture}

The matrix $M$ is block-lower triangular with three blocks on the diagonal, here the central block, which corresponds to the variables $\widetilde{\cal K}$, is block-upper triangular. Therefore the matrix $\Phi_H(t)$ possesses the same structure,

\begin{tikzpicture}[description/.style={fill=white,inner sep=1pt}]
	\matrix (m) [matrix of math nodes, row sep=2em, column sep=5.5em, text height=1.5ex, text depth=0.25ex]
	{
		\lambda& 0 & 0 & 0 \\
		0   & * & * & 0\\
		0   & 0 & C & 0\\
		0   & * & * & W\\
	};
	\draw[decorate, decoration = brace] (m-1-1.south west) -- node [left] {$\mu$} (m-1-1.north west);
	\draw[decorate, decoration = brace] (m-2-1.south west) -- node [left] {Primary brackets $\widetilde{\cal K}$} (m-2-1.north west);
	\draw[decorate, decoration = brace] (m-3-1.south west) -- node [left] {Secondary brackets $\widetilde{\cal K}$} (m-3-1.north west);
	\draw[decorate, decoration = brace] (m-4-1.south west) -- node [left] {$\widetilde w$} (m-4-1.north west);
	\draw[decorate, decoration = brace] (m-1-1.north west) -- node [above] {$\mu$} (m-1-1.north east);
	\draw[decorate, decoration = brace] (m-1-2.north west) -- node [above] {\tiny Primary brackets $\widetilde{\cal K}$} (m-1-2.north east);
	\draw[decorate, decoration = brace] (m-1-3.north west) -- node [above] {\tiny Secondary brackets $\widetilde{\cal K}$} (m-1-3.north east);
	\draw[decorate, decoration = brace] (m-1-4.north west) -- node [above] {$\widetilde w$} (m-1-4.north east);
\end{tikzpicture}

\noindent Here the matrix $W$ is proportional to the identity matrix, and the matrix $C$ is upper triangular.

Let us explicitly write out the differential equations on $\lambda$, $W$, and $C$:

\[
	\ddt \lambda = \Upsilon(t)\lambda,\qquad \lambda(0)=1.
\]

\noindent Denote $\lambda(\tau)=\rho$, where $\tau$ is the time on the trajectory $\gamma$ from $z_0$ to $z_1$. Further,

\[
	\ddt W = -\Upsilon(t)W,\quad W(0) = I\qquad\Longrightarrow\qquad W(\tau) = \rho^{-1}I.
\]

\noindent The matrix $C$ obeys the equation

\[
	\ddt C = \left(
		\begin{array}{ccc}
			-3\Upsilon(t) & * & * \\
			0 & -2\Upsilon(t) & * \\
			0 & 0 & -\Upsilon(t) \\
		\end{array}
	\right) C,\qquad C(0) = I.
\]

\noindent Hence the diagonal blocks of $C$ are given by $\rho^{-3}I$, $\rho^{-2}I$, and $\rho^{-1}I$.

Let us now compute the number $\rho$. Let the trajectory $\gamma'$ of system (\ref{eq:blowing_general_linearized_hamilton_vector_field}) be a lifting of $\gamma$ from ${\cal C}^H_0=\{\mu=0\}$ to $\{\mu\ne0\}$. In other words, the coordinates $\widetilde {\cal K}$ and $\widetilde w$ coincide on $\gamma$ and $\gamma'$, but on $\gamma'$ we have $\mu\ne 0$. Denote by $z'_0$ and $z'_1$ the corresponding liftings of $z_0$ and $z_1$.

\begin{lemma}
\label{lm:compute_rho}
	Denote $\mu_0=\mu(z'_0)$ and $\mu_1=\mu(z'_1)$. Then we have
	
	\[
		\frac{\mu_1}{\mu_0} = \rho.
	\]

\end{lemma}

\begin{proof}
	Indeed, $\mu_1$ is a function of the initial point $z'_0=(\mu_0,\widetilde{\cal K}_0,\widetilde w_0)$, i.e., $\mu_1 = \mu_1(\mu_0,\widetilde{\cal K}_0,\widetilde w_0)$. Clearly we have
	
	\[
		\mu_1(\kappa\mu_0,\widetilde{\cal K}_0,\widetilde w_0) = \kappa \mu_1(\mu_0,\widetilde{\cal K}_0,\widetilde w_0)\quad\forall\kappa\in\R.
	\]

	\noindent Differentiating this relation with respect to $\kappa$ at $\kappa=1$ we obtain
	
	\[
		\frac{\partial \mu_1}{\partial \mu_0} = \frac{\mu_1}{\mu_0}.
	\]

	\noindent The right-hand side of this equation does not change if $\mu_0$ is multiplied by $\kappa\ne 0$. We hence get that the derivative $\frac{\partial \mu_1}{\partial \mu_0}$ is constant on the whole vertical line $(\kappa\mu_0,\widetilde{\cal K}_0,\widetilde w_0)$, $\kappa\in\R$. However, on $\gamma$, i.e., for $\kappa=0$, we have that $\frac{\partial \mu_1}{\partial \mu_0}=\lambda(\tau) = \rho$. This completes the proof.
\end{proof}

\begin{remark}
	If we discard the coordinates $\widetilde w$ and the secondary Poisson brackets on the trajectory $\gamma'$, or equivalently, if we apply an inverse of the map $\Pi_+/g$, then we recover an optimal trajectory of model problem (\ref{problem:model}). Therefore we can apply the key Lemma \ref{lm:trajectory_exp_decrease} to it, which guarantees an exponentially fast convergence to zero of the coordinate $\mu(s)$ on this trajectory as $s\to+\infty$.
\end{remark}

Let us now consider the projection operator of the tangent space $T_{z_1}{\cal C}^H$ onto the hyperplane $T_{z_1}{\cal S}^H$ along the vector field of system (\ref{eq:blowing_general_linearized_hamilton_vector_field}). The vector field of system (\ref{eq:blowing_general_linearized_hamilton_vector_field}) is discontinuous at $z_1$, and we are interested in the velocity vector of the trajectory $\gamma$ when it arrives at $z_1$ for $t\to \tau-0$. Denote this vector by $v$. Then we have

\[
	v= \Bigl(\ \  0 \ \ \big|\ \  *\ \  \big|\ \  0\ \  \big| \ \ *\ \ \Bigr)^T.
\]

\noindent The co-vector $\alpha$ defining the hyperplane $T_{z_1}{\cal S}^H$ has the form

\[
	\alpha= \Bigl(\ \  0 \ \ \big|\ \  *\ \  \big|\ \  0\ \  \big| \ \ 0\ \ \Bigr).
\]

\noindent Hence the projection operator $1-\frac{\displaystyle v\otimes \alpha}{\displaystyle \alpha(v)}$ onto $T_{z_1}{\cal S}^H$ along $v$ has the form

\begin{tikzpicture}[description/.style={fill=white,inner sep=1pt}]
	\matrix (m) [matrix of math nodes, row sep=2em, column sep=5.5em, text height=1.5ex, text depth=0.25ex]
	{
		1   & 0 & 0 & 0 \\
		0   & * & 0 & 0\\
		0   & 0 & I & 0\\
		0   & * & 0 & I\\
	};
	\draw[decorate, decoration = brace] (m-1-1.south west) -- node [left] {$\mu$} (m-1-1.north west);
	\draw[decorate, decoration = brace] (m-2-1.south west) -- node [left] {Primary brackets $\widetilde{\cal K}$} (m-2-1.north west);
	\draw[decorate, decoration = brace] (m-3-1.south west) -- node [left] {Secondary brackets $\widetilde{\cal K}$} (m-3-1.north west);
	\draw[decorate, decoration = brace] (m-4-1.south west) -- node [left] {$\widetilde w$} (m-4-1.north west);
	\draw[decorate, decoration = brace] (m-1-1.north west) -- node [above] {$\mu$} (m-1-1.north east);
	\draw[decorate, decoration = brace] (m-1-2.north west) -- node [above] {\tiny Primary brackets $\widetilde{\cal K}$} (m-1-2.north east);
	\draw[decorate, decoration = brace] (m-1-3.north west) -- node [above] {\tiny Secondary brackets $\widetilde{\cal K}$} (m-1-3.north east);
	\draw[decorate, decoration = brace] (m-1-4.north west) -- node [above] {$\widetilde w$} (m-1-4.north east);
\end{tikzpicture}

Thus the differential of the Poincar\'e map $d\overline{\Phi_H}$ of the "nilpotentized" system (\ref{eq:blowing_general_linearized_hamilton_vector_field}) at the point $z_0$ is given by

\[
	d\overline{\Phi_H}|_{z_0} = \Bigl(1-\frac{\displaystyle v\otimes \alpha}{\displaystyle \alpha(v)}\Bigr) \Phi_H(\tau)
\]

\noindent and has the block structure

\[
	d\overline{\Phi_H}|_{z_0} =
	\left(\begin{array}{c|c|ccc|c}
		\rho \ &\  0\  & &  0 & &\  0\\[0.3cm]
		\hline&&&&&\\
		0    \ &\  A\  & &\  *\  & &\  0\\[0.3cm]
		\hline&&&&&\\
		       &     &\  \rho^{-3}I\  &\  *\  &\  *\  &\\[0.3cm]
		0    \ &\  0\  &\  0\  &\  \rho^{-2}I\  &\  *\  &0\\[0.3cm]
		       &     &\  0\  &\  0\  &\  \rho^{-1}I\  &\\[0.3cm]
		\hline&&&&&\\
		0    \ &\  *\  & &\  * \ & &\  \rho^{-1}I\\
	\end{array}\right)
\]

\begin{remark}
	Note that the matrix $A$ is obtained by conjugation with the differential of the map (\ref{eq:phi_psi_x_y_K_main}) from the corresponding matrix in the PMP system (\ref{eq:model_pmp_system}) for model problem (\ref{problem:model}). This follows immediately from Section \ref{subsec:model_problem_inside}. Moreover, since all considered trajectories lie in ${\cal D}_0^H$ the matrix appearing in the conjugation will be constant and independent of the choice of the trajectory $\gamma$ on ${\cal D}^H_0$.
\end{remark}

\bigskip

Let us now describe the matrix of the differential of the Poincar\'e map for the original, not "nilpotentized" system (\ref{eq:blowing_general_hamilton_vector_field_s}). Since the original system (\ref{eq:blowing_general_hamilton_vector_field_s}) and the "nilpotentized" system (\ref{eq:blowing_general_linearized_hamilton_vector_field}) coincide on ${\cal C}^H_0$, the matrices $d\Phi_H$ and $d\overline{\Phi_H}$ have the same elements everywhere except the first row and the first column. However, the first column is also identical, since the differential equation on $\mu$ differs only by terms of order $O(\mu^2)$ in the systems (\ref{eq:blowing_general_hamilton_vector_field_s}) and (\ref{eq:blowing_general_linearized_hamilton_vector_field}) in the neighbourhood of ${\cal C}_0^H$.

Hence we have proven

\begin{lemma}
\label{lm:block_structure_dPhi}
The differential $d\Phi_H$ of the Poincar\'e map for system (\ref{eq:blowing_general_hamilton_vector_field_s}) at the point $z_0\in{\cal D}_0^H\cap {\cal S}^H$ from ${\cal X}^0_H(x_0)$ or $\Xi^0_H(x_0)$ on the surface ${\cal S}^H$ has the form

\begin{tikzpicture}[description/.style={fill=white,inner sep=1pt}]
	\matrix (m) [matrix of math nodes, row sep=1em, column sep=1em, text height=1.5ex, text depth=0.25ex]
	{
		\rho& 0 & 0 & 0 & 0 & 0 & 0 & 0 & 0\\
		*   &   &   &   &   & * & * & * & 0\\
		*   &   &   &   &   & * & * & * & 0\\
		*   &   &   & A &   & * & * & * & 0\\
		*   &   &   &   &   & * & * & * & 0\\
		*   & 0 & 0 & 0 & 0 &\rho^{-3}I& * & * & 0\\
		*   & 0 & 0 & 0 & 0 & 0 &\rho^{-2}I& * & 0\\
		*   & 0 & 0 & 0 & 0 & 0 & 0 &\rho^{-1}I& 0\\
		*   & * & * & * & * & * & * & * &\rho^{-1}I\\
	};
	\draw[decorate, decoration = brace] (m-1-1.south west) -- node [left] {$\mu$} (m-1-1.north west);
	\draw[decorate, decoration = brace] (m-2-1.south west) -- node [left] {\tiny Primary brackets $\widetilde{\cal K}^1_r$} (m-2-1.north west);
	\draw[decorate, decoration = brace] (m-3-1.south west) -- node [left] {\tiny Primary brackets $\widetilde{\cal K}^2_{0r}$} (m-3-1.north west);
	\draw[decorate, decoration = brace] (m-4-1.south west) -- node [left] {\tiny Primary brackets $\widetilde{\cal K}^3_{00r}$} (m-4-1.north west);
	\draw[decorate, decoration = brace] (m-5-1.south west) -- node [left] {\tiny Primary brackets $\widetilde{\cal K}^4_{000r}$} (m-5-1.north west);
	\draw[decorate, decoration = brace] (m-6-1.south west) -- node [left] {\tiny Secondary brackets $\widetilde{\cal K}^2$} (m-6-1.north west);
	\draw[decorate, decoration = brace] (m-7-1.south west) -- node [left] {\tiny Secondary brackets $\widetilde{\cal K}^3$} (m-7-1.north west);
	\draw[decorate, decoration = brace] (m-8-1.south west) -- node [left] {\tiny Secondary brackets $\widetilde{\cal K}^4$} (m-8-1.north west);
	\draw[decorate, decoration = brace] (m-9-1.south west) -- node [left] {$\widetilde w$} (m-9-1.north west);
	\draw[decorate, decoration = brace] (m-1-1.north west) -- node [above] {$\mu$} (m-1-1.north east);
	\draw[decorate, decoration = brace] (m-1-2.north west) -- node [above] {Primary brackets $\widetilde{\cal K}$} (m-1-5.north east);
	\draw[decorate, decoration = brace] (m-1-6.north west) -- node [above] {Secondary brackets $\widetilde{\cal K}$} (m-1-8.north east);
	\draw[decorate, decoration = brace] (m-1-9.north west) -- node [above] {$\widetilde w$} (m-1-9.north east);
\end{tikzpicture}

\noindent where the matrix $A$ is up to substitution (\ref{eq:phi_psi_x_y_K_main}) equal to the corresponding matrix for system (\ref{eq:model_pmp_system}) of the Pontryagin maximum principle for model problem (\ref{problem:model}), and $\rho$ is obtained by virtue of lemma~\ref{lm:compute_rho}.

\end{lemma}

\subsection{finalization of the proof of the chaotic behaviour of the trajectories in a generic Hamiltonian system}

Lemma \ref{lm:block_structure_dPhi} allows us to use the results on the chaotic behaviour of the optimal synthesis for model problem (\ref{problem:model}). As we mentioned above, the sets $\Xi_H^0(x_0)\subset{\cal C}^H_0$ or ${\cal X}_H^0(x_0)\subset{\cal C}^H_0$ are constructed on the zero section ${\cal C}^H_0$ and are taken to the strange point $x_0$ by the inverse map $\blowingBig^{-1}$. In order to preserve the chaotic behaviour of the synthesis, we shall for every optimal trajectory from ${\cal X}$ or $\Xi$ construct two-dimensional stable submanifolds by virtue of the Hadamard-Perron theorem \cite{Katok}. These will leave the zero section ${\cal C}^H_0$. More precisely, we shall apply the Hadamard-Perron theorem to the Poincar\'e map $\Phi_H$ along the trajectory.

Let us fix a trajectory $\gamma$ from the set $\Xi_H^0(x_0)$ or the set ${\cal X}_H^0(x_0)$. The trajectory $\gamma$ is the image of some trajectory $\widehat\gamma$ from $\Xi/g$ or ${\cal X}/g$ under the map $\Pi_+/g$. Let us denote the intersection points of the switching surface ${\cal S}^H$  with the trajectory $\gamma$ consecutively by $\ldots,z_{-1},z_0,z_1,\ldots$, and with the trajectory $\widehat\gamma$ by $\ldots,\widehat z_{-1},\widehat z_0,\widehat z_1,\ldots$. The matrix $A$ and the number $\rho$ from Lemma \ref{lm:block_structure_dPhi} which correspond to the $k$-th switching point $z_k$ will be denoted by $A(z_k)=A_k$ and $\rho(z_k)=\rho_k$, respectively. By virtue of Lemma \ref{lm:block_structure_dPhi} we have for Hamiltonian system (\ref{eq:model_pmp_system})

\[
	d\Phi(\widehat z_k) = \left(
		\begin{array}{cc}
			\rho_k & 0\\
			0 & A_k\\
		\end{array}
	\right).
\]

\noindent Here and in the sequel $\Phi$ denotes the Poincar\'e map in the blown-up model system with vector field $\mu\xi$, where $\xi$ is the vector field of system (\ref{eq:blowing_model_hamilton_vector_field}).

First we show that the original Hamiltonian system with Hamiltonian $H$ is contracting in the vertical direction. The differential $d\Phi_H(z_k)$ corresponding to a single application of the Poincar\'e map $\Phi_H$ may in general have an eigenvalue $\rho_k$ which is larger than 1. We claim, however, that Lemma \ref{lm:trajectory_exp_decrease} guarantees the existence of a natural number $N$ such that $\Phi_H^N$ is contracting in the vertical direction. Indeed, every optimal trajectory of the model problem enters the origin in finite time. Thus we will eventually obtain a contraction in the vertical direction if we repeat the Poincar\'e map $\Phi_H$ often enough. We shall now prove this claim formally. Let us use the results obtained for model problem (\ref{problem:model}). From Lemma \ref{lm:trajectory_exp_decrease} it follows that on the optimal trajectory $\widehat\gamma$ on $M_+$ we have the inequalities 

\[
	c_1e^{-c_2 s}\le \frac{\mu(\widehat\gamma(s))}{\mu(\widehat\gamma(0))}\le c_3 e^{-c_4s},
\]

\noindent where the positive constants $c_j$, $k=1,\ldots,4$ are independent of the choice of trajectory on $M_+$. Therefore

\[
	\frac{\mu(s)}{\mu(0)}\le\frac{1}{2}\ \ \mbox{when}\ \ s\ge\frac{1}{c_4}\ln(2c_3).
\]

\noindent We now bound the transition time $\sigma(z)$ between two consecutive intersections of the switching surface ${\cal S}$ at points $z\in{\cal S}$ and $\Phi(z)\in{\cal S}$ when measured with the velocity of the field $\mu\xi$ along an arbitrary trajectory on $\Xi$ or ${\cal X}$. The function $\sigma$ is continuous and independent of $\mu$, and the sets $(\Xi/g)\cap {\cal S}$ and $({\cal X}/g)\cap {\cal S}$ are compact and have an empty intersection with ${\cal S}_{123}$. Therefore we get the bounds

\[
	s_{\min}\le\sigma(z)\le s_{\max}
\]

\noindent for some nonnegative constants $s_{\min}$ and $s_{\max}$. Since any continuous function achieves its minimum on a compact set, we have $s_{\max}\ge s_{\min}>0$. Hence we get by virtue of Lemma \ref{lm:compute_rho} 

\[
	\rho_k\rho_{k+1}\ldots\rho_{k+N-1} < \frac{1}{2}\qquad \forall k
	\quad\mbox{for}\quad
	N > \frac{1}{s_{\min}c_4}\ln(2c_3).
\]

Thus, if we fix a sufficiently large number $N$ and work only with the $N$-th power of the maps $\Phi$ and $\Phi_H$ instead of these maps themselves, we may assume that we have a contraction in the vertical direction. More precisely, we have proven 

\begin{lemma}
	Let us fix\footnote{The number $N$ is independent of the point $z_k$ and can be chosen universally for all trajectories.} $N > \frac{1}{s_{\min}c_4}\ln(2c_3)$ and denote $\Theta=\Phi^{-N}$ and $\Theta_H=\Phi_H^{-N}$. Set also $\varrho(z_{kN})=\varrho_k = (\rho_{kN-1}\rho_{kN-2}\ldots\rho_{(k-1)N})^{-1}>2$ and ${\cal A}(z_{kN})={\cal A}_k=(A_{kN-1}A_{kN-2}\ldots A_{(k-1)N})^{-1}$. Then we have

\[
	d\Theta(\widehat z_{kN}) = \left(
		\begin{array}{cc}
			\varrho_k & 0\\
			0 & {\cal A}_k\\
		\end{array}
	\right),
\]

\noindent and the matrix $d\Theta_H(z_{kN})$ has the structure described in Lemma \ref{lm:block_structure_dPhi} when replacing $\rho$ by $\varrho_k$ and $A$ by ${\cal A}_k$.
\end{lemma}

As a next step towards the application of the Hadamard-Perron theorem we will conduct a thorough study of the structure of the matrices $A_k$. We shall show that in some special decomposition of the tangent space all matrices $A_k$ become block-diagonal. For convenience we shall denote

\begin{equation}
\label{eq:defn_A_rho}
	d\Phi(\widehat z) = \left(
		\begin{array}{cc}
			\rho(\widehat z) & 0\\
			0 & A(\widehat z)\\
		\end{array}
	\right),
\end{equation}

\noindent where the linear operator $A(\widehat z)$ acts from $T_{\widehat z}({\cal C}_0\cap {\cal S})$ to $T_{\Phi(\widehat z)}{\cal C}_0\cap {\cal S}$.

\begin{lemma}
\label{lm:about_tangent_hyperbolic_decomposition}
	At the points $\widehat z$ of each of the sets $(\Xi/g)\cap {\cal S}$ and $({\cal X}/g)\cap {\cal S}$ there exists a decomposition of the tangent space\footnote{Recall that $\dim{{\cal C}_0\cap {\cal S}}=6$.} $T_{\widehat z}({\cal C}_0\cap {\cal S})=\widehat U(\widehat z)\oplus \widehat V(\widehat z)$, where $\dim \widehat U(\widehat z)=1$ and $\dim \widehat V(\widehat z)=5$, and a metric on the factors such that the operator $A(\widehat z)$ takes the form
	
	\[
		A(\widehat z) = \left(
			\begin{array}{cc}
				\widehat\alpha(\widehat z) & 0\\
				0 & \widehat\delta(\widehat z)\\
			\end{array}
		\right),
	\]

	\noindent where $\widehat\alpha(\widehat z):\widehat U(\widehat z)\to \widehat U(\Phi(\widehat z))$ and $\widehat\delta(\widehat z):\widehat V(\widehat z)\to \widehat V(\Phi(\widehat z))$. Here $\widehat\alpha$ is a contracting operator, and $\widehat\delta$ is an expanding operator, i.e.,
	
	\[
		\|\widehat\alpha(\widehat z)\|<1\qquad\mbox{and}\qquad \|\widehat\delta(\widehat z)^{-1}\|<1\qquad\forall \widehat z.
	\]
	
	\noindent Moreover, both the metric and the decomposition $T_{\widehat z}({\cal C}_0\cap {\cal S})=\widehat U(\widehat z)\oplus \widehat V(\widehat z)$ depend continuously on $\widehat z$.
\end{lemma}

\begin{proof}
	The proof of the lemma is divided into two parts. First we choose an "almost" correct decomposition of the tangent space at each of the points in the sets $(\Xi/g)\cap {\cal S}$ and $({\cal X}/g)\cap {\cal S}$, such that the matrix $A(\widehat z)$ becomes "almost" block-diagonal, and then we transform this decomposition. 
	
	We claim that in the neighbourhood of each of the sets $(\Xi/g)\cap {\cal S}$ and $({\cal X}/g)\cap {\cal S}$ there exist local coordinates and a Riemannian metric such that
	
	\begin{equation}
	\label{eq:A_alpha_beta_gamma_delta}
		A^{-1}(\widehat z) = \left(
			\begin{array}{cc}
				\alpha(\widehat z) & \beta(\widehat z)\\
				\gamma(\widehat z) & \delta(\widehat z)\\
			\end{array}
		\right)
	\end{equation}
	
	\noindent for every point $\widehat z$ in these neighbourhoods, where $\alpha(\widehat z)$ is a number, $\beta(\widehat z)$ and $\gamma(\widehat z)$ are $1\times 5$ and $5\times 1$ matrices, respectively, and $\delta(\widehat z)$ is a $5\times 5$ matrix. Moreover, there exist positive numbers $u_x$, $u_y$, $v_x$, and $v_y$ such that
	
	\[
		\|\alpha(\widehat z)^{-1}\| \le u_x^{-1},\quad \|\beta(\widehat z)\| \le u_y,\quad
		\|\gamma(\widehat z)\| \le v_x,\quad \|\delta(\widehat z)\| \le v_y.
	\]

	\noindent Here the diagonal elements are of hyperbolic nature and the off-diagonal elements are small compared to the diagonal elements. More precisely, the numbers $u_x$, $u_y$, $v_x$, and $v_y$ satisfy the inequalities
	
	\begin{equation}
	\label{eq:lipschitz_hyperbolic}
		\left\{
			\begin{array}{l}
				u_x>1>v_y;\\
				(u_x-1)(1-v_y) > u_yv_x.\\
			\end{array}
		\right.
	\end{equation}
	
	Indeed, the existence of such local coordinates in the neighbourhood of the set $(\Xi/g)\cap {\cal S}$ follows from the Grobman-Hartman theorem \cite{Katok} and the transversality of the homoclinic point (see Lemma \ref{lm:homoclinic_point_any_triangle}). The set $(\Xi/g)\cap {\cal S}$ from \ref{thm:model_chaos_any_triangle} can be reduced such that in an arbitrary metric the norm of the off-diagonal elements will be arbitrarily small, and the diagonal elements will be contracting and expanding operators of the same dimensions as for the 6-link chain, namely 5 and 1. 
	
	For the set $({\cal X}/g)\cap {\cal S}$ we prove this by explicitly describing appropriate local coordinates on the switching surface given by the relation $\psi_1=\psi_3$ with normalization $\phi_2=1$. First we reduce the number of variables to six, namely $\pi_1 = \frac{\sqrt{2}}{2}(y_1-y_3)$, $\pi_2 = \frac{1}{\sqrt{6}}(y_1-2y_2+y_3)$, $\pi_3 = \frac{\sqrt{2}}{2}(x_1-x_3)$, $\pi_4 = \frac{1}{\sqrt{6}}(x_1-2x_2+x_3)$, $\pi_5 = \frac{\sqrt{2}}{2}(\phi_1-\phi_3)$, $\pi_6 = \frac{1}{\sqrt{6}}(\psi_1-2\psi_2+\psi_3)$. Now in each of the 16 domains $AA.AA,AA.AC,\dots,CC.CC$ we apply a linear coordinate transformation $\eta = M\pi$, where the constant coefficient matrices for the 16 domains are given by
	
\begin{align*}
M_{AA.AA} &= \begin{pmatrix}
-0.5913 & -1.0677 & 0.0683 & 0.2937 & -0.0114 & 0.0316 \\
    0.0848 & -1.3179 & -0.9326 & -0.6262 & -0.0104 & 0.0737 \\
   -0.0023 & 0.6606 & 0.2677 & 0.7432 & 0.4258 & 0.2468 \\
   -0.2024 & 0.8600 & -0.7595 & 0.9189 & -1.5610 & 0.7982 \\
    0.5347 & -0.1061 & 1.3964 & 0.6468 & 2.3193 & -2.1665 \\
   -1.1822 & 0.5279 & -2.5279 & -0.2730 & -3.8024 & 2.6844
   \end{pmatrix}, \\
\end{align*}
\begin{align*}M_{AA.AC} &= \begin{pmatrix}
-0.3772 & -0.6812 & 0.0436 & 0.1874 & -0.0072 & 0.0201 \\
    0.0603 & -0.9371 & -0.6631 & -0.4452 & -0.0074 & 0.0524 \\
   -0.0016 & 0.4697 & 0.1903 & 0.5285 & 0.3028 & 0.1755 \\
   -0.1439 & 0.6115 & -0.5401 & 0.6534 & -1.1101 & 0.5675 \\
    0.3801 & -0.0754 & 0.9930 & 0.4599 & 1.6493 & -1.5406 \\
   -0.8407 & 0.3754 & -1.7976 & -0.1940 & -2.7037 & 1.9088
   \end{pmatrix}, \\
M_{AA.CA} &= \begin{pmatrix}
-0.3971 & -0.6704 & 0.1033 & 0.2038 & 0.0053 & -0.0444 \\
    0.1116 & -0.2444 & -0.2931 & -0.2188 & 0.0594 & 0.0028 \\
    0.0819 & 0.6562 & 0.5188 & 0.8690 & 1.3697 & 0.2361 \\
   -0.0770 & 0.5547 & -0.3887 & 0.6313 & -0.8937 & 0.5253 \\
    0.2177 & -0.0065 & 0.6332 & -0.1496 & 1.2464 & 0.3705 \\
   -0.1329 & 0.0749 & -0.4715 & -0.2392 & -0.9990 & 0.9895
   \end{pmatrix}, \\
M_{AA.CC} &= \begin{pmatrix}
-0.2941 & -0.4965 & 0.0765 & 0.1509 & 0.0039 & -0.0329 \\
    0.1180 & -0.2584 & -0.3099 & -0.2314 & 0.0629 & 0.0030 \\
    0.0865 & 0.6937 & 0.5485 & 0.9186 & 1.4479 & 0.2496 \\
   -0.0815 & 0.5864 & -0.4109 & 0.6674 & -0.9447 & 0.5552 \\
    0.2302 & -0.0068 & 0.6694 & -0.1581 & 1.3176 & 0.3916 \\
   -0.1405 & 0.0791 & -0.4985 & -0.2529 & -1.0561 & 1.0461
   \end{pmatrix}, \\
M_{AC.AA} &= \begin{pmatrix}
-1.3472 & -1.2301 & 0.8883 & 0.9348 & -0.4362 & -0.0239 \\
    0.1653 & -1.7355 & -1.7124 & -0.5395 & 0.6959 & 0.2788 \\
    0.3610 & 0.1926 & 2.3424 & -0.1152 & 4.4150 & -0.4106 \\
   -0.2216 & 1.5281 & -0.9719 & 1.9245 & -2.4828 & 1.4865 \\
    1.7063 & -1.3513 & 3.2009 & -1.9034 & 4.1924 & -0.2951 \\
   -0.4124 & 0.3207 & -1.2318 & -0.1906 & -2.4337 & 2.6356
   \end{pmatrix}, \\
M_{AC.AC} &= \begin{pmatrix}
-1.5370 & -1.4034 & 1.0134 & 1.0666 & -0.4977 & -0.0273 \\
    0.1456 & -1.5288 & -1.5084 & -0.4752 & 0.6131 & 0.2456 \\
    0.3180 & 0.1697 & 2.0633 & -0.1015 & 3.8890 & -0.3616 \\
   -0.1951 & 1.3460 & -0.8561 & 1.6952 & -2.1869 & 1.3094 \\
    1.5029 & -1.1902 & 2.8195 & -1.6767 & 3.6929 & -0.2599 \\
   -0.3632 & 0.2825 & -1.0851 & -0.1678 & -2.1437 & 2.3215
   \end{pmatrix}, \\
M_{AC.CA} &= \begin{pmatrix}
-1.4918 & -1.2360 & 1.0849 & 0.9744 & -0.6000 & -0.1150 \\
   -0.5929 & 0.5930 & 1.2489 & 0.5188 & -0.7829 & -0.2039 \\
    0.2087 & -2.0370 & 0.9934 & -2.4691 & 2.5841 & -1.8853 \\
    0.8390 & 1.3864 & 4.1958 & 1.8630 & 10.4128 & 0.9577 \\
    1.8997 & -0.1209 & 6.1266 & -1.0920 & 12.6540 & 3.2477 \\
   -0.5410 & 0.3880 & -1.7381 & -0.6350 & -3.6204 & 3.9319
   \end{pmatrix}, \\
M_{AC.CC} &= \begin{pmatrix}
-1.4492 & -1.2008 & 1.0539 & 0.9467 & -0.5829 & -0.1117 \\
   -0.6064 & 0.6066 & 1.2774 & 0.5307 & -0.8008 & -0.2085 \\
    0.2134 & -2.0836 & 1.0162 & -2.5256 & 2.6432 & -1.9285 \\
    0.8583 & 1.4182 & 4.2918 & 1.9056 & 10.6511 & 0.9796 \\
    1.9431 & -0.1237 & 6.2668 & -1.1170 & 12.9436 & 3.3220 \\
   -0.5534 & 0.3969 & -1.7778 & -0.6495 & -3.7032 & 4.0219
   \end{pmatrix}, \\
M_{CA.AA} &= \begin{pmatrix}
-0.8312 & -1.5009 & 0.0960 & 0.4129 & -0.0159 & 0.0445 \\
    0.1017 & -1.5797 & -1.1178 & -0.7505 & -0.0125 & 0.0883 \\
   -0.0028 & 0.7918 & 0.3209 & 0.8908 & 0.5104 & 0.2957 \\
   -0.2426 & 1.0308 & -0.9104 & 1.1014 & -1.8712 & 0.9568 \\
    0.6409 & -0.1272 & 1.6739 & 0.7752 & 2.7801 & -2.5969 \\
   -1.4171 & 0.6327 & -3.0301 & -0.3272 & -4.5577 & 3.2177
   \end{pmatrix}, \\
\end{align*}
\begin{align*}
M_{CA.AC} &= \begin{pmatrix}
-0.5590 & -1.0094 & 0.0646 & 0.2777 & -0.0107 & 0.0300 \\
    0.0284 & -1.2522 & -1.0645 & -0.6260 & -0.2454 & 0.2924 \\
    0.0136 & 0.5899 & 0.2568 & 0.6980 & 0.4571 & 0.2243 \\
   -0.1603 & 0.7743 & -0.6799 & 0.8766 & -1.3837 & 0.6740 \\
    0.4698 & 0.0295 & 1.3318 & 0.6845 & 2.0995 & -2.0201 \\
   -1.1106 & 0.4685 & -2.3717 & -0.2692 & -3.5420 & 2.4780
   \end{pmatrix}, \\
M_{CA.CA} &= \begin{pmatrix}
-0.6987 & -1.1795 & 0.1817 & 0.3586 & 0.0093 & -0.0782 \\
    0.2866 & -0.6276 & -0.7527 & -0.5620 & 0.1527 & 0.0072 \\
    0.2102 & 1.6846 & 1.3320 & 2.2311 & 3.5163 & 0.6063 \\
   -0.1980 & 1.4241 & -0.9979 & 1.6208 & -2.2944 & 1.3485 \\
    0.5591 & -0.0166 & 1.6256 & -0.3841 & 3.1999 & 0.9511 \\
   -0.3412 & 0.1922 & -1.2105 & -0.6143 & -2.5648 & 2.5404
   \end{pmatrix}, \\
M_{CA.CC} &= \begin{pmatrix}
-0.5165 & -0.8721 & 0.1343 & 0.2651 & 0.0068 & -0.0579 \\
    0.2310 & -0.5057 & -0.6066 & -0.4529 & 0.1230 & 0.0059 \\
    0.1694 & 1.3576 & 1.0734 & 1.7980 & 2.8339 & 0.4887 \\
   -0.1595 & 1.1476 & -0.8042 & 1.3062 & -1.8490 & 1.0867 \\
    0.4506 & -0.0134 & 1.3100 & -0.3096 & 2.5789 & 0.7664 \\
   -0.2750 & 0.1549 & -0.9756 & -0.4950 & -2.0670 & 2.0473
   \end{pmatrix}, \\
M_{CC.AA} &= \begin{pmatrix}
-1.7029 & -1.5549 & 1.1228 & 1.1818 & -0.5513 & -0.0303 \\
    0.2242 & -2.3541 & -2.3229 & -0.7318 & 0.9440 & 0.3782 \\
    0.4897 & 0.2613 & 3.1773 & -0.1562 & 5.9887 & -0.5568 \\
   -0.3006 & 2.0727 & -1.3182 & 2.6104 & -3.3677 & 2.0164 \\
    2.3143 & -1.8329 & 4.3417 & -2.5818 & 5.6867 & -0.4003 \\
   -0.5593 & 0.4349 & -1.6709 & -0.2585 & -3.3011 & 3.5749
   \end{pmatrix}, \\
M_{CC.AC} &= \begin{pmatrix}
-2.1389 & -1.9529 & 1.4103 & 1.4843 & -0.6925 & -0.0380 \\
    0.2145 & -2.2524 & -2.2225 & -0.7003 & 0.9032 & 0.3618 \\
    0.4685 & 0.2499 & 3.0400 & -0.1494 & 5.7300 & -0.5328 \\
   -0.2876 & 1.9832 & -1.2612 & 2.4976 & -3.2221 & 1.9292 \\
    2.2144 & -1.7536 & 4.1542 & -2.4703 & 5.4411 & -0.3829 \\
   -0.5352 & 0.4161 & -1.5988 & -0.2473 & -3.1585 & 3.4205
   \end{pmatrix}, \\
M_{CC.CA} &= \begin{pmatrix}
-2.1848 & -1.8102 & 1.5888 & 1.4272 & -0.8787 & -0.1684 \\
   -0.9631 & 0.3038 & 0.8826 & 0.2771 & -2.6832 & -1.3049 \\
    0.1995 & -2.1397 & 1.4883 & -2.6630 & 3.2584 & -1.9360 \\
    0.9908 & 1.4063 & 4.5037 & 1.9791 & 11.6620 & 0.8665 \\
    2.2134 & -0.2631 & 6.8269 & -1.3631 & 14.3844 & 3.5996 \\
   -0.6009 & 0.2926 & -2.2588 & -0.8628 & -4.4148 & 4.3259
   \end{pmatrix}, \\
M_{CC.CC} &= \begin{pmatrix}
-1.8994 & -1.5738 & 1.3813 & 1.2407 & -0.7640 & -0.1464 \\
   -0.6747 & 0.6749 & 1.4212 & 0.5903 & -0.8909 & -0.2319 \\
    0.2374 & -2.3179 & 1.1304 & -2.8096 & 2.9404 & -2.1454 \\
    0.9548 & 1.5775 & 4.7743 & 2.1199 & 11.8485 & 1.0898 \\
    2.1615 & -0.1375 & 6.9713 & -1.2427 & 14.3987 & 3.6954 \\
   -0.6155 & 0.4415 & -1.9777 & -0.7225 & -4.1196 & 4.4740
   \end{pmatrix}.
\end{align*}

	\noindent In these coordinates the constants $u_x$, $u_y$, $v_x$, $v_y$ for the Poincar\'e map on the domains $AA.AA$, $AA.AC$, $\ldots$, $CC.CC$ satisfy the inequalities
	
	\[
		|u_x| > 1.2641,\ \|u_y\| < 0.3894,\
		\|v_x\| < 0.3894,\ \|v_y\| < 0.4000.
	\]
	
	From now on the reasoning will be identical for both sets $(\Xi/g)\cap {\cal S}$ and $({\cal X}/g)\cap {\cal S}$. We shall then denote any of these sets by $X$. 
	
	Thus for every point $\widehat z$ in $X$ we have a decomposition of the tangent space
	
	\[
		T_{\widehat z}({\cal C}_0\cap {\cal S}) = U(\widehat z)\oplus V(\widehat z),
	\]

	\noindent depending continuously on $\widehat z\in X$. We have that $\alpha(\widehat z): U(\widehat z)\to U(\Phi^{-1}(\widehat z))$, $\beta(\widehat z): V(\widehat z)\to U(\Phi^{-1}(\widehat z))$, $\gamma(\widehat z): U(\widehat z)\to V(\Phi^{-1}(\widehat z))$, and $\delta(\widehat z): V(\widehat z)\to V(\Phi^{-1}(\widehat z))$\ \ are\ \  linear\ \ operators.\ \  In\ \ our\ \  case\ \  we\ \  have $\dim U(\widehat z)=1$ and $\dim V(\widehat z)=5$. This is not important, however, and our considerations below are valid for arbitrary dimensions $\dim U(\widehat z)=m_1$ and $\dim V(\widehat z)=m_2$. Accordingly, we shall consider $\alpha(\widehat z)$ not as a number, but as an invertible linear expanding operator which dilutes by a factor of at least $u_x$. Vectors in the spaces $U(\widehat z)$ and $V(\widehat z)$ will be denoted by $u$ and $v$.
	
	We shall now seek a linear coordinate transformation $\widetilde v = v + {\cal R}(\widehat z)u$ in each of the spaces $T_{\widehat z}({\cal C}_0\cap {\cal S})$ such that the lower left element of the matrix $A^{-1}(\widehat z)$ vanishes. Let us compute ${\cal R}(\widehat z)$. We have
	
	\[
		\left(\begin{array}{c}
			\widetilde u\\ \widetilde v\\
		\end{array}\right) =
		\left(\begin{array}{cc}
			1 & 0\\
			{\cal R}(\widehat z) & 1\\
		\end{array}\right)
		\left(\begin{array}{c}
			u\\ v\\
		\end{array}\right),
	\]

	\noindent hence

	\[
		A^{-1}(\widehat z)\left(\begin{array}{c}
			\widetilde u\\ \widetilde v\\
		\end{array}\right) =
		\left(\begin{array}{cc}
			1 & 0\\
			{\cal R}(\Phi^{-1}(\widehat z)) & 1\\
		\end{array}\right)
		\left(
			\begin{array}{cc}
				\alpha(\widehat z) & \beta(\widehat z)\\
				\gamma(\widehat z) & \delta(\widehat z)\\
			\end{array}
		\right)
		\left(\begin{array}{cc}
			1 & 0\\
			-{\cal R}(\widehat z) & 1\\
		\end{array}\right)
		\left(\begin{array}{c}
			\widetilde u\\ \widetilde v\\
		\end{array}\right),
	\]
	
	\noindent or in other words
	
	\[
		A^{-1}(\widehat z)\left(\begin{array}{c}
			\widetilde u\\ \widetilde v\\
		\end{array}\right) =
		\left(
			\begin{array}{cc}
				\widetilde\alpha(\widehat z) & \widetilde\beta(\widehat z)\\
				\widetilde\gamma(\widehat z) & \widetilde\delta(\widehat z)\\
			\end{array}
		\right)
		\left(\begin{array}{c}
			\widetilde u\\ \widetilde v\\
		\end{array}\right),
	\]

	\noindent where
	
	\[
		\left\{\begin{array}{l}
			\widetilde\alpha(\widehat z) = \alpha(\widehat z) - \beta(\widehat z){\cal R}(\widehat z)\\
			\widetilde\beta(\widehat z) = \beta(\widehat z)\\
			\widetilde\gamma(\widehat z) = \gamma(\widehat z) - \delta(\widehat z){\cal R}(\widehat z) + {\cal R}(\Phi^{-1}(\widehat z))\alpha(\widehat z) - {\cal R}(\Phi^{-1}(\widehat z))\beta(\widehat z){\cal R}(\widehat z)\\
			\widetilde\delta(\widehat z) = \delta(\widehat z) + {\cal R}(\Phi^{-1}(\widehat z))\beta(\widehat z).
		\end{array}\right.
	\]

	We look for a substitution such that $\widetilde\gamma(\widehat z)=0$ for all $\widehat z$ in $X$. We then get
	
	\begin{equation}
	\label{eq:R_new_R_old}
		{\cal R}(\Phi^{-1}(\widehat z)) = -\bigl(\gamma(\widehat z) - \delta(\widehat z){\cal R}(\widehat z)\bigr) \bigl(\alpha(\widehat z) - \beta(\widehat z){\cal R}(\widehat z)\bigr)^{-1}.
	\end{equation}

	Let us look for an operator ${\cal R}(\widehat z)$ of bounded norm not depending on $\widehat z$, i.e., $\|{\cal R}(\widehat z)\|\le\varepsilon$ for some fixed $\varepsilon>0$ and every $\widehat z\in X$. We now precisely choose $\varepsilon>0$. In order for the operator $\alpha(\widehat z) - \beta(\widehat z){\cal R}(\widehat z)$ to be invertible it is sufficient that
	
	\[
		u_x-\varepsilon u_y>0.
	\]

	In order for the map at the right-hand side (\ref{eq:R_new_R_old}) to define an operator of norm not exceeding $\varepsilon$ it is sufficient that
	
	\[
		\frac{v_x+\varepsilon v_y}{u_x-\varepsilon u_y}\le\varepsilon.
	\]

	Let us now demand another condition on $\varepsilon$, namely such that the map at the right-hand side of equation (\ref{eq:R_new_R_old}) is contracting. The differential of this map is given by
	
	\[
		\bigl(\delta(\widehat z) + {\cal R}(\Phi^{-1}(\widehat z))\beta(\widehat z)\bigr)
		d{\cal R}
		\bigl(\alpha(\widehat z) - \beta(\widehat z){\cal R}(\widehat z)\bigr)^{-1}.
	\]

	\noindent Hence a sufficient condition for the map to be contracting is
	
	\[
		\frac{v_y+\varepsilon u_y}{u_x-\varepsilon u_y} < 1.
	\]

	As a last condition on $\varepsilon$ we demand that the operator $A^{-1}(\widehat z)$ preserves the hyperbolic structure, namely such that $\|\widetilde \alpha(\widehat z)^{-1}\|<1$ and $\|\widetilde \delta(\widehat z)\|<1$. Sufficient conditions for this are
	
	\[
		\frac{1}{u_x-\varepsilon u_y}<1\qquad\mbox{and}\qquad v_y+\varepsilon u_y < 1.
	\]

	Gathering all conditions on $\varepsilon>0$, we obtain the system
	
	\begin{equation}
	\label{eq:hyperbolic_system_on_epsilon}
		\left\{
			\begin{array}{l}
				u_y\varepsilon < u_x;\\
				u_y\varepsilon^2 - (u_x-v_y)\varepsilon + v_x \le 0;\\
				u_y\varepsilon < \frac{1}{2}(u_x-v_y);\\
				u_y\varepsilon < u_x - 1;\\
				u_y\varepsilon < 1 - v_y.\\
			\end{array}
		\right.
	\end{equation}

	Let us show that this system has a feasible solution $\varepsilon$. If $u_y=0$, then the system has the trivial solution $\varepsilon\ge\frac{v_x}{u_x-v_y}$. Let us now assume $u_y>0$. The first condition of system (\ref{eq:hyperbolic_system_on_epsilon}) is a consequence of the fourth, and the third one is a consequence of the 4-th and the 5-th one. After the substitution $C=u_y\varepsilon$ we then obtain the equivalent system
	
	\[	
		\left\{\begin{array}{l}
			C^2 - (u_x-v_y) C + u_y v_x < 0;\\
			C < u_x - 1;\\
			C < 1-v_y.\\
		\end{array}\right.
	\]

	\noindent The discriminant of the quadratic trinomial in the first inequality is positive. Indeed, from (\ref{eq:lipschitz_hyperbolic}) we have
	
	\[
		D = (u_x-v_y)^2 - 4 u_yv_x \ge 4(u_x-1)(1-v_y)- 4 u_yv_x>0.
	\]

	\noindent Therefore the system has a feasible solution if the smaller root of the quadratic trinomial is smaller than both $u_x-1$ and $1-v_y$. We obtain
	
	\[
		\frac{u_x-v_y-\sqrt{D}}{2} < u_x-1\qquad\mbox{and}\qquad \frac{u_x-v_y-\sqrt{D}}{2} < 1-v_y,
	\]

	\noindent or equivalently
	
	\[
		\sqrt{D} > 2-u_x-v_y\qquad\mbox{and}\qquad \sqrt{D} > u_x+v_y-2.
	\]
	
	\noindent It remains to show that
	
	\[
		\sqrt{D} > |2-u_x-v_y|\quad \Longleftrightarrow \quad D > (2-u_x-v_y)^2,
	\]

	\noindent which is an immediate consequence of (\ref{eq:lipschitz_hyperbolic}) after expanding the parentheses and collecting similar terms.
	
	Let us now show that the linear map ${\cal R}(\widehat z):U(\widehat z)\to V(\widehat z)$, $\|R(\widehat z)\|\le\varepsilon$ exists and is continuous with respect to $\widehat z$. Consider the space $\mathfrak{C}$ of all continuous maps from $\widehat z\in X$ into the ball with radius $\varepsilon$ in the space $L(U(\widehat z),V(\widehat z))$ of linear maps from $U(\widehat z)$ into $V(\widehat z)$. The space $\mathfrak{C}$ is a complete metric space with respect to the standard supremum norm, because $X$ is compact and $L(U(\widehat z),V(\widehat z))$ is complete. On the space $\mathfrak{C}$ we can define a map $\mathfrak{P}:\mathfrak{C}\to\mathfrak{C}$ by virtue of (\ref{eq:R_new_R_old}), namely
	
	\[
		\mathfrak{P}{\cal R}(\Phi^{-1}(\widehat z)) = -\bigl(\gamma(\widehat z) - \delta(\widehat z){\cal R}(\widehat z)\bigr) \bigl(\alpha(\widehat z) - \beta(\widehat z){\cal R}(\widehat z)\bigr)^{-1}.
	\]

	\noindent The map $\mathfrak{P}$ is well-defined. Firstly because the image of the continuous map ${\cal R}$ is continuous. This follows because the operator $\bigl(\alpha(\widehat z) - \beta(\widehat z){\cal R}(\widehat z)\bigr)$ is bounded away from the set of non-invertible operators by virtue of the choice of $\varepsilon$. Secondly, by virtue of $|{\cal R}(\widehat z)|<\varepsilon$ we have $|\mathfrak{P}{\cal R}(\widehat z)|<\varepsilon$, again by the choice of $\varepsilon$. Moreover, the map $\mathfrak{P}$ is contracting for the same reasons. Hence the map $\mathfrak{P}$ has a unique fixed point ${\cal R}_0\in\mathfrak{C}$, i.e., ${\cal R}_0(\widehat z)$ is the sought linear operator, continuously depending on $\widehat z\in X$.
	
	Thus at the points of the set $X$ we have a continuous family of operators ${\cal R}_0(\widehat z)$ such that $\widetilde\gamma(\widehat z)=0$. Here $\|\widetilde\alpha(\widehat z)^{-1}\| < 1$ and $\|\widetilde\delta(\widehat z)\| < 1$. Replacing the subspaces $U(\widehat z)=\{v=0\}$ and $V(\widehat z)=\{u=0\}$ by $\widetilde U(\widehat z)=\{\widetilde v = v+{\cal R}_0(\widehat z)u=0\}$ and $\widetilde V(\widehat z) = V(\widehat z)$, we obtain another continuous decomposition $T_{\widehat z}({\cal C}_0\cap {\cal S})=\widetilde U(\widehat z)\oplus \widetilde V(\widehat z)$. In this decomposition the matrix $A^{-1}(\widehat z)$ has the form
	
	\[
		A^{-1}(\widehat z) =
		\left(
			\begin{array}{cc}
				\widetilde\alpha(\widehat z) & \beta(\widehat z)\\
				0 & \widetilde\delta(\widehat z)\\
			\end{array}
		\right)\mbox{ for all }\widehat z\mbox{ in }X.
	\]
	
 	Compactness of $X$ implies the existence of constants $\widetilde u_x>1$ and $\widetilde v_y<1$ such that $\|\widetilde\alpha(\widehat z)^{-1}\| < \widetilde u_x^{-1}$ and $\|\widetilde\delta(\widehat z)\| < \widetilde v_y$. In the new decomposition we have $\widetilde u_y=u_y$ and $\widetilde v_x=0$. Hence the conditions of (\ref{eq:lipschitz_hyperbolic}) are satisfied also in the new decomposition. Applying the procedure above again to the operator $A(\widehat z)$ in the new decomposition, we obtain the sought decomposition $T_{\widehat z}({\cal C}_0\cap {\cal S})=\widehat U(\widehat z)\oplus \widehat V(\widehat z)$\footnote{It can be shown that the decomposition $T_{\widehat z}({\cal C}_0\cap {\cal S})=\widehat U(\widehat z)\oplus \widehat V(\widehat z)$ is not only continuous but even H\"older with respect to $\widehat z\in X$, but this is not relevant for the following argument.}.
\end{proof}

\begin{remark}
	Let us explain why the operator ${\cal R}_0(\widehat z)$ is in general not smooth. Equation (\ref{eq:R_new_R_old}) is an equation on functions, it can be written in the simplest form as 
	
	\[
		y(x) = f\Bigr(x,y\bigl(\phi(x)\bigr)\Bigl),
	\]

	\noindent where $x$ is a coordinate, $f(x,y)$ and $\phi(x)$ given smooth functions, and $y(x)$ the sought function. The function $f$ satisfies the condition $\|f'_y\|<1$. We may even assume that $x$ and $y$ are 1-dimensional, this has no influence on the effects in connection with the smoothness of the solution. 
	
	If there were no function $\phi$, i.e., $\phi(x)=x$, then the relation would have the form $y=f(x,y)$. In this case the implicit function theorem guarantees the existence of a smooth solution. Namely, the inclusion $y\in C^1$, which implies also the inclusion $y\in C^\infty$, can be shown by the explicit formula
	
	\[
		y'(x) = \frac{f'_x(x,y(x))}{1-f'_y(x,y(x))}.
	\]

	\noindent Indeed, we obtain $y'\in C^1$, this in turn yields $y\in C^2$ and so on.
	
	If $\phi(x)$ is not the identity in $x$, but rather a map of some compact set on itself, then we can still show the existence and uniqueness of a continuous solution on this compact set, as was done in Lemma \ref{lm:about_tangent_hyperbolic_decomposition}. However, the solution of the functional equation need not to be smooth, just H\"older. For example, the equation $y(x) = \frac{1}{2}y(3x)$ has the solution $y(x)=|x|^{\log_3 2}$, which is not smooth at $x = 0$.
\end{remark}

We now consecutively apply the procedure described in Lemma \ref{lm:about_tangent_hyperbolic_decomposition} to the different decompositions of the tangent space. In this way we transform the $(-N)$-th power $d\Phi_H^{-N}=\Theta_H$ of the differential of the Poincar\'e map from Lemma \ref{lm:block_structure_dPhi} to the most convenient form. The following reasoning is almost identical for both sets $\Xi^0_H(x_0)$ and ${\cal X}^0_H(x_0)$. We hence denote by  $X^0$ any of the sets $\Xi^0_H(x_0)\cap {\cal S}^H$ or ${\cal X}^0_H(x_0)\cap {\cal S}^H$.

In order to bring the differential $d\Theta_H(z)$ to a block-diagonal form, we first restrict our attention to the zero section ${\cal C}_0^H$. Moreover, for the moment we shall not consider the coordinate $\widetilde w$. Denote the subspace of $T_{z}{\cal C}_0^H$ which corresponds to the primary brackets by $T^0_{z}{\cal C}_0^H$, and the subspace which corresponds to the secondary brackets by $T^1_{z}{\cal C}_0^H$. Decompose the subspace $T^0_{z}{\cal C}_0^H$ by virtue of Lemma \ref{lm:about_tangent_hyperbolic_decomposition} and substitution (\ref{eq:phi_psi_x_y_K_main}), such that $T^0_{z}{\cal C}_0^H=\widehat U(z)\oplus \widehat V(z)$ with $\dim\widehat U(z)=1$. The operator ${\cal A}(z)$ becomes block-diagonal and expands $\widehat U(z)$ and contracts $\widehat V(z)$. Now we extend the subspace $\widehat V(z)$ by adjoining the subspace $T^1_{z}{\cal C}_0^H$. In the obtained decomposition the operator $d\Theta_H|_{T^0_{z}{\cal C}_0^H\oplus T^1_{z}{\cal C}_0^H}$ has the form (\ref{eq:A_alpha_beta_gamma_delta}) with $\gamma(z)=0$ and $\|\alpha(z)^{-1}\|<1$. Since $\varrho(z)>2$ we can choose $N$ large enough to enforce $\|\delta(z)\| <  1$. The number $N$ can be chosen independent of $z$ by the compactness of the set $X^0$. Therefore we can apply the procedure described in Lemma \ref{lm:about_tangent_hyperbolic_decomposition} to $d\Theta_H|_{T^0_{z}{\cal C}_0^H\oplus T^1_{z}{\cal C}_0^H}$. As a result, the operator $d\Theta_H|_{T^K_{z}{\cal C}_0^H\oplus \widetilde T^K_{z}{\cal C}_0^H}$ becomes block-diagonal. Repeating this step, we consecutively adjoin the tangent subspaces corresponding to the variables $\widetilde w$ and $\mu$ to the decomposition. The dimension of the subspace $\widehat U(z)$ will increase by exactly 1 at the last step, namely the adjoining of the coordinate $\mu$.

\bigskip

Thus we obtain a decomposition $T_z {\cal C}^H\cap {\cal S}^H = U(z)\oplus V(z)$, continuously depending on $z \in X^0$, such that $\dim U(z)=2$, and the subspace $V(z)$ is horizontal, i.e., $V(z)\subset T_z{\cal C}_0^H$. Moreover, we have

\[
	d\Theta_H(z) =
	\left(\begin{array}{cc}
		\alpha(z) & 0 \\
		0 & \beta(z)\\
	\end{array}\right),
\]

\noindent where $\alpha(z):U(z)\to U(\Theta_H(z))$ and $\beta(z):V(z)\to V(\Theta_H(z))$ depend continuously on $z$, and $\|\alpha(z)^{-1}\|<1$, $\|\beta(z)\|<1$. By the compactness of $\Xi^0_H(x_0)\cap {\cal S}^H$ and ${\cal X}^0_H(x_0)\cap {\cal S}^H$ we may assume that the norms of these operators are bounded away from 1, i.e., $\|\alpha(z)^{-1}\|<\kappa$ and $\|\beta(z)\|<\kappa$ for some $\kappa<1$.

Let us now apply the Hadamard-Perron theorem. To this end we construct a system of local coordinates in the neighbourhood of $X^0$. Since the set $X^0$ is compact, for every $\varepsilon>0$ there exists $\delta>0$ such that for every point $z_0 \in X^0$ the difference $\Theta_H(z)-d\Theta_H(z_0)(z-z_0)$ is bounded by $\varepsilon$ in the $C^1$ metric for every $z$ from a $\delta$-neighbourhood of $z_0$. Let us cover $X^0$ with a finite number of $\delta/2$-neighbourhoods in ${\cal C}^H_0$ with centers $z_1^0$, $\ldots$, $z_M^0\in X^0$. In the $\delta$-neighbourhood of each point $z_m^0$ we introduce local coordinates, defined by the block-diagonal structure of the differential $d\Theta_H(z_m^0)$. If the point $z\in X_0$ lies in a $\delta/2$-neighbourhood of $z_m^0$, then the norm of the difference $d\Theta_H(z)-d\Theta_H(z_m^0)$ is bounded by $\varepsilon$. If $\varepsilon$ is sufficiently small, then the Hadamard-Perron theorem is applicable to the orbit $\Theta_H^k(z)$, $k\in\Z$. By virtue of this theorem, for each point $z_0\in X^0$ there exists a smooth 2-dimensional manifold $W^+(z_0)$ in the $\delta/2$-neighbourhood of $z_0$ in ${\cal C}^H_0\cap{\cal S}^H$, consisting of points $z$ which tend exponentially fast to the images of $z_0$ under iterations of the map $\Theta_H^{-1}$. Since $\Theta_H=\Phi_H^{-N}$ and the set $X^0$ is compact, we have that every point $z\in W^+(z_0)$ tends to the image of $z_0$ also under iterations of the map $\Phi_H$ itself,

\begin{eqnarray*}
	W^+(z_0) = \Bigl\{&&
		z\in{\cal C}^H\cap{\cal S}^H\mbox{ and }\|z-z_0\|<\frac{\delta}{2},\\
		&&\mbox{such that }\|\Phi_H^k(z)-\Phi_H^k(z_0)\|<C\lambda^k,\ k\in\N
	\quad\Bigr\}.\\
\end{eqnarray*}

\noindent Here the constants $C>0$ and $0<\lambda<1$ from the Hadamard-Perron theorem depend only on the choice of $\kappa$ and $\varepsilon$, and hence coincide for all points $z$. By virtue of compactness of the set $X^0$ we may choose constants $C_1>0$ and $0<\lambda_1<1$ such that $\|\Phi_H^k(z)-\Phi_H^k(z_0)\|<C_1\lambda_1^k$ for $k\in\N$.

We shall need some properties of the constructed sets $W^+(z_0)$.

\begin{enumerate}[(a)]
	\item We shall now describe the trajectories of the original Hamiltonian system with Hamiltonian $H$ which emanate from points in the set $W_+(z_0)$, $z_0\in X^0$. To this end we use the vector field $\mu\xi_H(x_0)$ on ${\cal C}_H$ in the neighbourhood of ${\cal C}^0_H$, which has been constructed in Section \ref{subsec:blowing_of_singularity}, cf.~also system (\ref{eq:blowing_general_hamilton_vector_field}). Let without loss of generality $\mu(z)>0$. The trajectory $\gamma(s)$ starting at the point $\gamma(0)=z\in W^+(z_0)$ intersects the surface ${\cal S}^H$ of discontinuity of the right-hand side of the Hamiltonian system a countable number of times as $s\to+\infty$, namely at the points $\Phi_H^k(z)$, $k\in\N$. This trajectory exists and is unique, because the vector field $\mu\xi_H(x_0)$ is transversal to ${\cal S}^H$ in a neighbourhood of $X^0$. The time spent between two consecutive switchings is bounded by $\frac{1}{2}s_{\min}$ and $2s_{\max}$ from below and from above, respectively. Therefore $\gamma(s)$ tends to ${\cal C}_0^H=\{\mu=0\}$ exponentially fast as $s\to +\infty$, i.e.,

	\[
		c_1 e^{-c_2 s} < \mu(\gamma(s)) < c_3 e^{-c_4 s}
	\]

	\noindent for some positive constants $c_i$, $i=1,\ldots,4$. Since $dt=\mu ds$, we have that the integral

	\[
		T(z)=\int_0^{+\infty}\mu(\gamma(s))ds<\infty
	\]

	\noindent is finite, and hence the preimage $\blowingBig^{-1}(\gamma(t))$ of the trajectory hits the point $x_0$ in finite time $T(z)$. In other words, the trajectory $X(t,\widetilde z)$, $z\in{\cal M}$ of the original Hamiltonian system with Hamiltonian $H$ which starts at the point $X(0,\widetilde z)=\widetilde z\in\blowingBig^{-1}(W^+(z_0))$ exists and is unique in the interval $t\in[0,T(\blowingBig(\widetilde z))]$. Moreover, it hits $x_0$ at the moment $T(\blowingBig(\widetilde z))$. In the sequel we will omit the symbol $\blowingBig$ for brevity and write $T(\widetilde z)$ instead of $T(\blowingBig(\widetilde z))$.
	
	Note that the primary brackets ${\cal K}^m(X(t,z))$ satisfy the inequalities claimed in Remark \ref{rm:main_bracket_ineq}.

	\item We shall now describe how two sets $W^+(z_0)$ and $W^+(z_1)$, defined in $\delta/2$-neighbourhoods of $z_0$ and $z_1$, may intersect. Let $U$ denote the intersection of the $\delta/2$-neighbourhoods of  $z_0$ and $z_1$. Then $W^+(z_0)$ and $W^+(z_1)$ either do not intersect, or coincide on $U$. We shall now prove this claim. If their intersection is not empty and contains some point $z\in U$, then the iterates of $z$ tend to both the images of $z_0$ and the images of $z_1$. Therefore the images of $z_0$ and $z_1$ approach each other exponentially fast. Consequently the iterates of an arbitrary point $z\in W^+(z_0)$ approach the iterates of $z_1$ and vice versa. Therefore the sets $W^+(z_i)$, $i=1,2$ coincide on $U$.	
\end{enumerate}

The sets $\Xi_H(x_0)$ and ${\cal X}_H(x_0)$ have the following properties. Their intersection with ${\cal S}^H$ is precisely the union of all fibers $\blowingBig^{-1}(W^+(z_0))$ over the points $z_0$,

\[
	\Xi_H(x_0)\cap {\cal S}^H = \bigcup_{z_0\in \Xi_H^0(x_0)} \blowingBig^{-1}(W^+(z_0)),
\]

\[
	{\cal X}_H(x_0)\cap {\cal S}^H = \bigcup_{z_0\in {\cal X}_H^0(x_0)} \blowingBig^{-1}(W^+(z_0)).
\]

\noindent The sets are obtained by releasing a trajectory of the Hamiltonian system from every point in ${\cal S}^H$ and continuing it up to the next intersection point with ${\cal S}^H$. Therefore properties (I) and (II) in both Theorem \ref{thm:main_chaos_hamilton_equilateral_triangle} and Theorem \ref{thm:main_chaos_hamilton_any_triangle} follow from property (a) of the set $W^+(z_0)$.

The maps $\Psi^H_\Gamma$ and $\Psi^H_{01}$ to the topological Markov chains $\Sigma_\Gamma^+$ and $\Sigma_{01}^+$, which have been announced in (III), are obtained by virtue of the corresponding maps for model problem  (\ref{problem:model}), namely as compositions of: (i) the projection $\pi_W$, which takes $z\in\blowingBig^{-1}(W^+(z_0))$ to the point $z_0\in{\cal C}_0^H$, (ii) identification of ${\cal X}^0_H(x_0)$ and $\Xi_H^0(x_0)$ with ${\cal X}^+/g$ and $\Xi/g$ by means of the map $\Pi_+/g$ from Lemma \ref{lm:Pi_plus}, (iii) application of the corresponding quotient map form model problem (\ref{problem:model}), and (iv) discarding some leading symbols in the obtained topological Markov chain. More precisely, for the map $\Psi_\Gamma^H$ in (ii) we use the map $\Psi_\Gamma^+/g$ from Theorem \ref{thm:model_chaos_equilateral_triangle}, and for the map $\Psi_{01}^H$ we use the map $\Psi_{01}/g$ from Theorem \ref{thm:model_chaos_any_triangle}. Note that the projection $z\mapsto z_0$ itself is not well-defined, since it is possible that $z\in\blowingBig^{-1}(W^+(z_0)\cap W^+(z_1))$. However, by virtue of property (b) the iterates of $z_0$ and $z_1$ approach each other exponentially fast in this case. Therefore there exists a number $K\in\N$ such that in the images $(\Psi_\Gamma^+/g)\bigl((\Pi_+/g)^{-1}(z_i)\bigr)\in\Sigma_\Gamma^+$ or $(\Psi_{01}/g)\bigl((\Pi_+/g)(z_i)\bigr)\in\Sigma_{01}$ differences can appear only at positions with indices smaller than $K$. At positions with numbers larger than $K$ these images have to coincide. By discarding all symbols with indices smaller than $K$ (denote this operator\footnote{For sequences in $\Sigma_\Gamma^+$, which are infinite only to the right, we discard the $K$ first symbols, and for sequences in $\Sigma_{01}$, which are bilaterally infinite, we discard the whole infinite sequence of symbols left of position $K$.} by $\mathfrak{d}$) we hence eliminate the ambiguity in the map $\pi_W$. Therefore

\[
	\Psi^H_\Gamma = \mathfrak{d} \circ (\Psi_\Gamma^+/g) \circ (\Pi_+/g)^{-1} \circ \pi_W
\]

\[
	\Psi^H_{01} = \mathfrak{d} \circ (\Psi_{01}/g) \circ (\Pi_+/g)^{-1} \circ \pi_W
\]

\noindent This proves claim (III) for both Theorems \ref{thm:main_chaos_hamilton_equilateral_triangle} and \ref{thm:main_chaos_hamilton_any_triangle}.

Claim (IV) of Theorem \ref{thm:main_chaos_hamilton_any_triangle} is proven by replacing the set $X_0$ by a finite set, consisting of the intersections of $R_{ijk}/g$ or $Q_i/g$ with ${\cal S}$ (or more precisely, of their images under the map $\Pi_+/g$). In the first case the set $X_0$ consists of three points, in the second case of four points. In the case of the three-link chain all eigenvalues of the operator $A^3$ from (\ref{eq:defn_A_rho}) have a modulus larger than 1 (see Section \ref{subsec:periodic_2_4_6_fuller}). Hence we obtain in this case that the manifold $W^+$ is 1-dimensional. For the case of the four-link chain the operator $A^4$ has exactly one eigenvalue with modulus smaller 1 (see Section \ref{subsec:periodic_2_4_6_fuller}). Hence in this case the manifold $W^+$ is 2-dimensional.

Item (IV) from Theorem \ref{thm:main_chaos_hamilton_equilateral_triangle} on the bounds on the dimension of ${\cal X}_H(x_0)$ follows from the corresponding item in Theorem \ref{thm:model_chaos_equilateral_triangle}. Indeed, for some sufficiently small $\mu_0$ the map $\blowingBig^{-1}$ is a diffeomorphism on ${\cal C}_0^H\times\{0<\mu<\mu_0\}$ and does not change dimensions. Secondly, if $dG(x_0)=0$, then the dimensions (Hausdorff or Kolmogorov) of the sets ${\cal X}^0_H(x_0)$ and ${\cal X}/g$ coincide, because in this case the map $\Pi_+$ is locally Lipschitz by virtue of Lemma \ref{lm:Pi_plus}. It remains to note that the sets ${\cal X}_H(x_0)$ and ${\cal X}^+$ are obtained from the sets ${\cal X}^0_H(x_0)\cap{\cal S}^H$ and $({\cal X}\cap {\cal S})/g$ by a similar procedure, namely two-dimensional stable submanifolds are constructed and trajectories of the system released.

Item (V) of Theorem \ref{thm:main_chaos_hamilton_equilateral_triangle} is identical with the corresponding item in Theorem \ref{thm:model_chaos_equilateral_triangle}.

Items (VI) of Theorem \ref{thm:main_chaos_hamilton_equilateral_triangle} and (V) of Theorem \ref{thm:main_chaos_hamilton_any_triangle} can be obtained by considering the Lipschitz surface $M_-$ instead of the Lipschitz surface $M_+$ of optimal trajectories. Recall that $M_-$ consists of trajectories which leave the origin of Hamiltonian system (\ref{eq:model_pmp_system}) determined by the Pontryagin maximum principle for model problem (\ref{problem:model}). This completes the proof of Theorems \ref{thm:main_chaos_hamilton_equilateral_triangle} and \ref{thm:main_chaos_hamilton_any_triangle}.

\begin{flushright}

$\square$

\end{flushright}

Let us now prove the claim of Remark \ref{rm:main_bracket_ineq}. The upper bound has been already obtained in Lemma \ref{lm:descending_system_general_nongeneral_bracket_order}. The lower bound is equivalent to the assertion that on every trajectory $X(t)$ from the set ${\cal X}_H(x_0)$ or $\Xi_H(x_0)$ we have the bound

\[
	\mu(X(t)) \ge c''(T-t),
\]

\noindent where $c''>0$ is some constant, and $T$ is the time instant of hitting the strange point $x_0$.

Let us transfer the trajectory $X(t)$ to the cylinder ${\cal C}_H$ by means of the blow-up map $\blowingBig$ and perform the substitution $ds=\frac{1}{\mu}dt$ of the time parameter. Then we obtain a trajectory $\widetilde X(s) = \blowing(X(t(s))$ of system (\ref{eq:blowing_general_hamilton_vector_field_s}). Hence we have to show that

\[
	\mu(\widetilde X(s)) \ge T-t(s) =
	c'' \int_s^{+\infty} \mu(\widetilde X(\sigma))\,d\sigma.
\]

Denote by $\widetilde X^0(s)$ the trajectory on ${\cal C}_0^H$ which emanates from ${\cal X}_H(x_0)$ or $\Xi^0_H(x_0)$ which the trajectory $\widetilde X(s)$ converges to as $s\to+\infty$. As we have mentioned above, the trajectory $\widetilde X^0(s)$ is also a trajectory of system (\ref{eq:blowing_general_linearized_hamilton_vector_field}), since we have $\mu(\widetilde X^0(s))=0$. Consider also some trajectory $\widetilde X^1(s)$ of system (\ref{eq:blowing_general_linearized_hamilton_vector_field}) which differs from $\widetilde X^0(s)$ only in the coordinate $\mu$, say $\mu(\widetilde X^1(s))>0$, such that

\[
	\left\{\begin{array}{rcl}
		\widetilde K(\widetilde X^1(s)) &=& \widetilde K(\widetilde X^0(s)),\\
		\widetilde w(\widetilde X^1(s)) &=& \widetilde w(\widetilde X^0(s)),\\
		\dds \mu(\widetilde X^1(s))     &=& \mu(\widetilde X^1(s))\Upsilon(\widetilde X^1(s)) = \mu(\widetilde X^1(s))\Upsilon(\widetilde X^0(s)),\\
		\mu(\widetilde X^1(0)) &=& \mu(\widetilde X(0)).\\
	\end{array}\right.
\]

Since the trajectory $\widetilde X^1(s)$ is the image of an optimal trajectory for model problem (\ref{problem:model}) under the map $\blowingBig\circ\Pi_+$, we have by virtue of Theorem \ref{thm:model_problem_bellman} the following inequality:

\[
	\mu(\widetilde X^1(s)) \ge
	c_0 \int_s^{+\infty} \mu(\widetilde X^1(\sigma))\,d\sigma
\]

\noindent on this trajectory for some constant $c_0>0$.

Let us now show that the values of $\mu$ do not differ too much on both trajectories $\widetilde X(s)$ and $\widetilde X^1(s)$, in the sense that 

\[
	\left\{\begin{array}{rcl}
	        \dds \mu(\widetilde X(s))     &=& \mu(\widetilde X(s))\Upsilon(\widetilde X(s)),\\
	        \dds \mu(\widetilde X^1(s))     &=& \mu(\widetilde X^1(s))\Upsilon(\widetilde X^1(s)),
	\end{array}\right.
\]

\noindent and

\[
	\dds \ln \frac{\mu(\widetilde X(s))}{\mu(\widetilde X^1(s))} =
	\Upsilon(\widetilde X(s)) - \Upsilon(\widetilde X^1(s)).
\]

\noindent Since both $\widetilde X(s)$ and $\widetilde X^1(s)$ tend to $\widetilde X^0(s)$ exponentially fast, the difference at the right-hand side in the last equation above also tends to 0 exponentially fast. Therefore we have for some constants $c_1,c_2>0$ that

\[
	-c_1e^{-c_2s}\le
	\dds \ln \frac{\mu(\widetilde X(s))}{\mu(\widetilde X^1(s))}
	\le c_1e^{-c_2s}.
\]

\noindent Taking into account that $\ln \frac{\mu(\widetilde X(0))}{\mu(\widetilde X^1(0))}=0$, we immediately obtain that for $c_3=e^{c_1/c_2}$ we have

\[
	\frac{1}{c_3}\mu(\widetilde X^1(s))\le
	\mu(\widetilde X(s)) \le
	c_3\mu(\widetilde X^1(s)).
\]

\noindent Thus

\[
	\mu(\widetilde X(s)) \ge \frac{1}{c_3}\mu(\widetilde X^1(s)) \ge
	\frac{c_0}{c_3} \int_s^{+\infty} \mu(\widetilde X^1(\sigma))\,d\sigma \ge
	\frac{c_0}{c_3^2} \int_s^{+\infty} \mu(\widetilde X(\sigma))\,d\sigma,
\]

\noindent what is what we had to show.

\begin{remark}
	By virtue of Remark \ref{rm:main_bracket_ineq} we can control how the sets $\Xi_H(x_0)$ and ${\cal X}_H(x_0)$ enter the strange point $x_0$. In geometric terms the sets $\Xi_H(x_0)$ and ${\cal X}_H(x_0)$ are tangent (in the sense of the order in $\mu$) to the cone $\blowingBig^{-1}\bigl({\cal D}_0\times\{\mu\in \R\}\bigr)$ .
\end{remark}

\begin{remark}
	By Definition \ref{defn:strange_point} a strange point can appear only if the number of degrees of freedom of the system is not smaller than 16. Indeed, in Definition \ref{defn:strange_point} we demand the linear independence of the differentials of the functions $F_r$, $(\ad F_i)F_r$ etc. up to the 4-th order. The minimal dimension where such a situation can occur can be easily computed by means of the dimension of the free nilpotent graded Lie algebra of depth 4 and with three generators ($F_3$ has to be excluded, since $F_3\equiv -F_1-F_2$). Namely, the growth vector of such an algebra can be calculated via Halls words, it equals $(3,6,14,32)$. Since we do not use $dF_0$ at the first level, the minimally possible dimension of $\cal M$ is $31$. But $\dim {\cal M}$ must be even, hence $\dim{\cal M}\ge 32$ and the system has at least 16 degrees of freedom. However, the conditions for the definition of a strange point are excessively strong, and that is why the number of degrees of freedom is so large. In fact, the conditions can be substantially weakened. Hamiltonian system (\ref{eq:model_pmp_system}), which is determined by the Pontryagin maximum principle for model problem (\ref{problem:model}), is a low-dimensional example (with 4 degrees of freedom) of a Hamiltonian system which exhibits the phenomenon described in Theorem \ref{thm:main_chaos_hamilton_equilateral_triangle}.
\end{remark}

\begin{remark}
\label{rm:structural_stability_equilateral_triangle}
	The appearance of strange points in large dimensions is not avoidable and they cannot be removed by a small perturbation of the Hamiltonians $H_i$, $i=1,2,3$ if we assume that the gradients of the commutators at $x_0$ are linearly independent not only up to the 4-th order, as demanded in Definition \ref{defn:strange_point}, but up to 5-th order. In this case the set $\cal ST$ of strange points is a smooth manifold in the neighbourhood of $x_0$ and its codimension can be computed by means of the dimension of the free nilpotent graded Lie algebra of depth 5 and with three generators. The growth vector of such an algebra is given by $(3,6,14,32,80)$. The codimension of $\cal ST$ is smaller than 80 by a number of $4$. Here 3 dimensions come from the symmetric form $B_{rr'}$ and one from the function $F_0$ at the first level. Hence we get 
	
	\[
		\codim {\cal ST} = 76.
	\]
\end{remark}

\bibliographystyle{plain}
\bibliography{all_eng}

\begin{thebibliography}{10}

\bibitem{Falconer}
K.~Falconer.
\newblock {\em Fractal Geometry}.
\newblock Mathematical Foundations and Applications, West Sussex, second
  edition edition, 2003.

\bibitem{Filippov}
A.~F. Filippov.
\newblock {\em Differential equations with discontinuous right-hand side (in
  russian)}.
\newblock Nauka, Moscow, 1985.

\bibitem{HLZ_Chaos}
R.~Hildebrand, L.V. Lokutsievskiy, and M.I. Zelikin.
\newblock Generic fractal structure of finite parts of trajectories of
  piecewise smooth hamiltonian systems.
\newblock {\em Russian Journal of Mathematical Physics}, 20(1):25--32, 2013.

\bibitem{Katok}
A.~Katok and B.~Hasselblatt.
\newblock {\em Introduction to the Modern Theory of Dynamical Systems}.
\newblock Cambridge University Press, Cambridge, 1996.

\bibitem{KelleyKoppMoyer}
H.J. Kelley, R.E. Kopp, and H.G. Moyer.
\newblock Singular extremals.
\newblock {\em Topics in Optimization, Academic Press, New York}, pages
  63--101, 1967.

\bibitem{Kupka}
I.~Kupka.
\newblock Fuller's phenomena.
\newblock {\em Progr. Systems Control Theory. Birkhauser, Boston.}, pages
  129--142, 1990.

\bibitem{Lewis}
R.M. Lewis.
\newblock Defenitions of order and junction condition in singular control
  problems.
\newblock {\em SIAM J. Control and Optimization}, 18(1):21--32, 1980.

\bibitem{LokutHSF}
L.V. Lokutsievskii.
\newblock The hamiltonian property of the flow of singular trajectories.
\newblock {\em Sb. Math.}, 205(3):432--458, 2014.

\bibitem{LokutMnogogr}
L.V. Lokutsievskii.
\newblock Singular regimes in controlled systems with multidimensional control
  in a polyhedron.
\newblock {\em Izv. RAN. Ser. Mat.}, 78(5):1006--1027, 2014.

\bibitem{LokutLipsChaos}
L.V. Lokutsievskii, M.I. Zelikin, and R.~Hildebrand.
\newblock Fractal structure of hyperbolic lipschitzian dynamical systems.
\newblock {\em Russian Journal of Mathematical Physics}, 19(1):27--44, 2012.

\bibitem{Lokut_HSF}
L.V. Lokutsievskiy.
\newblock Generic structure of the lagrangian manifold in chattering problems.
\newblock {\em Sbornik Mathematics}, 205(3):432--458, 2014.

\bibitem{Marchal73}
C.~Marchal.
\newblock Chattering arcs and chattering controls.
\newblock {\em J. Optimiz. Theory App.}, 11(5):441--468, 1973.

\bibitem{PBGM62}
L.S. Pontryagin, V.G. Boltyanski, R.V. Gamkrelidze, and E.F. Mishechenko.
\newblock {\em The mathematical theory of optimal processes}.
\newblock Wiley-Interscience, New York, London: Wiley, 1962.

\bibitem{ZelikinBorisov}
M.I. Zelikin and V.F. Borisov.
\newblock {\em Theory of Chattering Control with Applications to Astronautics,
  Robotics, Economics, and Engineering}.
\newblock Birkhauser, Boston, 1994.

\bibitem{HLZOrderTorus}
M.I. Zelikin, L.V. Lokutsievskiy, and R.~Hildebrand.
\newblock Geometry of neighborhoods of singular trajectories in problems with
  multidimensional control.
\newblock {\em Proceedings of the Steklov Institute of Mathematics},
  277:67--83, 2012.

\bibitem{HLZ_Chaos_DAN}
M.I. Zelikin, L.V. Lokutsievskiy, and R.~Hildebrand.
\newblock Stochastic dynamics of lie algebras of poisson brackets in
  neighborhoods of nonsmoothness points of hamiltonians.
\newblock {\em Doklady mathematics}, 87(3):259--263, 2013.

\bibitem{ZMHBasic}
M.I. Zelikin, N.B. Melnikov, and R.~Hildebrand.
\newblock Topological structure of a typical fibre of optimal synthesis for
  chattering problems.
\newblock {\em P. Steklov Inst. Math.}, 233:116--142, 2001.

\end{thebibliography}

\end{document}